\documentclass[12pt]{amsart}
\usepackage{amssymb}
\usepackage{amsmath}
\usepackage{amsthm}
\usepackage{graphicx}
\usepackage{enumerate}
\usepackage{tikz}
\usetikzlibrary{arrows}
\usetikzlibrary{arrows.meta}
\tikzset{>={Latex[width=1.2mm,length=1.7mm]}}
\usepackage{tabularx}
\usepackage{mathdots}
\usepackage{mathrsfs}
\usepackage{tabmac_avi}

\newtheorem{thm}{Theorem}[section]
\newtheorem{prop}[thm]{Proposition}
\newtheorem{cor}[thm]{Corollary}
\newtheorem{lem}[thm]{Lemma}
\newtheorem{obs}[thm]{Observation}
\newtheorem{conj}[thm]{Conjecture}
\newtheorem{prob}[thm]{Problem}
\theoremstyle{definition}
\newtheorem{defn}[thm]{Definition}
\newtheorem{alg}[thm]{Algorithm}

\numberwithin{equation}{section}

\newcommand{\sn}{\mathfrak{S}_n}

\newcommand{\mfs}[1]{\mathfrak{S}_{#1}}
\newcommand{\snn}{\mfs{\smash{[\ol n, n]}}}

\newcommand{\bn}{\mathfrak{B}_n}

\newcommand{\mfb}[1]{\mathfrak{B}_{#1}}

\newcommand{\zsn}{\mathbb{Z}[\sn]}
\newcommand{\qsn}{\mathbb{Q}[\sn]}
\newcommand{\zbn}{\mathbb{Z}[\bn]}
\newcommand{\qbn}{\mathbb{Q}[\bn]}

\newcommand{\zx}{\mathbb{Z}[x]}

\newcommand{\zqq}{\mathbb{Z}[q, q^{-1}]}

\newcommand{\cz}{\mathbb{C}[z]}
\newcommand{\zz}{\mathbb{Z}[z]}

\newcommand{\csn}{\mathbb{C}[\sn]}
\newcommand{\bitab}{\mathcal B}
\newcommand{\tab}{\mathcal U}
\newcommand{\bctab}{\mathcal U^{\msfBC}}
\newcommand{\atab}{\mathcal U^{\msfA}}

\newcommand{\qp}[2]{q^{\frac{#1}{#2}}}
\newcommand{\qm}[2]{q^{\negthinspace\Bar\,\frac{#1}{#2}}}

\newcommand{\qew}{\smash{\qp{\ell(w)}2}}

\newcommand{\qiew}{\qm{\ell(w)}2}

\newcommand{\ol}[1]{\overline{#1}}

\newcommand{\hnq}{H_n(q)}
\newcommand{\hbnq}{H_n^{\msfBC}(q)}
\newcommand{\hbq}[1]{H_{#1}^{\msfBC}(q)}
\newcommand{\hbnQ}{H_n^C(q,Q)}

\newcommand{\hlq}{H_\lambda(q)}

\newcommand{\atwonq}{\mathcal{A}(2n,q)}

\newcommand{\trsp}{\mathcal{T}}

\newcommand{\coh}{\mathrm{H}}
\newcommand{\icoh}{\mathrm{IH}}
\newcommand{\HH}{\mathcal{H}}

\newcommand{\wtc}[2]{\widetilde{C}_{#1}(#2)}
\newcommand{\btc}[2]{\widetilde{C}_{#1}^{\msfBC}(#2)}

\newcommand{\imm}[1]{\mathrm{Imm}_{#1}}

\newcommand{\bnimm}[1]{\mathrm{Imm}_{#1}^{\bn}}
\newcommand{\simm}[2]{\mathrm{Imm}_{#2}^{\smash{\mfs{#1}}}}

\newcommand{\sumsb}[1]{\sum_{\substack{#1}}}  
\newcommand{\prodsb}[1]{\prod_{\substack{#1}}}  
\newcommand{\inv}{\textsc{inv}}

\newcommand{\defeq}{:=} 
\newcommand{\dfct}{\mathrm{dfct}}

\newcommand{\sct}{\mathrm{sct}}
\newcommand{\des}{\mathrm{des}}
\newcommand{\exc}{\mathrm{exc}}
\newcommand{\stat}{\mathrm{stat}}
\newcommand{\frobch}{\mathrm{ch}}
\newcommand{\npfrobch}{\mathrm{nch}}
\newcommand{\pfrobch}{\mathrm{pch}}
\newcommand{\spn}{\mathrm{span}}
\newcommand{\sgn}{\mathrm{sgn}}
\newcommand{\triv}{\mathrm{triv}}
\newcommand{\wgt}{\mathrm{wgt}}
\newcommand{\src}{\mathrm{src}}
\newcommand{\snk}{\mathrm{snk}}
\newcommand{\type}{\mathrm{type}}
\newcommand{\ctype}{\mathrm{ctype}}
\newcommand{\inc}{\mathrm{inc}}

\newcommand{\pavoiding}{$3412$-avoiding, $4231$-avoiding }
\newcommand{\avoidsp}{avoids the patterns $3412$ and $4231${}}
\newcommand{\avoidp}{avoid the patterns $3412$ and $4231${}}
\newcommand{\avoidingp}{avoiding the patterns $3412$ and $4231${}}
\newcommand{\theps}{the patterns $3412$ and $4231${}}

\newcommand{\avoidssignedp}{avoids the signed patterns $1\ol2$,  $\ol21$, $\ol2\ol1$, $312$, $3\ol12${}}
\newcommand{\avoidingsignedp}{avoiding the signed patterns $1\ol2$, $\ol21$, $\ol2\ol1$, $312$, $3\ol12${}}
  
\newcommand{\ssec}[1]{\subsection{#1}{$\negthinspace$}}

\newcommand{\tr}{{\negthickspace \top \negthickspace}}
\newcommand{\nhtnsp}{\kern-0.08333em}
\newcommand{\ntnsp}{\negthinspace}
\newcommand{\ntksp}{\negthickspace}
\newcommand{\nTksp}{\negthickspace\negthickspace}

\newcommand{\nTtksp}{\negthickspace\negthickspace\negthickspace}
\newcommand{\nTTksp}{\negthickspace\negthickspace\negthickspace
	\negthickspace}

\newcommand{\bp}{\begin{prob}}
\newcommand{\ep}{\end{prob}}

\newcommand{\oqspnc}{{\mathcal O}_q(\mathrm{SP}_{2n} (\mathbb C))}

\newcommand{\uqspnc}{U_q(\mathfrak{sp}_{2n}(\mathbb{C}))}
\newcommand{\mat}[1]{\mathrm{Mat}_{#1 \times #1}}

\newcommand{\fl}[1]{\mathcal{F}_{#1}}
\newcommand{\PiBC}{\Pi^{\msfBC}}

\newcommand{\perm}{\mathrm{per}}
\newcommand{\rnk}{\mathrm{rank}}
\newcommand{\sort}{\mathrm{sort}}
\newcommand{\diag}{\mathrm{diag}}

\newcommand{\hess}{\mathrm{Hess}}
\newcommand{\ahess}{\mathrm{Hess^{\msfA}}}
\newcommand{\bhess}{\mathrm{Hess}^{\msfBB}}
\newcommand{\chess}{\mathrm{Hess}^{\msfC}}
\newcommand{\permmon}[2]{#1_{1,#2_1} \ntnsp\cdots #1_{n,#2_n}}
\newcommand{\bpermmon}[2]
           {#1_{\smash{\ol n,#2_{\ol n}}} \ntnsp\cdots #1_{\smash{\ol 1,#2_{\ol 1}}}
           #1_{1,#2_{1}} \ntnsp\cdots #1_{n,#2_n} }

\newcommand{\ssm}{\smallsetminus}
\newcommand{\multiu}{\Cup}
  
\newcommand{\bfx}{\mathbf x}

\newcommand{\circd}[1]{\raisebox{-9pt}{\textcircled{\raisebox{-.9pt}{#1}}}}

\newcommand{\sreg}{\left[ \begin{smallmatrix} 0 & 1 \\ 1 & 0 \end{smallmatrix}\right]}
\newcommand{\smallj}{\left[ \begin{smallmatrix} 0 & \phantom{\int} & 1 \\ & \smash{\iddots} & \phantom{T} \\ 1 & & 0 \end{smallmatrix}\right]}
\newcommand{\abdiag}[1]{\left[ \begin{smallmatrix} 1 & #1 \\ 0 & 1 \end{smallmatrix}\right]}
\newcommand{\bediag}[1]{\left[ \begin{smallmatrix} 1 & 0 \\ #1 & 1 \end{smallmatrix}\right]}
\newcommand{\upparrow}{\big \uparrow \nTksp \phantom{\uparrow}}

\newcommand{\precdot}{\prec \negthickspace \negthinspace \cdot\ }
\newcommand{\net}[2]{\mathcal F^{\mathsf{#1}}(#2)}
\newcommand{\snet}[2]{\mathcal S^{\mathsf{#1}}(#2)}
\newcommand{\znet}[2]{\mathcal S^{\mathsf{#1}}_{\mathrm Z}(#2)}
\newcommand{\dnet}[2]{\mathcal S^{\mathsf{#1}}_{\mathrm D}(#2)}

\newcommand{\msfA}{\mathsf{A}}

\newcommand{\msfBB}{\mathsf{B}}
\newcommand{\sfC}{\textsf{C}}
\newcommand{\msfC}{\mathsf{C}}

\newcommand{\msfBC}{\mathsf{BC}}
\newcommand{\YBCq}{Y_q^\mathsf{BC}}
\newcommand{\vertex}[1]{#1}

\newcommand{\phsum}{\phantom{\sum_A^Z}}
\newcommand{\phm}{\phantom M}
\newcommand{\phn}{\phantom{ni}}

\newcommand{\schub}[1]{\Omega_{#1}}
\newcommand{\Aschub}[1]{\Omega_{#1}^{\msfA}}
\newcommand{\Bschub}[1]{\Omega_{#1}^{\msfBB}}
\newcommand{\Cschub}[1]{\Omega_{#1}^{\msfC}}

\def\hhhsp{\def\baselinestretch{0.125}\large\normalsize}

\def\ssp{\def\baselinestretch{1.0}\large\normalsize}

\begin{document}
\author{Mark Skandera}
\title[Hyperoctahedral group characters and graph coloring]{Hyperoctahedral group characters and a type-BC analog of graph coloring}

\bibliographystyle{dart}

\date{\today}

\begin{abstract}
  We state combinatorial formulas for hyperoctahedral group ($\bn$)
  character evaluations of the form $\chi(\btc w1)$,
  where
  $\btc w1 \in \zbn$ is a type-$\msfBC$
  Kazhdan--Lusztig basis element,
  with $w \in \bn$ corresponding to simultaneously smooth
  type-$\msfBB$ and $\msfC$ Schubert varieties.
  We also extend the definition of symmetric group codominance
  to elements of $\bn$ and show that
  for each
  element $w \in \bn$ as above,
  there exists a $\msfBC$-codominant
  element $v \in \bn$ satisfying $\chi(\btc w1) = \chi(\btc v1)$
  for all $\bn$-characters $\chi$.
  Combinatorial structures and maps appearing in these formulas
  are type-$\msfBC$ extensions of planar networks, unit interval orders,
  indifference graphs, poset tableaux, and colorings.
  Using the ring of type-$\msfBC$ symmetric functions,
  we introduce natural generating functions $Y(\btc w1)$ for the above
  evaluations.
  These provide a new type-$\msfBC$ analog of Stanley's
  chromatic symmetric functions
  [{\em Adv.\ Math.} \textbf{111} (1995) pp.\ 166--194].
\end{abstract}

\maketitle

\section{Introduction}\label{s:intro}

Let $W$ be a Coxeter group, $H = H(W)$ its Hecke algebra, and $\trsp(H)$ the
space of Hecke algebra {\em traces},
linear functionals $\theta_q: H \rightarrow \zqq$ satisfying
$\theta_q(DD') = \theta_q(D'D)$ for all $D, D' \in H$.
Included in $\trsp(H)$ are the $H$-characters,
which encode much of the structure of $H$ in a condensed form.
Since traces are linear,
one might hope to solve the
following problem for particular bases
$\mathcal D = \{ D_w \,|\, w \in W \}$ of $H$
and $\Theta = \{ \theta_q^{(i)} \,|\, i = 1,\dotsc,p \}$
of $\trsp(H)$.
\bp\label{p:evaltrace}
Find combinatorial formulas for all evaluations
$\{ \theta_q^{(i)} (D_w) \,|\,
\theta_q^{(i)} \in \Theta, D_w \in \mathcal D\}$.
\ep

Unfortunately, trace evaluation is not always easy,
even in type $\msfA$,
when $W$ is the symmetric group $\sn$ with Hecke algebra $H = \hnq$.
(See, e.g., \cite[\S 1]{CHSSkanEKL}.)
Type-$\msfA$ solutions were given in
\cite{CSkanTNNChar}, \cite{KLSBasesQMBIndSgn},
using the induced sign character basis of $\trsp(\hnq)$,
and bases consisting of products of simple elements of
the (modified, signless) Kazhdan--Lusztig basis
$\{ \wtc wq \,|\, w \in \sn \}$ of $\hnq$.
It would be interesting to solve Problem~\ref{p:evaltrace}
for other pairs of type-$\msfA$ bases as well, as these evaluations are related to
facts and conjectures
concerning nonnegativity,
graph coloring, and Hessenberg varieties, e.g.,
\cite[Lem.\,1.1]{HaimanHecke},
\cite[Conj.\,2.1]{HaimanHecke},  
\cite[Conj.\,4.9]{SWachsChromQ},
\cite[Conj.\,5.5]{StanStemIJT}.

Partial type-$\msfA$ solutions to Problem~\ref{p:evaltrace}
were given in \cite{CHSSkanEKL}, \cite{SkanCCS}
for various bases of $\trsp(\hnq)$, and
the subset
\begin{equation}\label{eq:introsmooth}
  \{ \wtc wq \,|\, w \in \sn \text{ \avoidsp} \}
\end{equation}
of the Kazhdan--Lusztig basis of $\hnq$.
By \cite[Thm.\,4.3]{SkanNNDCB},
we have that for $w$ \avoidingp, there exists a planar network $F = F(w)$
which serves as a combinatorial interpretation for $\wtc wq$.
By \cite[Thm.\,7.4]{CHSSkanEKL},
there also exist
a poset $P = P(w)$
and graph $G = G(w)$
such that
evaluations $\theta_q( \wtc wq )$ may be computed combinatorially
by
\begin{enumerate}
\item filling Young diagrams with paths in $F$,
\item filling Young diagrams with elements of $P$,
\item coloring vertices of $G$, 
\item orienting edges of $G$,
\end{enumerate}
while obeying certain rules in each case.
(See also \cite{AthanPSE}, \cite{SWachsChromQF}, \cite{StanSymm}.)
While (\ref{eq:introsmooth}) is only a subset of the Kazhdan--Lusztig basis
of $\hnq$, it is conjectured~\cite[Conj.\,1.9]{ANigroUpdate},
\cite[Conj.\.3.1]{HaimanHecke} that an even smaller subset
\begin{equation}\label{eq:introcodom}
  \{ \wtc wq \,|\, w \in \sn
  \text{ is codominant, i.e., avoids the pattern } 312 \}
\end{equation}
explains trace evaluations at the entire Kazhdan--Lusztig basis.
It is known~\cite[Thm.\,4.6]{CHSSkanEKL}
that for each element $\wtc wq$ of (\ref{eq:introsmooth})
there exists an element $\wtc vq$ of (\ref{eq:introcodom})
with the property that $P(w) \cong P(v)$
and therefore that $\theta_q(\wtc wq) = \theta_q(\wtc vq)$ for all traces
$\theta_q \in \trsp(\hnq)$.

One could also answer Problem~\ref{p:evaltrace}
from the point of view of symmetric functions.
Let $\Lambda_n$ be the $\mathbb Z$-module of homogeneous, degree-$n$
symmetric functions.
Since the ranks of $\Lambda_n$ and $\trsp(\hnq)$ are
equal,
it is possible to define
a generating function in $\Lambda_n$
for evaluations of traces at
any fixed element $D \in \hnq$.
Following \cite[\S 2]{SkanCCS},
we use
the induced sign character basis $\{\epsilon_q^\lambda\}$ of $\trsp(\hnq)$
and monomial symmetric function basis $\{ m_\lambda \}$ of $\Lambda_n$
to define
\begin{equation*}
  Y_q(D) \defeq \sum_\lambda
  \epsilon_q^\lambda (D) m_\lambda \in \zqq \otimes \Lambda_n.
\end{equation*}
A certain pairing of six natural bases of $\trsp(\hnq)$
and six natural bases of $\Lambda_n$ then guarantees that for each pair
$(\{ \theta_q^\lambda \}, \{ g_\lambda \})$,
we have
\begin{equation*}
  Y_q(D) = \sum_\lambda \theta_q^\lambda(D) g_\lambda.
\end{equation*}
Thus $Y_q(D)$ is in fact a generating function for the evaluation
of all elements of these six trace bases
at $D$. (See, e.g., \cite[Prop.\,2.1]{SkanCCS}.)
Conveniently, the combinatorial computations mentioned above also
guarantee
that a certain
{\em chromatic (quasi-)symmetric function}
$X_{G,q}$, defined in terms of the proper colorings of $G$~\cite{SWachsChromQ},
satisfies $X_{G,q} = Y_q(\wtc wq)$~\cite[Thm.\,7.4]{CHSSkanEKL}.
Thus for $w \in \sn$ \avoidingp,
the graph $G$ essentially encodes all trace evaluations of the
form $\theta_q^\lambda(\wtc wq)$ for $\{ \theta_q^\lambda \}$ one of the six natural
bases of $\trsp(\hnq)$.

Some of the above results
from \cite{CHSSkanEKL}, \cite{SkanNNDCB}, \cite{SkanCCS}
have type-$\msfBC$ analogs,
i.e., extensions to the hyperoctahedral group $\bn$
and its Hecke algebra $\hbnq$.
In Sections~\ref{s:s2nbn} -- \ref{s:trace},
we present these algebras, their Kazhdan--Lusztig bases,
and their trace spaces.
In Section~\ref{s:planarnet}, we define type-$\msfBC$ analogs
of type-$\msfA$ planar networks, and use these to
graphically represent the subset
\begin{equation}\label{eq:BCintrosmooth}
  \{ \btc w1 \,|\, w \in \bn \text{ \avoidsp} \}
\end{equation}
of the Kazhdan--Lusztig basis of $\zbn$.
In Section~\ref{s:immtnn} we use immanants and total nonnegativity
to interpret trace evaluations at (\ref{eq:BCintrosmooth})
in terms of paths in the type-$\msfBC$ planar networks. 
In Sections~\ref{s:uio} -- \ref{s:incgraph}, we define type-$\msfBC$ analogs
$Q(w)$ and $\Gamma(w)$ of the type-$\msfA$ posets and graphs
associated to planar networks.
We define a type-$\msfBC$ analog of codominant permutations and
show that the posets correspond bijectively to the proper subset
\begin{equation}\label{eq:BCintrocodom}
  \{ \btc w1 \,|\, w \in \bn \text{ $\msfBC$-codominant} \}
\end{equation}
of (\ref{eq:BCintrosmooth}).
We use the above networks, posets, and graphs 
in Section~\ref{s:main}
to state and prove our main results on
the combinatorial computation of type-$\msfBC$ trace evaluations,
and in Section~\ref{s:equiv} to show that
for each element $\btc w1$ of (\ref{eq:BCintrosmooth})
there exists an element $\btc v1$ of (\ref{eq:BCintrocodom})
with the property that $Q(w) \cong Q(v)$
and therefore that $\theta(\btc w1) = \theta(\btc v1)$ for all traces
$\theta \in \trsp(\bn)$.
Formulas in Section~\ref{s:main} lead to natural 
type-$\msfBC$ analogs of
type-$\msfA$ chromatic symmetric functions 
in Section~\ref{s:symm}.
We finish in Section~\ref{s:hess}
with open problems concerning Hessenberg varieties.

\section{The symmetric and hyperoctahedral groups}\label{s:s2nbn}

The hyperoctahedral group $\bn$ is closely related to the symmetric
groups on $n$ and $2n$ letters.  To describe these relationships,
we will use 
subintervals of the set
\begin{equation*}
  [\ol n, n] \defeq \{ -n, \dotsc, n \} \ssm \{ 0 \},
  \end{equation*}
where we define
$\ol a = -a$ for all $a \in [\ol n,n]$.
We call any
subset
$[h,l] \defeq \{ h, \dotsc, l \} \ssm \{0 \}$ of $[\ol n,n]$
an {\em interval}, even if $h < 0 < l$.
Let $\mfs{[h,l]}$ denote
the group of permutations of letters in the interval $[h,l]$.
The group $\bn$ is naturally related both to $\snn$
and $\sn = \mfs{[1,n]}$.
To illustrate these relationships and prepare for our main results,
we will consider the groups' presentations, conjugacy classes, Bruhat
orders, and pattern-avoidance definitions.


\ssec{$\bn$ as a subgroup of $\mfs{\smash{[\ol n, n]}}$}\label{ss:bnassubgp}
The $2n$th symmetric group
$\mfs{\smash{[\ol n, n]}}$
is the Coxeter group
(see, e.g., \cite{BBCoxeter}) of type $\msfA_{2n-1}$,
with generators
  $s_{\ol{n-1}}, \dotsc, s_{\ol 1}, s_0, s_1, \dotsc, s_{n-1}$
and relations
\begin{equation*}
  \begin{alignedat}2
    s_i^2 &= e &\quad &\text{for $i = \ol{n-1}, \dotsc, n-1$,}\\
    s_is_j &= s_js_i &\quad &\text{for $|i-j| \geq 2$,}\\
    s_is_js_i &= s_js_is_j &\quad &\text{for $|i-j| = 1$.}
  \end{alignedat}
\end{equation*}

If an expression $s_{i_1} \cdots s_{i_\ell}$ for $w \in \snn$
is as short as possible,
then call it {\em reduced} and call $\ell = \ell(w)$ the {\em length} of $w$.
Define a (left) action of $\snn$ on rearrangements
$w_{\ol n} \cdots w_{\ol 1} w_1 \cdots w_n$ of the word
$\ol n \cdots \ol 1 1 \cdots n$ by
\begin{equation}\label{eq:laction}
  \begin{cases}
    s_i \text{ swaps letters in positions $i, i+1$}
      &\text{for $i = 1, \dotsc, n-1$},\\
    s_{\ol i} \text{ swaps letters in positions $\ol i, \ol{i+1}$}
      &\text{for $i = 1, \dotsc, n-1$},\\
    s_0 \text{ swaps letters in positions $\ol 1, 1$,}
  \end{cases}
\end{equation}
and
define the {\em one-line notation} of $w = s_{i_1} \cdots s_{i_r} \in \snn$ to be
\begin{equation}\label{eq:s2nrearrange}
  w_{\ol n} \cdots w_{\ol 1} w_1 \cdots w_n 
= s_{i_1}(s_{i_2}( \cdots (s_{i_r}( \ol n \cdots \ol 1 1 \cdots n)) \cdots )).
\end{equation}
For example, when $n=4$, the element $s_{\ol1}s_0s_1$
has one-line notation
\begin{equation*}
    s_{\ol1}(s_0( s_1(\ol4 \ol3 \ol2 \ol1 1234)))
    = s_{\ol1}(s_0(\ol4 \ol3 \ol2 \ol1 2 1 3 4))
    = s_{\ol1}(\ol4 \ol3 \ol2 2 \ol1 1 3 4)
    = \ol4\ol3 2 \ol2 \ol1 1 3 4.
  \end{equation*}
(By our definition, the right action of $s_i$ swaps the letters $i, i+1$,
wherever they are.)
It follows that $w_i^{-1}$ is the index $j$ satisfying $w_j = i$.
It is known that $\ell(w)$ equals the number of {\em inversions} in $w$:
\begin{equation*}
  \inv(w_{\ol n} \cdots w_{\ol1} w_1 \cdots w_n) \defeq
  \{ (j,i) \,|\, j > i \text{ and $j$ appears before $i$ in $w_{\ol n} \cdots w_{\ol1} w_1 \cdots w_n$} \}.
  \end{equation*}
Thus we have
  $\ell(\ol4\ol32\ol2\ol1134) =
  \inv(\ol4\ol32\ol2\ol1134) =
  |\{(2,\ol2), (2,\ol1), (2,1)\}| = 3$.

For $[a, b] \subseteq [h, l]$
let $\smash{s_{[a,b]}^{[h,l]}}$ be the permutation
in $\mfs{[h,l]}$ whose one-line notation has the form
\begin{equation*}
  h \cdots (a-1) \cdot b (b-1) \cdots (a+1) a \cdot (b+1) \cdots l.
\end{equation*}
When the interval $[h,l]$ is clear from context,
we will simply write $s_{[a,b]}$.
Call such an element a (type-$\msfA$) {\em reversal}.
Observe that the standard generators of $\snn$
are all reversals:
  $s_0 = s_{[-1,1]}$,
  and for $i \geq 1$ we have
  $s_i = s_{[i,i+1]}$,
  $s_{\smash{\ol i}} = s_{\smash{[\ol{i+1},\ol i]}}$. 
Also observe that each trivial reversal $s_{[a,a]}$
is equal to the identity element $e$, and that two reversals $s_{[a,b]}$,
$s_{[c,d]}$ commute
if their intervals $[a,b]$, $[c,d]$ do not intersect.

Let $\bn$ be the Coxeter group of type $\mathsf C_n = \mathsf B_n$,
i.e., the hyperoctahedral group.
We may view $\bn$ as the subgroup of $\snn$
generated by elements
\begin{equation*}
  t\ (\text{also written $s'_0$}) \defeq s_0, \qquad
  s'_i \defeq s_is_{\smash{\ol i}}, \text{ for } i = 1,\dotsc,n-1,
\end{equation*}
which satisfy the relations
\begin{equation*}
  \begin{alignedat}2
    \smash{{s'_i}^2} &= e &\quad &\text{for $i = 0, \dotsc, n-1$,}\\
    ts'_1ts'_1 &= s'_1ts'_1t, &\quad & \\
    s'_is'_j &= s'_js'_i &\quad &\text{for $i,j \geq 0$ and $|i-j| \geq 2$,}\\
    s'_is'_js'_i &= s'_js'_is'_j &\quad &\text{for $i,j \geq 1$ and $|i-j| = 1$.}
  \end{alignedat}
\end{equation*}
The one-line notation for elements of $\bn$ is inherited from that
of $\snn$.
For example, when $n=4$, the element $ts'_1s'_2 \in \mfb4 $
has one-line notation
\begin{equation*}
    t(s'_1(s'_2(\ol4 \ol3 \ol2 \ol1 1234)))
    = ts'_1(\ol4 \ol2 \ol3 \ol1 1 3 2 4))
    = t(\ol4 \ol2 \ol1 \ol3 3 1 2 4)
    = \ol4 \ol2 \ol1 3 \ol3 1 2 4.
  \end{equation*}
If an expression $s'_{i_1} \cdots s'_{i_\ell}$ for $w \in \bn$
is as short as possible,
then call it {\em reduced} and call $\ell = \ell(w)$ the {\em length} of $w$.
Let $\ell_t(w)$ be the number of ocurrences of $t$ in any (equivalently, every)
reduced expression for $w$.
Analogously, let $\ell_s(w) = \ell(w) - \ell_t(w)$ be the number
of ocurrences of $s'_1,\dotsc,s'_{n-1}$.
It is easy to see that one-line notations of
elements of $\bn$ are precisely the set of
permutations $w_{\ol n} \cdots w_{\ol 1} w_1 \cdots w_n \in \snn$
which satisfy $w_{\ol i} = \ol{w_i}$,
i.e., each is completely determined by the $n$-letter subword
$w_1 \cdots w_n$.
Call these words the {\em long} and
{\em short}
one-line notations of $w$, respectively.
%
We can read $\ell(w)$, $\ell_t(w)$, $\ell_s(w)$ from the
short one-line notation of $w$ by
\begin{equation*}
  \ell(w) = \inv(w_1 \cdots w_n) + \sumsb{i > 0\\ w_i < 0} |w_i|,
  \qquad
  \ell_t(w) = \# \{ i > 0 \,|\, w_i < 0 \},
  \end{equation*}
and $\ell_s(w) = \ell(w) - \ell_t(w)$.
Thus we have
  $\ell(\ol3124) = |3| = 3$, $\ell_t(\ol3124) = 1$, and $\ell_s(\ol3124) = 2$.


Define {\em type-$\msfBC$ reversals} to be those elements of $\bn$
having the forms
\begin{equation}\label{eq:creversal}
  \begin{aligned}
    s'_{\smash{[\ol a, a]}} &\defeq s_{\smash{[\ol a, a]}},
        \mbox{ for } 1 \leq a \leq n,\\
    s'_{\smash{[a,b]}} &\defeq s_{\smash{[\ol b, \ol a]}} s_{[a,b]},
    \mbox{ for } 1 \leq a \leq b \leq n,
  \end{aligned}
\end{equation}
where $s_{[a,b]}$, etc., are reversals in $\snn$.

\ssec{Conjugacy classes, partitions, tableaux, bipartitions, and bitableaux}\label{ss:conjclasses}
Conjugacy classes of $\sn$ correspond to {\em (integer) partitions} of $n$,
weakly decreasing positive integer sequences
$\lambda = (\lambda_1, \dotsc, \lambda_\ell)$
satisfying $\lambda_1 + \cdots + \lambda_\ell = n$.
The $\ell = \ell(\lambda)$ components of a partition $\lambda$
are called its {\em parts}
and we let the expressions
$|\lambda| = n$ and $\lambda \vdash n$ denote that $\lambda$ is
a partition of $n$.
Sometimes we use the notation $k^{a_k}$ to denote a
sequence of $a_k$ copies of the letter $k$.
Given $\lambda \vdash n$, we define the {\em transpose}
partition $\lambda^\tr = (\lambda_1^\tr, \dotsc,\lambda_{\lambda_1}^{\tr})$ by
$\lambda_i^\tr = \# \{ j \,|\, \lambda_j \geq i \}$. Thus $n^\tr = 1^n$.
We call $\lambda$ {\em self-transpose} if $\lambda^\tr = \lambda$
and we define the empty sequence $\emptyset$ to be the unique partition of the
integer $0$.
Generalizing integer partitions of $n$ are
{\em compositions} $\alpha = (\alpha_1,\dotsc, \alpha_r)$ of $n$,
which are simply positive integer sequences summing to $n$.
We let the notation $\alpha \vDash n$ denote that $\alpha$ is
a composition of $n$. (See, e.g., \cite[\S 1.2]{StanEC1}.)

The conjugacy class of $\sn$ corresponding to $\lambda \vdash n$
is the set of all permutations having cycle type $\lambda$.  We write
$\ctype(w) = \lambda$. Letting
$a_k$ be the multiplicity of $k$ in $\lambda$, for $k =1,\dotsc,n$,
we may express the cardinality of the $\lambda$-conjugacy class of $\sn$ as
$n!/z_\lambda$, where
\begin{equation}\label{eq:zlambda}
 z_\lambda = {1^{a_1} \cdots n^{a_n} a_1! \cdots a_n!}.
\end{equation}

Conjugacy classes of $\bn$ correspond to
{\em integer bipartitions of $n$}, pairs $(\lambda,\mu)$
of integer partitions with 
$|\lambda| + |\mu| = n$.
We let
$(\lambda,\mu) \vdash n$ denote that $(\lambda,\mu)$ is
a bipartition of $n$. 
To explicitly describe the conjugacy classes of $\bn$,
we define the homomorphism
\begin{equation}\label{eq:varphi}
  \begin{aligned}
    \varphi: \bn &\rightarrow \sn \\
    s'_i &\mapsto s_i, \quad i=1,\dotsc,n-1,\\
    t &\mapsto e,
  \end{aligned}
  \end{equation}
which replaces letters in the short one-line notation of
$v \in \bn$ by their absolute values.
For each element $v \in \bn$
and each cycle $C = (c_1, \dotsc, c_k = c_0)$
of $\varphi(v) \in \sn$,
define the {\em signed cycle}
$\hat C = (\hat c_1, \dotsc, \hat c_k)$
of $v$ by
\begin{equation*}
  \hat c_i = \begin{cases}
  c_i &\text{if $v(c_{i-1}) = c_i$},\\
  \ol c_i &\text{if $v(c_{i-1}) = \ol{c_i}$}.
  \end{cases}
\end{equation*}
Call $\hat C$ {\em positive} if it has an even number of negative letters,
and {\em negative} otherwise.
The conjugacy class of $\bn$ corresponding to $(\lambda,\mu) \vdash n$
is precisely the set of elements  
whose {\em signed cycle type}, the bipartition of
positive cycle cardinalities and negative cycle cardinalities,
is equal to $(\lambda,\mu)$.
We write $\sct(w) = (\lambda,\mu)$.
We may express the cardinality of the $(\lambda,\mu)$
conjugacy class of $\bn$ as $2^nn!/(z_\lambda z_\mu 2^{\ell(\lambda) + \ell(\mu)})$.
  

To each integer partition
$\lambda = (\lambda_1, \dotsc, \lambda_\ell) \vdash n$
we associate a {\em Young diagram of shape $\lambda$},
an arrangement of $n$ boxes into $\ell$ left-justified rows
with $\lambda_i$ boxes in row $i$.  By the French convention, row $1$ appears
on the bottom.
A Young diagram filled with
elements of a set $S$ is called a {\em tableau} or more specifically an
{\em $S$-tableau}.  If $S \subseteq \mathbb Z$ we also call it a
{\em Young tableau}.
Repeated elements are permitted.
Given a bipartition $(\lambda,\mu) \vdash n$, we define
a Young {\em bidiagram} of shape $(\lambda,\mu)$
to be an ordered pair of Young diagrams of shapes $\lambda$ and $\mu$.
We define {\em bitableaux} similarly.

\ssec{The Bruhat order}


For any Coxeter group $W$,
the {\em Bruhat order on $W$} is the poset
defined by declaring $v \leq_W w$ 
if
some (equivalently, every)
reduced expression for $w$ contains a reduced expression for $v$.
Ehresmann~\cite{Ehresmann}
showed that the Bruhat order on $\snn$ is isomorphic
to the (dual of the)
componentwise order on tableaux $\{ A(w) \,|\, w \in \snn \}$
of shape $(2n,2n-1,\dotsc,1)$ defined by placing the
increasing rearrangement of $w_i \cdots w_n$ in row $i$,
for $i = \ol n, \dotsc, n$.
For example, the type-$\msfA$ Bruhat order comparison
$\ol 3 2 1 \ol1 \ol2 3 \leq_{\snn} \ol3 2 1 3 \ol2 \ol1$
may be verified by the componentwise inequality
$A(\ol 3 2 1 \ol1 \ol2 3) \geq A(\ol3 2 1 3 \ol2 \ol1 )$,
\begin{equation*}
  \tableau[scY]{3|\ol2,3|\ol2,\ol1,3|\ol2,\ol1,1,3|\ol2,\ol1,1,2,3|\ol3,\ol2,\ol1,1,2,3|}
  \quad \geq \quad
  \tableau[scY]{\ol1|\ol2,\ol1|\ol2,\ol1,3|\ol2,\ol1,1,3|\ol2,\ol1,1,2,3|\ol3,\ol2,\ol1,1,2,3|}\ .
\end{equation*}
Proctor~\cite[Thm.\,5BC]{ProctorBruhat} showed that the Bruhat order on $\bn$
is isomorphic to a similar order on tableaux $\{B(w) \,|\, w \in \bn \}$
of shape $(n,n-1,\dotsc,1)$ defined by placing the
increasing rearrangement of $w_i \cdots w_n$ in row $i$,
for $i = 1, \dotsc, n$.
For example, the type-$\msfBC$ Bruhat order comparison
$\ol 3 2 1 \ol1 \ol2 3 \leq_{\bn} 1 2 \ol3 3 \ol2 \ol1$
may be verified by the componentwise inequality
$B(\ol 3 2 1 \ol1 \ol2 3) \geq B(1 2 \ol3 3 \ol2 \ol1)$,
\begin{equation*}
  \tableau[scY]{3|\ol2,3|\ol2,\ol1,3|}
  \quad \geq \quad
  \tableau[scY]{\ol1|\ol2,\ol1|\ol2,\ol1,3|}\ .
\end{equation*}



It is not difficult to show that the Bruhat order on $\bn$ is
an induced subposet of the Bruhat order on $\snn$.
\begin{prop}\label{p:abcbruhat}
  For $v, w \in \bn \subset \snn$, we have $v \leq_{\bn} w$ if and only if
  $v \leq_{\snn} w$.
\end{prop}
\begin{proof}
  Consider $v, w \in \bn \subset \snn$.
  If we have $v \leq_{\bn} w$,
  then there is a reduced $\bn$-expression $s'_{i_1} \cdots s'_{i_k}$ for $v$
  which is a subword of a reduced $\bn$-expression $s'_{j_1} \cdots s'_{j_k}$
  for $w$.
  Then the recipe
  \begin{equation*}
    s'_i \mapsto \begin{cases}
      s_{\ol i} s_i &\text{if $i > 0$}\\
      t &\text{if $i = 0$}
    \end{cases}
  \end{equation*}
  produces a reduced $\snn$-expression for $v$
  which is a subword of a reduced $\snn$-expression for $w$.
  Now suppose that $v \leq_{\snn} w$.
  By Ehresmann's criterion, we have the componentwise
  tableau inequality $A(v) \geq A(w)$.
  But the upper $n$ rows of these tableaux give the inequality $B(v) \geq B(w)$.
  This is precisely Proctor's criterion for $v \leq_{\mfb{n}} w$.
  \end{proof}
  

\ssec{Pattern avoidance}

Given a word $u = u_1 \cdots u_k$ in $\mfs k$,
and a word $y = y_1 \cdots y_k$ having $k$ distinct letters,
we say that {\em $y$ matches the pattern $u$} if the letters of $y$ appear
in the same relative order as those of $u$; that is, if we have
$u_i < u_j$ if and only if $y_i < y_j$ for all $i,j \in [k]$.
On the other hand, given a word $w = w_1\dotsc w_m$ having distinct letters,
e.g., $w \in \mfs n$ or $w \in \bn$,
we say that {\em $w$ avoids the pattern $u$} if no subword
of $w$ matches the pattern $u$.

In $\bn$, a second notion of pattern avoidance
involves signed letters and short one-line notation.
Let $v = v_1 \cdots v_k$ be the short one-line notation of
an element of $\mfb k$,
i.e., a word in
letters $[\ol k, k]$ with
$|v_1| \cdots |v_k| \in \mfs k$.
Let $y = y_1 \cdots y_k$ be a word in $[\ol n, n]$ such that
$|y_1| \cdots |y_k|$ has no repeated letters.
Say that {\em $y$ matches the signed pattern $v$} if
\begin{enumerate}
\item for $i = 1,\dotsc, k$, the letters $v_i$ and $y_i$ have the same sign,
\item for all $i, j$, $|v_i| < |v_j|$ if and only if $|y_i| < |y_j|$.
\end{enumerate}
Say that $w \in \bn$ {\em avoids the signed pattern $v$} if no subword of
the short one-line notation of $w$ matches the signed pattern $v$.
(See \cite[p.\,108]{BilleyLak}.)

Many properties of elements of $\bn$ can be expressed in terms of
signed pattern avoidance.
\begin{lem}\label{l:1ol2ol21}
  The element $w \in \bn$ avoids the signed patterns $1\ol2$ and $\ol21$
  if and only if
  the set of negative letters in $w_1 \cdots w_n$ is empty or
  forms an interval $[\ol b, \ol1]$ for some $b \geq 1$.
\end{lem}
\begin{proof}
  If all letters in $w_1 \cdots w_n$ are positive, then the claim is true.

  \noindent
  ($\Rightarrow$)
  Suppose therefore that the negative letters in this word do not form
  an interval of the desired form. Then for some $i$, $j$, we have
  $w_j \leq \ol1$ and $1 \leq w_i < |w_j|$.
  It follows that $w_iw_j$ matches the signed
  pattern $1\ol2$ or $w_jw_i$ matches the signed pattern $\ol2 1$.

  \noindent
  $(\Leftarrow$)
  If some subword $w_iw_j$ with $i < j$ matches the signed pattern
  $1 \ol2$ or $\ol21$, then clearly the negative letters in $w_1 \cdots w_n$
  do not form the desired interval.
\end{proof}

Avoidance of signed patterns can also imply the avoidance of ordinary patterns.
\begin{lem}\label{l:5signedp}
  If $w \in \bn$ \avoidssignedp, then $w$ avoids the unsigned patterns
  $3412$ and $4231$.
\end{lem}
\begin{proof}
  First we claim that if $w$ contains a subword $cdab$
  matching the unsigned pattern $3412$, then it must contain a subword
  matching one of the five signed patterns.
  Suppose that $cdab$ or just $dab$ is a subword of $w_1 \cdots w_n$.
  If $b>0$ then $dab$ matches the signed pattern $312$ or $3\ol12$.
  If $b<0$ then $ab$ matches the signed pattern $\ol2\ol1$.
  Now suppose that $cd$ appears in $w_{\ol n} \cdots w_{\ol1}$
  and $ab$ in $w_1 \cdots w_n$.
  If $b<0$ then $ab$ matches the signed pattern $\ol2\ol1$.
  If $b>0$ then $\ol c$ and $\ol d$ appear in $w_1 \cdots w_n$ without $\ol b$,
  contradicting Lemma~\ref{l:1ol2ol21}.
  Now suppose that $cdab$ or just $cda$
  appears in $w_{\ol n} \cdots w_{\ol 1}$.
  Then $\ol{badc}$ or $\ol{adc}$
  is a subword of $w_1 \cdots w_n$, matching the unsigned pattern
  $3412$ or $412$.  By the first case, $w_1 \cdots w_n$
  has a subword matching one of the signed patterns
  $312$, $3 \ol1 2$, $\ol2\ol1$.

  Now we claim that if $w$ contains a subword $dbca$
  matching the unsigned pattern $4231$, then it must contain a subword
  matching one of the five signed patterns.
  Suppose that $dbca$ is a subword of $w_1 \cdots w_n$.
  If $c > 0$ then $dbc$ matches the signed pattern $312$ or $3\ol12$.
  If $c < 0$ then $bc$ matches the signed pattern $\ol2\ol1$.
  Now suppose that $d$ appears in $w_{\ol n} \cdots w_{\ol 1}$
  and $bca$ in $w_1 \cdots w_n$.  If $c > 0$ then $\ol d$
  appears in $w_1 \cdots w_n$, contradicting Lemma~\ref{l:1ol2ol21}.
  If $c < 0$ then $bc$ matches the signed pattern $\ol2\ol1$.
  Now suppose that $db$ appears in $w_{\ol n} \cdots w_{\ol 1}$
  and $ca$ in $w_1 \cdots w_n$.
  If $c > 0$ then $\ol d$ ($<0$) appears in $w_1 \cdots w_n$,
  contradicting Lemma~\ref{l:1ol2ol21}.
  If $c < 0$ then $\ol b$ ($>0$) appears in $w_1 \cdots w_n$,
  also contradicting Lemma~\ref{l:1ol2ol21}.
  Finally, suppose that $dbca$ or just $dbc$
  appears in $w_{\ol n} \cdots w_{\ol 1}$.
  Then $\ol{acbd}$ or $\ol{cbd}$
  is a subword of $w_1 \cdots w_n$, matching the unsigned pattern
  $4231$ or $423$.  By the first and second cases above, $w_1 \cdots w_n$
  has a subword matching one of the five signed patterns.
\end{proof}

\section{Schubert varieties and Hecke algebras}\label{s:schuberthecke}

Our main results (Theorem~\ref{t:epsiloneta} --
Theorem~\ref{t:chi})
partially answer
Problem~\ref{p:evaltrace} using a linearly independent set in $\zbn$.
This set is best described in terms of a special basis of the
Hecke algebra $H(\bn)$ and smoothness of certain Schubert varieties.

\ssec{Schubert varieties}

Let $\mathrm G$ be a complex connected semisimple algebraic group,
choose a Borel subgroup $\mathrm B$ of $\mathrm G$, and consider
the quotient
$\mathrm G/\mathrm B$,
called a {\em flag variety}.
The action of $\mathrm B$ on $\mathrm G/\mathrm B$ by left multiplication 
partitions
it
into orbits often written
$\mathrm B\dot w\mathrm B$,
which are
parametrized by elements $w$ of the corresponding Weyl group $W$.
The Zariski closure $\schub w$ of
$\mathrm B \dot w \mathrm B$
in $\mathrm G/\mathrm B$
is called the {\em Schubert variety indexed by $w$}.
We have
$\schub v \supseteq \schub w$ if and only if $v \leq w$ in the Bruhat order.
(See, e.g., \cite[\S \,4.7]{BilleyLak}.)
Standard choices of $\mathrm G$
are
$\mathrm{SL}_n(\mathbb C)$
(type $\msfA$),
$\mathrm{SO}_{2n+1}(\mathbb C)$
(type $\msfBB$),
and $\mathrm{SP}_{n}(\mathbb C)$
(type $\msfC$).
The corresponding Weyl groups are $\sn$ (type $\msfA$), and
$\bn$
(types $\msfBB$ and $\msfC$).

Call the Schubert variety $\schub w$ {\em rationally smooth} if
its ordinary cohomology $\coh^*(\schub w)$ and
intersection cohomology $\icoh^*(\schub w)$ coincide.
(See \cite[Ch.\,6]{BilleyLak}.)
Call $\schub w$ {\em smooth} if the tangent space at every
point has dimension equal to the dimension of the variety.
It is known that every smooth
Schubert variety is rationally smooth.
Let $\Aschub w$ ($w \in \sn$), $\Bschub w$, $\Cschub w$ ($w \in \bn$)
denote the type-$\msfA$,
$\msfBB$, and
$\sfC$ Schubert varieties,
respectively.  Elements $w \in \sn$ for which $\Aschub w$ is rationally
smooth or smooth are characterized by pattern avoidance~\cite{LakSan}.
\begin{prop}\label{p:aavoid}
  For $w \in \sn$,
  the
  Schubert variety
  $\Aschub w$
  is smooth, equivalently rationally smooth,
  if and only if $w$ \avoidsp.
\end{prop}

Smoothness and rational smoothness of type-$\msfBB$ and
$\sfC$
Schubert varieties are charcterized by more intricate pattern avoidance.
The conditions on $w \in \bn$ which imply
rational smoothness of
$\Bschub w$
are the same as those which imply
rational smoothness of
$\Cschub w$~\cite[Thm.\,4.2]{BilleyPattern}:
$w$ must avoid the twenty-five
patterns listed in \cite[Eq.\,(13.3.5)]{BilleyLak}.
On the other hand, the conditions which imply smoothness of the two
Schubert varieties are different~\cite[Thm.\,8.3.17]{BilleyLak}:
$\Bschub w$
is smooth if and only if it is rationally smooth and
$w$ avoids the additional pattern $3412$;
$\Cschub w$
is smooth if and only if it is rationally smooth and
$w$ avoids the additional pattern $4231$.

We will be interested in those elements $w \in \bn$ for which
$\Bschub w$
and
$\Cschub w$
are simultaneously smooth.
These are precisely the elements $w \in \bn \subset \snn$ for which
$\Aschub w$ is smooth when $\mathrm G = \mathrm{SL}_{2n}(\mathbb C)$.
\begin{prop}\label{p:bcavoid}
  For $w \in \bn$,
  the
  Schubert varieties
  $\Bschub w$
  and
  $\Cschub w$
  are simultaneously smooth
  if and only if
  $w$ \avoidsp.
\end{prop}
\begin{proof}
  If $\Bschub w$
  and
  $\Cschub w$
  are both smooth, then
  by the above discussion,
  $w$ \avoidsp.
  Suppose that
  $w$ \avoidsp.
  It is straightforward to check that  
  each of the twenty-five patterns listed in
  \cite[Eq.\,(13.3.5)]{BilleyLak}
  contains $3412$ and/or $4231$.
  Thus
  $w$ avoids these twenty-five patterns as well,
  and
  $\Bschub w$,
  $\Cschub w$
  are both smooth.
\end{proof}


\ssec{Hecke algebras}\label{ss:hecke}

Given Coxeter group $W$ with generator set $S$, define
the {\em Hecke algebra} $H(W)$ of $W$ to be
the $\zqq$-span of $\{ T_w \,|\, w \in W \}$
with multiplicative unit $T_e$ and multiplication defined by
\begin{equation*}
  T_sT_w =
  \begin{cases}
    qT_{sw} + (q-1)T_w &\text{if $sw <_W w$},\\
    T_{sw} &\text{if $sw >_W w$},
  \end{cases}
\end{equation*}
where $s \in S$, $w \in W$, and $<_W$ is the Bruhat order on $W$.
This formula guarantees that for $w \in W$ and
any reduced expression $s_{i_1} \cdots s_{i_\ell}$ for $w$,
we have $T_w = T_{s_{i_1}} \cdots T_{s_{i_\ell}}$.
Call $\{ T_w \,|\, w \in W \}$ the {\em natural basis} of $H(W)$.
It is easy to see that the specialization of $H(W)$ at $q=1$
is isomorphic to $\mathbb Z[W]$.

A second basis~\cite{KLRepCH} of $H$ is the (modified, signless)
{\em Kazhdan--Lusztig basis} $\{ \wtc wq \,|\, w \in W \}$,
related to the natural basis by
\begin{equation*}
  \wtc wq = \sum_{v \leq_W w} P_{v,w}(q) T_v,
\end{equation*}
where
$\{ P_{v,w}(q) \,|\, v, w \in W \} \subseteq \mathbb Z[q]$ are the
{\em Kazhdan--Lusztig polynomials}
whose recursive definition appears in ~\cite{KLRepCH}.
Coefficients of these polynnomials may be interpreted in terms of
intersection cohomology $\icoh^*(\schub w)$~\cite{KLSchub}.  Specifically,
when $\schub w$ is rationally smooth,
all
polynomials
$\{P_{v,w}(q) \,|\, v \leq_{\sn} w\}$
are identically
$1$~\cite[Thm.\,A.2]{KLRepCH}.
Thus we have the following.
\begin{prop}\label{p:abcsmoothkl}
  For $W$ equal to $\sn$ or $\bn$ and $w \in W$ \avoidingp,
  the Kazhdan--Lusztig basis element $\wtc wq$
  of $H(W)$
  satisfies
  \begin{equation*}
    \wtc wq = \sum_{v \leq_W w} T_v.
  \end{equation*}
\end{prop}
\begin{proof}
  This follows from Propositions~\ref{p:aavoid}, \ref{p:bcavoid}.
  \end{proof}


In Sections~\ref{s:trace} -- \ref{s:hess} we will find it convenient to define
\begin{equation*}
  \hnq \defeq H(\sn), \qquad
  H_{[h,l]}(q) \defeq H(\mfs{[h,l]}), \qquad
  \hbnq \defeq H(\bn),
  \end{equation*}
and to let $\{ \wtc wq \,|\, w \in \sn \}$ and $\{ \btc wq \,|\, w \in \bn \}$
denote the Kazhdan--Lusztig bases of $\hnq$ and $\hbnq$, respectively.

\section{Trace spaces}\label{s:trace}


Given
Coxeter group $W$ and
Hecke algebra $H = H(W)$,
let $\trsp(H)$ be the $\zqq$-module of $H$-traces, 
linear functionals $\theta_q: H \rightarrow \zqq$ satisfying
$\theta_q(DD') = \theta_q(D'D)$ for all $D, D' \in H$.
This is the $\zqq$-span of all $H$-characters.
Let $\trsp(W) = \trsp(\mathbb Z[W])$
be the specialization of $\trsp(H)$ at $q=1$.
That is,
for each $H$-trace
$\theta_q \in \trsp(H)$ satisfying
$\theta_q(T_w) = f_w(q)$
for all $w \in W$,
define the $W$-trace
$\theta = \theta_1 \in \trsp(\zsn)$ by
$\theta(w) = f_w(1)$ for all $w \in W$.
(See, e.g., \cite{GPCharHecke}.)

The ranks of $\trsp(H)$ and $\trsp(W)$ are both equal to
the number of conjugacy classes of $W$.
We consider six bases of $\trsp(\hnq)$ and $\trsp(\sn)$,
and eleven bases of $\trsp(\hbnq)$ and $\trsp(\bn)$.

\ssec{The trace spaces $\trsp(\hnq)$ and $\trsp(\sn)$}\label{ss:snhnq}

The rank of $\trsp(\hnq)$ equals the number of partitions of $n$.
Three commonly used bases
consist of $\hnq$-characters.
These are the bases of
irreducible characters $\{ \chi_q^\lambda \,|\, \lambda \vdash n \}$,
induced trivial characters $\{ \eta_q^\lambda \,|\, \lambda \vdash n \}$,
and induced sign characters $\{ \epsilon_q^\lambda \,|\, \lambda \vdash n \}$,
where
\begin{equation}\label{eq:trivsign}
\eta_q^\lambda = \triv_q \upparrow_{\hlq}^{\hnq},
\qquad \triv_q(T_{s_i}) = q,
  \qquad \epsilon_q^\lambda = \sgn_q \upparrow_{\hlq}^{\hnq},
  \qquad \sgn_q(T_{s_i}) = -1,
\end{equation}
and $\hlq$ is the Young subalgebra of $\hnq$ generated by
\begin{equation*}
  \{ T_{s_1}, \dotsc, T_{s_{n-1}} \} \ssm
  \{ T_{s_{\lambda_1}}, T_{s_{\lambda_1 + \lambda_2}}, \dotsc, T_{s_{n-\lambda_\ell}} \}.
\end{equation*}
All $\hnq$-characters in $\trsp(\hnq)$ belong to $\spn_{\mathbb N[q]}\{ \chi_q^\lambda \,|\, \lambda \vdash n\}$. 
Three more non-character bases of $\trsp(\hnq)$ consist of traces which we call
{\em power sum traces} $\{ \psi_q^\lambda \,|\, \lambda \vdash n \}$,
{\em monomial traces} $\{ \phi_q^\lambda \,|\, \lambda \vdash n \}$, and
{\em forgotten traces} $\{ \gamma_q^\lambda \,|\, \lambda \vdash n \}$,
and which are defined by
\begin{equation*}
  \psi_q^\lambda = \sum_\mu \chi^\mu(\lambda) \chi_q^\mu,
  \qquad
  \phi_q^\lambda = \sum_\mu K^{-1}_{\lambda,\mu} \chi_q^\mu,
  \qquad
  \gamma_q^\lambda = \sum_\mu K^{-1}_{\lambda,\mu^\tr} \chi_q^\mu,
\end{equation*}
where $\chi^\mu(\lambda) \defeq \chi^\mu(w)$
for any $w \in \sn$ with $\ctype(w) = \lambda$,
and $\{K^{-1}_{\lambda,\mu} \,|\, \lambda,\mu \vdash n \}$ are the {\em inverse
Kostka numbers}.  (See \cite[\S 7]{StanEC2}.)
The specialization of the power sum trace basis at $q=1$
is 
essentially an indicator basis for conjugacy classes of $\sn$,
\begin{equation}\label{eq:psidef}
  \psi^{\lambda}(w) = \begin{cases}
    z_\lambda &\text{if $\ctype(w) = \lambda$},\\
    0       &\text{otherwise}.
  \end{cases}
\end{equation}
For few traces $\theta_q \in \trsp(\hnq)$
do we have cancellation-free formulas for all
evaluations of the form $\{\theta_q(T_w) \,|\, w \in \sn \}$.
Two examples are the trivial and sign characters
in (\ref{eq:trivsign}):
for all $w \in \sn$ we have
$\chi^n_q(T_w) = \eta^n_q(T_w) = q^{\ell(w)}$
and $\chi^{1^n}_q(T_w) = \epsilon^n_q(T_w) = (-1)^{\ell(w)}$.

\ssec{The trace spaces $\trsp(\bn)$ and $\trsp(\hbnq)$}\label{ss:bnhbnq}

The rank of $\trsp(\hbnq)$ equals the number of bipartitions of $n$.
Ten commonly used bases
can be constructed from pairs of type-$\msfA$
Hecke algebra trace bases, i.e., bases of
\begin{equation*}
  \bigoplus_{k=0}^n \trsp(H_k(q)) \otimes \trsp(H_{n-k}(q)),
\end{equation*}
and from the Young subalgebra
$\hbq{k,n-k}$ of $\hbnq$ generated by
\begin{equation*}
  \{ T_{t}, T_{\smash{s'_{\smash{1}}}}, \dotsc, T_{\smash{s'_{\smash{k-1}}}} \} \cup
  \{ T_{t_k}, T_{\smash{s'_{\smash{k+1}}}}, \dotsc, T_{\smash{s'_{\smash{n-1}}}} \},
\end{equation*}
where $t_k = s'_k \cdots s'_1 t s'_1 \cdots s'_k$.
Specifically, given bases
\begin{equation}\label{eq:basispair1}
  \{ \zeta_q^\lambda \,|\, \lambda \vdash k\}
  \subseteq \trsp(H_k(q)), \qquad
  \{ \xi_q^\mu \,|\, \mu \vdash n-k\}
  \subseteq \trsp(H_{n-k}(q)),
\end{equation}
we define traces
$q\zeta_q^\lambda \in \trsp(\hbq k)$,
$\delta\xi_q^\mu \in \trsp(\hbq{n-k})$ by
\begin{equation*}
  q\zeta_q^\lambda(T_w) \defeq q^{\ell_t(w)}\zeta_q^\lambda(T_{\varphi(w)}),
  \qquad
  \delta\xi_q^\mu(T_w) \defeq (-1)^{\ell_t(w)}\xi_q^\mu(T_{\varphi(w)}),
\end{equation*}
where $\ell_t$, $\varphi$ are defined as in
Subsections~\ref{ss:bnassubgp} -- \ref{ss:conjclasses}.
(When $\zeta^\lambda_q$ and $\xi^\mu_q$ are $\hnq$-characters, i.e.,
traces of matrix representations,
the modifications $q\zeta^\lambda_q$ and $\delta\xi^\mu_q$
correspond to
type-$\msfA$ matrix representations extended by the definitions
$T_t \mapsto qI$ and $T_t \mapsto -I$, respectively.)
Then we create a basis
$\{ (\zeta\xi)_q^{\lambda,\mu} \,|\, (\lambda,\mu) \vdash n \}$
of $\trsp(\hbnq)$ by
inducing
\begin{equation}\label{eq:inducedpair}
  (\zeta\xi)_q^{\lambda,\mu} \defeq (q\zeta_q^\lambda \otimes \delta\xi_q^\mu)
  \upparrow_{\hbq{k,n-k}}^{\hbnq}.
\end{equation}
This construction of the irreducible characters
\begin{equation*}
  \{ (\chi\chi)^{\lambda,\mu}_q \,|\, (\lambda,\mu) \vdash n \}
  \end{equation*}
of $\hbnq$ 
can be deduced from Hoefsmit~\cite[\S 2.2]{HoefsmitThesis}.
(See also \cite{DJRepHeckeB}, \cite[\S 5.5]{GPCharHecke}.)
One then verifies that other trace bases are related
to the irreducible character basis by
matrices described in \cite[\S 3]{RemTrans}.
When the bases in (\ref{eq:basispair1}) are type-$\msfA$ character bases,
the definition (\ref{eq:inducedpair}) gives a character basis of
$\trsp(\hbnq)$.
More examples are the
induced one-dimensional characters,
\begin{equation*}
\{ (\eta\eta)_q^{\lambda,\mu} \,|\, (\lambda,\mu) \vdash n \}, \quad
\{ (\eta\epsilon)_q^{\lambda,\mu} \,|\, (\lambda,\mu) \vdash n \}, \quad
\{ (\epsilon\eta)_q^{\lambda,\mu} \,|\, (\lambda,\mu) \vdash n \},\quad
\{ (\epsilon\epsilon)_q^{\lambda,\mu} \,|\, (\lambda,\mu) \vdash n \}.
\end{equation*}
Five
bases of $\trsp(\hbnq)$ which do not consist of characters are
formed from
the definition (\ref{eq:inducedpair}) and
type-$\msfA$ power sum, monomial, and forgotten traces,
\begin{equation*}
  \begin{gathered}
\{ (\psi\psi)_q^{\lambda,\mu} \,|\, (\lambda,\mu) \vdash n \}, \\
\{ (\phi\phi)_q^{\lambda,\mu} \,|\, (\lambda,\mu) \vdash n \}, \quad
\{ (\phi\gamma)_q^{\lambda,\mu} \,|\, (\lambda,\mu) \vdash n \}, \quad
\{ (\gamma\phi)_q^{\lambda,\mu} \,|\, (\lambda,\mu) \vdash n \},\quad
\{ (\gamma\gamma)_q^{\lambda,\mu} \,|\, (\lambda,\mu) \vdash n \}.
\end{gathered}
\end{equation*}
An eleventh basis of $\trsp(\hbnq)$,
\begin{equation*}
  \{ \iota_q^{\lambda,\mu} \,|\, (\lambda,\mu) \vdash n\},
\end{equation*}
may be defined in terms of irreducible characters by
\begin{equation}\label{eq:iotaqdef}
  \iota_q^{\lambda,\mu} =
  \sum_{(\alpha,\beta) \vdash n} (\chi\chi)^{\alpha,\beta}(\lambda,\mu)
  (\chi\chi)_q^{\alpha,\beta},
\end{equation}
where
we define
$(\chi\chi)^{\alpha,\beta}(\lambda,\mu) \defeq (\chi\chi)^{\alpha,\beta}(w)$
for any $w \in \bn$ having signed cycle type
$(\lambda,\mu)$. (See, e.g., \cite{AAERCharFormulasHyper}.)
The specialization of this basis at $q=1$
is 
essentially an indicator basis for conjugacy classes of $\bn$,
\begin{equation}\label{eq:iotadef}
  \iota^{\lambda,\mu}(w) = \begin{cases}
    z_\lambda z_\mu 2^{\ell(\lambda)+\ell(\mu)}
    &\text{if $\sct(w) = (\lambda,\mu)$},\\
    0 &\text{otherwise}.
  \end{cases}
\end{equation}

Unsurprisingly,
there are few traces $\theta_q \in \trsp(\hbnq)$
for which we have cancellation-free formulas
for $\{\theta_q(T_w) \,|\, w \in \bn \}$.
Four examples are the one-dimensional characters
constructed from (\ref{eq:trivsign}) and
(\ref{eq:inducedpair}): for all $w \in \bn$ we have
\begin{equation*}
  \begin{aligned}
(\chi\chi)^{(n,\emptyset)}_q(T_w) =
(\eta\eta)^{(n,\emptyset)}_q(T_w) =
(\eta\epsilon)^{(n,\emptyset)}_q(T_w) &= q^{\ell(w)},\\
(\chi\chi)^{(1^n,\emptyset)}_q(T_w) =
(\epsilon\eta)^{(n,\emptyset)}_q(T_w) =
(\epsilon\epsilon)^{(n,\emptyset)}_q(T_w) &= q^{\ell_t(w)}(-1)^{\ell_s(w)},\\
(\chi\chi)^{(\emptyset,n)}_q(T_w) =
(\eta\eta)^{(\emptyset,n)}_q(T_w) =
(\epsilon\eta)^{(\emptyset,n)}_q(T_w) &= (-1)^{\ell_t(w)}q^{\ell_s(w)},\\
(\chi\chi)^{(\emptyset,1^n)}_q(T_w) =
(\eta\epsilon)^{(\emptyset,n)}_q(T_w) =
(\epsilon\epsilon)^{(\emptyset,n)}_q(T_w) &= (-1)^{\ell(w)},
\end{aligned}
\end{equation*}
where $\ell_s$ is defined as in Subsection~\ref{ss:bnassubgp}.

\section{Planar networks}\label{s:planarnet}



Several partial solutions to the type-$\msfA$ case of Problem~\ref{p:evaltrace}
involve
the subset
\begin{equation}\label{eq:cwqsmooth}
  \{ \wtc wq \,|\, w \in \sn \text{ \avoidsp} \}
\end{equation}
of the Kazhdan--Lusztig basis of $\hnq$.
The graphical representation of these elements
by
planar networks
called type-$\msfA$ {\em zig-zag networks}~\cite[\S 3]{SkanNNDCB}
allows for simple combinatorial interpretation of certain trace
evaluations~\cite[\S 5--10]{SkanNNDCB}.
Moreover, the subset
\begin{equation}\label{eq:cwqcodom}
  \{ \wtc wq \,|\, w \in \sn \text{ avoids the pattern } 312 \}
\end{equation}
of (\ref{eq:cwqsmooth}) 
and its graphical representation by the subset of zig-zag networks
called {\em descending star networks}~\cite{SkanNNDCB}
captures much of the same information.

We will extend the above type-$\msfA$ results to types $\msfBB$ and $\sfC$
by defining {\em type-$\msfBC$ zig-zag networks}
to graphically represent the subset
\begin{equation}\label{eq:BCcwqsmooth}
  \{ \btc wq \,|\, w \in \bn \text{ \avoidsp} \}
\end{equation}
of the Kazhdan--Lusztig basis of $\hbnq$,
and {\em type-$\msfBC$ descending star networks}
to graphically represent the subset
\begin{equation}\label{eq:BCcwqcodom}
  \{ \btc wq \,|\, w \in \sn \text{ \avoidssignedp } \}
\end{equation}
of (\ref{eq:BCcwqsmooth}).
These graphical representations facilitate
simple combinatorial interpretation of certain trace evaluations
(Section~\ref{s:main}), when we specialize at $q=1$.



\ssec{Type-$\msfA$ planar networks and factoriztion}\label{ss:Aplanar}

Define a {\em type-$\msfA$ planar network with boundary vertices indexed by
the interval $[h,l]$}
to be a directed, planar,
acyclic multigraph which can be embedded in a disc so that
$2|[h,l]|$ boundary vertices can be labeled clockwise as
{\em source} $h, \dotsc,$ {\em source} $l$,
{\em sink} $l, \dotsc,$ {\em sink} $h$.
We will allow edges $(\vertex x, \vertex y)$
to be marked by a positive integer multiplicity $k$
and will say that such an edge contributes $k$ to the outdegree of $\vertex x$
and to the indegree of $\vertex y$.
We will assume all sources to have indegree $0$ and outdegree $1$,
and all sinks to have indegree $1$ and outdegree $0$.
Let $\net{A}{[h,l]}$ denote the set of such networks.

For each subinterval $[a,b]$ of $[h,l]$ we define a {\em simple star network}
$\smash{F_{[a,b]}^{[h,l]}} \in \net{A}{[h,l]}$ by
\begin{enumerate}
\item Sources $h,\dotsc,l$ lie on a vertical line to the left;
  sinks $h,\dotsc,l$ lie on a vertical line to the right.
  Both are labeled from bottom to top.
\item An interior vertex lies between the sources and sinks.
\item For $i = h, \dotsc, a-1$ and $i = b+1, \dotsc, l$, a directed edge
  begins at source $i$ and terminates at sink $i$.
\item For $i = a, \dotsc, b$, a directed edge begins at source $i$ and
  terminates at the interior vertex, and another directed edge begins
  at the interior vertex and terminates at sink $i$.
\item All edges have multiplicity $1$.
\end{enumerate}
When the set of source and sink labels is clear,
we omit the superscript $[h,l]$ and write $F_{[a,b]}$.
For zero- and one-element subintervals
we define the trivial network $F_{\emptyset} = F_{[h,h]} = \cdots = F_{[l,l]}$
to have no interior vertex, and $|[h,l]|$ horizontal edges, each from source $i$
to sink $i$, for $i = h,\dotsc,l$.
For example,
the (infinite) set $\net{A}{[\ol2, 2]}$ contains
seven simple star networks:
\begin{equation}\label{eq:simplestarnets}
\begin{gathered}
   \begin{tikzpicture}[scale=.5,baseline=-20]
\node at (-.4,0) {$\scriptstyle 2$};
\node at (-.4,-1) {$\scriptstyle 1$};
\node at (-.4,-2) {$\scriptstyle{\ol1}$};
\node at (-.4,-3) {$\scriptstyle{\ol2}$};  
\node at (1.4,0) {$\scriptstyle 2$};
\node at (1.4,-1) {$\scriptstyle 1$};
\node at (1.4,-2) {$\scriptstyle{\ol1}$};
\node at (1.4,-3) {$\scriptstyle{\ol2}$};  
\draw[-] (0,0) -- (1,-3);
\draw[-] (0,-1) -- (1,-2);
\draw[-] (0,-2) -- (1,-1);
\draw[-] (0,-3) -- (1,0);
\end{tikzpicture}
,\\
\phantom{\sum}\ntksp F_{[\ol2,2]}\phantom{\sum}
\end{gathered}
\begin{gathered}
\begin{tikzpicture}[scale=.5,baseline=-20]
\node at (-.4,0) {$\scriptstyle 2$};
\node at (-.4,-1) {$\scriptstyle 1$};
\node at (-.4,-2) {$\scriptstyle{\ol1}$};
\node at (-.4,-3) {$\scriptstyle{\ol2}$};  
\node at (1.4,0) {$\scriptstyle 2$};
\node at (1.4,-1) {$\scriptstyle 1$};
\node at (1.4,-2) {$\scriptstyle{\ol1}$};
\node at (1.4,-3) {$\scriptstyle{\ol2}$};  
\draw[-] (0,0) -- (1,-2);
\draw[-] (0,-1) -- (1,-1);
\draw[-] (0,-2) -- (1,0);
\draw[-] (0,-3) -- (1,-3);
\end{tikzpicture},\\
   \phantom{\sum}\ntksp F_{[\ol1,2]}\phantom{\sum}
 \end{gathered}
\begin{gathered}
\begin{tikzpicture}[scale=.5,baseline=-20]
\node at (-.4,0) {$\scriptstyle 2$};
\node at (-.4,-1) {$\scriptstyle 1$};
\node at (-.4,-2) {$\scriptstyle{\ol1}$};
\node at (-.4,-3) {$\scriptstyle{\ol2}$};  
\node at (1.4,0) {$\scriptstyle 2$};
\node at (1.4,-1) {$\scriptstyle 1$};
\node at (1.4,-2) {$\scriptstyle{\ol1}$};
\node at (1.4,-3) {$\scriptstyle{\ol2}$};  
\draw[-] (0,0) -- (1,0);
\draw[-] (0,-1) -- (1,-3);
\draw[-] (0,-2) -- (1,-2);
\draw[-] (0,-3) -- (1,-1);
\end{tikzpicture},\\
   \phantom{\sum}\ntksp F_{[\ol2,1]}\phantom{\sum}
 \end{gathered}
\begin{gathered}
\begin{tikzpicture}[scale=.5,baseline=-20]
\node at (-.4,0) {$\scriptstyle 2$};
\node at (-.4,-1) {$\scriptstyle 1$};
\node at (-.4,-2) {$\scriptstyle{\ol1}$};
\node at (-.4,-3) {$\scriptstyle{\ol2}$};  
\node at (1.4,0) {$\scriptstyle 2$};
\node at (1.4,-1) {$\scriptstyle 1$};
\node at (1.4,-2) {$\scriptstyle{\ol1}$};
\node at (1.4,-3) {$\scriptstyle{\ol2}$};  
\draw[-] (0,0) -- (1,-1);
\draw[-] (0,-1) -- (1,0);
\draw[-] (0,-2) -- (1,-2);
\draw[-] (0,-3) -- (1,-3);
\end{tikzpicture},\\
   \phantom{\sum}\ntksp F_{[1,2]}\phantom{\sum}
 \end{gathered}
\begin{gathered}
\begin{tikzpicture}[scale=.5,baseline=-20]
\node at (-.4,0) {$\scriptstyle 2$};
\node at (-.4,-1) {$\scriptstyle 1$};
\node at (-.4,-2) {$\scriptstyle{\ol1}$};
\node at (-.4,-3) {$\scriptstyle{\ol2}$};  
\node at (1.4,0) {$\scriptstyle 2$};
\node at (1.4,-1) {$\scriptstyle 1$};
\node at (1.4,-2) {$\scriptstyle{\ol1}$};
\node at (1.4,-3) {$\scriptstyle{\ol2}$};  
\draw[-] (0,0) -- (1,0);
\draw[-] (0,-1) -- (1,-2);
\draw[-] (0,-2) -- (1,-1);
\draw[-] (0,-3) -- (1,-3);
\end{tikzpicture},\\
   \phantom{\sum}\ntksp F_{[\ol1,1]}\phantom{\sum}
 \end{gathered}
\begin{gathered}
\begin{tikzpicture}[scale=.5,baseline=-20]
\node at (-.4,0) {$\scriptstyle 2$};
\node at (-.4,-1) {$\scriptstyle 1$};
\node at (-.4,-2) {$\scriptstyle{\ol1}$};
\node at (-.4,-3) {$\scriptstyle{\ol2}$};  
\node at (1.4,0) {$\scriptstyle 2$};
\node at (1.4,-1) {$\scriptstyle 1$};
\node at (1.4,-2) {$\scriptstyle{\ol1}$};
\node at (1.4,-3) {$\scriptstyle{\ol2}$};  
\draw[-] (0,0) -- (1,0);
\draw[-] (0,-1) -- (1,-1);
\draw[-] (0,-2) -- (1,-3);
\draw[-] (0,-3) -- (1,-2);
\end{tikzpicture},\\
   \phantom{\sum}\ntksp F_{[\ol2,\ol1]}\phantom{\sum}
 \end{gathered}
\
\begin{gathered}
\begin{tikzpicture}[scale=.5,baseline=-20]
\node at (-.4,0) {$\scriptstyle 2$};
\node at (-.4,-1) {$\scriptstyle 1$};
\node at (-.4,-2) {$\scriptstyle{\ol1}$};
\node at (-.4,-3) {$\scriptstyle{\ol2}$};  
\node at (1.4,0) {$\scriptstyle 2$};
\node at (1.4,-1) {$\scriptstyle 1$};
\node at (1.4,-2) {$\scriptstyle{\ol1}$};
\node at (1.4,-3) {$\scriptstyle{\ol2}$};  
\draw[-] (0,0) -- (1,0);
\draw[-] (0,-1) -- (1,-1);
\draw[-] (0,-2) -- (1,-2);
\draw[-] (0,-3) -- (1,-3);
\end{tikzpicture},\\
\phantom{\sum}\ntksp F_\emptyset \phantom{\sum}
 \end{gathered}
\end{equation}
where $F_\emptyset = F_{[\ol2,\ol2]} = F_{[\ol1,\ol1]} = F_{[1,1]} = F_{[2,2]}$.
In figures,
all edges in planar networks
should be understood to be oriented from left to right,
with vertices at both ends of all line
segments, and additional vertices at the centers of the
stars formed from crossing line segments.
Thus $F_{\smash{[\ol1,2]}}$ above can be more completely drawn as
\begin{equation*}
\begin{tikzpicture}[scale=.5,baseline=15]
\node at (-2.5,2.5) {$\scriptstyle{ \mathrm{source}\;2}$};
\node at (-2.5,1.5) {$\scriptstyle{ \mathrm{source}\;1}$};
\node at (-2.5,0.5) {$\scriptstyle{ \mathrm{source}\;\ol1}$};
\node at (-2.5,-0.5) {$\scriptstyle{ \mathrm{source}\;\ol2}$};  
\node at (2.25,2.5) {$\scriptstyle{ \mathrm{sink}\;2}$};
\node at (2.25,1.5) {$\scriptstyle{ \mathrm{sink}\;1}$};
\node at (2.25,0.5) {$\scriptstyle{ \mathrm{sink}\;\ol1}$};
\node at (2.25,-0.5) {$\scriptstyle{ \mathrm{sink}\;\ol2}$};  
\draw[->,>=stealth'] (-1,2.5) -- (-.2,1.65);
\draw[->,>=stealth'] (-1,1.5) -- (-.2,1.5);
\draw[->,>=stealth'] (-1,0.5) -- (-.2,1.35);
\draw[->,>=stealth'] (0,1.5) -- (.8,2.35);
\draw[->,>=stealth'] (0,1.5) -- (.8,1.5);
\draw[->,>=stealth'] (0,1.5) -- (.8,0.65);
\draw[->,>=stealth'] (-1,-0.5) -- (.8,-0.5);
\node at (-1,2.5) {$\bullet$}; 
\node at (-1,1.5) {$\bullet$}; 
\node at (-1,0.5) {$\bullet$}; 
\node at (-1,-0.5) {$\bullet$};
\node at (1,2.5) {$\bullet$}; 
\node at (1,1.5) {$\bullet$}; 
\node at (1,0.5) {$\bullet$}; 
\node at (1,-0.5) {$\bullet$};
\node at (0,1.5) {$\bullet$};
\end{tikzpicture}
\ .
\end{equation*}  
For economy, we will omit edge orientations and vertices from
drawings of planar networks.
When there is no danger of confusion, we will omit source and sink labels
as well. 



Given networks $E, F \in \net{A}{[h,l]}$,
in which all sources have
outdegree $1$ and all sinks have indegree $1$, define
the concatenation $E \circ F$ of $E$ and $F$
as follows.  For $i = h,\dotsc, l$, do
\begin{enumerate}
\item remove sink $i$ of $E$ and source $i$ of $F$,
\item merge each edge $(\vertex x, \text{sink }i)$ in $E$ with each edge
  $( \text{source }i, \vertex y )$ in $F$
  to form a single edge $(\vertex x, \vertex y)$ in $E \circ F$.
\end{enumerate}
Observe that for
nonintersecting intervals $[c_1,d_1]$, $[c_2,d_2]$,
the concatenations
$F_{[c_1,d_1]} \circ F_{[c_2,d_2]}$ and 
$F_{[c_2,d_2]} \circ F_{[c_1,d_1]}$ are
isomorphic as directed graphs.
Observe also that sometimes in a concatenation $E \circ F$,
there may exist vertices $\vertex x$ in $E$, $\vertex y$ in $F$
with
$m(\vertex x, \vertex y) > 1$ multiplicity-$1$ edges
incident upon both.
Define the {\em condensed concatenation}
$E \bullet F$ to be the
subdigraph of $E \circ F$ obtained
by removing, for all such pairs $(\vertex x, \vertex y)$,
all but one of the $m(\vertex x, \vertex y)$ edges incident upon both,
and by marking this edge with the multiplicity $m(\vertex x, \vertex y)$.
For example, in $\net A{[\ol2,2]}$ we have the isomorphic graphs
\begin{equation}\label{eq:allbutone1}
F_{[\ol2,\ol1]} \circ F_{[1,2]} =
F_{[\ol2,\ol1]} \bullet F_{[1,2]} =
\begin{tikzpicture}[scale=.5,baseline=0]
\node at (-.4,1.5) {$\scriptstyle 2$};
\node at (-.4,0.5) {$\scriptstyle 1$};
\node at (-.4,-0.5) {$\scriptstyle{\ol1}$};
\node at (-.4,-1.5) {$\scriptstyle{\ol2}$};  
\node at (2.4,1.5) {$\scriptstyle 2$};
\node at (2.4,0.5) {$\scriptstyle 1$};
\node at (2.4,-0.5) {$\scriptstyle{\ol1}$};
\node at (2.4,-1.5) {$\scriptstyle{\ol2}$};  
\draw[-] (0,1.5) -- (1,1.5) -- (2,0.5);
\draw[-] (0,0.5) -- (1,0.5) -- (2,1.5);
\draw[-] (0,-0.5) -- (1,-1.5) -- (2,-1.5);
\draw[-] (0,-1.5) -- (1,-0.5) -- (2,-0.5);
\end{tikzpicture}
\cong
\begin{tikzpicture}[scale=.5,baseline=0]
\node at (-.4,1.5) {$\scriptstyle 2$};
\node at (-.4,0.5) {$\scriptstyle 1$};
\node at (-.4,-0.5) {$\scriptstyle{\ol1}$};
\node at (-.4,-1.5) {$\scriptstyle{\ol2}$};  
\node at (1.4,1.5) {$\scriptstyle 2$};
\node at (1.4,0.5) {$\scriptstyle 1$};
\node at (1.4,-0.5) {$\scriptstyle{\ol1}$};
\node at (1.4,-1.5) {$\scriptstyle{\ol2}$};  
\draw[-] (0,1.5) -- (1,0.5);
\draw[-] (0,0.5) -- (1,1.5);
\draw[-] (0,-0.5) -- (1,-1.5);
\draw[-] (0,-1.5) -- (1,-0.5);
\end{tikzpicture}
\;\cong\;
F_{[1,2]} \circ F_{[\ol2,\ol1]} =
F_{[1,2]} \bullet F_{[\ol2,\ol1]},
\end{equation}
and the nonisomorphic graphs
\begin{equation}\label{eq:allbutone2}
  F_{[\ol2,1]} \circ F_{[\ol1,2]} \circ F_{[\ol2,1]} =
\begin{tikzpicture}[scale=.5,baseline=0]
\node at (-.4,1.5) {$\scriptstyle 2$};
\node at (-.4,0.5) {$\scriptstyle 1$};
\node at (-.4,-0.5) {$\scriptstyle{\ol1}$};
\node at (-.4,-1.5) {$\scriptstyle{\ol2}$};  
\node at (3.4,1.5) {$\scriptstyle 2$};
\node at (3.4,0.5) {$\scriptstyle 1$};
\node at (3.4,-0.5) {$\scriptstyle{\ol1}$};
\node at (3.4,-1.5) {$\scriptstyle{\ol2}$};  
\draw[-] (0,1.5) -- (1,1.5) -- (2,-0.5) -- (3,-0.5);
\draw[-] (0,0.5) -- (1,-1.5) -- (2,-1.5) -- (3,0.5);
\draw[-] (0,-0.5) -- (1,-0.5) -- (2,1.5) -- (3,1.5);
\draw[-] (0,-1.5) -- (1,0.5) -- (2,0.5) -- (3,-1.5);
\end{tikzpicture}
\; \not \cong \;
  F_{[\ol2,1]} \bullet F_{[\ol1,2]} \bullet F_{[\ol2,1]} =
\begin{tikzpicture}[scale=.5,baseline=0]
\node at (-0.9,1.5) {$\scriptstyle 2$};
\node at (-0.9,0.5) {$\scriptstyle 1$};
\node at (-0.9,-0.5) {$\scriptstyle{\ol1}$};
\node at (-0.9,-1.5) {$\scriptstyle{\ol2}$};  
\node at (2.15,-.35) {$\scriptstyle (2)$};
\node at (.85,-.35) {$\scriptstyle (2)$};
\node at (3.9,1.5) {$\scriptstyle 2$};
\node at (3.9,0.5) {$\scriptstyle 1$};
\node at (3.9,-0.5) {$\scriptstyle{\ol1}$};
\node at (3.9,-1.5) {$\scriptstyle{\ol2}$};  
\draw[-] (-0.5,1.5) -- (1,1.5) -- (1.5,0.5) -- (2,1.5) -- (3.5,1.5);
\draw[-] (-0.5,0.5) -- (-0,-0.5) -- (1.5,0.5) -- (3,-0.5) -- (3.5,0.5);
\draw[-] (-0.5,-0.5) -- (-0,-0.5) -- (1.5,0.5) -- (3,-0.5) -- (3.5,-0.5);
\draw[-] (-0.5,-1.5) -- (-0,-0.5) -- (.5,-1.5) -- (2.5,-1.5) -- (3,-0.5) -- (3.5,-1.5);
\end{tikzpicture},
\end{equation}
in which two pairs of edges are replaced by two single edges
marked with multiplicity $2$.
We refer to all iterations of concatenations and condensed concatenations
of simple star networks as {\em star networks}.
In fact, each element of $\net A{[h,l]}$ is isomorphic to a star network,
so we may think of $\net A{[h,l]}$ as a set of
star networks.


Given planar network $F \in \net A{[h,l]}$, define
its {\em path matrix} $A = A(F) = (a_{i,j})_{i,j \in [h,l]}$ by
\begin{equation}\label{eq:pathmatrix}
  a_{i,j} = \# \text{ paths in $F$ from source $i$ to sink $j$},
\end{equation}
ignoring multiplicities.
For instance, the star networks in
(\ref{eq:allbutone1}) -- (\ref{eq:allbutone2})
have path matrices
\begin{equation*}
  \begin{bmatrix}
    a_{\ol2,\ol2} & a_{\ol2, \ol1} & a_{\ol2, 1} & a_{\ol2, 2} \\
    a_{\ol1,\ol2} & a_{\ol1, \ol1} & a_{\ol1, 1} & a_{\ol1, 2} \\
    a_{1,\ol2}   & a_{1, \ol1}   & a_{1, 1}   & a_{1, 2} \\
    a_{2,\ol2}   & a_{2, \ol1}   & a_{2, 1}   & a_{2, 2} 
  \end{bmatrix}
  \quad
  = 
  \quad
  \begin{bmatrix}
    1 & 1 & 0 & 0 \\
    1 & 1 & 0 & 0 \\
    0 & 0 & 1 & 1 \\
    0 & 0 & 1 & 1     
  \end{bmatrix}\ntksp,
  \
  \begin{bmatrix}
    5 & 5 & 5 & 2 \\
    5 & 5 & 5 & 2 \\
    5 & 5 & 5 & 2 \\
    2 & 2 & 2 & 1
  \end{bmatrix}\ntksp,
  \
  \begin{bmatrix}
    2 & 2 & 2 & 1 \\
    2 & 2 & 2 & 1 \\
    2 & 2 & 2 & 1 \\
    1 & 1 & 1 & 1     
  \end{bmatrix}\ntksp,
\end{equation*}
repectively.

\begin{defn}\label{d:starnet}
Define $\snet{A}{[h,l]}$ to be set of all type-$\msfA$ planar networks of the
form
\begin{equation}\label{eq:bulletconcat}
    F = F_{[c_1,d_1]} \bullet \cdots \bullet F_{[c_t,d_t]}, \\
\end{equation}
and call these
{\em type-$\msfA$ condensed star networks
  (with boundary vertices indexed by $[h,l]$)}.
\end{defn}


We will be interested in two subclasses of these,
which we define as follows.
\begin{defn}\label{d:zz}
  Call a type-$\msfA$ condensed star network $F$
  (\ref{eq:bulletconcat})  
  a {\em type-$\msfA$ zig-zag network}
  if we have $F=F_\emptyset$ or 
  \begin{enumerate}
\item the intervals $[c_1,d_1], \dotsc, [c_t,d_t]$
  are distinct and pairwise nonnesting,
\item for all triples
  $i < j < k$ satisfying
  $[c_i,d_i] \cap [c_j,d_j] \neq \emptyset$ and 
$[c_j,d_j] \cap [c_k,d_k] \neq \emptyset$,
  we have
$c_i < c_j < c_k$ 
(and $d_i < d_j < d_k$) 
or $c_i > c_j > c_k$
(and $d_i > d_j > d_k$).
  \end{enumerate}
\end{defn}
\noindent
Let $\znet{A}{[h,l]}$ denote the set of type-$\msfA$ zig-zag networks
  with boundary vertices indexed by $[h,l]$.
\begin{defn}\label{d:adsn}
  Call a type-$\msfA$ condensed star network $F$ (\ref{eq:bulletconcat})  
  a {\em type-$\msfA$ descending star network}
  if we have $F=F_\emptyset$ or 
  \begin{enumerate}
\item the intervals $[c_1,d_1], \dotsc, [c_t,d_t]$
  are distinct and pairwise nonnesting,
\item for all pairs
  $i < j$ satisfying
  $[c_i,d_i] \cap [c_j,d_j] \neq \emptyset$
  we have $c_i > c_j$ (and $d_i > d_j$).
  \end{enumerate}
\end{defn}
\noindent
Let $\dnet{A}{[h,l]}$ denote the set of type-$\msfA$ descending star networks
with boundary vertices indexed by $[h,l]$.
Thus we have $\dnet{A}{[h,l]} \subseteq \znet{A}{[h,l]}$.
To illustrate, let us fix boundary vertices
indexed by any interval of cardinality $4$.
Then we have $14$ descending star networks,
\begin{equation}\label{eq:xfigures2}
\raisebox{-6mm}{
  \includegraphics[height=12mm]{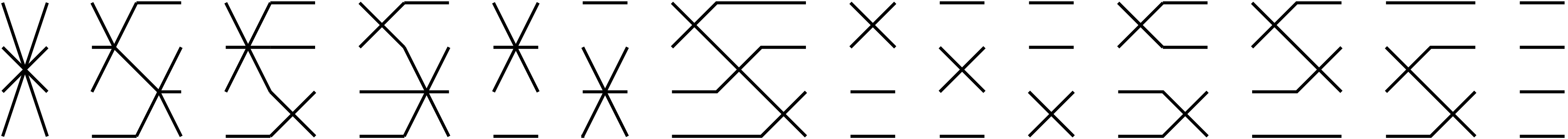}}\ ,
\hspace{-5in}{\tiny{^{_{(2)}}}} \hspace{5in}
\end{equation}
and $8$ more zig-zag networks
which are not descending star networks,
\begin{equation}\label{eq:xfigureszz}
\raisebox{-6mm}{
  \includegraphics[height=12mm]{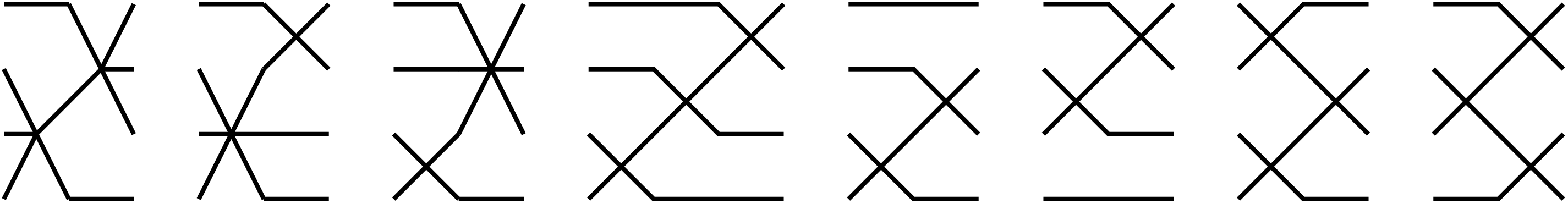}}\ .
\hspace{-3.7in}{\tiny{^{_{(2)}}}} \hspace{3.7in}
\end{equation}

The result \cite[Lem.\,3.5]{CHSSkanEKL}
describes intersections of paths in a descending star network.
\begin{lem}\label{l:lemma3.5}
  Let $\pi_{i_1}$, $\pi_{i_2}$ be paths in a descending star network $F$
  from sources $i_1 < i_2$ to sinks $m_1$, $m_2$, respectively.
  Then the two paths intersect if and only if there exists a path in $F$
  from $i_1$ to sink $m_2$.
  \end{lem}

By \cite[Thm.\,3.5, Lem.\,5.3]{SkanNNDCB}
and \cite[Thm.\,3.6]{CHSSkanEKL},
the sets
$\dnet{A}{[h,l]}$, $\znet{A}{[h,l]}$
are related to pattern avoidance in $\mfs{[h,l]}$.
\begin{prop}\label{p:anetpavoid}
There is a natural bijection $F \mapsto w(F)$ from 
$\znet{A}{[h,l]}$
to \pavoiding permutations in $\mfs{[h,l]}$,
which restricts to a bijection from $\dnet{A}{[h,l]}$ to
$312$-avoiding permutations in $\mfs{[h,l]}$.
\end{prop}
To describe
the bijection
explicitly
we
define
a relation $\precdot$
on the set of intervals appearing in (\ref{eq:bulletconcat})
by
declaring
\begin{equation}\label{eq:precdotdef}
  [c_i,d_i] \precdot [c_j,d_j]
  \end{equation}
if 
$i < j$ and
$[c_i,d_i] \cap [c_j,d_j] \ssm ( [c_{i+1},d_{i+1}] \cup \cdots \cup [c_{j-1},d_{j-1}]) \neq \emptyset$.
  The relation $\precdot$ may be viewed as an acyclic directed graph on 
  the intervals.
  The transitive, reflexive closure of $\precdot$ is a partial order $\preceq$.
  For $F \in \znet{A}{[h,l]}$, the directed graph is the Hasse diagram
  of the partial order; for other $F \in \snet A{[h,l]}$ this is not the case.
  For example, the networks
  $F_{[2,5]} \bullet F_{[1,3]} \bullet F_{[4,6]} \bullet F_{[6,7]}
  \in \znet A{[1,7]}$
  and
  $F_{[2,5]} \bullet F_{[1,2]} \bullet F_{[4,6]} \bullet F_{[2,5]}
  \in \snet{A}{[1,7]}$
and their corresponding interval digraphs and posets are
\begin{equation}\label{eq:posetex}
  \begin{gathered}
\begin{tikzpicture}[scale=.5,baseline=55]
  \node at (-.4,7) {$\scriptstyle 7$};
  \node at (-.4,6) {$\scriptstyle 6$};
  \node at (-.4,5) {$\scriptstyle 5$};
  \node at (-.4,4) {$\scriptstyle 4$};
  \node at (-.4,3) {$\scriptstyle 3$};
  \node at (-.4,2) {$\scriptstyle 2$};
  \node at (-.4,1) {$\scriptstyle 1$};
  \node at (3.4,7) {$\scriptstyle 7$};
  \node at (3.4,6) {$\scriptstyle 6$};
  \node at (3.4,5) {$\scriptstyle 5$};
  \node at (3.4,4) {$\scriptstyle 4$};
  \node at (3.4,3) {$\scriptstyle 3$};
  \node at (3.4,2) {$\scriptstyle 2$};
  \node at (3.4,1) {$\scriptstyle 1$};
\draw[-] (0,7) -- (2,7) -- (3,6);
\draw[-] (0,6) -- (1,6) -- (2,4) -- (3,4);
\draw[-] (0,5) -- (0.5,3.5) -- (1.5,2) -- (2,1) -- (3,1);
\draw[-] (0,4) -- (0.5,3.5) -- (1.5,2) -- (3,2);
\draw[-] (0,3) -- (0.5,3.5) -- (1.5,5) -- (3,5);
\draw[-] (0,2) -- (0.5,3.5) -- (1.5,5) -- (2.5,6.5) -- (3,7);
\draw[-] (0,1) -- (1,1) -- (2,3) -- (3,3);
\node at (0.75,4.6) {$\scriptstyle{(2)}$};
\node at (0.75,2.4) {$\scriptstyle{(2)}$};
\end{tikzpicture},
  \qquad \qquad
\begin{tikzpicture}[scale=.5,baseline=15]
\node at (0,1) {$[1, 3]$};            
\node at (-1.5,-1) {$[2,5]$};          
\node at (-3,1) {$[4,6]$};            
\node at (-3,3) {$[6,7]$};            
\draw[->] (-3,1.5) -- (-3,2.5);
\draw[<-] (0,0.5) -- (-1.5,-0.5);
\draw[<-] (-3,0.5) -- (-1.5,-0.5);
\end{tikzpicture},
  \qquad \qquad
\begin{tikzpicture}[scale=.5,baseline=15]
\node at (0,1) {$[1, 3]$};            
\node at (-1.5,-1) {$[2,5]$};          
\node at (-3,1) {$[4,6]$};            
\node at (-3,3) {$[6,7]$};            
\draw[-] (-3,1.5) -- (-3,2.5);
\draw[-] (0,0.5) -- (-1.5,-0.5);
\draw[-] (-3,0.5) -- (-1.5,-0.5);
\end{tikzpicture}\qquad\qquad \\
\qquad\qquad
\begin{tikzpicture}[scale=.5,baseline=55]
  \node at (-.4,7) {$\scriptstyle 7$};
  \node at (-.4,6) {$\scriptstyle 6$};
  \node at (-.4,5) {$\scriptstyle 5$};
  \node at (-.4,4) {$\scriptstyle 4$};
  \node at (-.4,3) {$\scriptstyle 3$};
  \node at (-.4,2) {$\scriptstyle 2$};
  \node at (-.4,1) {$\scriptstyle 1$};
  \node at (3.4,7) {$\scriptstyle 7$};
  \node at (3.4,6) {$\scriptstyle 6$};
  \node at (3.4,5) {$\scriptstyle 5$};
  \node at (3.4,4) {$\scriptstyle 4$};
  \node at (3.4,3) {$\scriptstyle 3$};
  \node at (3.4,2) {$\scriptstyle 2$};
  \node at (3.4,1) {$\scriptstyle 1$};
\draw[-] (0,7) -- (3,7);
\draw[-] (0,6) -- (1,6) -- (1.5,5) -- (2.5,3.5) -- (3,2);
\draw[-] (0,5) -- (1,2) -- (2,1) -- (3,1);
\draw[-] (0,4) -- (1,3) -- (2,3) -- (3,4);
\draw[-] (0,3) -- (0.5,3.5) -- (1.5,5) -- (2,6) -- (3,6);
\draw[-] (0,2) -- (0.5,3.5);
\draw[-] (2.5,3.5) -- (3,3);
\draw[-] (0,1) -- (1,1) -- (2,2) -- (3,5);
\node at (0.75,4.6) {$\scriptstyle{(2)}$};
\node at (2.3,4.6) {$\scriptstyle{(2)}$};
\end{tikzpicture},
  \qquad \qquad
\begin{tikzpicture}[scale=.5,baseline=15]
\node at (0,1) {$[1, 2]$};            
\node at (-1.5,-1) {$[2,5]$};         
\node at (-3,1) {$[4,6]$};            
\node at (-1.5,3) {$[2,5]$};            
\draw[->] (-3,1.5) -- (-1.6,2.5);
\draw[->] (0,1.5) -- (-1.4,2.5);
\draw[->] (-1.5,-0.5) -- (-1.5,2.5);
\draw[<-] (0,0.5) -- (-1.5,-0.5);
\draw[<-] (-3,0.5) -- (-1.5,-0.5);
\end{tikzpicture},
  \qquad \qquad
\begin{tikzpicture}[scale=.5,baseline=15]
\node at (0,1) {$[1,2]$};            
\node at (-1.5,-1) {$[2,5]$};          
\node at (-3,1) {$[4,6]$};            
\node at (-1.5,3) {$[2,5]$};            
\draw[-] (-3,1.5) -- (-1.5,2.5);
\draw[-] (0,1.5) -- (-1.5,2.5);
\draw[-] (0,0.5) -- (-1.5,-0.5);
\draw[-] (-3,0.5) -- (-1.5,-0.5);
\end{tikzpicture}.
\end{gathered}
\end{equation}
The bijection $F \mapsto w(F)$,
stated in
\cite[\S 3]{SkanNNDCB},
is given by the following algorithm.
\begin{alg}\label{a:FwF}
  Given $F$ as in (\ref{eq:bulletconcat}), do
  \begin{enumerate}
\item 
  Initialize the sequence of reversals
  $S \defeq (s_{[c_1,d_1]}, \dotsc, s_{[c_t,d_t]})$.
\item 
  For all pairs $(i,j)$ with $[c_i,d_i] \precdot [c_j,d_j]$ and
  $|[c_i,d_i] \cap [c_j,d_j]| > 1$,
 \begin{enumerate}[]
 \item Update $S$ by inserting $s_{[c_i,d_i] \cap [c_j,d_j]}$
   immediately after $s_{[c_i,d_i]}$.
 \end{enumerate}
\item Define $w(F)$ to be the product of reversals in $S$, from left to right.
\end{enumerate}
\end{alg}
We call the final sequence of reversals a
{\em zig-zag factorization} of $w(F)$.
For example, let $F$ be the first star network in (\ref{eq:posetex}).
This zig-zag network $F$ gives the
reversal sequence $(s_{[2,5]}, s_{[1,3]}, s_{[4,6]}, s_{[6,7]})$ which we update
by inserting $s_{[2,5] \cap [1,3]} = s_{[2,3]}$ after $s_{[2,5]}$, and then
$s_{[2,5] \cap [4,6]} = s_{[4,5]}$ after $s_{[2,5]}$ to obtain
the permutation
\begin{equation*}
  w = w(F) = s_{[2,5]}s_{[4,5]}s_{[2,3]}s_{[1,3]}s_{[4,6]}s_{[6,7]} = 3752146.
  \end{equation*}

The inverse of the map $F \mapsto w(F)$,
which we write
\begin{equation}\label{eq:wFw}
  w \mapsto F_w,
\end{equation}
is a bit intricate and is given in~\cite[\S 3]{SkanNNDCB}.
It turns out that the network $F$ above is $F_{3752146}$.
In (\ref{eq:xfigures2}), if we label sources and sinks $1, 2, 3, 4$ from bottom
to top, the descending star networks are
\begin{equation}\label{eq:dsnlist}
  F_{4321},F_{3421},F_{2431},F_{3241},F_{1432},F_{3214},F_{2341},F_{1243},F_{1324},F_{2134},
  F_{2143},F_{1342},F_{2314},F_{1234},
\end{equation}
respectively.
In (\ref{eq:xfigureszz}), the remaining zig-zag networks are
\begin{equation}\label{eq:zzlist}
  F_{4312}, F_{4213}, F_{4132}, F_{4123}, F_{3124}, F_{1423}, F_{1423}, F_{3142},F_{2413}.
  \end{equation}


The restriction of the map (\ref{eq:wFw}) to $312$-avoiding elements of $\sn$
is in fact rather simple.  Given word $w = w_1 \cdots w_n$
with distinct letters,
say that $w$ has a {\em record} at position $j$
if $w_j = \max\{ w_1, \dotsc, w_j \}$.

\begin{alg}\label{a:wFw}
  Given $w = w_1 \cdots w_n \in \sn$ avoiding the pattern $312$, do
  \begin{enumerate}
  \item Let $w$ have records at positions $1 = j_1, \dotsc, j_k$.
  \item Define
    $F_w = F_{[j_k,w_{j_k}]}
    \bullet \cdots \bullet F_{[j_1,w_{j_1}]}$.
  \end{enumerate}
\end{alg}

The bijection $F \mapsto w(F)$ is closely related to
families of source-to-sink paths in $F$, and also to Kazhdan--Lusztig
basis elements of the Hecke algebra of $\mfs{[h,l]}$. 
Given $F \in \net A{[h,l]}$,
call a sequence $\pi = (\pi_h,\dotsc,\pi_l)$ of paths in $F$
a {\em path family of type $w = w_h \cdots w_l$} if for $i = h,\dotsc,l$,
path $\pi_i$ begins at source $i$ and ends at sink $w_i$.
Say that a path family $\pi$ {\em covers} $F$ if every edge of $F$
appears in at least one path of $\pi$, and define the sets
\begin{equation}
  \begin{gathered}
    \Pi(F) = \{ \pi \,|\, \pi \text{ covers } F \},\\
    \Pi_w(F) = \{ \pi \in \Pi(F) \,|\, \type(\pi) = w \}.
  \end{gathered}
\end{equation}
For example,
the star network
and path family
\begin{equation}\label{eq:132313}
F = F_{[1,3]} \circ F_{[2,3]} \circ F_{[1,3]} = 
\begin{tikzpicture}[scale=.5,baseline=-10]
\node at (-.4,0.5) {$\scriptstyle 3$};
\node at (-.4,-0.5) {$\scriptstyle{2}$};
\node at (-.4,-1.5) {$\scriptstyle{1}$};  
\node at (3.4,0.5) {$\scriptstyle 3$};
\node at (3.4,-0.5) {$\scriptstyle{2}$};
\node at (3.4,-1.5) {$\scriptstyle{1}$};  
\draw[-] (0,0.5) -- (1,-1.5) -- (2,-1.5) -- (3,0.5);
\draw[-] (0,-0.5) -- (1,-0.5) -- (2,0.5) -- (3,-1.5);
\draw[-] (0,-1.5) -- (1,0.5) -- (2,-0.5) -- (3,-0.5);
\end{tikzpicture},\qquad
\pi = 
\begin{tikzpicture}[scale=.5,baseline=-10]
\node at (-.4,0.5) {$\scriptstyle 3$};
\node at (-.4,-0.5) {$\scriptstyle{2}$};
\node at (-.4,-1.5) {$\scriptstyle{1}$};  
\node at (3.4,0.5) {$\scriptstyle 3$};
\node at (3.4,-0.5) {$\scriptstyle{2}$};
\node at (3.4,-1.5) {$\scriptstyle{1}$};  
\draw[-, very thick] (0,0.5) -- (1,-1.5) -- (2,-1.5) -- (3,0.5);
\draw[-, thick, dashed] (0,-0.5) -- (1,-0.5) -- (2,0.5) -- (3,-1.5);
\draw[-] (0,-1.5) -- (1,0.5) -- (2,-0.5) -- (3,-0.5);
\end{tikzpicture}
\end{equation}
belong to $\net A{[1,3]}$ and $\Pi_{s_1}(F) \subset \Pi(F)$, respectively.
%
When $F$ is a zig-zag network,
we may characterize $w(F)$ in terms of $\Pi(F)$ as
follows~\cite[Lem.\,5.3]{SkanNNDCB}.
\begin{prop}\label{p:wFchar}  
  For $F \in \znet{A}{[h,l]}$,
$w(F)$ is the unique permutation of maximum length in
$\{ \type(\pi) \,|\, \pi \in \Pi(F) \} \subseteq \mfs{[h,l]}$.
\end{prop}

For all $F \in \net A{[h,l]}$, the set $\Pi(F)$ associates an element of
$\mathbb Z[\mfs{[h,l]}]$ to $F$:
we say that $F$ {\em graphically represents}
\begin{equation}\label{eq:Gtozsn}
\sum_{\pi \in \Pi(F)} \type(\pi)
\end{equation}
{\em as an element of $\mathbb Z[\mfs{[h,l]}]$}.
For example, the network $F$ in (\ref{eq:132313}) can be covered by
$72$ different path families: $12$ of each type $w \in \mfs 3$.
Thus it graphically represents
$12\,\wtc{s_1s_2s_1}1$
as an element of $\mathbb Z [\mfs 3]$.

Again for all $F \in \net A{[h,l]}$,
the set $\Pi(F)$ also associates an element of
$H_{[h,l]}(q)$ to $F$.
To describe this element explicitly, we first assume
that $F$ is formed by some iteration of ordinary or condensed concatenatation
of simple star networks $F_{[c_1,d_1]}, \dotsc, F_{[c_t,d_t]}$ with
internal vertices $\vertex z_1, \dotsc, \vertex z_t$. Observe that
the intersection of two source-to-sink paths $\pi_i$, $\pi_j$ in $F$
must be a disjoint union of the above internal vertices of $F$
and paths between these.
We say that $\pi_i$ and $\pi_j$ {\em meet} at a central vertex $\vertex z_k$
if both paths contain $\vertex z_k$, and enter it via different edges.
Given a path family $\pi$ covering $F$,
define a {\em defect} of $\pi$ to be a triple $(\pi_i, \pi_j, k)$ with
\begin{enumerate}
  \item $i < j$,
  \item $\pi_i$ and $\pi_j$ meet at vertex $\vertex z_k$ of
    $F$ after having crossed an odd number of times.
\end{enumerate}
Let $\dfct(\pi)$ denote the number of defects of $\pi$.
(This definition from \cite{CSkanTNNChar}
generalizes those of \cite{BWHex}, \cite{Deodhar90}.)
We say that $F$ {\em graphically represents}
\begin{equation}\label{eq:Gtozhnq}
\sum_{\pi \in \Pi(F)} q^{\dfct(\pi)} T_{\type(\pi)}
\end{equation}
{\em as an element of $H_{[h,l]}(q)$}.
For example,
the path family $\pi$ in (\ref{eq:132313})
satisfies $\dfct(\pi) = 3$: the defects are
$(\pi_1, \pi_2, 2)$, $(\pi_1, \pi_3, 3)$, $(\pi_2, \pi_3, 3)$.
It is possible to show that the network $F$ in (\ref{eq:132313})
graphically represents
$(1+q)^2(1+q+q^2) \wtc{s_1s_2s_1}q$
as an element of $H_3(q)$.

It is clear that if $F$ graphically represents
$D(q) \in H_{[h,l]}(q)$ as an element of $H_{[h,l]}(q)$, then
it graphically represents
$D(1)$ as an element of $\mathbb Z[\mfs{[h,l]}]$.
It is possible to show that
all star networks
graphically represent products of Kazhdan--Lusztig basis elements
(possibly divided by integers or polynomials in $q$).
In particular, the result~\cite[Thm.\,1]{BWHex}
shows that sometimes such a product consists of a single 
Kazhdan--Lusztig basis element, and that the star network is a
{\em wiring diagram}, i.e., all intervals $[c_i,d_i]$ satisfy $d_i = c_i + 1$.
\begin{prop}
  Let
  $s_{c_1} \cdots s_{c_t}$ be a reduced expression for
  $w \in \mfs{[h,l]}$
  avoiding the patterns
  $321$, $56781234$, $56718234$, $46781235$, $46718235$.
  Then the star network
  $F_{[c_1,c_1+1]} \bullet \cdots \bullet F_{[c_t, c_t+1]}$
  graphically represents $\smash{\wtc wq}$
  as an element of
  $H_{[h,l]}(q)$.
\end{prop}
The result \cite[Lem.\,5.3]{SkanNNDCB} shows that zig-zag networks
give graphical representations of other Kazhdan--Lusztig basis elements.
\begin{prop}\label{p:pavoidrep}
  For $w \in \mfs{[h,l]}$ \avoidingp,
  the zig-zag network $F_w$ graphically represents $\wtc wq$
  as an element of $H_{[h,l]}(q)$.
\end{prop}

This fact has the following consequence.
\begin{cor}\label{c:atmostonepath}
      For $v, w \in \mfs{[h,l]}$ with $w$ \avoidingp,
      the number of path families of type $v$ covering $F_w$ is
      $1$ if $v \leq_{\mfs{[h,l]}} w$, and is $0$ otherwise.
\end{cor}

\ssec{Type-$\msfBC$ planar networks and factoriztion}\label{ss:BCplanar}



For fixed $n$, define {\em type-$\msfBC$ simple star networks
  with boundary vertices indexed by $[\ol n, n]$} to be the
type-$\msfA$ star networks
\begin{equation}\label{eq:bcstardef}
  \begin{alignedat}{2}
    F'_{[a,b]} &\defeq F_{[a,b]} \bullet F_{[\ol a, \ol b]}
    = F_{[a,b]} \circ F_{[\ol a, \ol b]},
    &\qquad &1 \leq a \leq b \leq n, \\ 
    F'_{[\ol a, a]} &\defeq F_{[\ol a, a]},
    &\qquad &1 \leq a \leq n,
    \end{alignedat}
\end{equation}
which correspond naturally to the type-$\msfBC$ reversals (\ref{eq:creversal}).
For example the seven type-$\msfBC$
simple star networks
$F'_{\emptyset} = F'_{[1,1]} = F'_{[2,2]} = F'_{[3,3]}$ and $ F'_{[1,2]}, F'_{[2,3]}, F'_{[1,3]}, F'_{\smash{[\ol 1, 1]}}, F'_{\smash{[\ol 2, 2]}}, F'_{\smash{[\ol 3, 3]}}$,
\begin{equation}\label{eq:csimplestars}        
\begin{tikzpicture}[scale=.5,baseline=0]
  \node at (-.9,2.5) {$\scriptstyle 3$};
  \node at (-.9,1.5) {$\scriptstyle 2$};
  \node at (-.9,0.5) {$\scriptstyle 1$};
  \node at (0.9,2.5) {$\scriptstyle 3$};
  \node at (0.9,1.5) {$\scriptstyle 2$};
  \node at (0.9,0.5) {$\scriptstyle 1$};
  \node at (-.9,-2.5) {$\scriptstyle {\ol 3}$};
  \node at (-.9,-1.5) {$\scriptstyle {\ol 2}$};
  \node at (-.9,-0.5) {$\scriptstyle {\ol 1}$};
  \node at (0.9,-2.5) {$\scriptstyle {\ol 3}$};
  \node at (0.9,-1.5) {$\scriptstyle {\ol 2}$};
  \node at (0.9,-0.5) {$\scriptstyle {\ol 1}$};
\draw[-] (-.5,2.5) -- (.5,2.5);
\draw[-] (-.5,1.5) -- (.5,1.5);
\draw[-] (-.5,0.5) -- (.5,0.5);
\draw[-] (-.5,-0.5) -- (.5,-0.5);
\draw[-] (-.5,-1.5) -- (.5,-1.5);
\draw[-] (-.5,-2.5) -- (.5,-2.5);
\end{tikzpicture}
\ntnsp,\quad
\begin{tikzpicture}[scale=.5,baseline=0]
  \node at (-.9,2.5) {$\scriptstyle 3$};
  \node at (-.9,1.5) {$\scriptstyle 2$};
  \node at (-.9,0.5) {$\scriptstyle 1$};
  \node at (0.9,2.5) {$\scriptstyle 3$};
  \node at (0.9,1.5) {$\scriptstyle 2$};
  \node at (0.9,0.5) {$\scriptstyle 1$};
  \node at (-.9,-2.5) {$\scriptstyle {\ol 3}$};
  \node at (-.9,-1.5) {$\scriptstyle {\ol 2}$};
  \node at (-.9,-0.5) {$\scriptstyle {\ol 1}$};
  \node at (0.9,-2.5) {$\scriptstyle {\ol 3}$};
  \node at (0.9,-1.5) {$\scriptstyle {\ol 2}$};
  \node at (0.9,-0.5) {$\scriptstyle {\ol 1}$};
\draw[-] (-.5,2.5) -- (.5,2.5);
\draw[-] (-.5,1.5) -- (.5,0.5);
\draw[-] (-.5,0.5) -- (.5,1.5);
\draw[-] (-.5,-0.5) -- (.5,-1.5);
\draw[-] (-.5,-1.5) -- (.5,-0.5);
\draw[-] (-.5,-2.5) -- (.5,-2.5);
\end{tikzpicture}
\ntnsp,\quad
\begin{tikzpicture}[scale=.5,baseline=0]
  \node at (-.9,2.5) {$\scriptstyle 3$};
  \node at (-.9,1.5) {$\scriptstyle 2$};
  \node at (-.9,0.5) {$\scriptstyle 1$};
  \node at (0.9,2.5) {$\scriptstyle 3$};
  \node at (0.9,1.5) {$\scriptstyle 2$};
  \node at (0.9,0.5) {$\scriptstyle 1$};
  \node at (-.9,-2.5) {$\scriptstyle {\ol 3}$};
  \node at (-.9,-1.5) {$\scriptstyle {\ol 2}$};
  \node at (-.9,-0.5) {$\scriptstyle {\ol 1}$};
  \node at (0.9,-2.5) {$\scriptstyle {\ol 3}$};
  \node at (0.9,-1.5) {$\scriptstyle {\ol 2}$};
  \node at (0.9,-0.5) {$\scriptstyle {\ol 1}$};
\draw[-] (-.5,2.5) -- (.5,1.5);
\draw[-] (-.5,1.5) -- (.5,2.5);
\draw[-] (-.5,0.5) -- (.5,0.5);
\draw[-] (-.5,-0.5) -- (.5,-0.5);
\draw[-] (-.5,-1.5) -- (.5,-2.5);
\draw[-] (-.5,-2.5) -- (.5,-1.5);
\end{tikzpicture}
\ntnsp,\quad
\begin{tikzpicture}[scale=.5,baseline=0]
  \node at (-.9,2.5) {$\scriptstyle 3$};
  \node at (-.9,1.5) {$\scriptstyle 2$};
  \node at (-.9,0.5) {$\scriptstyle 1$};
  \node at (0.9,2.5) {$\scriptstyle 3$};
  \node at (0.9,1.5) {$\scriptstyle 2$};
  \node at (0.9,0.5) {$\scriptstyle 1$};
  \node at (-.9,-2.5) {$\scriptstyle {\ol 3}$};
  \node at (-.9,-1.5) {$\scriptstyle {\ol 2}$};
  \node at (-.9,-0.5) {$\scriptstyle {\ol 1}$};
  \node at (0.9,-2.5) {$\scriptstyle {\ol 3}$};
  \node at (0.9,-1.5) {$\scriptstyle {\ol 2}$};
  \node at (0.9,-0.5) {$\scriptstyle {\ol 1}$};
\draw[-] (-.5,2.5) -- (.5,0.5);
\draw[-] (-.5,1.5) -- (.5,1.5);
\draw[-] (-.5,0.5) -- (.5,2.5);
\draw[-] (-.5,-0.5) -- (.5,-2.5);
\draw[-] (-.5,-1.5) -- (.5,-1.5);
\draw[-] (-.5,-2.5) -- (.5,-0.5);
\end{tikzpicture}
\ntnsp,\quad
\begin{tikzpicture}[scale=.5,baseline=0]
  \node at (-.9,2.5) {$\scriptstyle 3$};
  \node at (-.9,1.5) {$\scriptstyle 2$};
  \node at (-.9,0.5) {$\scriptstyle 1$};
  \node at (0.9,2.5) {$\scriptstyle 3$};
  \node at (0.9,1.5) {$\scriptstyle 2$};
  \node at (0.9,0.5) {$\scriptstyle 1$};
  \node at (-.9,-2.5) {$\scriptstyle {\ol 3}$};
  \node at (-.9,-1.5) {$\scriptstyle {\ol 2}$};
  \node at (-.9,-0.5) {$\scriptstyle {\ol 1}$};
  \node at (0.9,-2.5) {$\scriptstyle {\ol 3}$};
  \node at (0.9,-1.5) {$\scriptstyle {\ol 2}$};
  \node at (0.9,-0.5) {$\scriptstyle {\ol 1}$};
\draw[-] (-.5,2.5) -- (.5,2.5);
\draw[-] (-.5,1.5) -- (.5,1.5);
\draw[-] (-.5,0.5) -- (.5,-0.5);
\draw[-] (-.5,-0.5) -- (.5,0.5);
\draw[-] (-.5,-1.5) -- (.5,-1.5);
\draw[-] (-.5,-2.5) -- (.5,-2.5);
\end{tikzpicture}
\ntnsp,\quad
\begin{tikzpicture}[scale=.5,baseline=0]
  \node at (-.9,2.5) {$\scriptstyle 3$};
  \node at (-.9,1.5) {$\scriptstyle 2$};
  \node at (-.9,0.5) {$\scriptstyle 1$};
  \node at (0.9,2.5) {$\scriptstyle 3$};
  \node at (0.9,1.5) {$\scriptstyle 2$};
  \node at (0.9,0.5) {$\scriptstyle 1$};
  \node at (-.9,-2.5) {$\scriptstyle {\ol 3}$};
  \node at (-.9,-1.5) {$\scriptstyle {\ol 2}$};
  \node at (-.9,-0.5) {$\scriptstyle {\ol 1}$};
  \node at (0.9,-2.5) {$\scriptstyle {\ol 3}$};
  \node at (0.9,-1.5) {$\scriptstyle {\ol 2}$};
  \node at (0.9,-0.5) {$\scriptstyle {\ol 1}$};
\draw[-] (-.5,2.5) -- (.5,2.5);
\draw[-] (-.5,1.5) -- (.5,-1.5);
\draw[-] (-.5,0.5) -- (.5,-0.5);
\draw[-] (-.5,-0.5) -- (.5,0.5);
\draw[-] (-.5,-1.5) -- (.5,1.5);
\draw[-] (-.5,-2.5) -- (.5,-2.5);
\end{tikzpicture}
\ntnsp,\quad
\begin{tikzpicture}[scale=.5,baseline=0]
  \node at (-.9,2.5) {$\scriptstyle 3$};
  \node at (-.9,1.5) {$\scriptstyle 2$};
  \node at (-.9,0.5) {$\scriptstyle 1$};
  \node at (0.9,2.5) {$\scriptstyle 3$};
  \node at (0.9,1.5) {$\scriptstyle 2$};
  \node at (0.9,0.5) {$\scriptstyle 1$};
  \node at (-.9,-2.5) {$\scriptstyle {\ol 3}$};
  \node at (-.9,-1.5) {$\scriptstyle {\ol 2}$};
  \node at (-.9,-0.5) {$\scriptstyle {\ol 1}$};
  \node at (0.9,-2.5) {$\scriptstyle {\ol 3}$};
  \node at (0.9,-1.5) {$\scriptstyle {\ol 2}$};
  \node at (0.9,-0.5) {$\scriptstyle {\ol 1}$};
\draw[-] (-.5,2.5) -- (.5,-2.5);
\draw[-] (-.5,1.5) -- (.5,-1.5);
\draw[-] (-.5,0.5) -- (.5,-0.5);
\draw[-] (-.5,-0.5) -- (.5,0.5);
\draw[-] (-.5,-1.5) -- (.5,1.5);
\draw[-] (-.5,-2.5) -- (.5,2.5);
\end{tikzpicture}
\ntksp ,
\end{equation}
correspond to the reversals
$s'_\emptyset$, $s'_{[1,2]}$, $s'_{[2,3]}$, $s'_{[1,3]}$, $s'_{[\ol1,1]}$, $s'_{[\ol2,2]}$,
$s'_{[\ol3,3]}$.
We refer to all iterations of concatenations and condensed concatenations
of type-$\msfBC$ simple star networks as {\em type-$\msfBC$ star networks},
and let $\net{BC}{[\ol n,n]}$ denote the set of these having boundary
vertices indexed by $[\ol n,n]$.
We will be interested in three subsets of these formed
by condensed concatenation of type-$\msfBC$ simple star networks.


\begin{defn}\label{d:bcstarnet}
  Define $\snet{BC}{[\ol n, n]}$ to be
  the set of all type-$\msfA$ condensed star networks of the form
\begin{equation}\label{eq:bcbulletconcat}
  F = F'_{[c_1,d_1]} \bullet \cdots \bullet F'_{[c_t,d_t]}
  \end{equation}
and call these
{\em type-$\msfBC$ condensed star networks
  (with boundary vertices indexed by $[\ol n,n]$)}.
\end{defn}
\begin{defn}\label{d:bczzn}
Call a type-$\msfBC$ condensed star network $F$ (\ref{eq:bcbulletconcat}) a
{\em type-$\msfBC$ zig-zag network} if
the intervals $[c_1,d_1],\dotsc,[c_t,d_t]$
satisfy the conditions of Definition~\ref{d:zz},
i.e., if the type-$\msfA$ star network
$F_{[c_1,d_1]} \bullet \cdots \bullet F_{[c_t,d_t]}$
is a type-$\msfA$ zig-zag network with boundary vertices indexed by
$[ \min\{1,c_1,\dotsc,c_t\}, n]$.
Let $\znet{BC}{[\ol n,n]}$
denote the set of type-$\msfBC$ zig-zag networks with boundary vertices
labeled by $[\ol n,n]$.
\end{defn}
\begin{defn}\label{d:bcdsn}
Call a type-$\msfBC$ star network $F$ (\ref{eq:bcbulletconcat}) a
{\em type-$\msfBC$ descending star network} if
the intervals $[c_1,d_1],\dotsc,[c_t,d_t]$
satisfy the conditions of Definition~\ref{d:adsn},
i.e., if the type-$\msfA$ star network
$F_{[c_1,d_1]} \bullet \cdots \bullet F_{[c_t,d_t]}$
is a type-$\msfA$ descending star network with boundary vertices indexed by
$[ \min\{1, c_1,\dotsc,c_t\}, n]$.
Let $\dnet{BC}{[\ol n,n]}$
denote the set of type-$\msfBC$ descending star
networks with boundary vertices
labeled by $[\ol n,n]$.
\end{defn}
\noindent
Thus we have $\dnet{BC}{[\ol n,n]} \subseteq \znet{BC}{[\ol n,n]}
\subseteq \znet{A}{[\ol n,n]},$
and each zig-zag network of type $\msfBC$
is $F_w$ for some $w \in \mfs{[\ol n, n]}$.
By the symmetry of these networks, we necessarily have $w \in \bn$. 

To illustrate, consider the set $\znet{BC}{[\ol3,3]}$
of twenty-two type-$\msfBC$ zig-zag networks.
Fourteen of these
are type-$\msfBC$ descending star networks
\begin{equation}\label{eq:bcdsn}
\begin{tikzpicture}[scale=.4,baseline=0]
\node at (0,2.5) {$\scriptstyle (3)$};
\node at (0,1.5) {$\scriptstyle (2)$};
\node at (0,0.5) {$\scriptstyle (1)$};
\node at (0,-0.5) {$\scriptstyle (\ol1)$};
\node at (0,-1.5) {$\scriptstyle (\ol2)$};
\node at (0,-2.5) {$\scriptstyle (\ol3)$};
\end{tikzpicture}
\begin{tikzpicture}[scale=.4,baseline=0]
\draw[-] (-.5,2.5) -- (.5,2.5);
\draw[-] (-.5,1.5) -- (.5,1.5);
\draw[-] (-.5,0.5) -- (.5,0.5);
\draw[-] (-.5,-0.5) -- (.5,-0.5);
\draw[-] (-.5,-1.5) -- (.5,-1.5);
\draw[-] (-.5,-2.5) -- (.5,-2.5);
\node at (0,-4.5) {$F_{123}$};
\end{tikzpicture}
\nTksp ,\ntnsp
\begin{tikzpicture}[scale=.4,baseline=0]
\draw[-] (-.5,2.5) -- (.5,2.5);
\draw[-] (-.5,1.5) -- (.5,1.5);
\draw[-] (-.5,0.5) -- (.5,-0.5);
\draw[-] (-.5,-0.5) -- (.5,0.5);
\draw[-] (-.5,-1.5) -- (.5,-1.5);
\draw[-] (-.5,-2.5) -- (.5,-2.5);
\node at (0,-4.5) {$F_{\ol123}$};            
\end{tikzpicture}
\nTksp ,\ntnsp
\begin{tikzpicture}[scale=.4,baseline=0]
\draw[-] (-.5,2.5) -- (.5,2.5);
\draw[-] (-.5,1.5) -- (.5,0.5);
\draw[-] (-.5,0.5) -- (.5,1.5);
\draw[-] (-.5,-0.5) -- (.5,-1.5);
\draw[-] (-.5,-1.5) -- (.5,-0.5);
\draw[-] (-.5,-2.5) -- (.5,-2.5);
\node at (0,-4.5) {$F_{213}$};
\end{tikzpicture}
\nTksp ,\ntnsp
\begin{tikzpicture}[scale=.4,baseline=0]
\draw[-] (-.5,2.5) -- (.5,1.5);
\draw[-] (-.5,1.5) -- (.5,2.5);
\draw[-] (-.5,0.5) -- (.5,0.5);
\draw[-] (-.5,-0.5) -- (.5,-0.5);
\draw[-] (-.5,-1.5) -- (.5,-2.5);
\draw[-] (-.5,-2.5) -- (.5,-1.5);
\node at (0,-4.5) {$F_{132}$};
\end{tikzpicture}
\nTksp ,\;
\begin{tikzpicture}[scale=.4,baseline=0]
\draw[-] (-.5,2.5) -- (.5,2.5);
\draw[-] (-.5,1.5) -- (.5,0.5);
\draw[-] (-.5,0.5) -- (.5,1.5);
\draw[-] (-.5,-0.5) -- (.5,-1.5);
\draw[-] (-.5,-1.5) -- (.5,-0.5);
\draw[-] (-.5,-2.5) -- (.5,-2.5);
\draw[-] (.5,2.5) -- (1.5,2.5);
\draw[-] (.5,1.5) -- (1.5,1.5);
\draw[-] (.5,0.5) -- (1.5,-0.5);
\draw[-] (.5,-0.5) -- (1.5,0.5);
\draw[-] (.5,-1.5) -- (1.5,-1.5);
\draw[-] (.5,-2.5) -- (1.5,-2.5);
\node at (0.5,-4.5) {$F_{2\ol13}$};
\end{tikzpicture}
,\ntnsp
\begin{tikzpicture}[scale=.4,baseline=0]
\draw[-] (-.5,2.5) -- (.5,1.5);
\draw[-] (-.5,1.5) -- (.5,2.5);
\draw[-] (-.5,0.5) -- (.5,-0.5);
\draw[-] (-.5,-0.5) -- (.5,0.5);
\draw[-] (-.5,-1.5) -- (.5,-2.5);
\draw[-] (-.5,-2.5) -- (.5,-1.5);
\node at (0,-4.5) {$F_{\ol132}$};
\end{tikzpicture}
\nTksp ,\;\;
\begin{tikzpicture}[scale=.4,baseline=0]
\draw[-] (-.5,2.5) -- (.5,1.5);
\draw[-] (-.5,1.5) -- (.5,2.5);
\draw[-] (-.5,0.5) -- (.5,0.5);
\draw[-] (-.5,-0.5) -- (.5,-0.5);
\draw[-] (-.5,-1.5) -- (.5,-2.5);
\draw[-] (-.5,-2.5) -- (.5,-1.5);
\draw[-] (.5,2.5) -- (1.5,2.5);
\draw[-] (.5,1.5) -- (1.5,0.5);
\draw[-] (.5,0.5) -- (1.5,1.5);
\draw[-] (.5,-0.5) -- (1.5,-1.5);
\draw[-] (.5,-1.5) -- (1.5,-0.5);
\draw[-] (.5,-2.5) -- (1.5,-2.5);
\node at (0.5,-4.5) {$F_{231}$};
\end{tikzpicture}
,\;\;
\begin{tikzpicture}[scale=.4,baseline=0]
\draw[-] (-.5,2.5) -- (.5,1.5);
\draw[-] (-.5,1.5) -- (.5,2.5);
\draw[-] (-.5,0.5) -- (.5,0.5);
\draw[-] (-.5,-0.5) -- (.5,-0.5);
\draw[-] (-.5,-1.5) -- (.5,-2.5);
\draw[-] (-.5,-2.5) -- (.5,-1.5);
\draw[-] (.5,2.5) -- (1.5,2.5);
\draw[-] (.5,1.5) -- (1.5,0.5);
\draw[-] (.5,0.5) -- (1.5,1.5);
\draw[-] (.5,-0.5) -- (1.5,-1.5);
\draw[-] (.5,-1.5) -- (1.5,-0.5);
\draw[-] (.5,-2.5) -- (1.5,-2.5);
\draw[-] (1.5,2.5) -- (2.5,2.5);
\draw[-] (1.5,1.5) -- (2.5,1.5);
\draw[-] (1.5,0.5) -- (2.5,-0.5);
\draw[-] (1.5,-0.5) -- (2.5,0.5);
\draw[-] (1.5,-1.5) -- (2.5,-1.5);
\draw[-] (1.5,-2.5) -- (2.5,-2.5);
\node at (1,-4.5) {$F_{23\ol1}$};
\end{tikzpicture}
\,,
\begin{tikzpicture}[scale=.4,baseline=0]
\draw[-] (-.5,2.5) -- (.5,0.5);
\draw[-] (-.5,1.5) -- (.5,1.5);
\draw[-] (-.5,0.5) -- (.5,2.5);
\draw[-] (-.5,-0.5) -- (.5,-2.5);
\draw[-] (-.5,-1.5) -- (.5,-1.5);
\draw[-] (-.5,-2.5) -- (.5,-0.5);
\node at (0,-4.42) {$F_{321}$};
\end{tikzpicture}
\nTksp, \;\,
\begin{tikzpicture}[scale=.4,baseline=0]
\draw[-] (-.5,2.5) -- (.5,0.5);
\draw[-] (-.5,1.5) -- (.5,1.5);
\draw[-] (-.5,0.5) -- (.5,2.5);
\draw[-] (-.5,-0.5) -- (.5,-2.5);
\draw[-] (-.5,-1.5) -- (.5,-1.5);
\draw[-] (-.5,-2.5) -- (.5,-0.5);
\draw[-] (.5,2.5) -- (1.5,2.5);
\draw[-] (.5,1.5) -- (1.5,1.5);
\draw[-] (.5,0.5) -- (1.5,-0.5);
\draw[-] (.5,-0.5) -- (1.5,0.5);
\draw[-] (.5,-1.5) -- (1.5,-1.5);
\draw[-] (.5,-2.5) -- (1.5,-2.5);
\node at (0.5,-4.5) {$F_{32\ol 1}$};
\end{tikzpicture}
, 
\begin{tikzpicture}[scale=.4,baseline=0]
\draw[-] (-.5,2.5) -- (.5,2.5);
\draw[-] (-.5,1.5) -- (.5,-1.5);
\draw[-] (-.5,0.5) -- (.5,-0.5);
\draw[-] (-.5,-0.5) -- (.5,0.5);
\draw[-] (-.5,-1.5) -- (.5,1.5);
\draw[-] (-.5,-2.5) -- (.5,-2.5);
\node at (0,-4.5) {$F_{\ol1\ol23}$};
\end{tikzpicture}
\nTksp , \;
\begin{tikzpicture}[scale=.4,baseline=0]
\draw[-] (-.5,2.5) -- (.5,1.5);
\draw[-] (-.5,1.5) -- (.5,2.5);
\draw[-] (-.5,0.5) -- (.5,0.5);
\draw[-] (-.5,-0.5) -- (.5,-0.5);
\draw[-] (-.5,-1.5) -- (.5,-2.5);
\draw[-] (-.5,-2.5) -- (.5,-1.5);
\draw[-] (.5,2.5) -- (1.5,2.5);
\draw[-] (.5,1.5) -- (1.5,-1.5);
\draw[-] (.5,0.5) -- (1.5,-0.5);
\draw[-] (.5,-0.5) -- (1.5,0.5);
\draw[-] (.5,-1.5) -- (1.5,1.5);
\draw[-] (.5,-2.5) -- (1.5,-2.5);
\node at (0.5,-4.5) {$F_{\ol13\ol2}$};
\end{tikzpicture}
,\;
\begin{tikzpicture}[scale=.4,baseline=0]s13s22
\draw[-] (-1,2.5) -- (-.5,1.5); \draw[-] (-.5,1.5) -- (0,2.5);
\draw[-] (-1,1.5) -- (-.5,1.5); \draw[-] (-.5,1.5) -- (.5,0);
\draw[-] (-1,0.5) -- (-.5,1.5);
\draw[-] (-1,-0.5) -- (-.5,-1.5); \draw[-] (-.5,-1.5) -- (.5,0);
\draw[-] (-1,-1.5) -- (-.5,-1.5); \draw[-] (-.5,-1.5) -- (0,-2.5);
\draw[-] (-1,-2.5) -- (-.5,-1.5);
\draw[-] (0,2.5) -- (1,2.5);
\draw[-] (.5,0) -- (1,1.5);
\draw[-] (.5,0) -- (1,0.5);
\draw[-] (.5,0) -- (1,-0.5);
\draw[-] (.5,0) -- (1,-1.5);
\draw[-] (0,-2.5) -- (1,-2.5);
\node at (0.25,1.25) {$\scriptstyle (2)$};
\node at (0.25,-1.25) {$\scriptstyle (2)$};
\node at (0,-4.5) {$F_{3\ol1\ol2}$};
\end{tikzpicture}
,
\begin{tikzpicture}[scale=.4,baseline=0]
\draw[-] (-.5,2.5) -- (.5,-2.5);
\draw[-] (-.5,1.5) -- (.5,-1.5);
\draw[-] (-.5,0.5) -- (.5,-0.5);
\draw[-] (-.5,-0.5) -- (.5,0.5);
\draw[-] (-.5,-1.5) -- (.5,1.5);
\draw[-] (-.5,-2.5) -- (.5,2.5);
\node at (0,-4.5) {$F_{\ol1\ol2\ol3}$};
\end{tikzpicture}
\nTksp,
\end{equation}
and eight are not,
\begin{equation}\label{eq:bczznotdsn}        
\begin{tikzpicture}[scale=.4,baseline=0]
\node at (0,2.5) {$\scriptstyle (3)$};
\node at (0,1.5) {$\scriptstyle (2)$};
\node at (0,0.5) {$\scriptstyle (1)$};
\node at (0,-0.5) {$\scriptstyle (\ol1)$};
\node at (0,-1.5) {$\scriptstyle (\ol2)$};
\node at (0,-2.5) {$\scriptstyle (\ol3)$};
\end{tikzpicture} \quad
\begin{tikzpicture}[scale=.4,baseline=0]
\draw[-] (-.5,2.5) -- (.5,2.5);
\draw[-] (-.5,1.5) -- (.5,1.5);
\draw[-] (-.5,0.5) -- (.5,-0.5);
\draw[-] (-.5,-0.5) -- (.5,0.5);
\draw[-] (-.5,-1.5) -- (.5,-1.5);
\draw[-] (-.5,-2.5) -- (.5,-2.5);
\draw[-] (.5,2.5) -- (1.5,2.5);
\draw[-] (.5,1.5) -- (1.5,0.5);
\draw[-] (.5,0.5) -- (1.5,1.5);
\draw[-] (.5,-0.5) -- (1.5,-1.5);
\draw[-] (.5,-1.5) -- (1.5,-0.5);
\draw[-] (.5,-2.5) -- (1.5,-2.5);
\node at (0.5,-4.5) {$F_{\ol213}$};
\end{tikzpicture}
,\quad
\begin{tikzpicture}[scale=.4,baseline=0]
\draw[-] (-.5,2.5) -- (.5,2.5);
\draw[-] (-.5,1.5) -- (.5,0.5);
\draw[-] (-.5,0.5) -- (.5,1.5);
\draw[-] (-.5,-0.5) -- (.5,-1.5);
\draw[-] (-.5,-1.5) -- (.5,-0.5);
\draw[-] (-.5,-2.5) -- (.5,-2.5);
\draw[-] (.5,2.5) -- (1.5,1.5);
\draw[-] (.5,1.5) -- (1.5,2.5);
\draw[-] (.5,0.5) -- (1.5,0.5);
\draw[-] (.5,-0.5) -- (1.5,-0.5);
\draw[-] (.5,-1.5) -- (1.5,-2.5);
\draw[-] (.5,-2.5) -- (1.5,-1.5);
\node at (0.5,-4.42) {$F_{312}$};
\end{tikzpicture}
,\quad
\begin{tikzpicture}[scale=.4,baseline=0]
\draw[-] (-.5,2.5) -- (.5,2.5);
\draw[-] (-.5,1.5) -- (.5,1.5);
\draw[-] (-.5,0.5) -- (.5,-0.5);
\draw[-] (-.5,-0.5) -- (.5,0.5);
\draw[-] (-.5,-1.5) -- (.5,-1.5);
\draw[-] (-.5,-2.5) -- (.5,-2.5);
\draw[-] (.5,2.5) -- (1.5,2.5);
\draw[-] (.5,1.5) -- (1.5,0.5);
\draw[-] (.5,0.5) -- (1.5,1.5);
\draw[-] (.5,-0.5) -- (1.5,-1.5);
\draw[-] (.5,-1.5) -- (1.5,-0.5);
\draw[-] (.5,-2.5) -- (1.5,-2.5);
\draw[-] (1.5,2.5) -- (2.5,1.5);
\draw[-] (1.5,1.5) -- (2.5,2.5);
\draw[-] (1.5,0.5) -- (2.5,0.5);
\draw[-] (1.5,-0.5) -- (2.5,-0.5);
\draw[-] (1.5,-1.5) -- (2.5,-2.5);
\draw[-] (1.5,-2.5) -- (2.5,-1.5);
\node at (1,-4.5) {$F_{\ol312}$};
\end{tikzpicture}
,\quad
\begin{tikzpicture}[scale=.4,baseline=0]
\draw[-] (-.5,2.5) -- (.5,1.5);
\draw[-] (-.5,1.5) -- (.5,2.5);
\draw[-] (-.5,0.5) -- (.5,-0.5);
\draw[-] (-.5,-0.5) -- (.5,0.5);
\draw[-] (-.5,-1.5) -- (.5,-2.5);
\draw[-] (-.5,-2.5) -- (.5,-1.5);
\draw[-] (.5,2.5) -- (1.5,2.5);
\draw[-] (.5,1.5) -- (1.5,0.5);
\draw[-] (.5,0.5) -- (1.5,1.5);
\draw[-] (.5,-0.5) -- (1.5,-1.5);
\draw[-] (.5,-1.5) -- (1.5,-0.5);
\draw[-] (.5,-2.5) -- (1.5,-2.5);
\node at (0.5,-4.5) {$F_{\ol231}$};
\end{tikzpicture}
,\quad
\begin{tikzpicture}[scale=.4,baseline=0]
\draw[-] (-.5,2.5) -- (.5,2.5);
\draw[-] (-.5,1.5) -- (.5,0.5);
\draw[-] (-.5,0.5) -- (.5,1.5);
\draw[-] (-.5,-0.5) -- (.5,-1.5);
\draw[-] (-.5,-1.5) -- (.5,-0.5);
\draw[-] (-.5,-2.5) -- (.5,-2.5);
\draw[-] (.5,2.5) -- (1.5,1.5);
\draw[-] (.5,1.5) -- (1.5,2.5);
\draw[-] (.5,0.5) -- (1.5,-0.5);
\draw[-] (.5,-0.5) -- (1.5,0.5);
\draw[-] (.5,-1.5) -- (1.5,-2.5);
\draw[-] (.5,-2.5) -- (1.5,-1.5);
\node at (0.5,-4.5) {$F_{3\ol12}$};
\end{tikzpicture}
,\quad
\begin{tikzpicture}[scale=.4,baseline=0]
\draw[-] (-.5,2.5) -- (.5,2.5);
\draw[-] (-.5,1.5) -- (.5,1.5);
\draw[-] (-.5,0.5) -- (.5,-0.5);
\draw[-] (-.5,-0.5) -- (.5,0.5);
\draw[-] (-.5,-1.5) -- (.5,-1.5);
\draw[-] (-.5,-2.5) -- (.5,-2.5);
\draw[-] (.5,2.5) -- (1.5,0.5);
\draw[-] (.5,1.5) -- (1.5,1.5);
\draw[-] (.5,0.5) -- (1.5,2.5);
\draw[-] (.5,-0.5) -- (1.5,-2.5);
\draw[-] (.5,-1.5) -- (1.5,-1.5);
\draw[-] (.5,-2.5) -- (1.5,-0.5);
\node at (0.5,-4.5) {$F_{\ol321}$};
\end{tikzpicture}
,\quad
\begin{tikzpicture}[scale=.4,baseline=0]
\draw[-] (-.5,2.5) -- (.5,2.5);
\draw[-] (-.5,1.5) -- (.5,-1.5);
\draw[-] (-.5,0.5) -- (.5,-0.5);
\draw[-] (-.5,-0.5) -- (.5,0.5);
\draw[-] (-.5,-1.5) -- (.5,1.5);
\draw[-] (-.5,-2.5) -- (.5,-2.5);
\draw[-] (.5,2.5) -- (1.5,1.5);
\draw[-] (.5,1.5) -- (1.5,2.5);
\draw[-] (.5,0.5) -- (1.5,0.5);
\draw[-] (.5,-0.5) -- (1.5,-0.5);
\draw[-] (.5,-1.5) -- (1.5,-2.5);
\draw[-] (.5,-2.5) -- (1.5,-1.5);
\node at (0.5,-4.5) {$F_{\ol1\ol32}$};
\end{tikzpicture}
,\quad
\begin{tikzpicture}[scale=.4,baseline=0]s22s13
\draw[-] (1,2.5) -- (.5,1.5); \draw[-] (.5,1.5) -- (0,2.5);
\draw[-] (1,1.5) -- (.5,1.5); \draw[-] (.5,1.5) -- (-.5,0);
\draw[-] (1,0.5) -- (.5,1.5);
\draw[-] (1,-0.5) -- (.5,-1.5); \draw[-] (.5,-1.5) -- (-.5,0);
\draw[-] (1,-1.5) -- (.5,-1.5); \draw[-] (.5,-1.5) -- (0,-2.5);
\draw[-] (1,-2.5) -- (.5,-1.5);
\draw[-] (0,2.5) -- (-1,2.5);
\draw[-] (-.5,0) -- (-1,1.5);
\draw[-] (-.5,0) -- (-1,0.5);
\draw[-] (-.5,0) -- (-1,-0.5);
\draw[-] (-.5,0) -- (-1,-1.5);
\draw[-] (0,-2.5) -- (-1,-2.5);
\node at (-0.25,1.25) {$\scriptstyle (2)$};
\node at (-0.25,-1.25) {$\scriptstyle (2)$};
\node at (0,-4.5) {$F_{\ol3\ol21}$};
\end{tikzpicture}
.
\end{equation}


By Corollary~\ref{c:atmostonepath} and the containment
$\bn \subseteq \snn$, we have that
for all $F_w \in \znet{BC}{[\ol n ,n]}$ and $v \in \bn$,
at most one
path family $\pi$ of type $v$
covers $F_w$.
For $i>0$, paths $\pi_i$ and $\pi_{\ol i}$ in this family
are necessarily mirror
images of one another, and we call $\pi_i$
{\em grounded} if it intersects path $\pi_{\ol i}$.
For example consider $F_{3\ol12}$ in (\ref{eq:bczznotdsn})
and the path families
$\pi$ of type $123$ and $\sigma$ of type $3\ol12$ covering it,
\begin{equation}\label{eq:pisigma}
  \pi = 
\begin{tikzpicture}[scale=.4,baseline=0]
  \node at (-.9,2.5) {$\scriptstyle 3$};
  \node at (-.9,1.5) {$\scriptstyle 2$};
  \node at (-.9,0.5) {$\scriptstyle 1$};
  \node at (1.9,2.5) {$\scriptstyle 3$};
  \node at (1.9,1.5) {$\scriptstyle 2$};
  \node at (1.9,0.5) {$\scriptstyle 1$};
  \node at (-.9,-2.5) {$\scriptstyle {\ol 3}$};
  \node at (-.9,-1.5) {$\scriptstyle {\ol 2}$};
  \node at (-.9,-0.5) {$\scriptstyle {\ol 1}$};
  \node at (1.9,-2.5) {$\scriptstyle {\ol 3}$};
  \node at (1.9,-1.5) {$\scriptstyle {\ol 2}$};
  \node at (1.9,-0.5) {$\scriptstyle {\ol 1}$};
  \draw[-, very thick]
  (-.5,2.5) -- (.5,2.5) -- (1,2) -- (1.5,2.5);
  \draw[-, dotted, very thick]
  (-.5,1.5) -- (0,1) -- (.5,1.5) -- (1,2) -- (1.5,1.5);
  \draw[-, dashed, very thick]
  (-.5,0.5) -- (0,1) -- (.5,0.5) -- (1,0) -- (1.5,0.5);
  \draw[-, thin]
  (-.5,-0.43) -- (0,-.93) -- (.5,-0.43) -- (1,0.07) -- (1.5,-0.43);
  \draw[-, thin]
  (-.5,-0.57) -- (0,-1.07) -- (.5,-0.57) -- (1,-0.07) -- (1.5,-0.57);
  \draw[-, dotted, thick]
  (-.5,-1.5) -- (0,-1) -- (.5,-1.5) -- (1,-2) -- (1.5,-1.5);
  \draw[-]
  (-.5,-2.5) -- (.5,-2.5) -- (1,-2) -- (1.5,-2.5);
\end{tikzpicture}\, ,
  \qquad \qquad
  \sigma = \begin{tikzpicture}[scale=.4,baseline=0]
  \node at (-.9,2.5) {$\scriptstyle 3$};
  \node at (-.9,1.5) {$\scriptstyle 2$};
  \node at (-.9,0.5) {$\scriptstyle 1$};
  \node at (1.9,2.5) {$\scriptstyle 3$};
  \node at (1.9,1.5) {$\scriptstyle 2$};
  \node at (1.9,0.5) {$\scriptstyle 1$};
  \node at (-.9,-2.5) {$\scriptstyle {\ol 3}$};
  \node at (-.9,-1.5) {$\scriptstyle {\ol 2}$};
  \node at (-.9,-0.5) {$\scriptstyle {\ol 1}$};
  \node at (1.9,-2.5) {$\scriptstyle {\ol 3}$};
  \node at (1.9,-1.5) {$\scriptstyle {\ol 2}$};
  \node at (1.9,-0.5) {$\scriptstyle {\ol 1}$};
  \draw[-, very thick]
  (-.5,2.5) -- (.5,2.5) -- (1,2) -- (1.5,1.5);
  \draw[-, dotted, very thick]
  (-.5,0.5) -- (0,1) -- (.5,1.5) -- (1,2) -- (1.5,2.5);
  \draw[-, dashed, very thick]
  (-.5,1.5) -- (0,1) -- (.5,0.5) -- (1,0) -- (1.5,-0.5);
  \draw[-]
  (-.5,-2.5) -- (.5,-2.5) -- (1,-2) -- (1.5,-1.5);
  \draw[-, dotted, thick]
  (-.5,-0.5) -- (0,-1) -- (.5,-1.5) -- (1,-2) -- (1.5,-2.5);
  \draw[-, thin]
  (-.5,-1.43) -- (0,-.93) -- (.5,-0.43) -- (1,0.07) -- (1.5,0.57);
  \draw[-, thin]
  (-.5,-1.57) -- (0,-1.07) -- (.5,-0.57) -- (1,-.07) -- (1.5,0.43);
\end{tikzpicture}\, .
\end{equation}
The path $\pi_1$ is grounded while $\pi_2$ and $\pi_3$ are not;
the path $\sigma_2$ is grounded while $\sigma_1$ and $\sigma_3$ are not.

By (\ref{eq:bcstardef}) the intervals appearing in the construction
of the type-$\msfBC$ star network $G$ (\ref{eq:bcbulletconcat})
are roughly half of those that appear in the type-$\msfA$ construction
of the same network.  The subposet of $\preceq$ induced by these intervals
satisfies the following.
\begin{prop}\label{p:atmostonesymminterval}
  For $F'_{\smash{[a_1,b_1]}} \bullet \cdots \bullet F'_{\smash{[a_t,b_t]}}$ 
  a type-$\msfBC$ zig-zag network,
  there is at most one interval $[a_i, b_i]$
  satisfying $a_i < 0$, $b_i = \ol{a_i}$.
  Furthermore, this interval is maximal or minimal
  (or both) in the poset $\preceq$ on $\{ [a_1,b_1], \dotsc, [a_t,b_t] \}$.
\end{prop}
\begin{proof}
  Condition (1) of Definition~\ref{d:zz} requires that the intervals
  $[a_1,b_1], \dotsc, [a_t,b_t]$ be distinct and form a nonnesting set.
  Thus at most one of these intervals satisfies $b_i = \ol{a_i}$.
  Let $[a_j,\ol{a_j}]$ be such an interval ($a_j < 0$) and suppose that
  it is neither maximal nor minimal in the partial order $\preceq$.
  Then there are indices $i, k$ with $i < j < k$ and
  $[a_i,b_i] \cap [a_j,\ol{a_j}] \neq \emptyset$,
  $[a_k,b_k] \cap [a_j,\ol{a_j}] \neq \emptyset$.  By Condition (2)
  of Definition~\ref{d:zz}, we must have $a_i < a_j < a_k$
  or $a_i > a_j > a_k$.  But this implies that $a_i < 0$ or $a_k < 0$, and
  therefore that $[a_j, \ol{a_j}]$ is properly contained in $[a_i, \ol{a_i}]$
  or $[a_k,\ol{a_k}]$, contradicting Condition (1).
  \end{proof}
As a consequence, the cardinalities of $\znet{BC}{[\ol n,n]}$ and
$\dnet{BC}{[\ol n,n]}$ are related to their type-$\msfA$ analogs,
with the second cardinality equal to a Catalan number.
\begin{thm}\label{t:catalan}
  For all $n$ we have
  \begin{enumerate}
  \item $|\znet{BC}{[\ol n,n]}| = |\znet{A}{[1,n+1]}|$,
  \item $|\dnet{BC}{[\ol n,n]}| = |\dnet{A}{[1,n+1]}| = \frac1{n+2} \tbinom{2n+2}{n+1}$.
    \end{enumerate}
\end{thm}
\begin{proof}
  (1) Define a map
  $\Upsilon: \znet{BC}{[\ol n, n]} \rightarrow \znet{A}{[1, n+1]}$ by
  $\Upsilon(F'_{[a_i,b_i]}) = F_{[\max\{a_i+1,1\},b_i+1]}$ and
  \begin{equation*}
      \Upsilon(F'_{[a_1,b_1]} \bullet \cdots \bullet F'_{[a_t,b_t]})
      =
    \Upsilon(F'_{[a_1, b_1]}) \bullet \cdots \bullet \Upsilon(F'_{[a_t,b_t]}),
  \end{equation*}
  so that all positive endpoints of intervals increase by one
  and all negative endpoints are replaced by $1$.
  To see that $\Upsilon$ is well defined, recall
  that by Proposition~\ref{p:atmostonesymminterval}
  at most one of the intervals $[a_i,b_i]$ satisfies $a_i < 0$.
  Thus the conditions of Definition~\ref{d:zz} are satisfied
  and $\Upsilon(F)$ belongs to $\znet{A}{[1,n+1]}$.
  Furthermore, for $F \in \dnet{BC}{[\ol n, n]}$,
  the inequalities $a_1 > \cdots > a_t$ imply that we have
  $a_1 + 1 > \cdots > \max\{a_t + 1, 1\}$ and
  $\Upsilon(F) \in \dnet{A}{[1,n+1]}$.
  To see that $\Upsilon$ is bijective, observe that we have
  \begin{equation*}
  \begin{gathered}
    \;[a_i,b_i] \subseteq [1,n] \;\Longleftrightarrow\;
      [a_i+1,b_i+1] \subseteq [2,n+1],\\
  a_i = b_i \in [1,n] \;\Longleftrightarrow\; a_i+1 = b_i+1 \in [2,n+1], \\
  [\ol{b_i},b_i] \subseteq [\ol n, n] \;\Longleftrightarrow\;
  [1,b_i+1] \subseteq [1,n+1].
  \end{gathered}
\end{equation*}
  Finally, by \cite[Thm.\,3.6]{CHSSkanEKL} we have
  $|\dnet{A}{[1,n]}| = \tfrac 1{n+1} \tbinom{2n}n$.
\end{proof}


In Theorems~\ref{t:znetbcpavoid} -- \ref{t:dnetbcpavoid}
we will characterize
$\znet{BC}{[\ol n,n]}$ and $\dnet{BC}{[\ol n,n]}$
as subsets of $\{ F_w \in \znet{A}{[\ol n,n]} \,|\, w \in \bn \}$
defined by $w$ avoiding certain patterns.
In order to do so, we
decompose certain
elements of $\bn$ into pairs $(u,v) \in \mfb k \times \mfs{n-k}$,
and certain
zig-zag networks in $\znet{BC}{[\ol n,n]}$
into pairs of components in $\znet{BC}{[\ol k,k]} \times \znet{A}{[1,n-k]}$.
Define the map
\begin{equation}\label{eq:oplusperm}
  \begin{aligned}
    \oplus: \mfb{k} \times \mfs{[1,n-k]} &\rightarrow \mfb{n}\\
    (u,v) &\mapsto u \oplus v = w_{\ol{n}} \cdots w_{\ol1} w_1 \cdots w_n
  \end{aligned}
\end{equation}
by 
\begin{equation*}
  w_i = \begin{cases}
    u_i &\text{if $i \in [\ol k, k]$},\\
    v_i + k &\text{if $i > k$},\\
    \ol{v_i + k} &\text{if $i < \ol k$}.
  \end{cases}
\end{equation*}
For example,
the elements
$u = \ol 1 \in \mfb1$, $v = 231 \in \mfs3$ give
$u \oplus v  =\ol2\ol4\ol3 1 \ol1 342 \in \mfb4$.
Observe that to make sense of the more general expression
$u \oplus v^{(1)} \oplus \cdots \oplus v^{(p)}$,
we must interpret it as
$( \cdots ((u \oplus v^{(1)}) \oplus v^{(2)}) \oplus \cdots \oplus v^{(p-1)}) \oplus v^{(p)}$, with $u \in \mfb{k}$, $v^{(i)} \in \mfs{[1,j_i]}$
for some $k, j_1, \dotsc, j_p$.
We will say that any element $w \in \bn$ which can be written as
$w = u \oplus v$ is {\em $\oplus$-decomposable}.
Equivalently, $w \in \bn$ is $\oplus$-decomposable if there is
some index $k$ such that
\begin{equation*}
  \{|w_1|, \dotsc, |w_k| \} = [1,k], \qquad \{w_{k+1}, \dotsc, w_n\} = [k+1,n].
\end{equation*}

We define a similar map
\begin{equation}\label{eq:oplusnet}
  \begin{aligned}
    \oplus: \snet{BC}{[\ol k,k]} \times \snet{A}{[1,n-k]} &\rightarrow \snet{BC}{[\ol n,n]}\\
      (E,F) &\mapsto E \oplus F
  \end{aligned}
\end{equation}
as follows.
\begin{enumerate}
\item Create $F^+ \in \snet{A}{[k+1,n]}$ by adding $k$ to the
  indices of all sources and sinks of $F$.
\item Create $F^- \in \snet{A}{[\ol n, \ol{k+1}]}$ by drawing $F^+$ upside-down
  and by multiplying each source and sink index by $-1$.
\item Vertically arrange the sources and sinks of these networks and $E$
  in order $(\ol n, \dotsc, n)$, so that we have $F^+$ above $E$ above $F^-$.
\end{enumerate}
For example, to construct
the network $F'_{[\ol1,1]} \oplus ( F_{[2,3]} \bullet F_{[1,2]} )$,
we place $F'_{[\ol1,1]}$ between
two copies of $F_{[2,3]} \bullet F_{[1,2]}$,
one upside-down,
to obtain
\begin{equation*}
  \begin{aligned}
  F'_{[\ol1,1]} &=
\begin{tikzpicture}[scale=.5,baseline=-25]
\node at (-.4,-1) {$\scriptstyle 1$};
\node at (-.4,-2) {$\scriptstyle{\ol1}$};
\node at (1.4,-1) {$\scriptstyle 1$};
\node at (1.4,-2) {$\scriptstyle{\ol1}$};
\draw[-] (0,-1) -- (1,-2);
\draw[-] (0,-2) -- (1,-1);
\end{tikzpicture},\\
\\
  F_{[2,3]} \bullet F_{[1,2]} &=
\begin{tikzpicture}[scale=.5,baseline=-18]
\node at (-.4,0)  {$\scriptstyle 3$};
\node at (-.4,-1) {$\scriptstyle 2$};
\node at (-.4,-2) {$\scriptstyle 1$};
\node at (2.4,0)  {$\scriptstyle 3$};
\node at (2.4,-1) {$\scriptstyle 2$};
\node at (2.4,-2) {$\scriptstyle 1$};
\draw[-] (0,0) -- (2,-2);
\draw[-] (0,-1) -- (1,0) -- (2,0);
\draw[-] (0,-2) -- (1,-2) -- (2,-1);
\end{tikzpicture},
\end{aligned}
  \qquad
  F'_{[\ol1,1]} \oplus ( F_{[2,3]} \bullet F_{[1,2]} ) = 
\begin{tikzpicture}[scale=.5,baseline=-5]
  \node at (-.4,3.5) {$\scriptstyle 4$};
  \node at (-.4,2.5) {$\scriptstyle 3$};
  \node at (-.4,1.5) {$\scriptstyle 2$};
  \node at (-.4,0.5) {$\scriptstyle 1$};
  \node at (-.4,-0.5) {$\scriptstyle \ol1$};
  \node at (-.4,-1.5) {$\scriptstyle \ol2$};
  \node at (-.4,-2.5) {$\scriptstyle \ol3$};
  \node at (-.4,-3.5) {$\scriptstyle \ol4$};
  \node at (2.4,3.5) {$\scriptstyle 4$};
  \node at (2.4,2.5) {$\scriptstyle 3$};
  \node at (2.4,1.5) {$\scriptstyle 2$};
  \node at (2.4,0.5) {$\scriptstyle 1$};
  \node at (2.4,-0.5) {$\scriptstyle \ol1$};
  \node at (2.4,-1.5) {$\scriptstyle \ol2$};
  \node at (2.4,-2.5) {$\scriptstyle \ol3$};
  \node at (2.4,-3.5) {$\scriptstyle \ol4$};
\draw[-] (0,3.5) -- (2,1.5);
\draw[-] (0,2.5) -- (1,3.5) -- (2,3.5);
\draw[-] (0,1.5) -- (1,1.5) -- (2,2.5);
\draw[-] (0,0.5) -- (1,-0.5) -- (2,-0.5);
\draw[-] (0,-3.5) -- (2,-1.5);
\draw[-] (0,-2.5) -- (1,-3.5) -- (2,-3.5);
\draw[-] (0,-1.5) -- (1,-1.5) -- (2,-2.5);
\draw[-] (0,-0.5) -- (1,0.5) -- (2,0.5);
\end{tikzpicture}
\cong F'_{[\ol1,1]} \bullet F'_{[2,3]} \bullet F'_{[1,2]}.
\end{equation*}


\begin{lem}\label{l:opluszz}
  For elements $u \in \mfb k$, $v \in \mfs{n-k}$
  and zig-zag networks
  $F_u \in \znet{BC}{[\ol k, k]}$,
  $F_v \in \znet{A}{[1,n-k]}$,
  we have the following.
  \begin{enumerate}
  \item $F_u \oplus F_v \in \znet{BC}{[\ol n, n]}$ is a zig-zag network
    satisfying $w(F_u \oplus F_v) = u \oplus v$.
  \item If $F_u$ and $F_v$ are descending star networks, then so is
    $F_u \oplus F_v$.
  \end{enumerate}
\end{lem}
\begin{proof}
  (1) To see that $F_u \oplus F_v$ belongs to $\znet{BC}{[\ol n, n]}$, write
  \begin{equation*}
      F_u = F'_{[c_1,d_1]} \bullet \cdots \bullet F'_{[c_t,d_t]}, \qquad
      F_v = F_{[a_1,b_1]} \bullet \cdots \bullet F_{[a_r,b_r]}.
  \end{equation*}
  By the definition (\ref{eq:oplusnet}) of the map $\oplus$ we have
  \begin{equation*}
    F_u \oplus F_v = F'_{[c_1,d_1]} \bullet \cdots \bullet F'_{[c_t,d_t]}
    \bullet F'_{[a_1+k,b_1+k]} \bullet \cdots \bullet F'_{[a_r+k,b_r+k]}.
  \end{equation*}
  It is easy to see that the set $\{[a_1+k,b_1+k],\dotsc,[a_r+k,b_r+k]\}$
  satisfies the conditions of Definition~\ref{d:zz},
  and since each interval $[c_i,d_i]$ is disjoint
  from each interval $[a_j+k,b_j+k]$, the union
  $\{[c_1,d_1],\dotsc,[c_t,d_t],[a_1+k,b_1+k], \dotsc, [a_r+k,b_r+k]\}$
  satisfies the conditions of Definition~\ref{d:zz} as well.

  Now let $w = w(F_u \oplus F_v) \in \bn$ and let $y = u \oplus v \in \bn$.
  To see that $w = y$,
  recall by Proposition~\ref{p:wFchar}
  that $w$ is the permutation in $\snn$
  which maximizes $\inv(z)$ over all $z \in \snn$
  for which there is a path family of type $z$ covering $F_u \oplus F_v$.
  By the disconectedness of $F_u \oplus F_v$, we have
  \begin{equation*}
    \{w_{\ol k}, \dotsc, w_k \} = [\ol k, k],
    \qquad
    \{w_{k+1}, \dotsc, w_n \} = [k+1,n],
    \qquad
    \{w_{\ol n}, \dotsc, w_{\ol{k+1}} \} = [\ol n, \ol{k+1}].
  \end{equation*}
  Since $w(F_u) = u$, it is clear that $w$ has as many inversions as possible
  among entries $\ol k, \dotsc, k$ when
  $w_{\ol k} \cdots w_k = u_{\ol k} \cdots u_k$.
  Similarly,
  $w$ has as many inversions as possible
  among entries $k+1, \dotsc, n$ when
  $w_{k+1} \cdots w_n$ matches the pattern $v_1 \cdots v_{n-k}$, i.e.,
  when $w_{k+i} = v_i + k$ for $i = 1,\dotsc,n-k$.
  In this case,
  we also have
  $w_{\ol n} \cdots w_{\ol{k+1}} = \ol{w_n} \cdots \ol{w_{k+1}}$.
  Thus we have $w = u \oplus v$.

  \noindent
  (2) The fact that $F_u \oplus F_v$ belongs to $\dnet{BC}{[\ol n,n]}$
  follows immediately from Definition~\ref{d:bcdsn}.
\end{proof}

Now we may characterize $\znet{BC}{[\ol n, n]}$ and $\dnet{BC}{[\ol n, n]}$
in terms of pattern avoidance.


\begin{thm}\label{t:znetbcpavoid}
  Elements of $\znet{BC}{[\ol n,n]}$ correspond bijectively to
  \pavoiding elements of $\bn$.  Specifically we have
  \begin{equation}\label{eq:znetbcseteq}
    \znet{BC}{[\ol n, n]} =
    \{ F_w \in \znet{A}{[\ol n,n]} \,|\, w \in \bn \}.
    \end{equation}
\end{thm}
\begin{proof}
  ($\subseteq$)
  Consider $F \in \znet{BC}{[\ol n,n]} \subset \znet{A}{[\ol n,n]}$.
  By \cite[\S 3]{SkanNNDCB}, $F$ has the form $F_w$
  for some $w \in \snn$ \avoidingp, and factors as in
  (\ref{eq:bcbulletconcat})
  and Definition~\ref{d:zz}.
  By Proposition~\ref{p:atmostonesymminterval},
  at most one of the intervals $[c_i,d_i]$ appearing in
  (\ref{eq:bcbulletconcat}) satisfies $c_i = \ol{d_i}$.
  If such an interval exists, then we may assume that it appears
  first or last.
  Thus we may factor $F$ as
  \begin{equation*}
    F'_{[c_0,d_0]} \bullet
    (F'_{[c_1,d_1]} \bullet F'_{[\ol{d_1},\ol{c_1}]}) \bullet \cdots \bullet
    (F'_{[c_t,d_t]} \bullet F'_{[\ol{d_t},\ol{c_t}]})
    \bullet F'_{[c_{t+1},d_{t+1}]},
  \end{equation*}
  with one or both of the intervals $[c_0,d_0]$ and $[c_{t+1},d_{t+1}]$
  satisfying $c_i=d_i$, and at most one of these satisfying $c_i = \ol{d_i}$.
  Algorithm~\ref{a:FwF}
  then gives a reversal factorization of $w$.
  This factorization consists of the subsequence of reversals
  \begin{equation}\label{eq:mainreversals}
    (s_{[c_0,d_0]}, s_{[c_1,d_1]}, s_{[\ol{d_1},\ol{c_1}]}, \dotsc,
    s_{[c_t,d_t]}, s_{[\ol{d_t},\ol{c_t}]}, s_{[\ol{d_{t+1}},\ol{c_{t+1}}]}),
  \end{equation}
  and more pairs of reversals 
  \begin{equation}\label{eq:intreversals}
    s_{[c_i,d_i] \cap [c_j,d_j]}, \quad  s_{[\ol{d_i},\ol{c_i}] \cap [\ol{d_j},\ol{c_j}]}
  \end{equation}
  inserted between these.
  Since the only intervals appearing in
  (\ref{eq:mainreversals}) -- (\ref{eq:intreversals})  
  which can contain both positive and negative integers
  are $[c_0,d_0]$, $[c_{t+1},d_{t+1}]$,
  we may reorder the sequence of reversals
  to place each pair $s_{[a,b]}$ and $s_{[\ol b, \ol a]}$ consecutively.
  Thus $w$ equals a product of type-$\msfBC$ reversals of the
  forms $s_{[c_0,d_0]}$, $s'_{[a,b]} = s_{[a,b]}s_{[\ol b, \ol a]}$,
  $s_{[c_{t+1},d_{t+1}]}$, and belongs to $\bn$.

  ($\supseteq$)
  We claim that for each element $w \in \bn$ \avoidingp,
  we have $F_w \in \znet{BC}{[\ol n, n]}$.
  This is true when $n = 1$ because
  $\znet{A}{[\ol1,1]} = \{ F_\emptyset, F_{[\ol1,1]} \} = \znet{BC}{[\ol1,1]}$.
  Now suppose that the
  statement is true
  for $w \in \mfb1,\dotsc,\mfb{n-1}$ and consider $w \in \bn$.

  If $w$ is $\oplus$-decomposable then we can write
  $w = u \oplus v$ for $u \in \mfb k$, $v \in \mfs{n-k}$, and
  $1 \leq k < n$.
  Then by induction we have $F_u \in \znet{BC}{[\ol k,k]}$ and
  $F_v \in \znet{A}{[1,n-k]}$.
By Lemma \ref{l:opluszz} the network $F_w \in \znet{A}{[\ol n,n]}$
satisfies
$F_w = F_u \oplus F_v \in \znet{BC}{[\ol n,n]}$.

If
$w$ is not $\oplus$-decomposable,
then we may apply \cite[Obs.\,3.3]{SkanNNDCB} to find a zig-zag factorization
of $w$ (as in the paragraph following (\ref{eq:posetex}))
and to obtain an expression (\ref{eq:bulletconcat})
for $F_w$ which satisfies the conditions of Definition~\ref{d:zz}.
In particular, we compare the lengths $\ell$, $m$
of the longest decreasing prefixes of $w$ and $w^{-1}$ respectively,
\begin{equation*}
  w_{\smash{\ol n}} > \cdots > w_{\smash{\ol{n-\ell+1}}},
  \qquad (w^{-1})_{\smash{\ol n}} > \cdots > (w^{-1})_{\smash{\ol{n-m+1}}}.
\end{equation*}
If $\ell = m$, then $w = s_{\smash{[\ol n, n]}}$ and $F_w$ is the type-$\msfBC$
zig-zag network $F_{\smash{[\ol n,n]}}$.
If $\ell < m$, then $w$ has a type-$\msfA$ zig-zag factorization beginning
with
\begin{equation*}
  s_{\smash{[\ol n, \ol{n-\ell+1}]}}s_{\smash{[\ol{n-m},\ol{n-\ell+1}]}}s_{\smash{[\ol{n-m},k]}}
\end{equation*}
for some $k > \ol{n-m}$,
and the interval $[\ol n, \ol{n-\ell+1}]$ is $\prec$-minimal.
By the $\bn$-skew-symmetry of $w$,
it also has a type-$\msfA$ zig-zag factorization beginning with
\begin{equation*}
  s_{[n-\ell+1,n]}s_{[n-\ell+1,n-m]}s_{\smash{[\ol k,n-m]}},
\end{equation*}
and the interval $[n-\ell+1,n]$ is also $\prec$-minimal.
In other words, we can write $w = s'_{\smash{[n-\ell+1,n]}}w'$
for some $w' \in \bn$ satisfying
$w'_i = i$ for $i = \ol n, \dotsc, \ol{n-m+1}, n-m+1,\dotsc,n$, i.e.,
\begin{equation*}
  w' = 
  w'_{\smash{\ol{n-m}}} \cdots w'_{n-m} \oplus 1 \cdots m.
\end{equation*}
It follows that we have $F_w \cong F'_{\smash{[n-\ell+1,n]}} \bullet F_{w'}$.
By induction $F_{w'}$ is a type-$\msfBC$ zig-zag network, and so is $F_w$.
\end{proof}


\begin{thm}\label{t:dnetbcpavoid}
  Elements of $\dnet{BC}{[\ol n,n]}$ correspond bijectively to
  elements of $\bn$ \avoidingsignedp.
  Specifically we have
  \begin{equation}\label{eq:dnetbcseteq}
    \dnet{BC}{[\ol n, n]} =
    \{ F_w \in \znet{A}{[\ol n,n]} \,|\, w \in \bn
    \text{ \avoidssignedp }\}.
  \end{equation}
\end{thm}
\begin{proof}
  First we observe that by Lemma~\ref{l:5signedp},
  avoidance of the signed patterns
  $1\ol2$, $\ol21$, $\ol2\ol1$, $312$, $3\ol12$ implies
  avoidance of the unsigned patterns $3412$ and $4231$.
  Thus the right-hand side of (\ref{eq:dnetbcseteq})
  includes one zig-zag network $F_w$ for every element $w\in \bn$
  avoiding the five signed patterns.
  Next, consider the subset of $\dnet{BC}{[\ol n, n]}$
  consisting of networks $F_w$ factoring as
  \begin{equation}\label{eq:bcbc}
    F'_{[c_1,d_1]} \bullet \cdots \bullet F'_{[c_t,d_t]}
  \end{equation}
  with $c_1 > \cdots > c_t > 0$.
  By Proposition~\ref{p:anetpavoid},
  these networks are precisely
  \begin{equation*}
    \{ F_u \in \znet{A}{[\ol n,n]} \,|\, u \in \bn,
    u_1 \cdots u_n
    \text{ nonnegative and avoiding the pattern } 312 \}.
\end{equation*}
  Therefore we may prove the proposition
  by proving (\ref{eq:dnetbcseteq}),
  restricting our attention
  on the right-hand-side
  to networks $F_u$
  with $u \in \bn$ having
  at least one negative letter in the subword $u_1 \cdots u_n$, and
  on the left-hand-side to networks $F_w \in \dnet{BC}{[\ol n, n]}$
  factoring as
  (\ref{eq:bcbc})
  with $c_1 > \cdots > c_t = \ol{d_t}$.
  Such networks $F_w$ correspond bijectively to networks
  $F_v \in \dnet{A}{[c_t,n]}$
  factoring as
  \begin{equation}\label{eq:abc}
  F_v \defeq F_{[c_1,d_1]} \bullet \cdots \bullet F_{[c_t,d_t]},  
  \end{equation}
  with $v$ and $w$ satisfying
  \begin{equation}\label{eq:vw}
    \begin{gathered}
    v_{\ol{d_t}} \cdots v_{\ol 2} v_{\ol 1} = d_t \cdots 21, \qquad
    v_1 \cdots v_n = w_1 \cdots w_n,\\
    \{w_1, \dotsc, w_n \} = \{\ol1, \dotsc, \ol{d_t}, d_t+1,\dotsc, n\}.
   \end{gathered}
  \end{equation}
  
  \noindent
  ($\subseteq$)
  Consider
  $F_w \in \dnet{BC}{[\ol n, n]} \subset \znet{A}{[\ol n, n]}$
  factoring as
  (\ref{eq:bcbc})
  with $c_1 > \cdots > c_t = \ol{d_t}$.
  Since the related network $F_v$ (\ref{eq:abc})
  belongs to
  $\dnet{A}{[c_t, n]}$,
  the element $v = v_{c_t} \cdots v_{\ol1} v_1 \cdots v_n \in \mfs{[c_t,n]}$
  avoids the ordinary pattern $312$.
  Thus $v_1 \cdots v_n = w_1 \cdots w_n$
  avoids the signed patterns $312$ and $3 \ol1 2$.
  By \cite[Obs.\,3.2]{SkanNNDCB}
  we have that
  $\ol 1 \cdots \ol{d_t}$ is a subword of $w_1 \cdots w_n$.
  Thus $w_1 \cdots w_n$ contains a negative letter and also
  avoids the signed patterns $1 \ol2$, $\ol21$, $\ol2\ol1$. 
  
  \noindent ($\supseteq$) Consider $F_w$ on the right-hand side of
  (\ref{eq:dnetbcseteq}) with $w_1 \cdots w_n$ containing a negative
  letter.
  By (\ref{eq:znetbcseteq})
  we have
  $F_w \in \znet{BC}{[\ol n,n]}$.
  By Proposition~\ref{p:atmostonesymminterval},
  there exists a factorization
  (\ref{eq:bcbc}) of $F_w$
  in which exactly one interval $[c_i,d_i]$ satisfies $c_i = \ol{d_i}$,
  and this interval must be maximal or minimal (or both)
  in the partial order $\preceq$.

  Assume that this interval is minimal and that $i = 1$.
  By \cite[Obs.\,3.2]{SkanNNDCB}
  we have
  \begin{equation}\label{eq:decseq}
    w_{\ol{d_i}} > \cdots w_{\ol1} > w_1 > \cdots > w_{d_i},
  \end{equation}
  with $w_{\ol1} > 0 > w_1$ since $w \in \bn$.
  If some positive letter $j \leq d_i$ does not appear in these positions,
  then $w_1 \cdots w_n$ contains either
  the subword $w_{d_i} j$ which matches the pattern $\ol 21$,
  or the subword $w_{d_i} \ol j$ which matches the pattern $\ol 2\ol1$.
  This contradicts our choice of $F_w$.
  Thus
  the $2d_i$ letters in these positions
  must be
  $d_i > \cdots > 1 > \ol 1 > \cdots > \ol{d_i}$,
  and the interval $[\ol{d_i},d_i]$ is both minimal and maximal.
  It follows that we have $w = s_{[\ol{d_i},d_i]} \oplus u$ for some
  $u \in \mfs{[d_i+1,n]}$.
  Since $w$ avoids the signed pattern $312$, we have that
  $u$ avoids the ordinary pattern $312$,
  $F_u$ belongs to $\dnet{A}{[1,n-d_i]}$,
  and $F_w = F'_{[\ol{d_i},d_i]} \oplus F_u$ belongs to $\dnet{BC}{[\ol n,n]}$.

  Now assume that the interval $[c_i,d_i] = [\ol{d_i},d_i]$
  in the factorization (\ref{eq:bcbc}) of $F_w$
  is maximal with $i=t$.
  Define $F_v \in \znet{A}{[c_t,n]}$ as in (\ref{eq:abc}).
  We claim that
  $v_{c_t} \cdots v_{\ol 1} v_1 \cdots v_n \in \mfs{[c_t,n]}$
  avoids the ordinary pattern $312$.
  To obtain a contradiction,
  assume that some subword
  $v_{j_1} v_{j_2} v_{j_3}$
  matches the ordinary pattern $312$.
  Suppose first that $j_1 \geq 1$. Then 
  $w_{j_1} w_{j_2} w_{j_3}$
  matches one of the signed patterns $312$, $3\ol12$, $3\ol2\ol1$,
  $\ol1\ol3\ol2$ and this contradicts the containment of $F_w$ on the
  right-hand side of (\ref{eq:dnetbcseteq}).
  Now suppose that $j_1 \leq \ol 1$ and $j_2 \geq 1$.
  Then the letter $v_{j_1}$ is positive by (\ref{eq:vw}),
  and letters $\ol1, \dotsc, \ol{v_{j_1}}$ appear in
  $v_1 \cdots v_n = w_1 \cdots w_n$.
  Since $w$ avoids the signed pattern $\ol2\ol1$, it is  
  impossible for
  letters $v_{j_2}v_{j_3}$ in $v_1 \cdots v_n$
  to complete the ordinary pattern $312$.
  Now suppose that $j_1 < j_2 \leq \ol 1$.
  Since $v_{j_1}$ and $v_{j_2}$ are both positive,
  all of the letters $v_{j_2} + 1, \dotsc, v_{j_1}-1$ appear between
  these two letters, none can complete the pattern $312$.
%
  Thus no subword of $v_{c_t} \cdots v_{\ol 1} v_1 \cdots v_n$ matches
  the pattern $312$, and
  $F_w$ belongs to $\dnet{BC}{[\ol n, n]}$.
\end{proof}
It is easy to see that the list of signed patterns in
Theorem~\ref{t:dnetbcpavoid} can not be shortened.
For $w \in \bn \subset \snn$,
failure to avoid the signed pattern $1 \ol2$ or $\ol2\ol1$
implies failure to avoid the ordinary pattern $3412$ or $4231$,
which implies that $F_w$ is not a type-$\msfBC$ zig-zag network.
Furthermore, inspection of $F_{\ol21}$
($F_{\ol213}$ with highest and lowest edges removed),
$F_{312}$, and $F_{\smash{3\ol12}}$ in (\ref{eq:bczznotdsn})
shows that these are not type-$\msfBC$ descending star networks.

Since any element $w \in \bn$ \avoidingp{}
can be viewed as a permutation in $\mfs{[\ol n,n]}$
and any zig-zag network in $\znet{BC}{[\ol n,n]}$
can be viewed as a zig-zag network in $\znet{A}{[\ol n,n]}$,
the bijection $F \mapsto w(F)$ guaranteed by
Theorems~\ref{t:znetbcpavoid} -- \ref{t:dnetbcpavoid}
can be realized by Algorithm~\ref{a:FwF}.
The inverse $w \mapsto F_w$
of the map can be realized as in~\cite[\S 3]{SkanNNDCB},
or as follows in the special case that $w$ \avoidssignedp.
\begin{alg}
  Given $w \in \bn$ \avoidingsignedp, do
  \begin{enumerate}
  \item Set
  \begin{equation*}
    v = \begin{cases}
      |h| \cdots 1 w_1 \cdots w_n
       &\text{if $w_1 \dotsc w_n$ contains negative letters $h \cdots \ol 1$,}\\
      w_1 \cdots w_n
       &\text{if $w_1 \cdots w_n$ contains no negative letters}.
    \end{cases}
  \end{equation*}
  \item Apply Algorithm~\ref{a:wFw} to $v$ to obtain
      $F_{[c_1,d_1]} \bullet \cdots \bullet F_{[c_t,d_t]}$.
  \item Set $F_w = F'_{[c_1,d_1]} \bullet \cdots \bullet F'_{[c_t,d_t]}$.
  \end{enumerate}
\end{alg}

Theorems \ref{t:znetbcpavoid}
and \ref{t:dnetbcpavoid} suggest defining
type-$\msfBC$ analogs of path families and graphical representation
(\ref{eq:Gtozsn}), (\ref{eq:Gtozhnq}).
Given $F \in \net{BC}{[\ol n,n]}$,
and
$\pi=(\pi_{\ol n}, \dotsc, \pi_{\ol 1}, \pi_1, \dotsc, \pi_{\ol n})$
covering $F$, call $\pi$ a {\em $\msfBC$-path family} if
for each factor $F'_{[a,b]}$ of $F$
and each index $i \in [1,n]$, there exist indices $j$, $k$, such that
paths $\pi_i$ and $\pi_{\ol i}$ enter $F'_{[a,b]}$ via sources $j$, $\ol j$
and exit $F'_{[a,b]}$ via sinks $k$, $\ol k$, respectively.
In other words, $\pi_{\ol i}$ must be a reflection of $\pi_i$.
For $F \in \net{BC}{[\ol n,n]}$ and $u \in \bn$, define the sets
\begin{equation}\label{eq:BCpathfams}
  \begin{gathered}
    \PiBC(F) = \{ \pi \,|\, \pi \text{ a $\msfBC$-path family covering } F \},\\
    \PiBC_u(F) = \{ \pi \in \PiBC(F) \,|\, \type(\pi) = u \}.
    \end{gathered}
\end{equation}
For example, the two path families in (\ref{eq:pisigma})
belong to
$\PiBC(F_{3\ol12})$ with
$\pi \in \PiBC_e(F_{3\ol12})$, $\sigma \in \PiBC_{3\ol12}(F_{3\ol12})$.
On the other hand,
the path family
\begin{equation}\label{eq:notBCpathfam}
\begin{tikzpicture}[scale=.4,baseline=0]
  \node at (-1.4,1.5) {$\scriptstyle 2$};
  \node at (-1.4,0.5) {$\scriptstyle 1$};
  \node at (1.4,1.5) {$\scriptstyle 2$};
  \node at (1.4,0.5) {$\scriptstyle 1$};
  \node at (-1.4,-1.5) {$\scriptstyle {\ol 2}$};
  \node at (-1.4,-0.5) {$\scriptstyle {\ol 1}$};
  \node at (1.4,-1.5) {$\scriptstyle {\ol 2}$};
  \node at (1.4,-0.5) {$\scriptstyle {\ol 1}$};
  \draw[-, very thick]
  (-1,1.5) -- (0,.5) -- (1,1.5);
  \draw[-, dotted, very thick]
  (-1,0.5) -- (0,1.5) -- (1,0.5);
  \draw[-, dotted, thick ]
  (-1,-0.5) -- (-.5,-1) -- (0,-.5) -- (.5,-1) -- (1,-0.5);
  \draw[-, thick]
  (-1,-1.5) -- (-.5,-1) -- (0,-1.5) -- (.5,-1) -- (1,-1.5);
\end{tikzpicture}
\end{equation}
is not a $\msfBC$-path family,
even though it has type $e \in \mfb2$.

The set $\PiBC(F)$ associates elements of $\zbn$ and $\hbnq$
to $F$. Specifically, we say that
$F$ {\em graphically represents}
\begin{equation}\label{eq:Gtozsnbc}
\sum_{\pi \in \PiBC(F)} \nTksp \type(\pi)
\end{equation}
{\em as an element of $\zbn$}.
To describe the corresponding element of $\hbnq$ we first extend
the definition of {\em type-$\msfA$ defects} from
Subsection~\ref{ss:Aplanar} (and
\cite{BWHex},
\cite{CSkanTNNChar},
\cite{Deodhar90}).
Assume that $F$
is formed by some iteration
of ordinary or condensed concatenatation
of simple star networks $F'_{[c_1,d_1]}, \dotsc, F'_{[c_t,d_t]}$.
Each factor $F'_{[c_k, d_k]}$ contributes a single internal vertex
if $c_k = \ol{d_k}$, and two such vertices otherwise.
Given a $\msfBC$-path family $\pi$ covering $F$,
define a {\em type-$\msfBC$ defect} of $\pi$ to be
a triple $(\pi_i, \pi_j, k)$ with
\begin{enumerate}
\item $|i| \leq j$,
\item $\pi_i$ and $\pi_j$ meet at one of the internal vertices of
  $F'_{[c_k,d_k]}$ after having crossed an odd number of times.
\end{enumerate}
(The first condition prevents the double-counting of path meetings
which occur in pairs when $|i| \neq |j|$.)
Let $\dfct^{\msfBC}(\pi)$ denote the number of type-$\msfBC$ defects of $\pi$.
For example, consider the star network
and path family
\begin{equation}\label{eq:all010}
  F'_{[\ol2,2]} \circ F'_{[\ol1,1]} \circ F'_{[1,2]} \circ F'_{[\ol2,2]} =
\begin{tikzpicture}[scale=.5,baseline=-5]
\node at (-.4,1.5) {$\scriptstyle{2}$};
\node at (-.4,0.5) {$\scriptstyle{1}$};  
\node at (-.4,-0.5) {$\scriptstyle{\ol1}$};
\node at (-.4,-1.5) {$\scriptstyle{\ol2}$};  
\node at (4.4,1.5) {$\scriptstyle{2}$};
\node at (4.4,0.5) {$\scriptstyle{1}$};  
\node at (4.4,-0.5) {$\scriptstyle{\ol1}$};
\node at (4.4,-1.5) {$\scriptstyle{\ol2}$};  
\draw[-] (0,1.5) -- (1,-1.5) -- (2,-1.5) -- (3,-.5) -- (3.5,0) -- (4,.5);
\draw[-] (0,-1.5) -- (1,1.5) -- (2,1.5) -- (3,.5) -- (3.5,0) -- (4,-.5);
\draw[-] (0,.5) -- (1,-.5) -- (2,.5) -- (3,1.5) -- (3.5,0) -- (4,1.5);
\draw[-] (0,-.5) -- (1,.5) -- (2,-.5) -- (3,-1.5) -- (3.5,0) -- (4,-1.5);
\end{tikzpicture},\qquad
\pi = 
\begin{tikzpicture}[scale=.5,baseline=-5]
\node at (-.4,1.5) {$\scriptstyle{2}$};
\node at (-.4,0.5) {$\scriptstyle{1}$};  
\node at (-.4,-0.5) {$\scriptstyle{\ol1}$};
\node at (-.4,-1.5) {$\scriptstyle{\ol2}$};  
\node at (4.4,1.5) {$\scriptstyle{2}$};
\node at (4.4,0.5) {$\scriptstyle{1}$};  
\node at (4.4,-0.5) {$\scriptstyle{\ol1}$};
\node at (4.4,-1.5) {$\scriptstyle{\ol2}$};  
\draw[-, very thick]
(0,1.5) -- (1,-1.5) -- (2,-1.5) -- (3,-.5) -- (3.5,0) -- (4,-.5);
\draw[-]
(0,.5) -- (1,-.5) -- (2,.5) -- (3,1.5) -- (3.5,0) -- (4,1.5);
\draw[-, dashed, thick]
(0,-.5) -- (1,.5) -- (2,-.5) -- (3,-1.5) -- (3.5,0) -- (4,-1.5);
\draw[-, very thick, dotted]
(0,-1.5) -- (1,1.5) -- (2,1.5) -- (3,.5) -- (3.5,0) -- (4,.5);
\end{tikzpicture}.
\end{equation}
The defects of $\pi$ are
$(\pi_{\ol1}, \pi_1, 2)$,
$(\pi_{\ol1}, \pi_2, 3)$,
$(\pi_1, \pi_2, 4)$,
$(\pi_{\ol2}, \pi_2, 4)$,
and we have $\dfct^{\msfBC}(\pi) = 4$.
We say that $F$ {\em graphically represents}
\begin{equation}\label{eq:Gtozbnq}
\sum_{\pi \in \PiBC(F)} \nTksp q^{\dfct^{\msfBC}(\pi)} T_{\type(\pi)}
\end{equation}
{\em as an element of $\hbnq$}.
Specializing at $q=1$, we see that if $F$ graphically represents
$D(q)$ as an element of $\hbnq$, then
it graphically represents
$D(1)$ as an element of $\zbn$.

In the special case that $F = F_w \in \znet{BC}{[\ol n,n]}$,
it graphically represents a Kazhdan--Lusztig basis element.
\begin{thm}\label{t:zigzagbc}
  For $w \in \bn$ \avoidingp, the zig-zag network $F_w$ represents
  $\btc wq$ as an element of $\hbnq$.
  \end{thm}
\begin{proof}
  By \cite[Lem.\,5.3]{SkanNNDCB}, we have that for all $u \in \snn$,
  there exists exactly one path family of type $u$ covering $F_w$
  if $u \leq_{\snn} w$ and no such path family otherwise.
  In particular, this is true for $u \in \bn \subset \snn$.
  But by Proposition~\ref{p:abcbruhat}
  we have $u \leq_{\bn} w$ if and only if $u \leq_{\snn} w$.
  Since $w$ \avoidsp, the network $F_w$ belongs to $\znet{BC}{[\ol n,n]}$,
  and every path family
  $\pi \in \PiBC(F_w)$ satisfies $\dfct(\pi) = 0$.  Thus the sum
  (\ref{eq:Gtozbnq}) becomes
  \begin{equation*}
    \sum_{u \leq_{\bn} w} \ntksp T_u,
  \end{equation*}
  which by (\ref{p:abcsmoothkl}) is $\btc wq$.
\end{proof}

\begin{cor}\label{c:zigzagbc}
  For $v, w \in \bn$ with $w$ \avoidingp,
  the number of $\msfBC$-path families of type $v$ covering $F_w$ is
  $1$ if $v \leq_{\bn} w$, and is $0$ otherwise.
\end{cor}

\section{Immanants and total nonnegativity}\label{s:immtnn}

In order to use Section~\ref{s:planarnet} to produce partial solutions
to Problem~\ref{p:evaltrace} for the subsets
(\ref{eq:cwqsmooth}) -- (\ref{eq:BCcwqcodom})
of the Kazhdan--Lusztig bases,
we rely heavily upon methods borrowed from the study of total nonnegativity
and upon trace generating functions
in a ring $\mathbb Z[\bfx]$
where $\bfx = (\bfx_{i,j})_{i,j \in [\ol n,n]}$ is viewed
as the $2n \times 2n$ matrix
\begin{equation}\label{eq:xmatrix}
  \bfx = \begin{bmatrix}
    \bfx_{\ol n, \ol n} & \cdots & \bfx_{\ol n, \ol 1} & \bfx_{\ol n,1} & \cdots & \bfx_{\ol n, n} \\
    \vdots       &        & \vdots        & \vdots   &        & \vdots   \\
    \bfx_{\ol 1, \ol n} & \cdots & \bfx_{\ol 1, \ol 1} & \bfx_{\ol 1,1} & \cdots & \bfx_{\ol 1, n} \\
    \bfx_{1, \ol n}   & \cdots  & \bfx_{1, \ol 1}    & \bfx_{1,1}    & \cdots & \bfx_{1, n} \\
    \vdots       &        & \vdots        & \vdots   &        & \vdots   \\
    \bfx_{n, \ol n}   & \cdots  & \bfx_{n, \ol 1}    & \bfx_{n,1}    & \cdots & \bfx_{n, n}
    \end{bmatrix}.
\end{equation}
For subsets $I, J \subseteq [\ol n, n]$ we define the submatrix
$\bfx_{I,J} \defeq (\bfx_{i,j})_{i\in I, j \in J}$.
To economize notation, we abbreviate
\begin{equation}\label{eq:ndef}
  [n] \defeq [1,n].
\end{equation}
Thus $\bfx_{[n],[n]}$ denotes the submatrix of positively indexed entries of
$\bfx$.
Given polynomial
$p(\bfx) \in \mathbb Z[\bfx]$,
and $2n \times 2n$ matrix $A = (a_{i,j})_{i,j \in [\ol n, n]}$, we define
$p(A)$ to be the expression obtained by evaluating $p(\bfx)$ at
$\bfx_{i,j} = a_{i,j}$, for all $i,j \in [\ol n,n]$.

\ssec{Type-$\msfA$ immanants}\label{ss:Aimm}

For certain $\theta \in \trsp(\sn)$
and for all $w \in \sn$ \avoidingp, 
combinatorial formulas for $\theta(\wtc w1)$
depend upon generating functions
which are polynomials in entries of the submatrix $\bfx_{[n],[n]}$ of $\bfx$.
Following Littlewood~\cite{LittlewoodTGC} and Stanley~\cite{StanPos},
we define the {\em (type-$\msfA$) $\theta$-immanant} to be
\begin{equation}\label{eq:immdef}
  \simm n{\theta}(\bfx_{[n],[n]}) \defeq
  \sum_{w \in \sn} \theta(w) \permmon \bfx w
  \in \mathbb Z[\bfx_{1,1}, \bfx_{1,2}, \dotsc, \bfx_{n,n}].
\end{equation}
When $\theta$ is an induced one-dimensional character
$\eta^\lambda$ or
$\epsilon^\lambda$ with $\lambda = (\lambda_1, \dotsc,\lambda_r) \vdash n$,
we may neatly express its corresponding immanant in terms of
permanents or determinants, and
{\em ordered set partitions of type $\lambda$},
i.e., sequences $(J_1,\dotsc,J_r)$ of subsets of $[n]$ with
\begin{enumerate}
  \item $J_1 \uplus \cdots \uplus J_r = [n]$,
  \item $|J_i| = \lambda_i$ for $i = 1,\dotsc,r$.
    \end{enumerate}
In particular,
we have the Littlewood--Merris--Watkins
identities~\cite{LittlewoodTGC}, \cite{MerWatIneq},
\begin{equation}\label{eq:lmw}
  \begin{aligned}
    \simm{n}{\epsilon^\lambda}(\bfx_{[n],[n]})
    = \nTksp \sum_{(J_1,\dotsc,J_r)} \nTksp \det(\bfx_{J_1,J_1}) \cdots \det(\bfx_{J_r,J_r}),\\
    \simm{n}{\eta^\lambda}(\bfx_{[n],[n]})
    = \nTksp \sum_{(J_1,\dotsc,J_r)} \nTksp \perm(\bfx_{J_1,J_1}) \cdots \perm(\bfx_{J_r,J_r}),
  \end{aligned}
\end{equation}
where the sums are over ordered set partitions $(J_1,\dotsc,J_r)$ of $[n]$
of type $\lambda$. 
(See \cite[Thm.\,2.1]{KSkanQGJ} for a $q$-analog.)
We also have
\begin{equation}\label{eq:pimm}
  \simm{n}{\psi^\lambda}(\bfx_{[n],[n]}) =
  z_\lambda \nTksp \ntksp \sumsb{w\\ \ctype(w) = \lambda} \ntksp \nTksp \permmon \bfx w,
\end{equation}
where $z_\lambda$ is defined as in (\ref{eq:zlambda}).
(See \cite{GJMaster} for work on $\simm{n}{\chi^\lambda}(\bfx_{[n],[n]})$.)

Immanants and trace evaluations of the form $\theta(\wtc w1)$ are connected
by the following identity~\cite[Eqn.\,(3.5)]{CHSSkanEKL}.
\begin{thm}
  Fix $w \in \sn$ \avoidingp{} with corresponding zig-zag network $F_w$
  having path matrix $A = A(w)$.  Then for any linear functional
  $\theta: \csn \rightarrow \mathbb C$ we have
\begin{equation}\label{eq:immid}
  \theta(\wtc w1) = \simm n\theta(A),
\end{equation}
where $\simm n\theta(A)$ should be interpreted as
$\smash{\simm n\theta(\bfx_{[n],[n]})}$
evaluated at $\bfx_{i,j} = a_{i,j}$.
\end{thm}
Thus each combinatorial interpretation of $\smash{\simm n\theta(A)}$
yields a combinatorial interpretation of $\theta(\wtc w1)$. 
To produce such combinatorial interpretations, we appeal
to methods of total nonnegativity, namely,
Lindstrom's Lemma~\cite{KMG}, \cite{LinVrep}
and some simple extensions.
\begin{prop}\label{p:detpermpsiinterpbasic}
    Fix $w \in \sn$ \avoidingp{} with
    corresponding zig-zag network $F_w$ having path matrix $A = A(w)$.
    We have
\begin{gather}
  \label{eq:detinterpbasic}
  \simm n{\epsilon^n}(A)  = \det(A)
  = \# \{ \pi \in \Pi_e(F_w) \,|\, \pi_1,\dotsc,\pi_n \text{ pairwise nonintersecting\,} \}, \\
  \label{eq:perminterpbasic}
    \simm n{\eta^n}(A) = \perm(A)
    = \# \Pi(F_w),\\
    \label{eq:pimminterpbasic}
    \simm n{\psi^n}(A)
    = n \cdot \# \{ \pi \in \Pi_u(F_w) \,|\, u \in \sn, \ctype(u) = n \}.
    \end{gather}
\end{prop}
Proposition~\ref{p:detpermpsiinterpbasic}
implies simple interpretations of
$\simm n{\epsilon^\lambda}(A)$,
$\simm n{\eta^\lambda}(A)$,
$\simm n{\psi^\lambda}(A)$
as well, for $\lambda \vdash n$ arbitrary.
We will return to these in Subection~\ref{ss:mainA}.
For $q$-analogs, see \cite{CHSSkanEKL}.


\ssec{Type-$\msfBC$ immanants}\label{ss:BCimm}
To create a generating function for $\theta \in \trsp(\bn)$,
we define
the {\em (type-$\msfBC$) $\theta$-immanant} to be
\begin{equation}\label{eq:bimmdef}
  \bnimm{\theta}(\bfx) \defeq \sum_{w \in \bn}
  \theta(w) \bpermmon \bfx w \in \mathbb Z[\bfx].
\end{equation}
This is a special case of the {\em wreath product immanant}
defined in \cite[Eqns.\,(26)--(27)]{SkanGenFnWreath}, and generalizes the
Littlewood -- Stanley immanant (\ref{eq:immdef}).
When $\theta$ is an induced character of the form
$(\zeta \otimes \delta \xi)\upparrow_{\mfb m \times \mfb{n-m}}^{\bn}$
for symmetric group characters $\zeta$, $\xi$,
as in (\ref{eq:inducedpair}) with $q=1$,
then
we may neatly express its corresponding immanant
in terms of type-$\msfA$ immanants
and $n \times n$ matrices
$\bfx^+ = (\bfx^+_{i,j})_{i,j \in [n]}$,
$\bfx^- = (\bfx^-_{i,j})_{i,j \in [n]}$ defined in terms of
the $2n \times 2n$ matrix $\bfx$ (\ref{eq:xmatrix}) by
\begin{equation}\label{eq:pm}
  \bfx^+_{i,j} = \bfx_{i,j}\bfx_{\ol i, \ol j} + \bfx_{i, \ol j}\bfx_{\ol i, j},
  \qquad
  \bfx^-_{i,j} = \bfx_{i,j}\bfx_{\ol i, \ol j} - \bfx_{i, \ol j}\bfx_{\ol i, j}.
\end{equation}
For $I, J \subseteq [n]$, we let
$\bfx^+_{I,J} \defeq (\bfx^+)_{I,J}$ 
and $\bfx^-_{I,J} \defeq (\bfx^-)_{I,J}$ denote the $I,J$ submatrices of these.
For example,
\begin{equation*}
  \bfx^+_{12,14} = \begin{bmatrix}
    \bfx_{1,1}\bfx_{\ol1,\ol1} + \bfx_{1,\ol1}\bfx_{\ol1,1} &
    \bfx_{1,4}\bfx_{\ol1,\ol4} + \bfx_{1,\ol4}\bfx_{\ol1,4} \\
    \bfx_{2,1}\bfx_{\ol2,\ol1} + \bfx_{2,\ol1}\bfx_{\ol2,1} &
    \bfx_{2,4}\bfx_{\ol2,\ol4} + \bfx_{2,\ol4}\bfx_{\ol2,4} 
  \end{bmatrix}.
\end{equation*}
By \cite[Thm.\,3.1]{SkanGenFnWreath} we have for bipartitions
$(\lambda,\mu) \vdash n$
with $|\lambda| = k$ that
\begin{equation}\label{eq:wrthimm2}
  \bnimm{(\epsilon\epsilon)^{\lambda,\mu}}(\bfx) =
  \sum_I\simm{k}{\epsilon^\lambda}(\bfx^+_{I,I})
  \simm{n-k}{\epsilon^\mu}(\bfx^-_{\smash{[n]\ssm I,[n]\ssm I}}),
  \end{equation}
where the
sum is over all $m$-element subsets $I$ of $[n]$.
More generally we have the following formula,
which is a type-$\msfBC$ analog of
\cite[Prop.\,2.4]{StemConj}.
\begin{lem}\label{l:tensorimm}
  Given symmetric group traces $\zeta \in \trsp(\mfs{k})$,
  $\xi \in \trsp(\mfs{n-k})$,
  and hyperoctahedral group trace $\theta \in \trsp(\bn)$ satisfying
  $\theta = (\zeta\otimes \delta\xi) \upparrow_{\mfb{k,n-k}}^{\bn}$,
  we have
  \begin{equation}\label{eq:tensorimm}
    \bnimm{\theta}(\bfx)
    = \sumsb{I \subseteq [n]\\|I| = k}
    \simm{k}{\zeta}(\bfx_{I,I}^+) \simm{n-k}{\xi}(\bfx_{[n] \ssm I, [n] \ssm I}^-).
    \end{equation}
\end{lem}
\begin{proof}
  Expand $\zeta$, $\xi$ in the induced sign character bases
  of $\trsp(\mfs{k})$, $\trsp(\mfs{n-k})$ as 
  \begin{equation*}
    \zeta = \sum_{\lambda \vdash k} a_\lambda \epsilon^\lambda,
    \qquad
    \xi = \sum_{\mu \vdash n-k} b_\mu \epsilon^\mu.
  \end{equation*}
  Then
  we have
  \begin{equation*}
      \theta =
      \Big(\sum_{\lambda\vdash k} a_\lambda \epsilon^\lambda \otimes
      \delta \nTksp \sum_{\mu \vdash n-k} \nTksp b_\mu \epsilon^\mu \Big)
      \upparrow_{\mfb{k,n-k}}^{\bn} 
      = \nTksp \sumsb{\lambda \vdash k\\ \mu \vdash n-k} \nTksp a_\lambda b_\mu
      (\epsilon^\lambda \otimes \delta \epsilon^\mu)
      \upparrow_{\mfb{k,n-k}}^{\bn} 
      = \nTksp \sumsb{\lambda \vdash k\\ \mu \vdash n-k} \nTksp a_\lambda b_\mu
      (\epsilon\epsilon)^{\lambda,\mu},
  \end{equation*}
  and the left-hand side of (\ref{eq:tensorimm}) is
  \begin{equation*}
  \sumsb{\lambda \vdash k\\ \mu \vdash n-k} \nTksp a_\lambda b_\mu
  \bnimm{(\epsilon\epsilon)^{\lambda,\mu}}(\bfx).
  \end{equation*}
  But by (\ref{eq:wrthimm2}), this is
  \begin{multline*}
  \smash{\sumsb{\lambda \vdash k\\ \mu \vdash n-k} \nTksp a_\lambda b_\mu
  \ntksp \sumsb{I \subseteq [n]\\ |I|=k}
  \ntksp \simm{k}{\epsilon^\lambda}(\bfx^+_{I,I})
  \,\simm{n-k}{\epsilon^\mu}(\bfx^-_{[n]\ssm I,[n]\ssm I})}\\
  = \sumsb{I \subseteq [n]\\ |I|=k}^{\phantom a}
  \sum_{\lambda \vdash k} a_\lambda 
  \simm{k}{\epsilon^\lambda}(\bfx^+_{I,I})
  \ntksp \sum_{\mu \vdash n-k} \nTksp b_\mu
  \simm{n-k}{\epsilon^\mu}(\bfx^-_{[n]\ssm I,[n]\ssm I}),
  \end{multline*}
  which is the right-hand side of (\ref{eq:tensorimm}).
    \end{proof}

Evaluating the immanants (\ref{eq:bimmdef}) at path matrices of
type-$\msfBC$ zig-zag networks
\begin{equation*}
  \{ F_w \,|\, w \in \bn \text{ \avoidsp } \}
\end{equation*}
gives the following type-$\msfBC$ analog of (\ref{eq:immid}) which allows
us to use type-$\msfBC$ immanants to compute trace evaluations of the
form $\theta(\btc w1)$.

\begin{thm}\label{t:charevalimmbc}
  Let $w \in \bn$ \avoidp,
  and let zig-zag network $F_w$ have path matrix $A$.
  Then for any linear functional $\theta:\bn \rightarrow \mathbb C$ we have
  \begin{equation}\label{eq:charevalimmbc}
    \theta(\btc w1) = \bnimm{\theta}(A).
  \end{equation}
\end{thm}
\begin{proof}
  The right-hand side of (\ref{eq:charevalimmbc}) is
  \begin{equation}\label{eq:charevalimmbc2}
    \sum_{v \in \bn} \theta(v)
    a_{\ol n, v_{\ol n}} \cdots a_{\ol 1, v_{\ol 1}} a_{1,v_1} \cdots a_{n,v_n}.
  \end{equation}
  Since $F_w$ is a type-$\msfBC$ zig-zag network of order $2n$,
  it is also a type-$\msfA$ zig-zag network of order $2n$.
  By \cite[Lem.\,5.3]{SkanNNDCB}, the product
  $a_{\ol n, v_{\ol n}} \cdots a_{\ol 1, v_{\ol 1}} a_{1,v_1} \cdots a_{n,v_n}$
  is $1$ when $v \leq_{\snn} w$ and is $0$ otherwise.
  Thus by Proposition~\ref{p:abcbruhat} it
  is $1$ when $v \leq_{\bn} w$ and is $0$ otherwise,
  and 
  the sum (\ref{eq:charevalimmbc2}) is
  \begin{equation*}
    \sum_{v \leq_{\bn} w} \nTksp \theta(v) = \theta \Big(\ntksp \sum_{v \leq_{\bn} w\nTksp} \ntksp v \Big)
    = \theta(\btc w1).
  \end{equation*}
\end{proof}

Thus each combinatorial interpretation of $\bnimm \theta(A)$
yields a combinatorial interpretation of $\theta(\btc w1)$. 
Taking the special cases of Lemma~\ref{l:tensorimm} corresponding to
$\zeta$, $\xi$ equal to $\triv$, $\sgn$, or $\psi^n$
and evaluating
$\sn$-immanants at $A^+$ and $A^-$,
we have the following type-$\msfBC$ analogs of the sets of path families
appearing in Proposition~\ref{p:detpermpsiinterpbasic}.

\begin{prop}\label{p:detDBPB}
Fix $w \in \bn$ \avoidingp, 
let $F_w \in \znet{BC}{[\ol n,n]}$
have path matrix $A$,
and define $A^+$, $A^-$ as in (\ref{eq:pm}).
We have
\begin{equation}\label{eq:permap}
  \perm(A^+) = \# \PiBC(F_w),
\end{equation}
\begin{equation}\label{eq:permam}
  \begin{aligned}
    \perm(A^-) &= \# \{ \pi \in \PiBC(F_w) \,|\,
    \pi_i, \pi_j \text{ may share a vertex only if  $i,j < 0$ or $i,j > 0$} \}\\
    &= \begin{cases}
      \# \PiBC(F_w) &\text{if $\ell_t(w) = 0$},\\
      0 &\text{otherwise},
    \end{cases}
  \end{aligned}
\end{equation}
\begin{equation}\label{eq:detap}
  \begin{aligned}
    \det(A^+) &= \# \{ \pi \in \PiBC(F_w) \,|\,
    \pi_i, \pi_j \text{ may share a vertex only if } -1 \leq i,j \leq 1 \}\\
        &= \begin{cases}
      2^{\ell(w)} &\text{if $w \in \{e,t\}$},\\
      0 &\text{otherwise},
    \end{cases}
  \end{aligned}
\end{equation}
\begin{equation}\label{eq:detam}
  \begin{aligned}
    \det(A^-) &= \# \{ \pi \in \PiBC(F_w) \,|\,
    \pi_i, \pi_j \text{ are vertex-disjoint for all } i \neq j \}\\
            &= \begin{cases}
      1 &\text{if $w = e$},\\
      0 &\text{otherwise},
    \end{cases}
  \end{aligned}
\end{equation}
\begin{equation}\label{eq:psiap}
    \simm n{\psi^n}(A^+) = n \cdot \# \{ \pi \in \PiBC_u(F_w) \,|\,
    u \in \bn, \ctype(\varphi(u)) = n \},
\end{equation}
\begin{equation}\label{eq:psiam}
    \simm n{\psi^n}(A^-) = \begin{cases}
      n \cdot \# \{ \pi \in \PiBC_u(F_w) \,|\,
      u \in \bn, \ctype(\varphi(u)) = n \}, &\text{if $\ell_t(w) = 0$},\\
      0 &\text{otherwise}.
      \end{cases}
  \end{equation}
\end{prop}
\begin{proof}
  Define $\ell_t$, $\ell_s$, $\varphi$ as in
Subsections~\ref{ss:bnassubgp} -- \ref{ss:conjclasses}.
Observe that for $v \in \sn$ we have
  \begin{equation*}
    \begin{gathered}
      \permmon{a^+}v =
      \sumsb{u \in \bn\\\varphi(u)=v} \bpermmon au
      = \sumsb{u \in \bn\\\varphi(u)=v} |\PiBC_u(F_w)|,\\
      \permmon{a^-}v =
      \sumsb{u \in \bn\\\varphi(u)=v} (-1)^{\ell_t(u)} \bpermmon au
      = \sumsb{u \in \bn\\\varphi(u)=v} (-1)^{\ell_t(u)} |\PiBC_u(F_w)|.
    \end{gathered}
  \end{equation*}
Thus we have
\begin{equation}\label{eq:sums}
  \begin{gathered}
    \perm(A^+) = \sum_{u \in \mfb n}  |\PiBC_u(F_w)|,\qquad
    \perm(A^-) = \sum_{u \in \mfb n} (-1)^{\ell_t(u)} |\PiBC_u(F_w)|,\\
    \det(A^+) = \sum_{u \in \mfb n} (-1)^{\ell_s(u)} |\PiBC_u(F_w)|,\qquad
    \det(A^-) = \sum_{u \in \mfb n} (-1)^{\ell(u)} |\PiBC_u(F_w)|,\\
    \simm n{\psi^n}(A^+) =
    \ntksp \nTksp \nTksp
    \sumsb{u \in \bn\\ \ctype(\varphi(u))=n} \ntksp \nTksp \nTksp
    n |\PiBC_u(F_w)|, \qquad
    \simm n{\psi^n}(A^-) =
    \ntksp \nTksp \nTksp
    \sumsb{u \in \bn\\ \ctype(\varphi(u))=n} \ntksp \nTksp \nTksp
    (-1)^{\ell_t(u)} n |\PiBC_u(F_w)|. 
  \end{gathered}
  \end{equation}
  By Corollary~\ref{c:zigzagbc} the cardinality
  $|\PiBC_u(F_w)|$ is $1$ if $u \leq_{\bn} w$
  and is $0$ otherwise.

  The interpretations (\ref{eq:permap}), (\ref{eq:psiap}) follow 
  from the subtraction-free expressions for $\perm(A^+)$
  and $\simm n{\psi^n}(A^+)$ in (\ref{eq:sums}).

  Now consider the interpretations (\ref{eq:permam}), (\ref{eq:psiam}).
  If $\ell_t(w) = 0$,
  then all elements
  $u \leq_{\bn} w$ also satisfy $\ell_t(u) = 0$.
  Thus the expressions
  for $\perm(A^-)$ and $\simm n{\psi^n}(A^-)$ in
  (\ref{eq:sums})
  are subtraction-free and have the claimed interpretations.
  Furthermore, since there is no path in $F_w$
  from source $1$ to sink $\ol 1$ (or source $\ol1$ to sink $1$),
  in any path family $\pi$ covering $F_w$
  paths $\pi_i$ and $\pi_j$ cannot intersect unless $i,j < 0$ or $i,j > 0$.
  On the other hand if  
  $\ell_t(w) \neq 0$,
  then $F_w$ has a factorization of the form
  (\ref{eq:bcstardef})
  which begins or ends with
  $F'_{\smash{[\ol k, k]}}$
  for some $k$.
  If the factorization begins with
  $F'_{\smash{[\ol k, k]}}$,
  define an involution on $\PiBC(F_w)$ by
  $\pi \mapsto \pi'$ where $\pi'$ is obtained from $\pi$ by swapping
  paths $\pi_1$
  and $\pi_{\ol 1}$
  after they touch at the central vertex of $F'_{\smash{[\ol k, k]}}$.
  This map satisfies
  \begin{equation*}
    \ctype(\varphi(\type(\pi'))) = \ctype(\varphi(\type(\pi))), \qquad
    \type(\pi') = t \cdot \type(\pi).
    \end{equation*}
  Thus the two families contribute to the expressions for $\det(A^-)$
  and $\simm n{\psi^n}(A^-)$ in (\ref{eq:sums}), specifically contributing
  \begin{equation*}
    (-1)^{\smash{\ell_t(\type(\pi))}} + (-1)^{\smash{\ell_t(\type(\pi))+1}} = 0
    \end{equation*}
  to each.
  If the factorization of $F_w$ ends with $F'_{\smash{[\ol k, k]}}$,
  form $\pi'$ from $\pi$
  by swapping the final portions
  (from the central vertex of $F'_{\smash{[\ol k, k]}}$ to the end)
  of paths terminating at sinks $1$, $\ol1$.
  Then we have $\type(\pi') = \type(\pi) \cdot t$ and again
  the two families together contribute $0$ to $\perm(A^-)$
  and to $\imm{\psi^n}(A^-)$.


  Now consider the interpretation (\ref{eq:detap}).
  If
  $\ell_s(w) = 0$,
  then we have $w \in \{ e, t \}$ and
  the third sum in (\ref{eq:sums})
  is subtraction-free.
  It has two terms equal to $1$ if $w = t$, and one such term if $w = e$.
  On the other hand,
  if
  $\ell_s(w) \neq 0$,
  then $F_w$ has a factorization of the form
  (\ref{eq:bcstardef})
  which contains at least one factor of the form
  $F'_{\smash{[k_1, k_2]}}$
  with $1 \leq k_1 < k_2 \leq n$
  and with $[k_1,k_2]$ maximal or minimal with respect to $\preceq$. 
  If $[k_1,k_2]$ is minimal,
  define an involution on $\PiBC(F_w)$ by $\pi \mapsto \pi'$
  where $\pi'$ is obtained from $\pi$ by swapping paths $\pi_{k_1}$ and
  $\pi_{k_1+1}$
  (and $\pi_{\ol{k_1}}$ and $\pi_{\ol{k_1+1}}$) after they intersect at the
  central vertices of $F'_{\smash{[k_1, k_2]}}$.
  Then we have $\type(\pi') = s'_{k_1} \cdot \type(\pi)$ and the two families
  together contribute
  \begin{equation*}
    (-1)^{\ell_s(\type(\pi))} + (-1)^{\ell_s(\type(\pi))\pm1} = 0
  \end{equation*}
  to $\det(A^+)$.
  If $[k_1,k_2]$ is maximal, then form $\pi'$ from $\pi$ by swapping
  the final portions (from the central vertices of $F'_{\smash{[k_1, k_2]}}$ to
  the end) of the paths terminating at sinks $k_1$, $k_1+1$
  (and $\ol{k_1}$, $\ol{k_1+1}$).
  Then we have $\type(\pi') = \type(\pi) \cdot s'_{k_1}$ and the two families
  together contribute $0$ to $\det(A^+)$.

  Finally consider the interpretation (\ref{eq:detam}).
  Repeating either of the above arguments with $\ell(w)$ in place of
  $\ell_t(w)$ or $\ell_s(w)$, we see that any network $F_w$ with $w \neq e$
  leads to a bijection in which all pairs of paths families contribute $0$.
  The only path families which are counted by $\det(A^-)$
  are those of type $e$ covering the network $F_e$.
\end{proof}

It would be interesting to define an appropriate noncommutative
ring in the variables (\ref{eq:xmatrix})
to extend the above results for $\bn$-characters
to analogous results for $\hbnq$-characters.

\bp
State and prove $q$-analogs of
Lemma~\ref{l:tensorimm} -- Proposition~\ref{p:detDBPB}.
\ep

\section{Unit interval orders}\label{s:uio}

More partial solutions to Problem~\ref{p:evaltrace} for the
subsets (\ref{eq:cwqsmooth}) -- (\ref{eq:BCcwqcodom})
of the Kazhdan--Lusztig bases employ posets called
{\em unit interval orders}, those posets for which no induced
four-element subposet is isomorphic to a disjoint union
of two two-element chains ($\mathbf2 + \mathbf2$)
or of a three-element chain and a single element ($\mathbf3 + \mathbf 1$).

In type $\msfA$,
a map $w \mapsto P(w)$
from \pavoiding permutations in $\sn$ to
unit interval orders
facilitates
combinatorial interpretations of
trace evaluations~\cite[\S 4--10]{CHSSkanEKL}.
The restriction of this map to $312$-avoiding permutations
is bijective.
In types $\msfBB$ and $\sfC$, we define an analogous map
$w \mapsto Q(w)$ from \pavoiding elements of $\bn$ to posets we
call {\em type-$\msfBC$ unit interval orders}.
The restriction of this map to elements \avoidingsignedp{}
is bijective.
These graphical representations facilitate
combinatorial interpretation of
trace evaluations
(Section~\ref{s:main}) when we specialize at $q=1$.


\ssec{Type-$\msfA$ unit interval orders}\label{ss:auio}


Fix $w \in \mfs{[h,n]}$ ($h \in \{\ol n, 1 \}$)
\avoidingp, and let $F_w$ be the planar network
corresponding to $w$ by the bijection following (\ref{eq:wFw}),
i.e., in \cite[\S 3]{SkanNNDCB}.
Given path family $\pi = (\pi_h,\dotsc,\pi_n)$ covering $F_w$,
we define a partial order $P(\pi)$ on these paths by
declaring $\pi_i <_{P(\pi)}  \pi_j$ if
\begin{enumerate}
\item $i < j$ as integers,
\item $\pi_i$ does not intersect $\pi_j$.
\end{enumerate}
For every zig-zag network $F_w$, there is a unique path family of type $e$
which covers $F_w$.  If $\pi$ is this path family, we define
\begin{equation}\label{eq:pw}
  P(w) \defeq P(\pi),
\end{equation}
and we
label the elements
of $P(w)$ by
$h,\dotsc,n$
rather than by $\pi_h,\dotsc,\pi_n$.
For example, consider the descending star networks
(\ref{eq:xfigures2})
in $\dnet{A}{[1,4]}$,
labeled $F_{4321}, \dotsc, F_{1234}$ as in (\ref{eq:dsnlist}).
The unit interval orders $P(4321), \dotsc, P(1234)$ are
\begin{equation}\label{eq:uioa}
\begin{tikzpicture}[scale=.50,baseline=0]
\draw[fill] (0,0) circle (1.2mm); \node at (0.,.5) {$\scriptstyle{1}$};
\draw[fill] (.6,0) circle (1.2mm); \node at (0.6,.5) {$\scriptstyle{2}$};
\draw[fill] (1.2,0) circle (1.2mm); \node at (1.2,.5) {$\scriptstyle{3}$};
\draw[fill] (1.8,0) circle (1.2mm); \node at (1.8,.5) {$\scriptstyle{4}$};
\end{tikzpicture}
\
\begin{tikzpicture}[scale=.50,baseline=0]
\draw[fill] (0,.5) circle (1.2mm); \node at (0,1) {$\scriptstyle{4}$};
\draw[fill] (0,-.5) circle (1.2mm); \node at (0,-1) {$\scriptstyle{1}$};
\draw[fill] (.6,0) circle (1.2mm); \node at (0.6,.5) {$\scriptstyle{2}$};
\draw[fill] (1.2,0) circle (1.2mm); \node at (1.2,.5) {$\scriptstyle{3}$};
\draw[-] (0,.5) -- (0,-.5);
\end{tikzpicture}
\
\begin{tikzpicture}[scale=.50,baseline=0]
\draw[fill] (-.3,.5) circle (1.2mm); \node at (-.3,1) {$\scriptstyle{3}$};
\draw[fill] (.3,.5) circle (1.2mm); \node at (.3,1) {$\scriptstyle{4}$};
\draw[fill] (0,-.5) circle (1.2mm); \node at (0,-1) {$\scriptstyle{1}$};
\draw[fill] (.8,0) circle (1.2mm); \node at (0.8,.5) {$\scriptstyle{2}$};
\draw[-] (-.3,.5) -- (0,-.5);
\draw[-] (.3,.5) -- (0,-.5);
\end{tikzpicture}
\
\begin{tikzpicture}[scale=.50,baseline=0]
\draw[fill] (-.3,-.5) circle (1.2mm); \node at (-.3,-1) {$\scriptstyle{1}$};
\draw[fill] (.3,-.5) circle (1.2mm); \node at (.3,-1) {$\scriptstyle{2}$};
\draw[fill] (0,.5) circle (1.2mm); \node at (0,1) {$\scriptstyle{4}$};
\draw[fill] (.8,0) circle (1.2mm); \node at (0.8,.5) {$\scriptstyle{3}$};
\draw[-] (-.3,-.5) -- (0,.5);
\draw[-] (.3,-.5) -- (0,.5);
\end{tikzpicture}
\
\begin{tikzpicture}[scale=.50,baseline=0]
\draw[fill] (-.5,.5) circle (1.2mm); \node at (-.5,1) {$\scriptstyle{2}$};
\draw[fill] (0,.5) circle (1.2mm); \node at (0,1) {$\scriptstyle{3}$};
\draw[fill] (.5,.5) circle (1.2mm); \node at (0.5,1) {$\scriptstyle{4}$};
\draw[fill] (0,-.5) circle (1.2mm); \node at (0,-1) {$\scriptstyle{1}$};
\draw[-] (0,-.5) -- (-.5,.5);
\draw[-] (0,-.5) -- (0,.5);
\draw[-] (0,-.5) -- (.5,.5);
\end{tikzpicture}
\
\begin{tikzpicture}[scale=.50,baseline=0]
\draw[fill] (-.5,-.5) circle (1.2mm); \node at (-.5,-1) {$\scriptstyle{1}$};
\draw[fill] (0,-.5) circle (1.2mm); \node at (0,-1) {$\scriptstyle{2}$};
\draw[fill] (.5,-.5) circle (1.2mm); \node at (0.5,-1) {$\scriptstyle{3}$};
\draw[fill] (0,.5) circle (1.2mm); \node at (0,1) {$\scriptstyle{4}$};
\draw[-] (0,.5) -- (-.5,-.5);
\draw[-] (0,.5) -- (0,-.5);
\draw[-] (0,.5) -- (.5,-.5);
\end{tikzpicture}
\ \,
\begin{tikzpicture}[scale=.50,baseline=0]
\draw[fill] (0,.5) circle (1.2mm); \node at (0,1) {$\scriptstyle{3}$};
\draw[fill] (0,-.5) circle (1.2mm); \node at (0,-1) {$\scriptstyle{1}$};
\draw[fill] (.8,-.5) circle (1.2mm); \node at (0.8,-1) {$\scriptstyle{2}$};
\draw[fill] (.8,.5) circle (1.2mm); \node at (.8,1) {$\scriptstyle{4}$};
\draw[-] (0,.5) -- (0,-.5);
\draw[-] (.8,.5) -- (.8,-.5);
\draw[-] (0,-.5) -- (.8,.5);
\end{tikzpicture}
\
\begin{tikzpicture}[scale=.50,baseline=0]
\draw[fill] (-.3,1) circle (1.2mm); \node at (-.3,1.5) {$\scriptstyle{3}$};
\draw[fill] (.3,1) circle (1.2mm); \node at (.3,1.5) {$\scriptstyle{4}$};
\draw[fill] (0,0) circle (1.2mm); \node at (-.5,0) {$\scriptstyle{2}$};
\draw[fill] (0,-1) circle (1.2mm); \node at (0,-1.5) {$\scriptstyle{1}$};
\draw[-] (0,-1) -- (0,0) -- (-.3,1);
\draw[-]           (0,0) -- (.3,1);
\end{tikzpicture}
\
\begin{tikzpicture}[scale=.50,baseline=0]
\draw[fill] (0,1) circle (1.2mm); \node at (0,1.5) {$\scriptstyle{4}$};
\draw[fill] (-.3,0) circle (1.2mm); \node at (-.8,0) {$\scriptstyle{2}$};
\draw[fill] (.3,0) circle (1.2mm); \node at (.8,0) {$\scriptstyle{3}$};
\draw[fill] (0,-1) circle (1.2mm); \node at (0,-1.5) {$\scriptstyle{1}$};
\draw[-] (0,-1) -- (-.3,0) -- (0,1);
\draw[-] (0,-1) -- (.3,0) -- (0,1);
\end{tikzpicture}
\
\begin{tikzpicture}[scale=.50,baseline=0]
\draw[fill] (-.3,-1) circle (1.2mm); \node at (-.3,-1.5) {$\scriptstyle{1}$};
\draw[fill] (.3,-1) circle (1.2mm); \node at (.3,-1.5) {$\scriptstyle{2}$};
\draw[fill] (0,0) circle (1.2mm); \node at (.5,0) {$\scriptstyle{3}$};
\draw[fill] (0,1) circle (1.2mm); \node at (0,1.5) {$\scriptstyle{4}$};
\draw[-] (0,1) -- (0,0) -- (-.3,-1);
\draw[-]           (0,0) -- (.3,-1);
\end{tikzpicture}
\
\begin{tikzpicture}[scale=.50,baseline=0]
\draw[fill] (0,.5) circle (1.2mm); \node at (0,1) {$\scriptstyle{3}$};
\draw[fill] (0,-.5) circle (1.2mm); \node at (0,-1) {$\scriptstyle{1}$};
\draw[fill] (.8,-.5) circle (1.2mm); \node at (0.8,-1) {$\scriptstyle{2}$};
\draw[fill] (.8,.5) circle (1.2mm); \node at (.8,1) {$\scriptstyle{4}$};
\draw[-] (0,.5) -- (0,-.5);
\draw[-] (.8,.5) -- (.8,-.5);
\draw[-] (0,-.5) -- (.8,.5);
\draw[-] (0,.5) -- (.8,-.5);
\end{tikzpicture}
\
\begin{tikzpicture}[scale=.50,baseline=0]
\draw[fill] (.3,-1) circle (1.2mm); \node at (.3,-1.5) {$\scriptstyle{1}$};
\draw[fill] (0,0) circle (1.2mm); \node at (-.5,0) {$\scriptstyle{2}$};
\draw[fill] (0,1) circle (1.2mm); \node at (0,1.5) {$\scriptstyle{4}$};
\draw[fill] (.6,.5) circle (1.2mm); \node at (.6,1) {$\scriptstyle{3}$};
\draw[-] (0,1) -- (0,0) -- (.3,-1);
\draw[-]           (.6,.5) -- (.3,-1);
\end{tikzpicture}
\
\begin{tikzpicture}[scale=.50,baseline=0]
\draw[fill] (.3,1) circle (1.2mm); \node at (.3,1.5) {$\scriptstyle{4}$};
\draw[fill] (0,0) circle (1.2mm); \node at (-.5,0) {$\scriptstyle{3}$};
\draw[fill] (0,-1) circle (1.2mm); \node at (0,-1.5) {$\scriptstyle{1}$};
\draw[fill] (.6,-.5) circle (1.2mm); \node at (.6,-1) {$\scriptstyle{2}$};
\draw[-] (0,-1) -- (0,0) -- (.3,1);
\draw[-]           (.6,-.5) -- (.3,1);
\end{tikzpicture}
\
\begin{tikzpicture}[scale=.50,baseline=0]
\draw[fill] (0,1.5) circle (1.2mm); \node at (-.5,1.5) {$\scriptstyle{4}$};
\draw[fill] (0,.5) circle (1.2mm); \node at (-.5,.5) {$\scriptstyle{3}$};
\draw[fill] (0,-.5) circle (1.2mm); \node at (-.5,-.5) {$\scriptstyle{2}$};
\draw[fill] (0,-1.5) circle (1.2mm); \node at (-.5,-1.5) {$\scriptstyle{1}$};
\draw[-] (0,1.5) -- (0,.5) -- (0,-.5) -- (0,-1.5);
\end{tikzpicture}\; ,
\end{equation}
respectively.
The map $w \mapsto P(w)$ is a surjection from \pavoiding permutations
in $\mfs{[h,n]}$ to unit interval orders on $|[h,n]|$ elements.
Furthermore, we have the following~\cite[Thm.\,4.4]{CHSSkanEKL}.
\begin{thm}\label{t:312bij}
The restriction of the map $w \mapsto P(w)$
to the subset of $312$-avoiding permutations in $\mfs{[h,n]}$
is a bijection.
\end{thm}
One may construct $P(w)$ directly from $w$ as follows.
\begin{alg}\label{a:wtop}
  Given $w = w_h \cdots w_n \in \mfs{[h,n]}$ avoiding the pattern
  $312$, do
  \begin{enumerate}
  \item Define the word $m_h \cdots m_n$ by
  $m_i = \max\{w_h,\dotsc,w_i\}$.
  \item For $i=h,\dotsc,n$ define $i <_{P(w)} j$ if and only if $j > m_i$.
    \end{enumerate}
  \end{alg}
The labels which paths in $F_w$ assign to poset elements are redundant in the
sense that they are determined up to automorphism by the structure of the poset.
Specifically, for each poset element $y$ define
\begin{equation}\label{eq:betadef}
  \beta(y) =
  \# \{ x \in P \,|\, x \leq_P y \} - \# \{ z \in P \,|\, z \geq_P y \}.
\end{equation}
It is easy to see that
the labels of $P(w)$
inherited from the zig-zag network $F_w$ satisfy
$i < j$ (as integers) if $\beta(i) < \beta(j)$.
The inverse of Algorithm~\ref{a:wtop} is the following.
\begin{alg}\label{a:ptow}
  Given unlabeled unit interval order $P$ on $|[h,n]|$ elements, do
  \begin{enumerate}
  \item For each element $y \in P$, compute
  $\beta(y) \defeq
  \# \{ x \in P \,|\, x \leq_P y \} - \# \{ z \in P \,|\, z \geq_P y \}$.
\item Label the poset elements by $[h,n]$
so that we have
  $\beta(h) \leq \cdots \leq \beta(n)$.
\item Define
$w = w_h \cdots w_n$ by
$w_j = \max ( \{ i \in [h,n] \,|\, i \not >_P j \} \ssm \{ w_h, \dotsc, w_{j-1} \} )$.
\end{enumerate}
\end{alg}

Observe that the path families of type $e$ covering the zig-zag networks
(\ref{eq:xfigureszz}), which are not descending star networks and which
have the form $F_w$ for $w$ containing the pattern $312$,
form posets isomorphic to posets 2, 4, 3, 7, 13, 12, 7, 7, respectively,
in (\ref{eq:uioa}).  It is straightforward to show that
the poset labeling inherited from $\pi$ (\ref{eq:pw})
guarantees that for some indices $i$, $j$, the
minimal and maximal elements of $P(w)$ are
given by intervals $[h,i]$ and $[j,n]$, respectively.
Furthermore we have the following. (See, e.g., \cite[p.\,33]{Fish}, \cite[\S 8.2]{Trott}.)
\begin{prop}\label{p:antichain}
  Fix $w \in \mfs{[h,n]}$ \avoidingp{} and define $P = P(w)$.
  \begin{enumerate}
  \item If $i$, $j$ are incomparable in $P$ with $i < j$ in $\mathbb Z$,
    then $[i,j]$ is an antichain in $P$.
  \item If $i <_P j$ then all elements $h,\dotsc,i$ are less than all elements
    $j,\dotsc, n$ in $P$.
  \end{enumerate}
\end{prop}


\ssec{Type-$\msfBC$ unit interval orders}\label{ss:bcuio}

For each element $w \in \bn \subseteq \snn$
\avoidingp{}, the zig-zag network $F_w$ and poset $P(w)$
are defined as in Subsections \ref{ss:Aplanar}, \ref{ss:auio}.
For example, the fourteen
posets corresponding to the
descending star networks in $\dnet{BC}{[\ol 3,3]}$ (\ref{eq:bcdsn}) are
\begin{equation}\label{eq:bcuios}
\begin{gathered}
\begin{tikzpicture}[scale=.5,baseline=0]
\draw[fill] (0,2.5) circle (1.2mm); \node at (-.5,2.5) {$\scriptstyle{3}$};
\draw[fill] (0,1.5) circle (1.2mm); \node at (-.5,1.5) {$\scriptstyle{2}$};
\draw[fill] (0,0.5) circle (1.2mm); \node at (-.5,0.5) {$\scriptstyle{1}$};
\draw[fill] (0,-0.5) circle (1.2mm); \node at (-.5,-0.5) {$\scriptstyle{\ol1}$};
\draw[fill] (0,-1.5) circle (1.2mm); \node at (-.5,-1.5) {$\scriptstyle{\ol2}$};
\draw[fill] (0,-2.5) circle (1.2mm); \node at (-.5,-2.5) {$\scriptstyle{\ol3}$};
\draw[-] (0,2.5) -- (0,1.5) -- (0,0.5) -- (0,-0.5) -- (0,-1.5) -- (0,-2.5);
\end{tikzpicture}
\,\, \quad
\begin{tikzpicture}[scale=.5,baseline=0]
\draw[fill] (0,1.5) circle (1.2mm); \node at (.5,1.5) {$\scriptstyle{3}$};
\draw[fill] (0,0.5) circle (1.2mm); \node at (.5,0.7) {$\scriptstyle{2}$};
\draw[fill] (-.5,0) circle (1.2mm); \node at (-1,0) {$\scriptstyle{\ol1}$};
\draw[fill] (.5,0) circle (1.2mm); \node at (1,0) {$\scriptstyle{1}$};
\draw[fill] (0,-0.5) circle (1.2mm); \node at (-.5,-0.7) {$\scriptstyle{\ol2}$};
\draw[fill] (0,-1.5) circle (1.2mm); \node at (-.5,-1.5) {$\scriptstyle{\ol3}$};
\draw[-] (0,1.5) -- (0,0.5) -- (-.5,0) -- (0,-0.5) -- (0,-1.5);
\draw[-]            (0,0.5) -- (0.5,0) -- (0,-0.5);
\end{tikzpicture}
\quad
\begin{tikzpicture}[scale=.5,baseline=0]
\draw[fill] (0,1) circle (1.2mm); \node at (0,1.5) {$\scriptstyle{3}$};
\draw[fill] (-.5,0.5) circle (1.2mm); \node at (-1,0.5) {$\scriptstyle{1}$};
\draw[fill] (.5,0.5) circle (1.2mm); \node at (1,0.5) {$\scriptstyle{2}$};
\draw[fill] (-.5,-0.5) circle (1.2mm); \node at (-1,-0.5) {$\scriptstyle{\ol2}$};
\draw[fill] (.5,-0.5) circle (1.2mm); \node at (1,-.5) {$\scriptstyle{\ol1}$};
\draw[fill] (0,-1) circle (1.2mm); \node at (0,-1.5) {$\scriptstyle{\ol3}$};
\draw[-] (0,1) -- (-.5,0.5) -- (-.5,-0.5) -- (0,-1) -- (.5,-0.5) -- (.5,0.5) -- (0,1);
\draw[-]           (-.5,0.5) -- (0.5,-0.5);
\draw[-]           (-.5,-0.5) -- (0.5,0.5);
\end{tikzpicture}
\quad
\begin{tikzpicture}[scale=.5,baseline=0]
\draw[fill] (0,0.5) circle (1.2mm); \node at (.5,0.3) {$\scriptstyle{1}$};
\draw[fill] (-.5,1) circle (1.2mm); \node at (-.5,1.5) {$\scriptstyle{2}$};
\draw[fill] (.5,1) circle (1.2mm); \node at (.5,1.5) {$\scriptstyle{3}$};
\draw[fill] (-.5,-1) circle (1.2mm); \node at (-.5,-1.5) {$\scriptstyle{\ol3}$};
\draw[fill] (.5,-1) circle (1.2mm); \node at (.5,-1.5) {$\scriptstyle{\ol2}$};
\draw[fill] (0,-.5) circle (1.2mm); \node at (-.5,-0.3) {$\scriptstyle{\ol1}$};
\draw[-] (-.5,1) -- (0,0.5) -- (0,-0.5) -- (-.5,-1);
\draw[-] (.5,1) -- (0,0.5);
\draw[-] (.5,-1) -- (0,-0.5);
\end{tikzpicture}
\quad
\begin{tikzpicture}[scale=.5,baseline=0]
\draw[fill] (0,1) circle (1.2mm); \node at (0,1.5) {$\scriptstyle{3}$};
\draw[fill] (-.5,0.5) circle (1.2mm); \node at (-1,0.5) {$\scriptstyle{1}$};
\draw[fill] (.5,0.5) circle (1.2mm); \node at (1,0.5) {$\scriptstyle{2}$};
\draw[fill] (-.5,-0.5) circle (1.2mm); \node at (-1,-0.5) {$\scriptstyle{\ol2}$};
\draw[fill] (.5,-0.5) circle (1.2mm); \node at (1,-.5) {$\scriptstyle{\ol1}$};
\draw[fill] (0,-1) circle (1.2mm); \node at (0,-1.5) {$\scriptstyle{\ol3}$};
\draw[-] (0,1) -- (-.5,0.5) -- (-.5,-0.5) -- (0,-1) -- (.5,-0.5) -- (.5,0.5) -- (0,1);
\draw[-]           (-.5,-0.5) -- (0.5,0.5);
\end{tikzpicture}
\quad
\begin{tikzpicture}[scale=.5,baseline=0]
\draw[fill] (-.5,1) circle (1.2mm); \node at (-.5,1.5) {$\scriptstyle{2}$};
\draw[fill] (-.5,0) circle (1.2mm); \node at (-1,0) {$\scriptstyle{\ol1}$};
\draw[fill] (-.5,-1) circle (1.2mm); \node at (-.5,-1.5) {$\scriptstyle{\ol3}$};
\draw[fill] (.5,1) circle (1.2mm); \node at (.5,1.5) {$\scriptstyle{3}$};
\draw[fill] (.5,0) circle (1.2mm); \node at (1,0) {$\scriptstyle{1}$};
\draw[fill] (.5,-1) circle (1.2mm); \node at (.5,-1.5) {$\scriptstyle{\ol2}$};
\draw[-] (-.5,1) -- (-.5,0) -- (-.5,-1);
\draw[-] (.5,1) -- (.5,0) -- (.5,-1);
\draw[-] (-.5,1) -- (.5,0) -- (-.5,-1);
\draw[-] (.5,1) -- (-.5,0) -- (.5,-1);
\end{tikzpicture}
\quad
\begin{tikzpicture}[scale=.5,baseline=0]
\draw[fill] (-.5,1.25) circle (1.2mm); \node at (-1,1.25) {$\scriptstyle{3}$};
\draw[fill] (-.5,.25) circle (1.2mm); \node at (-1,.25) {$\scriptstyle{1}$};
\draw[fill] (-.5,-.75) circle (1.2mm); \node at (-1,-.75) {$\scriptstyle{\ol2}$};
\draw[fill] (.5,.75) circle (1.2mm); \node at (1,.75) {$\scriptstyle{2}$};
\draw[fill] (.5,-.25) circle (1.2mm); \node at (1,-.250) {$\scriptstyle{\ol1}$};
\draw[fill] (.5,-1.25) circle (1.2mm); \node at (1,-1.25) {$\scriptstyle{\ol3}$};
\draw[-] (-.5,1.25) -- (-.5,0.25) -- (-.5,-.75);
\draw[-] (.5,-1.25) -- (.5,-.25) -- (.5,.75);
\draw[-] (-.5,.25) -- (.5,-.25);
\draw[-] (-.5,-.75) -- (.5,.75);
\end{tikzpicture}
\quad
\begin{tikzpicture}[scale=.5,baseline=0]
\draw[fill] (-.5,1.25) circle (1.2mm); \node at (-1,1.25) {$\scriptstyle{3}$};
\draw[fill] (-.5,.25) circle (1.2mm); \node at (-1,.25) {$\scriptstyle{1}$};
\draw[fill] (-.5,-.75) circle (1.2mm); \node at (-1,-.75) {$\scriptstyle{\ol2}$};
\draw[fill] (.5,.75) circle (1.2mm); \node at (1,.75) {$\scriptstyle{2}$};
\draw[fill] (.5,-.25) circle (1.2mm); \node at (1,-.250) {$\scriptstyle{\ol1}$};
\draw[fill] (.5,-1.25) circle (1.2mm); \node at (1,-1.25) {$\scriptstyle{\ol3}$};
\draw[-] (-.5,1.25) -- (-.5,0.25) -- (-.5,-.75);
\draw[-] (.5,-1.25) -- (.5,-.25) -- (.5,.75);
\draw[-] (-.5,1.25) -- (.5,-.25);
\draw[-] (-.5,.25) -- (.5,-1.25);
\draw[-] (-.5,-.75) -- (.5,.75);
\end{tikzpicture}
\\
\begin{tikzpicture}[scale=.5,baseline=0]
\draw[fill] (-1,.5) circle (1.2mm); \node at (-1,1) {$\scriptstyle{1}$};
\draw[fill] (-1,-.5) circle (1.2mm); \node at (-1,-1) {$\scriptstyle{\ol3}$};
\draw[fill] (0,.5) circle (1.2mm); \node at (0,1) {$\scriptstyle{2}$};
\draw[fill] (0,-.5) circle (1.2mm); \node at (0,-1) {$\scriptstyle{\ol2}$};
\draw[fill] (1,.5) circle (1.2mm); \node at (1,1) {$\scriptstyle{3}$};
\draw[fill] (1,-.5) circle (1.2mm); \node at (1,-1) {$\scriptstyle{\ol1}$};
\draw[-] (-1,.5) -- (-1,-.5);
\draw[-] (-1,.5) -- (0,-.5);
\draw[-] (-1,.5) -- (1,-.5);
\draw[-] (0,.5) -- (-1,-.5);
\draw[-] (0,.5) -- (0,-.5);
\draw[-] (0,.5) -- (1,-.5);
\draw[-] (1,.5) -- (-1,-.5);
\draw[-] (1,.5) -- (0,-.5);
\draw[-] (1,.5) -- (1,-.5);
\end{tikzpicture}
\quad
\begin{tikzpicture}[scale=.5,baseline=0]
\draw[fill] (-1,.5) circle (1.2mm); \node at (-1,1) {$\scriptstyle{1}$};
\draw[fill] (-1,-.5) circle (1.2mm); \node at (-1,-1) {$\scriptstyle{\ol3}$};
\draw[fill] (0,.5) circle (1.2mm); \node at (0,1) {$\scriptstyle{2}$};
\draw[fill] (0,-.5) circle (1.2mm); \node at (0,-1) {$\scriptstyle{\ol2}$};
\draw[fill] (1,.5) circle (1.2mm); \node at (1,1) {$\scriptstyle{3}$};
\draw[fill] (1,-.5) circle (1.2mm); \node at (1,-1) {$\scriptstyle{\ol1}$};
\draw[-] (-1,.5) -- (-1,-.5);
\draw[-] (-1,.5) -- (0,-.5);
\draw[-] (0,.5) -- (-1,-.5);
\draw[-] (0,.5) -- (0,-.5);
\draw[-] (0,.5) -- (1,-.5);
\draw[-] (1,.5) -- (-1,-.5);
\draw[-] (1,.5) -- (0,-.5);
\draw[-] (1,.5) -- (1,-.5);
\end{tikzpicture}
\quad
\begin{tikzpicture}[scale=.5,baseline=0]
\draw[fill] (0,1) circle (1.2mm); \node at (0,1.5) {$\scriptstyle{3}$};
\draw[fill] (-1.5,0) circle (1.2mm); \node at (-1.9,0) {$\scriptstyle{\ol2}$};
\draw[fill] (-.5,0) circle (1.2mm); \node at (-.9,0) {$\scriptstyle{\ol1}$};
\draw[fill] (.5,0) circle (1.2mm); \node at (.9,0) {$\scriptstyle{1}$};
\draw[fill] (1.5,0) circle (1.2mm); \node at (1.9,0) {$\scriptstyle{2}$};
\draw[fill] (0,-1) circle (1.2mm); \node at (0,-1.5) {$\scriptstyle{\ol3}$};
\draw[-] (0,1) -- (-1.5,0) -- (0,-1);
\draw[-] (0,1) -- (-.5,0) -- (0,-1);
\draw[-] (0,1) -- (.5,0) -- (0,-1);
\draw[-] (0,1) -- (1.5,0) -- (0,-1);
\end{tikzpicture}
\quad
\begin{tikzpicture}[scale=.5,baseline=0]
\draw[fill] (0,1) circle (1.2mm); \node at (0,1.5) {$\scriptstyle{3}$};
\draw[fill] (-1.5,0) circle (1.2mm); \node at (-1.9,0) {$\scriptstyle{\ol2}$};
\draw[fill] (-.5,0) circle (1.2mm); \node at (-.9,0) {$\scriptstyle{\ol1}$};
\draw[fill] (.5,0) circle (1.2mm); \node at (.9,0) {$\scriptstyle{1}$};
\draw[fill] (1.5,0) circle (1.2mm); \node at (1.9,0) {$\scriptstyle{2}$};
\draw[fill] (0,-1) circle (1.2mm); \node at (0,-1.5) {$\scriptstyle{\ol3}$};
\draw[-] (0,1) -- (-1.5,0);
\draw[-] (0,1) -- (-.5,0) -- (0,-1);
\draw[-] (0,1) -- (.5,0) -- (0,-1);
\draw[-] (1.5,0) -- (0,-1);
\end{tikzpicture}
\quad
\begin{tikzpicture}[scale=.5,baseline=0]
\draw[fill] (-1,.5) circle (1.2mm); \node at (-1,1) {$\scriptstyle{1}$};
\draw[fill] (-1,-.5) circle (1.2mm); \node at (-1,-1) {$\scriptstyle{\ol3}$};
\draw[fill] (0,.5) circle (1.2mm); \node at (0,1) {$\scriptstyle{2}$};
\draw[fill] (0,-.5) circle (1.2mm); \node at (0,-1) {$\scriptstyle{\ol2}$};
\draw[fill] (1,.5) circle (1.2mm); \node at (1,1) {$\scriptstyle{3}$};
\draw[fill] (1,-.5) circle (1.2mm); \node at (1,-1) {$\scriptstyle{\ol1}$};
\draw[-] (-1,.5) -- (-1,-.5);
\draw[-] (0,.5) -- (-1,-.5);
\draw[-] (1,.5) -- (-1,-.5);
\draw[-] (1,.5) -- (0,-.5);
\draw[-] (1,.5) -- (1,-.5);
\end{tikzpicture}
\quad
\begin{tikzpicture}[scale=.5,baseline=0]
\draw[fill] (-2,0) circle (1.2mm); \node at (-2,.5) {$\scriptstyle{\ol3}$};
\draw[fill] (-1.2,0) circle (1.2mm); \node at (-1.2,.5) {$\scriptstyle{\ol2}$};
\draw[fill] (-.4,0) circle (1.2mm); \node at (-0.4,.5) {$\scriptstyle{\ol1}$};
\draw[fill] (.4,0) circle (1.2mm); \node at (0.4,.5) {$\scriptstyle{1}$};
\draw[fill] (1.2,0) circle (1.2mm); \node at (1.2,.5) {$\scriptstyle{2}$};
\draw[fill] (2,0) circle (1.2mm); \node at (2,.5) {$\scriptstyle{3}$};
\end{tikzpicture}\;,
\end{gathered}
\end{equation}
respectively.
Observe that the path families of type $e$ covering the
zig-zag networks (\ref{eq:bczznotdsn})
which are not descending star networks
form posets isomorphic to posets
$5$, $7$, $8$, $8$, $8$, $10$, $12$, $13$, respectively in (\ref{eq:bcuios}).

The conditions preceding (\ref{eq:BCpathfams}),
which define $\msfBC$-path families, guarantee that each such poset
$P(w)$
is self-dual with antiautomorphism
$i \mapsto \ol i$.
Thus
it belongs to the class of {\em type-$\msfC$ posets}
defined in \cite[Defn.\,10]{CMRIndexSpec}.
Since $P(w)$ is a unit interval order, we also have the following.
\begin{prop}\label{p:signedposet}
  Fix $w \in \bn$ \avoidingp{} and define $P = P(w)$.
  Let $\pi$
  be the unique path familiy of type $e$ covering $F_w$, and let
  $i+1$ be the smallest element of $[1,n]$ such that $\pi_{i+1}$ is not grounded.
  Then we have
  \begin{enumerate}
  \item if $i>0$ then $[\ol i, i]$ is an antichain in $P$,
  \item $\ol n,\dotsc, \ol{i+1}$ are less than $1, \dotsc,n$ in $P$,
  \item $\ol n, \dotsc, \ol1$ are less than $i+1,\dotsc,n$ in $P$.
  \end{enumerate}
\end{prop}
\begin{proof} (1) Since $\pi_{\ol i}$ and $\pi_i$ intersect in $F_w$,
  elements $\ol i$ and $i$ are incomparable in $P$.
  By Proposition~\ref{p:antichain}, $[\ol i, i]$ is an antichain in $P$.

  \noindent (2),(3)  Suppose that $\ol{i+1}$ is incomparable to $1$ in $P$.
  By symmetry, $\ol1$ is incomparable to $i+1$ as well.
  Then $\pi_{\ol{i+1}}$ and $\pi_1$ intersect, as do $\pi_{\ol1}$ and $\pi_{i+1}$.
  Factor $F_w$ as in (\ref{eq:bcbulletconcat})
  and suppose that paths $\pi_{\ol{i+1}}$, $\pi_1$
  meet in $F'_{[c_k,d_k]}$.
  By the definition of $\msfBC$-path family, 
  paths $\pi_{{i+1}}$, $\pi_{\ol1}$ meet there as well.
  If $c_k \geq 1$ then $\pi_{\ol1}$, $\pi_1$ cross twice, contradicting
  the uniquess of $\pi$ of type $e$ covering $F_w$.  Thus we have that
  $c_k = \ol{d_k}$.  But then $\pi_{\ol{i+1}}$, $\pi_{i+1}$ meet as well,
  contradicting the assumption that these paths are not grounded.
  We conclude that $\ol{i+1} <_P 1$ and $\ol1 <_P i+1$.
  Now Proposition~\ref{p:antichain} gives the desired results.
\end{proof}
The self-duality $i \mapsto \ol i$ of $P(w)$
and \cite[Lem.\,1.1]{SFischerThesis} show that $P(w)$
is a {\em signed poset} as
defined in~\cite{SFischerThesis}, \cite{ReinerParset}.
By Proposition~\ref{p:signedposet}
the information in $P(w)$ can be recorded by the subposet induced by elements
$[1,n]$,
if we circle elements corresponding to grounded paths of $\pi$.
(This is not true of signed posets in general.)
Call this decorated poset $Q(w)$, and in general,
define a {\em type-$\msfBC$ unit interval order} to be
a unit interval order decorated by circling
a (possibly empty) subset of minimal elements, declared to be {\em grounded},
with the property that if element $i$ is grounded and $j$ is not,
then $\beta(i) \leq \beta(j)$, where $\beta$ is the function defined in
Algorithm~\ref{a:ptow}.
We define an isomorphism of type-$\msfBC$ unit interval orders
to be a poset isomorphism which respects circled elements.
For example, the $3$-element
type-$\msfBC$ unit interval orders $Q(w)$ 
corresponding to the $6$-element unit interval orders $P(w)$ in
(\ref{eq:bcuios}) are
\begin{equation}\label{eq:uioc}
\begin{gathered}
\begin{tikzpicture}[scale=.5,baseline=0]
\draw[fill] (0,1) circle (1.2mm); \node at (-.5,1) {$\scriptstyle{3}$};
\draw[fill] (0,0) circle (1.2mm); \node at (-.5,0) {$\scriptstyle{2}$};
\draw[fill] (0,-1) circle (1.2mm); \node at (-.5,-1) {$\scriptstyle{1}$};
\draw[-] (0,1) -- (0,0) -- (0,-1);
\end{tikzpicture}
\quad
\begin{tikzpicture}[scale=.5,baseline=0]
\draw[fill] (0,1) circle (1.2mm); \node at (-.5,1) {$\scriptstyle{3}$};
\draw[fill] (0,0) circle (1.2mm); \node at (-.5,0) {$\scriptstyle{2}$};
\draw[fill] (0,-1) circle (1.2mm); \node at (-.5,-1) {$\scriptstyle{1}$};
\draw (0,-1) circle (2.8mm); 
\draw[-] (0,1) -- (0,0) -- (0,-1);
\end{tikzpicture}
\quad
\begin{tikzpicture}[scale=.5,baseline=0]
\draw[fill] (0,.5) circle (1.2mm); \node at (0,1) {$\scriptstyle{3}$};
\draw[fill] (.5,-.5) circle (1.2mm); \node at (-.5,-1.1) {$\scriptstyle{1}$};
\draw[fill] (-.5,-.5) circle (1.2mm); \node at (.5,-1.1) {$\scriptstyle{2}$};
\draw[-] (-.5,-.5) -- (0,.5) -- (.5,-.5);
\end{tikzpicture}
\quad
\begin{tikzpicture}[scale=.5,baseline=0]
\draw[fill] (0,-.5) circle (1.2mm); \node at (0,-1.1) {$\scriptstyle{1}$};
\draw[fill] (.5,.5) circle (1.2mm); \node at (-.5,1) {$\scriptstyle{2}$};
\draw[fill] (-.5,.5) circle (1.2mm); \node at (.5,1) {$\scriptstyle{3}$};
\draw[-] (-.5,.5) -- (0,-.5) -- (.5,.5);
\end{tikzpicture}
\quad
\begin{tikzpicture}[scale=.5,baseline=0]
\draw[fill] (0,.5) circle (1.2mm); \node at (0,1) {$\scriptstyle{3}$};
\draw[fill] (.5,-.5) circle (1.2mm); \node at (.5,-1.1) {$\scriptstyle{2}$};
\draw[fill] (-.5,-.5) circle (1.2mm); \node at (-.5,-1.1) {$\scriptstyle{1}$};
\draw (-.5,-.5) circle (2.8mm); 
\draw[-] (-.5,-.5) -- (0,.5) -- (.5,-.5);
\end{tikzpicture}
\quad
\begin{tikzpicture}[scale=.5,baseline=0]
\draw[fill] (0,-.5) circle (1.2mm); \node at (0,-1.1) {$\scriptstyle{1}$};
\draw[fill] (.5,.5) circle (1.2mm); \node at (-.5,1) {$\scriptstyle{2}$};
\draw[fill] (-.5,.5) circle (1.2mm); \node at (.5,1) {$\scriptstyle{3}$};
\draw (0,-.5) circle (2.8mm); 
\draw[-] (-.5,.5) -- (0,-.5) -- (.5,.5);
\end{tikzpicture}
\quad
\begin{tikzpicture}[scale=.5,baseline=0]
\draw[fill] (0,-.5) circle (1.2mm); \node at (0,-1.1) {$\scriptstyle{1}$};
\draw[fill] (0,.5) circle (1.2mm); \node at (0,1) {$\scriptstyle{3}$};
\draw[fill] (.5,0) circle (1.2mm); \node at (1,0) {$\scriptstyle{2}$};
\draw[-] (0,-.5) -- (0,.5);
\end{tikzpicture}
\quad
\begin{tikzpicture}[scale=.5,baseline=0]
\draw[fill] (0,-.5) circle (1.2mm); \node at (0,-1.1) {$\scriptstyle{1}$};
\draw[fill] (0,.5) circle (1.2mm); \node at (0,1) {$\scriptstyle{3}$};
\draw[fill] (.5,0) circle (1.2mm); \node at (1,0) {$\scriptstyle{2}$};
\draw (0,-.5) circle (2.8mm); 
\draw[-] (0,-.5) -- (0,.5);
\end{tikzpicture}
\\
\begin{tikzpicture}[scale=.5,baseline=0]
\draw[fill] (-.65,0) circle (1.2mm); \node at (-.65,.6) {$\scriptstyle{1}$};
\draw[fill] (0,0) circle (1.2mm); \node at (0,.6) {$\scriptstyle{2}$};
\draw[fill] (.65,0) circle (1.2mm); \node at (.65,.6) {$\scriptstyle{3}$};
\end{tikzpicture}
\quad
\begin{tikzpicture}[scale=.5,baseline=0]
\draw[fill] (-.65,0) circle (1.2mm); \node at (-.65,.6) {$\scriptstyle{1}$};
\draw[fill] (0,0) circle (1.2mm); \node at (0,.6) {$\scriptstyle{2}$};
\draw[fill] (.65,0) circle (1.2mm); \node at (.65,.6) {$\scriptstyle{3}$};
\draw (-.65,0) circle (2.8mm); 
\end{tikzpicture}
\quad
\begin{tikzpicture}[scale=.5,baseline=0]
\draw[fill] (0,.5) circle (1.2mm); \node at (0,1) {$\scriptstyle{3}$};
\draw[fill] (.5,-.5) circle (1.2mm); \node at (-.5,-1.1) {$\scriptstyle{1}$};
\draw[fill] (-.5,-.5) circle (1.2mm); \node at (.5,-1.1) {$\scriptstyle{2}$};
\draw (.5,-.5) circle (2.8mm);
\draw (-.5,-.5) circle (2.8mm); 
\draw[-] (-.5,-.5) -- (0,.5) -- (.5,-.5);
\end{tikzpicture}
\quad
\begin{tikzpicture}[scale=.5,baseline=0]
\draw[fill] (0,-.5) circle (1.2mm); \node at (0,-1.1) {$\scriptstyle{1}$};
\draw[fill] (0,.5) circle (1.2mm); \node at (0,1) {$\scriptstyle{3}$};
\draw[fill] (.5,0) circle (1.2mm); \node at (1,0) {$\scriptstyle{2}$};
\draw (0,-.5) circle (2.8mm);
\draw (.5,0) circle (2.8mm); 
\draw[-] (0,-.5) -- (0,.5);
\end{tikzpicture}
\quad
\begin{tikzpicture}[scale=.5,baseline=0]
\draw[fill] (-.65,0) circle (1.2mm); \node at (-.65,.6) {$\scriptstyle{1}$};
\draw[fill] (0,0) circle (1.2mm); \node at (0,.6) {$\scriptstyle{2}$};
\draw[fill] (.65,0) circle (1.2mm); \node at (.65,.6) {$\scriptstyle{3}$};
\draw (-.65,0) circle (2.8mm);
\draw (0,0) circle (2.8mm);
\end{tikzpicture}
\quad
\begin{tikzpicture}[scale=.5,baseline=0]
\draw[fill] (-.65,0) circle (1.2mm); \node at (-.65,.6) {$\scriptstyle{1}$};
\draw[fill] (0,0) circle (1.2mm); \node at (0,.6) {$\scriptstyle{2}$};
\draw[fill] (.65,0) circle (1.2mm); \node at (.65,.6) {$\scriptstyle{3}$};
\draw (-.65,0) circle (2.8mm);
\draw (0,0) circle (2.8mm);
\draw (.65,0) circle (2.8mm); 
\end{tikzpicture}\,.
\end{gathered}
\end{equation}

If we remove labels from the map $w \mapsto Q(w)$,
we obtain a surjection from \pavoiding
elements of $\bn$ to type-$\msfBC$ unit interval orders.
The restriction of this map to the subset of $\bn$ \avoidingsignedp{}
is a bijection.
Equivalently, we have the following.


\begin{prop}\label{p:rtorw}
  The map $F_w \mapsto Q(w)$
  from
  $\dnet{BC}{[\ol n,n]}$
  to type-$\msfBC$ unit interval orders is bijective.
\end{prop}
\begin{proof}
  To see that the map is injective, consider
  $F_v \neq F_w$ in $\dnet{BC}{[\ol n,n]}$.
  By \cite[Thm.\,4.4]{CHSSkanEKL} we have $P(v) \neq P(w)$,
  since for each fixed unit interval order $P$ on $2n$ elements,
  the set
  $\{ F_w \in \znet{BC}{[\ol n,n]} \,|\, P(w) = P \}$
  contains exactly one type-$\msfBC$
  descending star network: the rearrangement
  $F'_{[a_{u_1},b_{u_1}]} \bullet \cdots \bullet F'_{[a_{u_t},b_{u_t}]}$ of
  (\ref{eq:bcbulletconcat}) satisfying $a_{u_1} > \cdots > a_{u_t}$
  (and $b_{u_1} > \cdots > b_{u_t}$).
  Now let $P'(v)$, $P'(w)$ be the subposets of $P(v)$ and $P(w)$ induced
  by elements $\{ 1, \dotsc, n \}$.  If $P'(v) \neq P'(w)$ then
  we clearly have $Q(v) \neq Q(w)$.  Suppose
  therefore that $P'(v) = P'(w)$.  Since $P(v) \neq P(w)$,
  there must be two indices $i \neq j$ such that
  elements $1, \dotsc, i$ of $P'(v)$ are grounded, and elements
  $1, \dotsc, j$ of $P'(w)$ are grounded.
  Again we have $Q(v) \neq Q(w)$.

  To see that the map is surjective, consider a type-$\msfBC$ unit interval
  order $Q$ on $n$ elements with elements labeled
  as in Algorithm~\ref{a:qtow}
  and with a subset $\{1, \dotsc, i\}$ of minimal elements circled,
  for some $i$.
  Let $F_u \in \dnet{A}{[1,n]}$, $u \in \sn$,
  be the descending star network corresponding
  to $Q$ viewed as an ordinary poset, ignoring circles,
  and write $F_u = F_{[a_1,b_1]} \bullet \cdots \bullet F_{[a_t,b_t]}$ as
  in Definition~\ref{d:adsn}.  Now construct
  $F'_{[a_1,b_1]} \bullet \cdots \bullet F'_{[a_t,b_t]}
  \bullet F'_{[\ol i, i]}$ in $\dnet{BC}{[\ol n,n]}$ and call this $F_w$
  for $w \in \bn$.
  It is easy to see that we have $F_w \mapsto Q$,
  i.e., $Q = Q(w)$.
\end{proof}

The bijection $w \mapsto Q(w)$, which we have defined to be the composition
\begin{equation}\label{eq:composition}
  w \mapsto F_w \mapsto P(w) \mapsto Q(w)
  \end{equation}
of the three maps
described in \cite[\S 3]{SkanNNDCB}, (\ref{eq:pw}),
and before (\ref{eq:uioc}),
can also be described
by the following algorithm.

\begin{alg}\label{a:wtoq}
    Given $w \in \bn$ \avoidingsignedp, do
    \begin{enumerate}
    \item Let $b$ be the least positive letter in $\{w_1, \dotsc w_n, n+1\}$.
  \item  
    Define
    the word $m_1 \cdots m_n$ by
    $m_j = \max \{ b-1, w_1,\dotsc, w_j \}$. 
 \item For $j = 1,\dotsc, n-1$, define $j <_{Q(w)} m_j+1, \dotsc, n$.
  \item For $j = 1,\dotsc, n$, if $w_j < 0$ then circle element $|w_j|$.
  \end{enumerate}
  \end{alg}

\begin{prop}\label{p:wtoqworks}
  For $w \in \bn$ \avoidingsignedp, the composition
  (\ref{eq:composition}) agrees with 
  Algorithm~\ref{a:wtoq}.
\end{prop}
\begin{proof}
  Computing $Q(w)$ via the composition (\ref{eq:composition}),
  we let $F_w$ be the descending star network
  given by (\ref{eq:wFw}), i.e., \cite[\S 3]{SkanNNDCB}.
  To construct $P(w)$, let
  \begin{equation*}
    \pi = (\pi_{\ol n}, \dotsc, \pi_{\ol 1}, \pi_1,\dotsc, \pi_n),
    \qquad
    \sigma = (\sigma_{\ol n}, \dotsc, \sigma_{\ol 1}, \sigma_1,\dotsc, \sigma_n)
  \end{equation*}
  be the unique path families of types $e$ and $w$ covering $F_w$,
  and for $j = 1,\dotsc, n-1$ find the elements
  $k \in P$ satisfying $j <_P k$.
  First we claim that for $i^*$ maximizing $\{ w_i \,|\, i \in [\ol n,j]\}$,
  we have that
  \begin{equation}\label{eq:premax}
    j \not <_P \ol n, \dotsc, w_{i^*}.
  \end{equation}
  By the definition of $P(w)$ we have $j \not <_P \ol n,\dotsc, j$,
  and by the pigeonhole principle, we have
  $w_{i^*} \geq j$ (as integers).
  Since the path $\sigma_{i^*}$ from source $i^*$ to sink $w_{i^*}$ intersects
  the path $\pi_j$ from source $j$ to sink $j$,
  we have a path from source $j$ to sink $w_{i^*}$.
  This path in turn intersects all paths $\pi_{j+1}, \dotsc, \pi_{w_{i^*}}$,
  and we have paths from source $j$ to all sinks $j+1, \dotsc, w_{i^*}$.
  Since
  the subnetwork of $F$ covered by paths $\pi_1, \dotsc, \pi_n$
  is isomorphic to a type-$\msfA$ descending star network,
  we may apply Lemma~\ref{l:lemma3.5} to conclude that $\pi_j$ intersects
  $\pi_{j+1}, \dotsc, \pi_{w_{i^*}}$.
  Thus we 
  obtain the remaining
  inequalities $j \not <_P j+1, \dotsc, w_{i^*}$ in
  (\ref{eq:premax}).
  
Now we claim that
  \begin{equation}\label{eq:postmax}
    j <_P w_{i^*}+1, \dotsc, n.
  \end{equation}
  Consider the paths $\pi_k$ for $k > w_{i^*}$.
Again by Lemma~\ref{l:lemma3.5}, paths $\pi_k$ and $\pi_j$ do not intersect,
since there is no path in $F_w$ from source $j$ to sink $k$.  Thus we have
$j <_P k$ as in (\ref{eq:postmax}).  

Now we define $b$ to be 
the least positive letter in $\{ w_1, \dotsc, w_n, n+1 \}$, and
we claim that
  \begin{equation}\label{eq:istarclaim}
      w_{i^*} = \max\{ b-1, w_1, \dotsc, w_j \}.
  \end{equation}
  By Lemma~\ref{l:1ol2ol21}, the set of positive letters in $w_1 \cdots w_n$
  is empty or forms the interval $[b,n]$.
  If this set is empty, then avoidance of the signed pattern $\ol2\ol1$ implies
  that $w_1 \cdots w_n = \ol1 \cdots \ol n$.  Thus we have
  $w_{i^*} = \max \{ n, \dotsc, 1, \ol1, \dotsc, \ol j \} = n$,
  $b = n+1$, and the right-hand-side of (\ref{eq:istarclaim}) is $n$.
  Suppose therefore that the positive letters are $[b,n]$.
  Then the positive letters $[1,b-1]$ appear in $w_{\ol n}\cdots w_{\ol1}$.
  This allows us to write
  \begin{equation*}
      w_{i^*} = \max\{ w_i \,|\, i \leq j \}
      = \max(\{ w_{\ol n}, \dotsc, w_{\ol 1} \} \cup \{ w_1, \dotsc, w_j \})
      = \max\{ b-1, w_1, \dotsc, w_j \}.
  \end{equation*}

  The subposet of $P(w)$ induced by $[1,n]$, which will become $Q(w)$,
  now agrees with steps (1) -- (3) of Algorithm~\ref{a:wtoq}. 
  To complete the construction of $Q(w)$ by (\ref{eq:composition}),
  we circle grounded elements of $P(w)$, if there are any.
  If no path of $\pi$ is grounded, then we do nothing.
  In this case,
  no path of $\sigma$ has a source and sink with different signs,
  all letters in $w_1 \cdots w_n$ are positive, and nothing is done in
  step (4) of Algorithm~\ref{a:wtoq}.
  On the other hand, if some $2k$ paths of $\pi$ are grounded,
  then by Proposition~\ref{p:signedposet}, these paths are
  $(\pi_{\ol k},\dotsc, \pi_{\ol1}, \pi_1, \dotsc, \pi_k)$,
  and we circle elements $1, \dotsc, k$ of $P(w)$ to form $Q(w)$
  by (\ref{eq:composition}).
  In this case,
  $\smash{F'_{[c_t,d_t]} = F'_{[\ol k, k]}}$ is the last factor
  in the expression
  (\ref{eq:bcbulletconcat}) 
  for $F_w$,
  and the letters $\ol 1,\dotsc, \ol k$
  appear in $w_1 \cdots w_n$.
  Thus in step (4) of Algorithm~\ref{a:wtoq},
  elements $1, \dotsc, k$ are circled.
\end{proof}

Like Algorithm~\ref{a:FwF},
Algorithm~\ref{a:wtoq} is invertible even if labels of the poset $Q$ are not
given.
\begin{alg}\label{a:qtow}
  Given unlabeled type-$\msfBC$ unit interval order $Q$
  with $p$ circled elements, do
  \begin{enumerate}
  \item For each element $y \in Q$, compute
      $\beta(y) \defeq
  \# \{ x \in Q \,|\, x \leq_Q y \} - \# \{ z \in Q \,|\, z \geq_Q y \}$.
\item Label the poset elements by $[1,n]$
so that we have
$\beta(1) \leq \cdots \leq \beta(n)$, and so that circled elements
form the interval $[1,p]$.
\item Define the word
  $a_1 \cdots a_n = \ol 1 \cdots \ol p (p+1) \cdots n$.
\item Define
$w = w_1 \cdots w_n$ by
$w_j = \max ( \{ a_i \,|\, i \not >_Q j \} \ssm \{ w_1, \dotsc, w_{j-1} \} )$.
\end{enumerate}
\end{alg}
To see that
Algorithm~\ref{a:qtow} inverts Algorithm~\ref{a:wtoq},
we consider a close relationship between certain descending
star networks of types $\msfA$ and $\msfBC$.
\begin{lem}\label{l:abc}
  Fix $w \in \bn$ \avoidingsignedp{} with $p>0$
  negative
  letters $(\ol 1, \dotsc, \ol p)$ appearing in $w_1 \cdots w_n$,
  and
  type-$\msfBC$ descending star network
  \begin{equation*}
    F_w =
    F'_{[c_1,d_1]} \bullet \cdots \bullet F'_{[c_{t-1},d_{t-1}]} \bullet F'_{[\ol p,p]},
  \end{equation*}
  with factors defined as in (\ref{eq:bcstardef}).
  Define $u \in \mfs{[\ol p, n]}$ to be the $312$-avoiding permutation
  corresponding to the type-$\msfA$ descending star network
  \begin{equation*}
    F_u =
    F_{[c_1,d_1]} \bullet \cdots \bullet F_{[c_{t-1},d_{t-1}]} \bullet F_{[\ol p,p]},
  \end{equation*}
  with factors defined as in (\S\ref{ss:Aplanar}).
  Then the one-line notation of $u$ is $p \cdots 21w_1 \cdots w_n$ and
  the subposet $P_{[1,n]}$ of $P(u)$ induced by $[1,n]$ satisfies
  $P_{[1,n]} \cong Q(w)$ (as undecorated posets).
  \end{lem}
\begin{proof}
  Let $\pi' = (\pi'_{\ol n}, \dotsc, \pi'_{\ol 1}, \pi'_1, \dotsc, \pi'_n)$ and
  $\pi = (\pi_{\ol p}, \dotsc, \pi_{\ol 1}, \pi_1, \dotsc, \pi_n)$
  be the unique path families of type $e$ covering $F_w$ and $F_u$,
  respectively.  
  By Definition~\ref{d:bcdsn} and Proposition~\ref{p:atmostonesymminterval}
  we have $c_1 > \cdots > c_{t-1} \geq 1$.
  Thus for $1 \leq i < j \leq n$ we have that $\pi_i$ intersects $\pi_j$
  if and only if $\pi'_i$ intersects $\pi'_j$. It follows that
  $P_{[1,n]} \cong Q(w)$.

  Now let
  $\sigma' = (\sigma'_{\ol n}, \dotsc, \sigma'_{\ol 1},
  \sigma'_1, \dotsc, \sigma'_n)$ and
  $\sigma = (\sigma_{\ol p}, \dotsc, \sigma_{\ol 1}, \sigma_1, \dotsc, \sigma_n)$
  be the unique path families of types $w$ and $u$ covering $F_w$ and $F_u$,
  respectively.  Both families have the property that any two paths which
  intersect must cross.  Thus paths $\sigma_{\ol p}, \dotsc, \sigma_{\ol 1}$,
  which intersect only at the central vertex of $F_{[\ol p, p]}$, have
  sinks $p, \dotsc, 1$, respectively. Thus
  $u_{\ol p} \cdots u_{\ol 1} = p \cdots 1$.
  Also, paths $\sigma_1, \dotsc, \sigma_n$ pass through the same stars
  as $\sigma'_1, \dotsc, \sigma'_n$, respectively, and have the same sinks.
  Thus $u_1 \cdots u_n = w_1 \cdots w_n$. 
  \end{proof}
  \begin{prop}
  Algorithm~\ref{a:qtow} inverts Algorithm~\ref{a:wtoq}.
\end{prop}
\begin{proof}
  Fix $w \in \bn$. If no negative letters appear in
  $w_1 \cdots w_n$,
  then we may interpret this word as an element of $\sn$.
  The applications of Algorithms~\ref{a:wtop} and \ref{a:wtoq} to $w$ agree
  and produce the poset $P(w) = Q(w)$.
  Since this poset has no circled elements,
  the applications of Algorithms~\ref{a:ptow} and \ref{a:qtow} to it agree,
  producing $w$ since Algorithm~\ref{a:ptow} inverts Algorithm~\ref{a:wtop}.
  It follows that Algorithm~\ref{a:qtow} inverts Algorithm~\ref{a:wtoq} as well.
  
  Now suppose that $p > 0$ negative letters appear in $w_1 \cdots w_n$,
  and define
  \begin{equation*}
    u = u_{\ol p} \cdots u_{\ol 1} u_1 \cdots u_n
    = p \cdots 2 1 w_1 \cdots w_n \in \mfs{[\ol p, n]}.
  \end{equation*}
  By Lemma~\ref{l:abc}, $u$ avoids the ordinary pattern $312$.
  Define $P_{[1,n]}$ to be the subposet of $P(u)$ induced by elements $[1,n]$.
  Applying Algorithm~\ref{a:ptow} to $P(u)$, we obtain $u$.
  It follows that for $j = 1, \dotsc, n$, we have
  \begin{equation}\label{eq:wj1}
    \begin{aligned}
    u_j = 
    w_j &= \max (\{ i \in [\ol p, n] \,|\,
    i \not >_{P(u)} j \} \ssm \{u_{\ol p}, \dotsc, u_{j-1} \} ) \\
    &= \max (\{ i \in [\ol p, \ol 1] \cup [p+1,n] \,|\,
    i \not >_{P(u)} j \} \ssm \{w_1, \dotsc, w_{j-1} \}). 
    \end{aligned}
  \end{equation}
  Since $[\ol p, p] \subseteq P(u)$ is an antichain of minimal elements,
  each pair $(i,j) \in [\ol p, \ol 1] \times [1,n]$
  satisfies
  $i \not >_{P(u)} j$ if and only if $\ol i \not >_{P_{[1,n]}} j$.
  Thus we may rewrite (\ref{eq:wj1}) as
  \begin{equation*}
    w_j =
    \max (\{ i \in [\ol p, \ol 1] \,|\, \ol i \not >_{P_{[1,n]}} j \}
                \cup \{ i \in [p+1, n] \,|\, i \not >_{P_{[1,n]}} j \}
                \ssm \{ w_1, \dotsc, w_{j-1} \} ).
  \end{equation*}
  On the other hand, applying Algorithm~\ref{a:qtow} to $Q(w)$,
  we obtain a word $v_1 \cdots v_n$ satisfying
  \begin{equation*}
    v_j = \max (\{ i \in [\ol p, \ol 1] \,|\, \ol i \not >_{Q(w)} j \}
                \cup \{ i \in [p+1, n] \,|\, i \not >_{Q(w)} j \}
    \ssm \{ v_1, \dotsc, v_{j-1} \} ).
  \end{equation*}
  By Lemma~\ref{l:abc}, we have $P_{[1,n]} \cong Q(w)$, and therefore
  $v_1 \cdots v_n = w_1 \cdots w_n$.
  Again, Algorithm~\ref{a:qtow} inverts Algorithm~\ref{a:wtoq}.
\end{proof}

\section{Indifference graphs}\label{s:incgraph}

More partial solutions to Problem~\ref{p:evaltrace} for the
subsets (\ref{eq:cwqsmooth}) -- (\ref{eq:BCcwqcodom})
of the Kazhdan--Lusztig bases employ graphs
called
{\em indifference graphs}, those graphs
whose vertices correspond
to elements of a unit interval order $P$ and whose edges correspond to
unordered pairs $\{i, j\}$ of poset elements which are incomparable,
i.e., $i \not \leq_P j$ and $j \not \leq_P i$.

In type $\msfA$, we have a map $w \mapsto G(w)$,
from \pavoiding permutations in $\sn$ to indifference graphs
whose colorings and edge orientations
facilitate simple combinatorial interpretations of
trace evaluations~\cite[\S 5--10]{CHSSkanEKL}.
In types $\msfBB$ and $\sfC$, we define an analogous map
$w \mapsto G(w)$ from \pavoiding elements of $\bn$ to
objects which we call
{\em type-$\msfBC$ indifference graphs}.
These graphical representations facilitate
simple combinatorial interpretation of certain trace evaluations
(Section~\ref{s:main}), when we specialize at $q=1$.



\ssec{Type-$\msfA$ indifference graphs, coloring, and orientation}\label{ss:aincg}

Given any poset $P$,
we define its {\em incomparability graph} $\inc(P)$
to be the graph whose vertices are the elements of $P$ and whose
edges are the pairs of incomparable elements of $P$.
When $P = P(w)$ is a unit interval order, write $G(w) = \inc(P)$
and call this an {\em indifference graph}.
It is possible to have $P(w) \not \cong P(v)$ and $G(w) \cong G(v)$.
For example, the incomparability graphs
of the fourteen unit interval orders (\ref{eq:uioa}) are the nine
nonisomorphic indifference graphs
\begin{equation}\label{eq:uioag}
\begin{tikzpicture}[scale=.55,baseline=0]
\draw[fill] (0,-.5) circle (1.2mm); 
\draw[fill] (0,.5) circle (1.2mm); 
\draw[fill] (.5,0) circle (1.2mm);
\draw[fill] (-.5,0) circle (1.2mm); 
\draw[-] (0,-.5) -- (.5,0);
\draw[-] (0,-.5) -- (0,.5);
\draw[-] (0,-.5) -- (-.5,0);
\draw[-] (-.5,0) -- (.5,0);
\draw[-] (-.5,0) -- (0,.5);
\draw[-] (0,.5) -- (.5,0);
\end{tikzpicture}
\qquad
\begin{tikzpicture}[scale=.55,baseline=0]
\draw[fill] (0,-.5) circle (1.2mm); 
\draw[fill] (0,.5) circle (1.2mm); 
\draw[fill] (.5,0) circle (1.2mm);
\draw[fill] (-.5,0) circle (1.2mm); 
\draw[-] (0,-.5) -- (.5,0);
\draw[-] (0,-.5) -- (-.5,0);
\draw[-] (-.5,0) -- (.5,0);
\draw[-] (-.5,0) -- (0,.5);
\draw[-] (0,.5) -- (.5,0);
\end{tikzpicture}
\qquad
\begin{tikzpicture}[scale=.55,baseline=0]
\draw[fill] (0,-.5) circle (1.2mm); 
\draw[fill] (0,.5) circle (1.2mm); 
\draw[fill] (.5,0) circle (1.2mm);
\draw[fill] (-.5,0) circle (1.2mm); 
\draw[-] (0,-.5) -- (.5,0);
\draw[-] (-.5,0) -- (.5,0);
\draw[-] (-.5,0) -- (0,.5);
\draw[-] (0,.5) -- (.5,0);
\end{tikzpicture}
\qquad
\begin{tikzpicture}[scale=.55,baseline=0]
\draw[fill] (0,-.5) circle (1.2mm); 
\draw[fill] (0,.5) circle (1.2mm); 
\draw[fill] (.5,0) circle (1.2mm);
\draw[fill] (-.5,0) circle (1.2mm); 
\draw[-] (-.5,0) -- (.5,0);
\draw[-] (-.5,0) -- (0,.5);
\draw[-] (0,.5) -- (.5,0);
\end{tikzpicture}
\qquad
\begin{tikzpicture}[scale=.55,baseline=0]
\draw[fill] (0,-.5) circle (1.2mm); 
\draw[fill] (0,.5) circle (1.2mm); 
\draw[fill] (.5,0) circle (1.2mm);
\draw[fill] (-.5,0) circle (1.2mm); 
\draw[-] (0,-.5) -- (.5,0);
\draw[-] (-.5,0) -- (.5,0);
\draw[-] (-.5,0) -- (0,.5);
\end{tikzpicture}
\qquad
\begin{tikzpicture}[scale=.55,baseline=0]
\draw[fill] (0,-.5) circle (1.2mm); 
\draw[fill] (0,.5) circle (1.2mm); 
\draw[fill] (.5,0) circle (1.2mm);
\draw[fill] (-.5,0) circle (1.2mm); 
\draw[-] (-.5,0) -- (.5,0);
\end{tikzpicture}
\qquad
\begin{tikzpicture}[scale=.55,baseline=0]
\draw[fill] (0,-.5) circle (1.2mm); 
\draw[fill] (0,.5) circle (1.2mm); 
\draw[fill] (.5,0) circle (1.2mm);
\draw[fill] (-.5,0) circle (1.2mm); 
\draw[-] (0,-.5) -- (-.5,0);
\draw[-] (0,.5) -- (.5,0);
\end{tikzpicture}
\qquad
\begin{tikzpicture}[scale=.55,baseline=0]
\draw[fill] (0,-.5) circle (1.2mm); 
\draw[fill] (0,.5) circle (1.2mm); 
\draw[fill] (.5,0) circle (1.2mm);
\draw[fill] (-.5,0) circle (1.2mm); 
\draw[-] (0,-.5) -- (.5,0);
\draw[-] (0,.5) -- (.5,0);
\end{tikzpicture}
\qquad
\begin{tikzpicture}[scale=.55,baseline=0]
\draw[fill] (0,-.5) circle (1.2mm); 
\draw[fill] (0,.5) circle (1.2mm); 
\draw[fill] (.5,0) circle (1.2mm);
\draw[fill] (-.5,0) circle (1.2mm); 
\end{tikzpicture}\;.
\end{equation}
In Section~\ref{s:main}
we will combinatorially evaluate certain traces
at Kazhdan--Lusztig basis elements $\wtc w1 \in \zsn$ with $w$ \avoidingp\ by
coloring the vertices of $G(w)$ or by orienting the edges of $G(w)$.

Given any graph $G = (V, E)$ call a map
$\kappa: V \rightarrow \mathbb N \ssm \{0\}$
a {\em proper coloring of $G$} if $\{a,b\} \in E$
implies that $\kappa(a) \neq \kappa(b)$.
More specifically, say that a proper coloring has {\em type}
$\alpha = (\alpha_1, \dotsc, \alpha_r) \vDash n$
if $\alpha_k$ vertices have color $k$ for $k = 1, \dotsc, r$.
If $G = \inc(P)$
then a proper coloring of $\inc(P)$ of type $\lambda \vdash n$ corresponds
to a sequence of pairwise disjoint chains in $P$ having
weakly decreasing cardinalities $(\lambda_1,\dotsc,\lambda_r)$.

Call a directed graph $O = (V, E')$
an {\em orientation} of $G = (V,E)$ if $O$ is obtained from $G$
by replacing each undirected edge $\{a, b\} \in E$
with exactly one of the directed edges $(a, b)$ or $(b, a)$.
Call $O$ {\em acyclic} if it has no directed cycles.
Acyclic orientations of $G(w)$
correspond to
sequences $(v_h, \dotsc, v_n)$ of elements of $P(w)$
satisfying $v_i \not \geq_{P(w)} v_{i+1}$ for $i=h,\dotsc,n-1$.
We call these {\em $P(w)$-descent-free} sequences.
(See \cite[\S 4]{AthanPSE} and references there.)
\begin{prop}\label{p:aodesfree}
  For $w \in \mfs{[h,n]}$, acyclic orientations of $G(w)$
  correspond bijectively to $P(w)$-descent-free sequences of elements
  of $P(w)$.
\end{prop}
Specifically, this bijection from acyclic orientations to $P(w)$-descent-free
sequences is given by the following algorithm.
\begin{alg}\label{a:OtoP}
  Given $w \in \mfs{[h,l]}$ and an acyclic orientation $O$ of $G(w)$, do
\begin{enumerate}
  \item Set $O(h) = O$.
\item For $i = h, \dotsc, l$,
  \begin{enumerate}
  \item Let $j$ be the least integer appearing as a
    vertex
    in $O(i)$ and having indegree $0$.
  \item Set $v_i = j$.
  \item Form $O(i+1)$ by removing vertex $j$ and its incident edges from $O(i)$.
  \end{enumerate}
  \item Output the sequence $(v_h,\dotsc,v_l)$.
\end{enumerate}
\end{alg}
The inverse of Algorithm~\ref{a:OtoP} is very simple.
\begin{alg}\label{a:PtoO}
  Given $w \in \mfs{[h,l]}$,
  undirected graph $G(w)$ with vertices labeled $\{ h, \dotsc, l\}$,
  and $P(w)$-descent-free sequence $v = (v_h,\dotsc,v_l)$,
  do
  \begin{enumerate}
  \item
  Orient each edge $\{ a,b\}$ of $G(w)$
    as $(a,b)$ if $a$ appears before $b$ in $v$, and as $(b,a)$ otherwise.
  \end{enumerate}
  \end{alg}


\ssec{Type-$\msfBC$ indifference graphs, coloring, and orientation}\label{ss:bcincg}

Given a type-$\msfBC$ unit interval order $Q$, define its {\em incomparability
  graph} $\inc(Q)$ to be the decorated graph
whose vertices are the elements of $Q$, maintaining circles,
and whose edges are the pairs of incomparable elements of $Q$.
For $w \in \bn$ \avoidingp, write $\Gamma(w) = \inc(Q(w))$
and call $\Gamma(w)$ a {\em type-$\msfBC$ indifference graph}.  
We define an isomorphism of type-$\msfBC$ indifference graphs
to be a graph isomorphism which respects circled elements.
Again, it is possible to have $Q(w) \not \cong Q(v)$
and $\Gamma(w) \cong \Gamma(v)$.
For instance, the
fourteen (labeled) type-$\msfBC$ indifference graphs on three elements are
\begin{equation}\label{eq:uiocg}
\begin{tikzpicture}[scale=.55,baseline=0]
\draw[fill] (0,-.5) circle (1.2mm); \node at (0,-1.1) {$\scriptstyle{1}$};
\draw[fill] (0,.5) circle (1.2mm); \node at (0,1) {$\scriptstyle{3}$};
\draw[fill] (.55,0) circle (1.2mm); \node at (.55,-.5) {$\scriptstyle{2}$};
\end{tikzpicture}
\quad\ntnsp
\begin{tikzpicture}[scale=.55,baseline=0]
\draw[fill] (0,-.5) circle (1.2mm); \node at (0,-1.1) {$\scriptstyle{1}$};
\draw[fill] (0,.5) circle (1.2mm); \node at (0,1) {$\scriptstyle{3}$};
\draw[fill] (.55,0) circle (1.2mm); \node at (.55,-.5) {$\scriptstyle{2}$};
\draw (0,-.5) circle (2.8mm); 
\end{tikzpicture}
\quad\ntnsp
\begin{tikzpicture}[scale=.55,baseline=0]
\draw[fill] (0,-.5) circle (1.2mm); \node at (0,-1.1) {$\scriptstyle{1}$};
\draw[fill] (0,.5) circle (1.2mm); \node at (0,1) {$\scriptstyle{3}$};
\draw[fill] (.55,0) circle (1.2mm); \node at (.55,-.5) {$\scriptstyle{2}$};
\draw[-] (0,-.5) -- (.55,0);
\end{tikzpicture}
\quad\ntnsp
\begin{tikzpicture}[scale=.55,baseline=0]
\draw[fill] (0,-.5) circle (1.2mm); \node at (0,-1.1) {$\scriptstyle{1}$};
\draw[fill] (0,.5) circle (1.2mm); \node at (0,1) {$\scriptstyle{3}$};
\draw[fill] (.55,0) circle (1.2mm); \node at (.55,-.5) {$\scriptstyle{2}$};
\draw[-] (0,.5) -- (.55,0);
\end{tikzpicture}
\quad\ntnsp
\begin{tikzpicture}[scale=.55,baseline=0]
\draw[fill] (0,-.5) circle (1.2mm); \node at (0,-1.1) {$\scriptstyle{1}$};
\draw[fill] (0,.5) circle (1.2mm); \node at (0,1) {$\scriptstyle{3}$};
\draw[fill] (.55,0) circle (1.2mm); \node at (.55,-.5) {$\scriptstyle{2}$};
\draw (0,-.5) circle (2.8mm); 
\draw[-] (0,-.5) -- (.55,0);
\end{tikzpicture}
\quad\ntnsp
\begin{tikzpicture}[scale=.55,baseline=0]
\draw[fill] (0,-.5) circle (1.2mm); \node at (0,-1.1) {$\scriptstyle{1}$};
\draw[fill] (0,.5) circle (1.2mm); \node at (0,1) {$\scriptstyle{3}$};
\draw[fill] (.55,0) circle (1.2mm); \node at (.55,-.5) {$\scriptstyle{2}$};
\draw (0,-.5) circle (2.8mm); 
\draw[-] (0,.5) -- (.55,0);
\end{tikzpicture}
\quad\ntnsp
\begin{tikzpicture}[scale=.55,baseline=0]
\draw[fill] (0,-.5) circle (1.2mm); \node at (0,-1.1) {$\scriptstyle{1}$};
\draw[fill] (0,.5) circle (1.2mm); \node at (0,1) {$\scriptstyle{3}$};
\draw[fill] (.55,0) circle (1.2mm); \node at (.55,-.5) {$\scriptstyle{2}$};
\draw[-] (0,-.5) -- (.55,0) -- (0,.5);
\end{tikzpicture}
\quad\ntnsp
\begin{tikzpicture}[scale=.55,baseline=0]
\draw[fill] (0,-.5) circle (1.2mm); \node at (0,-1.1) {$\scriptstyle{1}$};
\draw[fill] (0,.5) circle (1.2mm); \node at (0,1) {$\scriptstyle{3}$};
\draw[fill] (.55,0) circle (1.2mm); \node at (.55,-.5) {$\scriptstyle{2}$};
\draw (0,-.5) circle (2.8mm); 
\draw[-] (0,-.5) -- (.55,0) -- (0,.5);
\end{tikzpicture}
\quad\ntnsp
\begin{tikzpicture}[scale=.55,baseline=0]
\draw[fill] (0,-.5) circle (1.2mm); \node at (0,-1.1) {$\scriptstyle{1}$};
\draw[fill] (0,.5) circle (1.2mm); \node at (0,1) {$\scriptstyle{3}$};
\draw[fill] (.55,0) circle (1.2mm); \node at (.55,-.5) {$\scriptstyle{2}$};
\draw[-] (0,-.5) -- (.55,0) -- (0,.5) -- (0,-.5);
\end{tikzpicture}
\quad\ntnsp
\begin{tikzpicture}[scale=.55,baseline=0]
\draw[fill] (0,-.5) circle (1.2mm); \node at (0,-1.1) {$\scriptstyle{1}$};
\draw[fill] (0,.5) circle (1.2mm); \node at (0,1) {$\scriptstyle{3}$};
\draw[fill] (.55,0) circle (1.2mm); \node at (.55,-.5) {$\scriptstyle{2}$};
\draw (0,-.5) circle (2.8mm); 
\draw[-] (0,-.5) -- (.55,0) -- (0,.5) -- (0,-.5);
\end{tikzpicture}
\quad\ntnsp
\begin{tikzpicture}[scale=.55,baseline=0]
\draw[fill] (0,-.5) circle (1.2mm); \node at (0,-1.1) {$\scriptstyle{1}$};
\draw[fill] (0,.5) circle (1.2mm); \node at (0,1) {$\scriptstyle{3}$};
\draw[fill] (.55,0) circle (1.2mm); \node at (.55,-.5) {$\scriptstyle{2}$};
\draw (0,-.5) circle (2.8mm);
\draw (.55,0) circle (2.8mm); 
\draw[-] (0,-.5) -- (.55,0);
\end{tikzpicture}
\quad\ntnsp
\begin{tikzpicture}[scale=.55,baseline=0]
\draw[fill] (0,-.5) circle (1.2mm); \node at (0,-1.1) {$\scriptstyle{1}$};
\draw[fill] (0,.5) circle (1.2mm); \node at (0,1) {$\scriptstyle{3}$};
\draw[fill] (.55,0) circle (1.2mm); \node at (.55,-.5) {$\scriptstyle{2}$};
\draw (0,-.5) circle (2.8mm);
\draw (.55,0) circle (2.8mm); 
\draw[-] (0,-.5) -- (.55,0) -- (0,.5);
\end{tikzpicture}
\quad\ntnsp
\begin{tikzpicture}[scale=.55,baseline=0]
\draw[fill] (0,-.5) circle (1.2mm); \node at (0,-1.1) {$\scriptstyle{1}$};
\draw[fill] (0,.5) circle (1.2mm); \node at (0,1) {$\scriptstyle{3}$};
\draw[fill] (.55,0) circle (1.2mm); \node at (.55,-.5) {$\scriptstyle{2}$};
\draw (0,-.5) circle (2.8mm);
\draw (.55,0) circle (2.8mm); 
\draw[-] (0,-.5) -- (.55,0) -- (0,.5) -- (0,-.5);
\end{tikzpicture}
\quad\ntnsp
\begin{tikzpicture}[scale=.55,baseline=0]
\draw[fill] (0,-.5) circle (1.2mm); \node at (0,-1.1) {$\scriptstyle{1}$};
\draw[fill] (0,.5) circle (1.2mm); \node at (0,1) {$\scriptstyle{3}$};
\draw[fill] (.55,0) circle (1.2mm); \node at (.55,-.5) {$\scriptstyle{2}$};
\draw (0,-.5) circle (2.8mm);
\draw (0,.5) circle (2.8mm);
\draw (.55,0) circle (2.8mm); 
\draw[-] (0,-.5) -- (.55,0) -- (0,.5) -- (0,-.5);
\end{tikzpicture}
\end{equation}
with the third and fourth graphs being isomorphic.


Analogous to type-$\msfA$ indifference graphs,
type-$\msfBC$ indifference graphs have colorings and edge
orientations which facilitate the evaluation of certain type-$\msfBC$
traces at Kazhdan--Lusztig basis elements $\smash{\btc w1} \in \mathbb Z[\bn]$
when $w$ \avoidsp.
Given a type-$\msfBC$ indifference graph $\Gamma = (V,E)$,
define a {\em marked $\msfBC$-coloring}
\begin{equation*}
\kappa = (\kappa_1, \kappa_2): V \rightarrow
(\mathbb Z \ssm \{0\}) \times \{0,1\}
\end{equation*}
of $\Gamma$ to be an assignment
of a nonzero color $\kappa_1(\vertex b)$ and possibly a star
(if $\kappa_2(\vertex b) = 1$)
to each vertex $\vertex b \in V$,
with the properties that
\begin{enumerate}
  \item for vertex $\vertex b$ grounded we have $\kappa_1(\vertex b) > 0$,
  \item for vertex $\vertex b$ not grounded we have $\kappa_2(\vertex b) = 0$.
\end{enumerate}
Say that $\kappa$ has {\em type
$(\lambda,\mu) = ((\lambda_1,\dotsc,\lambda_m),(\mu_1,\dotsc,\mu_k))$} if
\begin{enumerate}
\item $\lambda_i$ vertices have color $i$, for $i=1,\dotsc, m$,
\item $\mu_i$ vertices have color $\ol i$, for $i=1,\dotsc, k$,
\end{enumerate}
and that $\kappa$ is {\em proper}
if $\{a,b\} \in E$ implies that $\kappa_1(a) \neq \kappa_1(b)$.
As before, monochromatic sets of $\Gamma(w)$
correspond to chains in $Q(w)$;
now each such chain
contains at most one grounded element.
Thus a proper $\msfBC$-coloring of $\Gamma$ of type $(\lambda,\mu)$
may be represented by a pair $(U,V)$ of $Q$-tableaux
in which column $i$ of $U$
$(i = 1,\dotsc,m)$
contains the color-$i$ chain of $Q$
with at most one grounded element marked with a star,
and
column $i$ of $V$
$(i = 1,\dotsc,k)$
contains the color-$\ol i$ chain of $Q$
with no grounded elements.

Define a {\em marked acyclic orientation} of a type-$\msfBC$ indifference graph
to be a directed graph $O$ on the same verices, with some subset of
grounded vertices marked by stars,
in which each undirected edge $\{ a, b \}$ is replaced with one of
the directed edges $(a,b)$ or $(b,a)$.
For example, the type-$\msfBC$ unit interval order $Q = Q(\ol14365\ol2)$,
its incomparability graph $\inc(Q)$,
a marked acyclic orientation of $\inc(Q)$, and 
a marked coloring of $\inc(Q)$ of type $((2,1,1),(2))$ are
\begin{equation}\label{eq:incQ}
\begin{tikzpicture}[scale=.7,baseline=15]
\draw[fill] (-1,1.5) circle (1.2mm); \node at (-1,1.9) {$\scriptstyle{4}$};
\draw[fill] (0,0) circle (1.2mm); \node at (0,-.55) {$\scriptstyle{1}$};
\draw (0,0) circle (2.8mm);
\draw[fill] (0,1) circle (1.2mm); \node at (0.25,.65) {$\scriptstyle{3}$};
\draw[fill] (0,2) circle (1.2mm); \node at (0,2.4) {$\scriptstyle{5}$};
\draw[fill] (1,2) circle (1.2mm); \node at (1,2.4) {$\scriptstyle{6}$};
\draw[fill] (1,1) circle (1.2mm); \node at (1,.45) {$\scriptstyle{2}$};
\draw (1,1) circle (2.8mm);
\draw[-] (-1,1.5) -- (0,0) -- (0,1) -- (0,2) -- (1,1) -- (1,2) -- (0,1);
\end{tikzpicture}\;,
\qquad 
\begin{tikzpicture}[scale=.7,baseline=15]
\draw[fill] (-1,1.5) circle (1.2mm); \node at (-1,1.9) {$\scriptstyle{4}$};
\draw[fill] (0,0) circle (1.2mm); \node at (0,-.55) {$\scriptstyle{1}$};
\draw (0,0) circle (2.8mm);
\draw[fill] (0,.9) circle (1.2mm); \node at (-0.25,.65) {$\scriptstyle{3}$};
\draw[fill] (0,2.1) circle (1.2mm); \node at (0,2.5) {$\scriptstyle{5}$};
\draw[fill] (1,2) circle (1.2mm); \node at (1,2.4) {$\scriptstyle{6}$};
\draw[fill] (1,1) circle (1.2mm); \node at (1,.45) {$\scriptstyle{2}$};
\draw (1,1) circle (2.8mm);
\draw[-] (1,1) -- (0,0);
\draw[-] (0,.9) -- (1,1);
\draw[-] (0,.9) -- (-1,1.5);
\draw[-] (1,1) -- (-1,1.5);
\draw[-] (-1,1.5) -- (0,2.1);
\draw[-] (-1,1.5) -- (1,2);
\draw[-] (1,2) -- (0,2.1);
\end{tikzpicture}\;,
\qquad
\begin{tikzpicture}[scale=.7,baseline=15]
\draw[fill] (-1,1.5) circle (1.2mm); \node at (-1,1.9) {$\scriptstyle{4}$};
\draw[fill] (0,0) circle (1.2mm); \node at (0,-.6) {$\scriptstyle{1}$};
\draw (0,0) circle (2.8mm);
\draw[fill] (0,.9) circle (1.2mm); \node at (-0.25,.65) {$\scriptstyle{3}$};
\draw[fill] (0,2.1) circle (1.2mm); \node at (0,2.5) {$\scriptstyle{5}$};
\draw[fill] (1,2) circle (1.2mm); \node at (1,2.4) {$\scriptstyle{6}$};
\draw[fill] (1,1) circle (1.2mm); \node at (1.1,.45) {$\scriptstyle{2^\star}$};
\draw (1,1) circle (2.8mm);
\draw[->] (1,1) -- (0.25,0.25);
\draw[->] (0,.9) -- (.68,1);
\draw[->] (0,.9) -- (-.9,1.3);
\draw[->] (1,1) -- (-.82,1.45);
\draw[->] (-1,1.5) -- (-.15,2);
\draw[->] (-1,1.5) -- (.8,1.9);
\draw[->] (1,2) -- (0.2,2.1);
\end{tikzpicture},
\qquad
\left(\, \tableau[scY]{5 | 1^\star, 4, 2}\,, \tableau[scY]{6 | 3}\, \right).
\end{equation}
To connect acyclic orientations of $\Gamma(w)$
to $Q(w)$-descent-free sequences as in
Proposition~\ref{p:aodesfree},
we define {\em marked $Q(w)$-descent-free}
sequences to be
those $Q(w)$-descent-free sequences in which some subset of
grounded elements is marked.
\begin{prop}\label{p:bcodesfree}
Marked acyclic orientations of $\Gamma(w)$
correspond to marked $Q(w)$-descent-free
sequences
of elements of $Q(w)$.
\end{prop}
\begin{proof}
The correspondence is given by Algorithms~\ref{a:OtoP} -- \ref{a:PtoO},
modified so that marked graph vertices correspond to marked poset elements.
\end{proof}
\noindent
For example,
the $Q$-descent-free sequence corresponding to the acyclic orientation
in (\ref{eq:incQ}) is $(3, 2^\star\ntksp,1, 4, 6, 5)$.

We remark that other authors have defined $\msfBC$-analogs of
graphs~\cite{HararyNBSG}, \cite{ReinerParset},
have associated these to posets generalizing
type-$\msfBC$ unit interval orders~\cite{CsarThesis}, \cite{ReinerParset},
and have studied their colorings~\cite{KurTsuChrom}, \cite{ZaslavskySGC}.  
However, it is not clear that such graphs and colorings
are closely related to ours.
In particular, the other authors' graphs have edges describing
comparability of poset elements rather than
incomparability, and their colorings include restrictions on
pairs of vertices
whose colors can share an absolute value,
whereas ours do not.

\section{Combinatorial trace evaluations: path tableaux, poset tableaux,
  and acyclic orientations}\label{s:main}


Our main results,
Theorem~\ref{t:epsiloneta} -- Theorem~\ref{t:chi},
combinatorially interpret
trace evaluations
$\theta(\btc w1)$ for
certain $\theta \in \trsp(\bn)$
and all $w \in \bn$ \avoidingp.
Analogous to known type-$\msfA$ results,
our new type-$\msfBC$ evaluations use the type-$\msfBC$
unit interval orders and their incomparability graphs
defined in Subsections~\ref{ss:bcuio}, \ref{ss:bcincg}.

\ssec{Type-$\msfA$ trace evaluations}\label{ss:mainA}

To state these interpretations, we fill
(French) Young diagrams with paths
and we call the resulting structures
{\em path tableaux}.
If the paths are a family $\pi = (\pi_1,\dotsc,\pi_n)$ which covers $F_w$,
we will more specifically call the path tableau an {\em $F_w$-tableau},
or a {\em $\pi$-tableau.}
If $\pi$ has type $v \in \sn$, then we also say that each
$\pi$-tableau has {\em type $v$}.
Since $\pi$ can be viewed as the poset $Q(\pi)$ defined in
Subsection \ref{ss:auio},
$\pi$-tableaux are special cases of
Gessel and Viennot's {\em poset tableaux}~\cite{GV},
Young diagrams filled with elements of a poset.
Thus if $\pi$ is the unique family of type $e$ covering $F_w$,
then a $P(\pi)$-tableau
is a $P(w)$-tableau.
For any tableau
$U$, let $U_i$ be the $i$th row of $U$,
and let $U_{i,j}$ be the $j$th entry in row $i$.
Let $\tab(\pi,\lambda)$ denote the set of all $\pi$-tableaux of shape
$\lambda$, and let $\tab(F_w,\lambda)$ 
denote the set of {\em all} $F_w$-tableaux of shape $\lambda$, i.e.,
containing all path families covering $F_w$,
\begin{equation}\label{eq:upidef}
  \tab(F_w,\lambda) = \nTksp \bigcup_{\pi \in \Pi(F_w)} \nTksp
  \tab(\pi,\lambda).
\end{equation}
If $\pi$ is the unique path family of type $e$ covering $F_w$, then define
$\tab(P(w),\lambda) \defeq \tab(\pi,\lambda)$.
For example, consider $F_{2341}$,
the seventh zig-zag network in (\ref{eq:xfigureszz}),
and let $\pi = (\pi_1,\pi_2,\pi_3,\pi_4)$
be the unique path family of type $e$ covering $F_{2341}$.
Then $P(2341) = Q(\pi)$ is the seventh unit interval order in (\ref{eq:uioa}).
Labeling each element $\pi_i$ of $P(2341)$ by $i$ and forming a few
$P(2341)$-tableaux, we have 
\begin{equation}\label{eq:posettableaux}
  \begin{gathered}
    P(2341) =
\begin{tikzpicture}[scale=.7,baseline=-13]
\draw[fill] (0,0) circle (1mm); \node at (-.5,0) {$3$};
\draw[fill] (0,-1) circle (1mm); \node at (-.5,-1) {$1$};
\draw[fill] (.8,0) circle (1mm); \node at (1.3,0) {$4$};
\draw[fill] (.8,-1) circle (1mm); \node at (1.3,-1) {$2$};
\draw[-,thick] (0,0) -- (0,-1);
\draw[-,thick] (0,-1) -- (.8,0);
\draw[-,thick] (.8,0) -- (.8,-1);
\end{tikzpicture},
\qquad 
  S = {\tableau[scY]{4|1,2,3}}\,, \qquad
  T = {\tableau[scY]{2|1,3,4}}\,, \qquad
  U = {\tableau[scY]{1|4,3,2}}\,, \\
  V = {\tableau[scY]{3|4,1,2}}\,, \qquad
  W = {\tableau[scY]{3|1,4,2}}\,, \qquad
  W_1 = {\tableau[scY]{1,4,2}}\,, \qquad
W_{1,2} = 4.\qquad
\end{gathered}
\end{equation}
The tableaux $S$, $T$, $U$, $V$, $W$ all belong to
$\tab(P(2341),31)$.
%

Several properties which path-tableaux may posess can be defined for
poset tableaux. Let $P$ be any labeled poset and let
$U$ be a $P$-tableau.
Call an entry $U_{i,j}$ a {\em record} in $U$
if it is greater in $P$ than $U_{i,1}, \dotsc, U_{i,j-1}$.
Call a record $U_{i,j}$ {\em nontrivial} if $j > 1$.
Call a row of $U$
{\em left anchored} ({\em right anchored})
if its leftmost (rightmost) element is less in $\mathbb Z$
than all other elements
in the row.
Call elements $(a,b)$ a {\em $P$-inversion} in $U$
if the elements are incomparable in $P$ with $a < b$ 
in $\mathbb Z$ and $b$ appearing in an earlier column than $a$.
Let $\inv_P(U)$ denote the number of $P$-inversions in $U$.
Call elements $(U_{i,j},U_{i,j+1})$ a {\em $P$-descent} in $U$
if $U_{i,j} >_P U_{i,j+1}$. 
Let $\des_P(U)$ denote the number of $P$-descents in $U$.
Define $\sort(U)$ to be the tableau obtained from $U$ by sorting entries
in each row so that labels increase to the right.
Define a $P$-excedance in $U$ to be a position $(i,j)$ such that
$U_{i,j} >_P \sort(U)_{i,j}$.
Let $\exc_P(U)$ be the number of $P$-excedances in $U$.

Call a $P$-tableau $U$
\begin{enumerate}
\item {\em column-strict} if the entries of each column satisfy
  $U_{i,j} <_P U_{i+1,j}$,
\item {\em descent-free} or {\em row-semistrict}
  if $\des_P(U) = 0$,
\item {\em cyclically row-semistrict} if it is row-semistrict,
  and if the last entry $U_{i,\lambda_i}$ of each row satisfies
  $U_{i,\lambda_i} \not >_P U_{i,1}$,
\item {\em standard} if it is column-strict and row-semistrict,
\item {\em excedance-free} if $\exc_P(U) = 0$,
\item {\em record-free} if no row has a nontrivial $P$-record,
\item {\em left anchored} ({\em right anchored}) if each row is
  left anchored (right anchored).
\end{enumerate}
For example, we may examine the tableaux in (\ref{eq:posettableaux})
for these properties to obtain the table
\hhhsp
\begin{center}
\newcolumntype{R}{>{$}c<{$}}
\begin{tabularx}{135.3mm}{|R|R|R|R|R|R|}%
\hline
\phsum & \phm S\phm & \phm T\phm & \phm U\phm & \phm V\phm & \phn W\phn \\
\hline 
\phsum\mbox{column-strict}\phsum
& \checkmark &            &            &            & \checkmark \\
\phsum\mbox{row-semistrict}\phsum
& \checkmark & \checkmark & \checkmark &            &           \\
\phsum\nTksp\mbox{cyclically row-semistrict}\phsum\nTksp
&            &            & \checkmark &            &           \\
\phsum\mbox{standard}\phsum
& \checkmark &            &            &            &           \\
\phsum\mbox{excedance-free}\phsum
& \checkmark & \checkmark &            &            &            \\
\phsum\mbox{record-free}\phsum
& \checkmark &            & \checkmark & \checkmark &            \\
\phsum\mbox{left anchored}\phsum
& \checkmark & \checkmark &            &            & \checkmark \\
\phsum\mbox{right anchored}\phsum
&            &            & \checkmark & \checkmark &            \\
\hline
\end{tabularx}\, ,
\end{center}
\ssp
where
the row-semistrict tableaux $S$ and $T$ fail to be
cyclically row-semistrict because their first rows begin with $1$ and
end with $3 >_P 1$ and $4 >_P 1$, respectively.

Other properties of path-tableaux depend upon the fact that
each path
$\pi_j$
in a path family
has a source vertex
$\src(\pi_j)$
and a sink vertex
$\snk(\pi_j)$.
Given a path-tableau $U$,  
let $\src(U)$ and $\snk(U)$ denote the Young tableaux of integers
obtained from $U$ by replacing paths $\pi_1,\dotsc,\pi_n$
with their corresponding source and sink indices, respectively.
If $U$ is a path-tableau, call $U$
\begin{enumerate}
\item {\em row-closed} if for each index $i$,
  $\snk(U_i)$ is a permutation of $\src(U_i)$,
\item {\em left row-strict} if entries of $\src(U)$
 strictly increase in each row,
\item {\em cylindrical} if in each row $U_i = (U_{i,1}, \dotsc, U_{i,k})$
  we have
  $\snk(U_i) = \src(U_{i,2}, \dotsc, U_{i,k}, U_{i,1})$.
\end{enumerate}

For example consider $F_{2341}$ again,
the unique path families
$\rho$,
$\sigma$,
and
$\tau$
of type $2314$, $2134$, and $2341$
which cover $F_{2341}$,
\begin{equation}\label{eq:dpi}
  F_{2341} = \,
\begin{tikzpicture}[scale=.44,baseline=15]
  \node at (.5,3) {$4$};  \node at (4.5,3) {$4$};
  \node at (.5,2) {$3$};  \node at (4.5,2) {$3$};
  \node at (.5,1) {$2$};  \node at (4.5,1) {$2$};
  \node at (.5,0) {$1$};  \node at (4.5,0) {$1$};
\draw[-,thick] (1,3) -- (4,0);
\draw[-,thick] (1,2) -- (2,3) -- (4,3);
\draw[-,thick] (1,1) -- (2,1) -- (3,2) -- (4,2);
\draw[-,thick] (1,0) -- (3,0) -- (4,1);
\end{tikzpicture}\,,
\qquad \qquad
  \begin{tikzpicture}[scale=.44,baseline=15]
  \node at (.3,3) {$\rho_4$};  
  \node at (.3,2) {$\rho_3$};  
  \node at (.3,1) {$\rho_2$};  
  \node at (.3,0) {$\rho_1$};  
\draw[-, thick, dashed] (1,3) -- (1.5, 2.5) -- (2,3) -- (4,3);
\draw[-, ultra thick] (1,2) -- (1.5, 2.5) -- (4,0);
\draw[-, ultra thick, dotted] (1,1) -- (2,1) -- (3,2) -- (4,2);
\draw[-, thin] (1,0) -- (3,0) -- (4,1);
\end{tikzpicture}\,\,,
\qquad
  \begin{tikzpicture}[scale=.44,baseline=15]
  \node at (.3,3) {$\sigma_4$};  
  \node at (.3,2) {$\sigma_3$};  
  \node at (.3,1) {$\sigma_2$};  
  \node at (.3,0) {$\sigma_1$};  
\draw[-, thick, dashed] (1,3) -- (1.5, 2.5) -- (2,3) -- (4,3);
\draw[-, ultra thick] (1,2) -- (1.5, 2.5) -- (2.5, 1.5) -- (3,2) -- (4,2);
\draw[-, ultra thick, dotted] (1,1) -- (2,1) -- (2.5,1.5) -- (4,0);
\draw[-, thin] (1,0) -- (3,0) -- (4,1);
\end{tikzpicture}\,\,,
\qquad
  \begin{tikzpicture}[scale=.44,baseline=15]
  \node at (.3,3) {$\tau_4$};  
  \node at (.3,2) {$\tau_3$};  
  \node at (.3,1) {$\tau_2$};  
  \node at (.3,0) {$\tau_1$};  
\draw[-, thick, dashed] (1,3) -- (1.5, 2.5) -- (4,0);
\draw[-, ultra thick] (1,2) -- (1.5, 2.5) -- (2,3) -- (4,3);
\draw[-, ultra thick, dotted] (1,1) -- (2,1) -- (3,2) -- (4,2);
\draw[-, thin] (1,0) -- (3,0) -- (4,1);
\end{tikzpicture}\,\,,
\end{equation}
and the path tableaux
\begin{equation}\label{eq:pathtableaux}
  SS = {\tableau[scY]{\rho_4|\rho_2,\rho_3,\rho_1}}\,, \qquad
  TT = {\tableau[scY]{\rho_4|\rho_1,\rho_2,\rho_3}}\,, \qquad
  UU = {\tableau[scY]{\sigma_4|\sigma_1,\sigma_2,\sigma_3}}\,, \qquad
  VV = {\tableau[scY]{\tau_3|\tau_1,\tau_2,\tau_4}}\,,
\end{equation}
belonging to $\tab(\rho,31)$, $\tab(\sigma,31)$, $\tab(\tau,31)$.
To inspect these tableaux for the properties defined above,
we replace each path
$\pi_j$
with the ordered pair
$(\src(\pi_j),\snk(\pi_j))$,
\begin{equation}\label{eq:sinktableax}
  \phantom{SS = } {\tableau[scY]{44|23,31,12}}\,, \qquad
  \phantom{TT = } {\tableau[scY]{44|12,23,31}}\,, \qquad
  \phantom{UU = } {\tableau[scY]{44|12,21,33}}\,, \qquad
  \phantom{VV = } {\tableau[scY]{34|12,23,41}}\,,
\end{equation}
and we obtain the summary
\hhhsp
\begin{center}
\newcolumntype{R}{>{$}c<{$}}
\begin{tabularx}{117mm}{|R|R|R|R|R|}%
\hline
\phsum & \phm SS\phm & \phm TT\phm & \phm UU\phm & \phm VV\phm \\
\hline 
\phsum\mbox{row-closed}\phsum
& \checkmark & \checkmark & \checkmark &           \\
\phsum\mbox{left row-strict}\phsum
&            & \checkmark & \checkmark & \checkmark \\
\phsum\mbox{cylindrical}\phsum
& \checkmark & \checkmark &            &           \\
\hline
\end{tabularx}\, .
\end{center}
\ssp




Using Lindstrom's Lemma~\cite{KMG}, \cite{LinVrep},
its permanental analogs \cite[Thm.\,4.15]{SkanCCS},
and its power sum immanant analogs \cite[Thm.\,4.16]{SkanCCS}
we may now extend Proposition~\ref{p:detpermpsiinterpbasic}
to include more combinatorial interpretations.
\begin{prop}\label{p:detpermpsiinterplong}
    Fix $w \in \sn $ \avoidingp, with
    corresponding zig-zag network $F_w$ having path matrix $A$.
    Let $P = P(w)$ and $G = \inc(P)$ be the corresponding
    unit interval order and incomparability graph.
    We have
    \begin{equation}\label{eq:detinterp}
      \begin{aligned}
  \simm n{\epsilon^n}(A)  = \det(A)
  &= \# \{ U \in \tab(P,1^n) \,|\, U \text{ column-strict } \},\\   
  &= \begin{cases} 1 &\text{if $G$ is an independent set ($w=e$ and $P$ is a chain)},\\
     0 &\text{otherwise}.
   \end{cases}\\
   \end{aligned}
\end{equation}
\begin{equation}\label{eq:perminterp}
  \begin{aligned}  
    \simm n{\eta^n}(A) = \perm(A)
    &= \# \{ U \in \tab(F_w,n) \,|\,
    U \text{ left row-strict}\,\},\\
    &=\# \{ U \in \tab(P,n) \,|\,
    U \text{ row-semistrict } \},\\
    &= \# \{ U \in \tab(P,n) \,|\, 
    U \text{ excedance-free } \},\\
    &= \# \text{ acyclic orientations of $G$. }
  \end{aligned}
\end{equation}
\begin{equation}\label{eq:pimminterp}
  \begin{aligned}  
    \simm n{\psi^n}(A)
    &= \# \{ U \in \tab(F_w,n) \,|\,
    U \text{ cylindrical}\,\},\\
    &=\# \{ U \in \tab(P,n) \,|\,
    U \text{ cyclically row-semistrict } \},\\
    &= \# \{ U \in \tab(P,n) \,|\, 
    U \text{ record-free, row-semistrict } \},\\
    &= n \cdot \# \{ U \in \tab(P,n) \,|\, 
    U \text{ right-anchored, row-semistrict } \},\\
    &= \# \text{ acyclic orientations of $G$ having exactly one source. }
  \end{aligned}
\end{equation}
\end{prop}
By (\ref{eq:immid}),
Proposition~\ref{p:detpermpsiinterplong}
gives interpretations of
$\epsilon^\lambda(\wtc w1)$,
$\eta^\lambda(\wtc w1)$,
$\psi^\lambda(\wtc w1)$
in the special case that $\lambda = n$.
The identities
(\ref{eq:lmw}) -- (\ref{eq:pimm}) then lead
to results for general $\lambda$.
(See \cite[Thm.\,4.7]{CHSSkanEKL}, \cite[Thms.\,30--31]{SkanCCS}.)

\begin{thm}\label{t:wtc1interps}
  Fix $w \in \sn $ \avoidingp{} with
  corresponding zig-zag network $F_w$
  and unit interval order $P = P(w)$ as in (\ref{eq:pw}).
  For each partition $\lambda \vdash n$ we have the following.
\begin{enumerate}
\item[$(i$-$a)$] $\epsilon^\lambda(C'_w(1)) =
  \# \{ U \in \tab(P,\lambda^\tr\,) \,|\, 
  U \text{ column-strict}\, \}$.
\item[$(i$-$b)$] $\epsilon^\lambda(C'_w(1)) =
  \# \text{ colorings of $\inc(P)$ of type $\lambda$.}$
\item[$(ii$-$a)$] $\eta^\lambda(C'_w(1)) =
  \# \{ U \in \tab(F_w,\lambda) \,|\, 
U \text{ row-closed, left row-strict}\, \}$.
\item[$(ii$-$b)$] $\eta^\lambda(C'_w(1)) =
  \# \{ U \in \tab(P,\lambda) \,|\, 
U \text{ row-semistrict}\, \}$.
\item[$(ii$-$c)$] $\eta^\lambda(C'_w(1)) =
  \# \{ U \in \tab(P,\lambda) \,|\, 
U \text{ excedance-free}\, \}$.
\item[$(iii)$] $\chi^\lambda(C'_w(1)) =
  \# \{ U \in \tab(P,\lambda) \,|\, 
U \text{ standard}\, \}$.
\item[$(iv$-$a)$] $\psi^\lambda(C'_w(1)) =
  \# \{ U \in \tab(F_w,\lambda) \,|\, 
U \text{ cylindrical}\,\, \}$.
\item[$(iv$-$b)$] $\psi^\lambda(C'_w(1)) =
  \# \{ U \in \tab(P,\lambda) \,|\, 
U \text{ cyclically row-semistrict}\, \}$.
\item[$(iv$-$c)$] $\psi^\lambda(C'_w(1)) =
  \# \{ U \in \tab(P,\lambda) \,|\, 
U \text{ record-free, row-semistrict}\, \}$.
\item[$(iv$-$d)$] $\psi^\lambda(C'_w(1)) =
  \lambda_1 \cdots \lambda_r \cdot
  \# \{ U \in \tab(P,\lambda) \,|\, 
U \text{ right-anchored, row-semistrict}\, \}$.
\end{enumerate}
\end{thm}
\noindent
See \cite[Thm.\,31]{SkanCCS} for a proof of statement ($ii$-$c$)
and its $q$-analog;
see \cite[Thm.\,4.7]{CHSSkanEKL}
for proofs of other statements
and \cite[Cor.\,31]{BChowUIODot}, \cite[\S 5--9]{CHSSkanEKL}
for proofs of their $q$-analogs.
We may also interpret $\eta_q^\lambda(C'_w(1))$
and $\psi_q^\lambda(C'_w(1))$ in terms of acyclic orientations of sequences of
subgraphs of $\inc(P(w))$~\cite[Thm.\,10, Thm.\,13]{SkanCCS}.
\begin{thm}\label{t:wtc1interpssubg}
  Fix $w \in \sn$ \avoidingp,
  and define $P = P(w)$ as in Subsection~\ref{ss:auio}.
  For all $\lambda = (\lambda_1, \dotsc, \lambda_r) \vdash n$ we have
  \begin{enumerate}
  \item $\eta^\lambda(C'_w(1))$ equals the number of acyclic orientations
    of subgraph sequences
    \begin{equation}\label{eq:graphseq}
      (\inc(P_{I_1}), \dotsc, \inc(P_{I_r})),
    \end{equation}
    where $(I_1,\dotsc,I_r)$ is an ordered set partition of $[n]$
    of type $\lambda$.
  \item $\psi^\lambda(C'_w(1))$ equals the number of acyclic orientations
    of subgraph sequences (\ref{eq:graphseq}) in which
    each subgraph $\inc(P_{I_j})$ is connected and its orientation has
    a unique source.
  \end{enumerate}
\end{thm}


\ssec{Type-$\msfBC$ trace evaluations}\label{ss:mainBC}

It is possible to extend Proposition~\ref{p:detDBPB} to include
interpretations of the functions there in terms of path tableaux,
poset tableaux, and acyclic orientations,
just as Theorem~\ref{t:wtc1interps} extends
Propositions~\ref{p:detpermpsiinterpbasic} and \ref{p:detpermpsiinterplong}.
To do this, we define $\msfBC$-analogs of poset tableaux and path tableaux
(and use the marked acylic orientations defined at the end of
Subsection~\ref{ss:bcincg}).

Given a type-$\msfBC$ unit interval order $Q$,
define a {\em marked $Q$-tableau}
to be a Young diagram
filled with elements of $Q$, in which a (possibly empty)
subset of grounded elements of $Q$ is marked with stars.
Define $\bctab(Q,\lambda)$ to be the set of
marked $Q$-tableaux of shape $\lambda$.
The seven properties of $P$-tableaux stated after (\ref{eq:posettableaux})
carry over in a straightforward way to $Q$-tableaux.
For example, the type-$\msfBC$ unit interval order $Q = Q(\ol1 3 \ol2)$
and a few
row-semistrict
marked $Q$-tableaux of shape $3$ are
\begin{equation*}
  \begin{tikzpicture}[scale=.5,baseline=0]
  \draw[fill] (0,1) circle [radius=0.15];
  \draw[fill] (1,0) circle [radius=0.15];
  \draw[fill] (0,-1) circle [radius=0.15];
  \draw[thick] (1,0) circle [radius=0.35];
  \draw[thick] (0,-1) circle [radius=0.35];
\node at (-.7,1) {$\scriptstyle 3$};
\node at (1.7,0) {$\scriptstyle 2$};
\node at (-.7,-1) {$\scriptstyle 1$};
\draw[-, thick] (0,1) -- (0,-1);
\end{tikzpicture}\, ,
\qquad
\begin{gathered}
  {\tableau[scY]{1,2,3}}\,, \quad
  {\tableau[scY]{1^\star,2,3}}\,,\quad
  {\tableau[scY]{1,2^\star,3}}\,, \quad
  {\tableau[scY]{1^\star,2^\star,3}}\,, \\
  {\tableau[scY]{3,2,1}}\,, \quad
  {\tableau[scY]{3,2^\star,1}}\,, \quad
  {\tableau[scY]{3,2,1^\star}}\,, \quad
  {\tableau[scY]{3,2^\star,1^\star}}\,.
\end{gathered}
\end{equation*}

Given
type-$\msfBC$ zig-zag network
$F_w \in \znet{BC}{[\ol n,n]}$
and a path family
\begin{equation*}
  \pi = (\pi_{\ol n}, \dotsc, \pi_{\ol1}, \pi_1, \dotsc, \pi_n) \in \PiBC(F_w),
  \end{equation*}
define an {\em $F_w$-tableau},
or more specifically a {\em $\pi$-tableau}, to be a
Young diagram
filled
with paths $(\pi_1,\dotsc,\pi_n)$.
Define
$\bctab(F_w,\lambda)$ to be the set of
(unmarked) $F_w$-tableaux of shape $\lambda$.
Properties of such tableaux are simple extensions of
properties of type-$\msfA$
path tableaux stated before (\ref{eq:dpi}),
with sink indices replaced by their absolute values.
For $U \in \bctab(F_w,\lambda)$,
call $U$
\begin{enumerate}
\item {\em row-closed} if
  $\{|\snk(U_{i,1})|, \dotsc, |\snk(U_{i,\lambda_i})|\}
  = \{\src(U_{i,1}), \dotsc, \src(U_{i,\lambda_i})\}$ for all $i$,
\item {\em left row-strict} if
  $\src(U_{i,1}) < \cdots < \src(U_{i,\lambda_i})$ for all $i$,
\item {\em cylindrical} if $|\snk(U_{i,j})| = \src(U_{i,j+1})$ for $j=1,\dotsc,\lambda_i-1$ and
  $|\snk(U_{i,\lambda_i})| = \src(U_{i,1})$.
  \end{enumerate}
For example, consider 
$F_{\ol13\ol2}$ in (\ref{eq:tauclass})
and let $\pi = (\pi_{\ol3}, \pi_{\ol2}, \pi_{\ol1}, \pi_1, \pi_2, \pi_3)$
be the fourth path family shown there.
Then $\bctab(F_w,3)$ and $\bctab(F_w,21)$ contain row-closed tableaux
such as
\begin{equation*}
  {\tableau[scY]{\pi_1,\pi_2,\pi_3}}\,, \quad
  {\tableau[scY]{\pi_1|\pi_3,\pi_2}}\,,
\end{equation*}
the first of which is left-row strict and the second of which is cylindrical.


Left row-strict $F_w$-tableaux of shape $n$
correspond bijectively
to path families in $\PiBC(F_w)$:
\begin{equation}\label{eq:obviousbij}
  {\tableau[scY]{\pi_1}}\,\cdots {\tableau[scY]{\pi_n}}
  \quad \leftrightarrow \quad
  (\pi_{\ol n}, \dotsc, \pi_{\ol1}, \pi_1, \dotsc, \pi_n).
\end{equation}
%
%
These tableaux and path families also correspond bijectively to
marked acyclic orientations of $\inc(Q(w))$ and to
certain subsets of marked $Q(w)$-tableaux.
To describe these correspondences, we first
define an equivalence relation on $\PiBC(F_w)$
by declaring
\begin{equation}\label{eq:varphiequiv}
  \pi \sim \tau \quad \text{ if } \quad
  \varphi(\type(\pi)) = \varphi(\type(\tau)),
  \end{equation}
  where $\varphi: \bn \rightarrow \sn$
  is the map defined in (\ref{eq:varphi}).
  In terms of paths in the two families, $\pi \sim \tau$ if
  $|\snk(\pi_i)| = |\snk(\tau_i)|$ for $i = 1,\dotsc,n$.

  The cardinality of an equivalence class (\ref{eq:varphiequiv})
  depends on the number of
  positive sources of $F_w$ from which there exists
  a path to a negative sink.
  Specifically, if the related type-$\msfBC$ unit interval order $Q(w)$ has $k$
  grounded elements, then 
  each equivalence class consists of $2^k$ families,
  with exactly one family $\tau$ in each class satisfying
  $\ell_t(\type(\tau)) = 0$, i.e., $\snk(\pi_i) > 0$ for $i = 1,\dotsc, n$.
Thus we have the bijection
\begin{equation} \label{eq:taupairs1}
  \begin{aligned}
    \PiBC(F_w) &\rightarrow
    \{(\tau,K) \in \PiBC(F_w) \times 2^{[k]} \,|\, \ell_t(\type(\tau)) = 0\}\\
    \pi &\mapsto
    (\tau, \{ \snk(\pi_{\ol n}), \dotsc, \snk(\pi_{\ol1}) \} \cap \mathbb N ).
    \end{aligned}
\end{equation}
Since
    the positively indexed paths $\tau_1,\dotsc,\tau_n$ of $\tau$
    cover the upper half of $F_w$, i.e.,
    the planar network $F_{\varphi(w)} \in \dnet{A}{[n]}$,
    we also have the bijection
\begin{equation} \label{eq:taupairs2}
  \begin{aligned}
    \bctab(F_w,n) &\rightarrow
    \atab(F_{\varphi(w)},n) \times 2^{[k]}\\
    {\tableau[scY]{\pi_{\ntnsp u_1}}} \cdots {\tableau[scY]{\pi_{\ntnsp u_n}}} &\mapsto
    \big(\;{\tableau[scY]{\tau_{\ntnsp u_1}}} \cdots {\tableau[scY]{\tau_{\ntnsp u_n}}},
    \{ \snk(\pi_{\ol n}), \dotsc, \snk(\pi_{\ol1}) \} \cap \mathbb N \big),
    \end{aligned}
\end{equation}
which preserves the row-closed, left row-strict, and cylindrical
properties of tableaux.
  For example consider the network $F_{\ol13\ol2}$
  and the unique path family $\tau \in \PiBC(F_{\ol13\ol2})$ of type $132$.
  Since the type-$\msfBC$
  unit interval order $Q(\ol13\ol2)$ has $2$ grounded elements,
  the equivalence class of $\tau$ consists of four
  path families encoded by $(\tau,K)$
  for subsets $K \subseteq \{1,2\}$,
  \begin{equation}\label{eq:tauclass}
\begin{tikzpicture}[scale=.5,baseline=0]
\draw[-] (-.5,2.5) -- (.5,1.5) -- (1,0) -- (1.5,1.5);
  \draw[-, dashed] (-.5,-2.5) -- (.5,-1.5) -- (1,0) -- (1.5,-1.5);
\draw[-, very thick, loosely dotted] (-.5,1.5) -- (.5,2.5) -- (1.5,2.5);
  \draw[-, very thick] (-.5,-1.5) -- (.5,-2.5) -- (1.5,-2.5);
\draw[-, ultra thick, densely dashed] (-.5,0.5) -- (.5,0.5) -- (1,0) -- (1.5,0.5);
  \draw[-, very thick, densely dotted] (-.5,-0.5) -- (.5,-0.5) -- (1,0) -- (1.5,-0.5);
\node at (-.9,2.5) {$\scriptstyle 3$};
\node at (-.9,1.5) {$\scriptstyle 2$};
\node at (-.9,0.5) {$\scriptstyle 1$};
\node at (-.9,-0.5) {$\scriptstyle \ol 1$};
\node at (-.9,-1.5) {$\scriptstyle \ol 2$};
\node at (-.9,-2.5) {$\scriptstyle \ol 3$};
\node at (1.9,2.5) {$\scriptstyle 3$};
\node at (1.9,1.5) {$\scriptstyle 2$};
\node at (1.9,0.5) {$\scriptstyle 1$};
\node at (1.9,-0.5) {$\scriptstyle \ol 1$};
\node at (1.9,-1.5) {$\scriptstyle \ol 2$};
\node at (1.9,-2.5) {$\scriptstyle \ol 3$};
\node at (0.5,-4.5) {$(\tau, \emptyset)$};
\end{tikzpicture}\, ,
\qquad
\begin{tikzpicture}[scale=.5,baseline=0]
\draw[-] (-.5,2.5) -- (.5,1.5) -- (1,0) -- (1.5,1.5);
  \draw[-, dashed] (-.5,-2.5) -- (.5,-1.5) -- (1,0) -- (1.5,-1.5);
\draw[-, very thick, loosely dotted] (-.5,1.5) -- (.5,2.5) -- (1.5,2.5);
  \draw[-, very thick] (-.5,-1.5) -- (.5,-2.5) -- (1.5,-2.5);
\draw[-, ultra thick, densely dashed] (-.5,0.5) -- (.5,0.5) -- (1,0) -- (1.5,-0.5);
  \draw[-, very thick, densely dotted] (-.5,-0.5) -- (.5,-0.5) -- (1,0) -- (1.5,0.5);
\node at (-.9,2.5) {$\scriptstyle 3$};
\node at (-.9,1.5) {$\scriptstyle 2$};
\node at (-.9,0.5) {$\scriptstyle 1$};
\node at (-.9,-0.5) {$\scriptstyle \ol 1$};
\node at (-.9,-1.5) {$\scriptstyle \ol 2$};
\node at (-.9,-2.5) {$\scriptstyle \ol 3$};
\node at (1.9,2.5) {$\scriptstyle 3$};
\node at (1.9,1.5) {$\scriptstyle 2$};
\node at (1.9,0.5) {$\scriptstyle 1$};
\node at (1.9,-0.5) {$\scriptstyle \ol 1$};
\node at (1.9,-1.5) {$\scriptstyle \ol 2$};
\node at (1.9,-2.5) {$\scriptstyle \ol 3$};
\node at (0.5,-4.5) {$(\tau, \{1\})$};
\end{tikzpicture}\, ,
\qquad
\begin{tikzpicture}[scale=.5,baseline=0]
\draw[-] (-.5,2.5) -- (.5,1.5) -- (1,0) -- (1.5,-1.5);
  \draw[-, dashed] (-.5,-2.5) -- (.5,-1.5) -- (1,0) -- (1.5,1.5);
\draw[-, very thick, loosely dotted] (-.5,1.5) -- (.5,2.5) -- (1.5,2.5);
  \draw[-, very thick] (-.5,-1.5) -- (.5,-2.5) -- (1.5,-2.5);
\draw[-, ultra thick, densely dashed] (-.5,0.5) -- (.5,0.5) -- (1,0) -- (1.5,0.5);
  \draw[-, very thick, densely dotted] (-.5,-0.5) -- (.5,-0.5) -- (1,0) -- (1.5,-0.5);
\node at (-.9,2.5) {$\scriptstyle 3$};
\node at (-.9,1.5) {$\scriptstyle 2$};
\node at (-.9,0.5) {$\scriptstyle 1$};
\node at (-.9,-0.5) {$\scriptstyle \ol 1$};
\node at (-.9,-1.5) {$\scriptstyle \ol 2$};
\node at (-.9,-2.5) {$\scriptstyle \ol 3$};
\node at (1.9,2.5) {$\scriptstyle 3$};
\node at (1.9,1.5) {$\scriptstyle 2$};
\node at (1.9,0.5) {$\scriptstyle 1$};
\node at (1.9,-0.5) {$\scriptstyle \ol 1$};
\node at (1.9,-1.5) {$\scriptstyle \ol 2$};
\node at (1.9,-2.5) {$\scriptstyle \ol 3$};
\node at (0.5,-4.5) {$(\tau, \{2\})$};
\end{tikzpicture}\, ,
\qquad
\begin{tikzpicture}[scale=.5,baseline=0]
\draw[-] (-.5,2.5) -- (.5,1.5) -- (1,0) -- (1.5,-1.5);
  \draw[-, dashed] (-.5,-2.5) -- (.5,-1.5) -- (1,0) -- (1.5,1.5);
\draw[-, very thick, loosely dotted] (-.5,1.5) -- (.5,2.5) -- (1.5,2.5);
  \draw[-, very thick] (-.5,-1.5) -- (.5,-2.5) -- (1.5,-2.5);
\draw[-, ultra thick, densely dashed] (-.5,0.5) -- (.5,0.5) -- (1,0) -- (1.5,-0.5);
  \draw[-, very thick, densely dotted] (-.5,-0.5) -- (.5,-0.5) -- (1,0) -- (1.5,0.5);
\node at (-.9,2.5) {$\scriptstyle 3$};
\node at (-.9,1.5) {$\scriptstyle 2$};
\node at (-.9,0.5) {$\scriptstyle 1$};
\node at (-.9,-0.5) {$\scriptstyle \ol 1$};
\node at (-.9,-1.5) {$\scriptstyle \ol 2$};
\node at (-.9,-2.5) {$\scriptstyle \ol 3$};
\node at (1.9,2.5) {$\scriptstyle 3$};
\node at (1.9,1.5) {$\scriptstyle 2$};
\node at (1.9,0.5) {$\scriptstyle 1$};
\node at (1.9,-0.5) {$\scriptstyle \ol 1$};
\node at (1.9,-1.5) {$\scriptstyle \ol 2$};
\node at (1.9,-2.5) {$\scriptstyle \ol 3$};
\node at (0.5,-4.5) {$(\tau, \{1,2\})$};
\end{tikzpicture}\ntksp.
  \end{equation}
  We can now relate certain sets of
  $F_w$-tableaux and marked $Q(w)$-tableaux as follows.
  \begin{lem}\label{l:twobijections}
    For $F_w \in \dnet{BC}{[\ol n,n]}$, and corresponding
    $Q = Q(w)$,
    we have bijections
    \begin{enumerate}[(i)]
      \item
      $\{ U \in \bctab(F_w,n) \,|\, U \text{ left row-strict}\, \}
        \overset{1-1}\longleftrightarrow
      \{ U \in \bctab(Q,n) \,|\, U \text{ descent-free}\, \},$
      \item $\{ U \in \bctab(F_w,n) \,|\, U \text{ cylindrical}\, \}
        \overset{1-1}\longleftrightarrow
      \{ U \in \bctab(Q,n) \,|\, U \text{ cyclically row-semistrict}\, \}$.
      \end{enumerate}
  \end{lem}
  \begin{proof}
    Let $k$ be the number of grounded elements of $Q(w)$.
    By
    (\ref{eq:taupairs2}),
    tableaux
    on the left-hand side
    of (i)
    correspond bijectively to pairs
    \begin{equation*}
      \{ (U,K) \in \atab(F_{\varphi(w)},n) \times 2^{[k]} \,|\,
      U \text{ left row-strict} \},
    \end{equation*}
    and by (\ref{eq:perminterp}) these
    correspond bijectively to
    \begin{equation*}
      \{ (V,K) \in \atab(Q(w),n) \times 2^{[k]} \,|\,
      V \text{ descent-free}\, \}.
    \end{equation*}
    Elements of this set correspond bijectively to tableaux
    on the right-hand side of $(i)$:
    simply modify $V$ by 
    marking entries belonging to $K$.
    Similarly,
    tableaux in the first set
    of (ii)
    correspond bijectively to pairs
    \begin{equation*}
      \{ (U,K) \in \atab(F_{\varphi(w)},n) \times 2^{[k]} \,|\,
      U \text{ cylindrical}\, \},
    \end{equation*}
    and by (\ref{eq:pimminterp}) these
    correspond bijectively to
    \begin{equation*}
      \{ (V,K) \in \atab(Q(w),n) \times 2^{[k]} \,|\,
      V \text{ cyclically row-semistrict}\, \}.
    \end{equation*}
    Elements of this set correspond bijectively to
    tableaux on the right-hand side of (ii):
    again modify $V$ by marking entries belonging to $K$.
  \end{proof}
      

  Combining these bijections with Proposition~\ref{p:detDBPB},
  we obtain the following type-$\msfBC$ analogs
  of the results in 
  Proposition~\ref{p:detpermpsiinterplong}.

\begin{prop}\label{p:colstrictetc}
  Fix $w \in \bn$ \avoidingp{} with $F_w \in \znet{BC}{[\ol n,n]}$
  having path matrix $A$,
and define $A^+$, $A^-$ as in (\ref{eq:pm}).
Let $Q = Q(w)$ be the type-$\msfBC$ unit interval order defined before
Proposition~\ref{p:rtorw}.
We have
\begin{equation*}
  \begin{aligned}
    \perm(A^+) &= \# \{ U \in \bctab(F_w,n) \,|\, U \text{ left row-strict}\, \}\\
    &= \# \{ U \in \bctab(Q,n) \,|\, U \text{ descent-free}\, \}\\
    &= \# \{ U \in \bctab(Q,n) \,|\, U \text{ excedance-free}\, \}\\
    &= \# \text{ marked acyclic orientations of $\inc(Q)$},
  \end{aligned}
\end{equation*}
\begin{equation*}
  \begin{aligned}
    \perm(A^-) &= \# \{ U \in \bctab(F_w,n) \,|\, U \text{ left row-strict with no grounded paths}\, \}\\
    &= \# \{ U \in \bctab(Q,n) \,|\, U \text{ descent-free with no grounded elements}\, \}\\
    &= \# \{ U \in \bctab(Q,n) \,|\, U \text{ excedance-free with no grounded elements}\, \}\\
    &= \# \text{ acyclic orientations of $\inc(Q)$ with no grounded vertices},
  \end{aligned}
\end{equation*}
\begin{equation*}
  \begin{aligned}
    \det(A^+) &= \# \{ U \in \bctab(Q,1^n) \,|\, U
    \text{ column-strict with at most $1$
      grounded element}\, \}\\
      &= \# \text{ proper marked $\msfBC$-colorings of $\inc(Q)$ of type } (n,\emptyset),\\
      &= \begin{cases}
      2^k &\text{if $Q$ is a chain with $k \leq 1$ grounded elements},\\
      0 &\text{otherwise},
    \end{cases}
  \end{aligned}
\end{equation*}
\begin{equation*}
  \begin{aligned}
    \det(A^-) &= \# \{ U \in \bctab(Q,1^n) \,|\, U \text{ column-strict with no grounded elements}\, \}\\
          &= \# \text{ proper marked $\msfBC$-colorings of $\inc(Q)$ of type } (\emptyset,n),\\
            &= \begin{cases}
      1 &\text{if $Q$ is a chain with no grounded elements},\\
      0 &\text{otherwise},
      \end{cases}
  \end{aligned}
\end{equation*}
\begin{equation*}
  \begin{aligned}
    \simm n{\psi^n}(A^+)
    &= \# \{ U \in \bctab(F_w,n) \,|\, U
    \text{ cylindrical}\, \},\\
    &= \# \{ U \in \bctab(Q,n) \,|\, U
    \text{ cyclically row-semistrict}\, \},\\
    &= \# \{ U \in \bctab(Q,n) \,|\, U
    \text{ record-free, row-semistrict}\, \},\\
    &= n \cdot \# \{ U \in \bctab(Q,n) \,|\, U
    \text{ right-anchored, row-semistrict}\, \},\\
    &= \# \text{ marked acyclic orientations of $\inc(Q)$ with one source},
  \end{aligned}
\end{equation*}
\begin{equation*}
  \begin{aligned}
    &\simm n{\psi^n}(A^-)
    = \# \{ U \in \bctab(F_w,n) \,|\, U
    \text{ cylindrical with no grounded paths}\, \},\\
    &\qquad= \# \{ U \in \bctab(Q,n) \,|\, U
    \text{ cyclically row-semistrict with no grounded elements}\, \},\\
    &\qquad= \# \{ U \in \bctab(Q,n) \,|\, U
    \text{ record-free, row-semistrict with no grounded elements}\, \},\\
    &\qquad= n \cdot \# \{ U \in \bctab(Q,n) \,|\, U
    \text{ right-anchored, row-semistrict with no grounded elements}\, \},\\
    &\qquad= \# \text{ acyclic orientations of $\inc(Q)$ with one source and
    no grounded vertices}.
  \end{aligned}
\end{equation*}
\end{prop}
\begin{proof}
  By Proposition~\ref{p:detDBPB},
  $\perm(A^+)$ counts all families in $\PiBC(F_w)$,
equivalently (\ref{eq:obviousbij}) all left row-strict tableaux
in $\bctab(F_w,n)$.  By Lemma~\ref{l:twobijections} this equals
the number of descent-free marked $Q$-tableaux, and by
(\ref{eq:perminterp}) it equals the number of
excedance-free marked $Q$-tableaux.
By the algorithms at the end of Subsection~\ref{ss:aincg},
this also equals the number of marked acyclic orientations of $\inc(Q)$.
By Proposition~\ref{p:detDBPB}, $\perm(A^-)$ counts the same
  assuming that $\ell_t(w) = 0$, i.e., that
  $Q$ has no grounded elements, and is $0$ otherwise.

  Now let
  $\sigma = (\sigma_{\ol n}, \dotsc, \sigma_{\ol 1}, \sigma_1, \dotsc, \sigma_n)$
  be the unique path family in $\PiBC_e(F_w)$,
  so that $Q$ is the poset on $\sigma_1, \dotsc, \sigma_n$,
  with grounded elements $\sigma_1, \dotsc, \sigma_k$ for some $k$.
  By Proposition~\ref{p:detDBPB},
  $\det(A^+)$ counts the families $\pi \in \PiBC(F_w)$
  in which only $\pi_{\ol 1}$ and $\pi_1$ may share a vertex.
  This number is nonzero if and only if $Q$ is an $n$-element chain.
  If $\sigma_{\ol 1}$ and $\sigma_1$ do not share a vertex, then $\sigma$
  is the unique such family and $\det(A^+) = 1$.
  In this case, $Q$ has no grounded element and $|\bctab(Q,1^n)| = 1$.
  If on the other hand $\sigma_{\ol 1}$ and $\sigma_1$ do share a vertex,
  then exactly one other path family $\pi \in \PiBC(F_w)$
  is counted and we have $\det(A^+) = 2$.
  In this case, $Q$ has one grounded element
  and $|\bctab(Q,1^n)| = 2$ because the element may appear
  with or without a star in a marked $Q$-tableau.  
  More simply,
  $\det(A^-)$ is $1$ if $F_w = F_e$ and is $0$ otherwise.
  Equivalently, $\det(A^-)$ is $1$
  if $Q$ is a chain with no grounded element and
  is $0$ otherwise.

  By Proposition~\ref{p:detDBPB}, $\simm n{\psi^n}(A^+)$ equals
  $n$ times the number of path families $\pi$
  in $\PiBC(F_w)$ such that
  $\varphi(\type(\pi)) \in \sn$ is an $n$-cycle.
  This is the number of cylindrical tableaux in $\bctab(F_w,n)$
  because each family $\pi$ can be arranged in the orders
  $(\pi_j, \pi_{v(j)}, \pi_{v(v(j))}, \dotsc, \pi_{v^{-1}(j)})$
  for $j = 1, \dotsc, n$ to create $n$ cylindrical $\pi$-tableaux. 
  The remaining interpretations follow from
  (\ref{eq:pimminterp}).
  $\imm{\psi^n}(A^-)$ is the same,
  assuming that $\ell_t(w) = 0$, i.e., that
  $Q$ has no grounded elements, and is $0$ otherwise.
\end{proof}


To combinatorially interpret evaluations of $\bn$-traces,
and state type-$\msfBC$ analogs of the results in Theorem~\ref{t:wtc1interps},
we will use {\em pairs} of path-tableaux, poset-tableaux,
and acyclic orientations.
Define a {\em Young bidiagram of shape $(\lambda,\mu)$} to be a
pair of Young diagrams of shapes $\lambda$ and $\mu$.
Given a type-$\msfBC$ path family
$\pi = (\pi_{\ol n}, \dotsc, \pi_{\ol 1}, \pi_1,\dotsc,\pi_n)$
covering a zig-zag network $F_w \in \znet{BC}{[\ol n,n]}$,
we fill Young bidiagram of shape $(\lambda,\mu)$
with the paths
$(\pi_1,\dotsc,\pi_n)$, keeping grounded paths
in the left diagram.
We call the resulting pair $(U,V)$ of tableaux an {\em $F_w$-bitableau} or
more specifically, a {\em $\pi$-bitableau} of shape $(\lambda,\mu)$.
If $\pi$ has type $v \in \bn$,
then we also say that each $\pi$-bitableau has {\em type $v$}.
Let $\bitab(F_w,\lambda,\mu)$ be the set of all $F_w$-bitableaux
of shape $(\lambda,\mu)$.
If $(U,V)$ is a $\pi$-bitableau of type $v$ with $\varphi(v) = e$,
then
we may use
(\ref{eq:taupairs1}) to replace paths $\pi_1,\dotsc,\pi_n$ in $(U,V)$
with the elements of $Q(w)$, marking each grounded element $i$ in $U$
with a star if $\snk(\pi_i) < 0$.  We call the resulting structure a
{\em marked $Q(w)$-tableau}.
Let $\bitab(Q(w),\lambda,\mu)$ be the set of all marked $Q(w)$-bitableaux
of shape $(\lambda,\mu)$.

For example, consider $F_{2345\ol1} \in \dnet{BC}{[\ol5,5]}$ and the
unique path families $\pi$ of type $2\ol1345$ and $\tau$ of
type $21345$ covering $F_{2345\ol1}$,
\begin{equation}\label{eq:pathex}
  F_{2345\ol1} =
\begin{tikzpicture}[scale=.5,baseline=0]
  \node at (-2.9,4.5) {\scriptsize 5};  \node at (2.9,4.5) {\scriptsize 5};
  \node at (-2.9,3.5) {\scriptsize 4};  \node at (2.9,3.5) {\scriptsize 4};
  \node at (-2.9,2.5) {\scriptsize 3};  \node at (2.9,2.5) {\scriptsize 3};
  \node at (-2.9,1.5) {\scriptsize 2};  \node at (2.9,1.5) {\scriptsize 2};
  \node at (-2.9,0.5) {\scriptsize 1};  \node at (2.9,0.5) {\scriptsize 1};
  \node at (-2.9,-0.5) {\scriptsize $\ol1$};  \node at (2.9,-0.5) {\scriptsize $\ol1$};
  \node at (-2.9,-1.5) {\scriptsize $\ol2$};  \node at (2.9,-1.5) {\scriptsize $\ol2$};
  \node at (-2.9,-2.5) {\scriptsize $\ol3$};  \node at (2.9,-2.5) {\scriptsize $\ol3$};
  \node at (-2.9,-3.5) {\scriptsize $\ol4$};  \node at (2.9,-3.5) {\scriptsize $\ol4$};  
  \node at (-2.9,-4.5) {\scriptsize $\ol5$};  \node at (2.9,-4.5) {\scriptsize $\ol5$};  
\draw[-] (-2.5,4.5) -- (-1.5,3.5);
\draw[-] (-2.5,3.5) -- (-1.5,4.5);
\draw[-] (-2.5,2.5) -- (-1.5,2.5);
\draw[-] (-2.5,1.5) -- (-1.5,1.5);
\draw[-] (-2.5,0.5) -- (-1.5,0.5);
\draw[-] (-2.5,-0.5) -- (-1.5,-0.5);
\draw[-] (-2.5,-1.5) -- (-1.5,-1.5);
\draw[-] (-2.5,-2.5) -- (-1.5,-2.5);
\draw[-] (-2.5,-3.5) -- (-1.5,-4.5);
\draw[-] (-2.5,-4.5) -- (-1.5,-3.5);
\draw[-] (-1.5,4.5) -- (-.5,4.5);
\draw[-] (-1.5,3.5) -- (-.5,2.5);
\draw[-] (-1.5,2.5) -- (-.5,3.5);
\draw[-] (-1.5,1.5) -- (-.5,1.5);
\draw[-] (-1.5,0.5) -- (-.5,0.5);
\draw[-] (-1.5,-0.5) -- (-.5,-0.5);
\draw[-] (-1.5,-1.5) -- (-.5,-1.5);
\draw[-] (-1.5,-2.5) -- (-.5,-3.5);
\draw[-] (-1.5,-3.5) -- (-.5,-2.5);
\draw[-] (-1.5,-4.5) -- (-.5,-4.5);
\draw[-] (-.5,4.5) -- (.5,4.5);
\draw[-] (-.5,3.5) -- (.5,3.5);
\draw[-] (-.5,2.5) -- (.5,1.5);
\draw[-] (-.5,1.5) -- (.5,2.5);
\draw[-] (-.5,0.5) -- (.5,0.5);
\draw[-] (-.5,-0.5) -- (.5,-0.5);
\draw[-] (-.5,-1.5) -- (.5,-2.5);
\draw[-] (-.5,-2.5) -- (.5,-1.5);
\draw[-] (-.5,-3.5) -- (.5,-3.5);
\draw[-] (-.5,-4.5) -- (.5,-4.5);
\draw[-] (.5,4.5) -- (1.5,4.5);
\draw[-] (.5,3.5) -- (1.5,3.5);
\draw[-] (.5,2.5) -- (1.5,2.5);
\draw[-] (.5,1.5) -- (1.5,0.5);
\draw[-] (.5,0.5) -- (1.5,1.5);
\draw[-] (.5,-0.5) -- (1.5,-1.5);
\draw[-] (.5,-1.5) -- (1.5,-0.5);
\draw[-] (.5,-2.5) -- (1.5,-2.5);
\draw[-] (.5,-3.5) -- (1.5,-3.5);
\draw[-] (.5,-4.5) -- (1.5,-4.5);
\draw[-] (1.5,4.5) -- (2.5,4.5);
\draw[-] (1.5,3.5) -- (2.5,3.5);
\draw[-] (1.5,2.5) -- (2.5,2.5);
\draw[-] (1.5,1.5) -- (2.5,1.5);
\draw[-] (1.5,0.5) -- (2.5,-0.5);
\draw[-] (1.5,-0.5) -- (2.5,0.5);
\draw[-] (1.5,-1.5) -- (2.5,-1.5);
\draw[-] (1.5,-2.5) -- (2.5,-2.5);
\draw[-] (1.5,-3.5) -- (2.5,-3.5);
\draw[-] (1.5,-4.5) -- (2.5,-4.5);
\end{tikzpicture}\,,
\qquad \qquad
\begin{tikzpicture}[scale=.5,baseline=0]
  \node at (-3.1,4.5) {$\pi_5$};
  \node at (-3.1,3.5) {$\pi_4$};  
  \node at (-3.1,2.5) {$\pi_3$};  
  \node at (-3.1,1.5) {$\pi_2$};  
  \node at (-3.1,0.5) {$\pi_1$};  
  \node at (-3.1,-0.5) {$\pi_{\ol1}$};  
  \node at (-3.1,-1.5) {$\pi_{\ol2}$};  
  \node at (-3.1,-2.5) {$\pi_{\ol3}$};  
  \node at (-3.1,-3.5) {$\pi_{\ol4}$};
  \node at (-3.1,-4.5) {$\pi_{\ol5}$};
\draw[-] (-2.5,4.5) -- (-2,4) -- (-1.5,4.5) -- (2.5,4.5);
\draw[-, ultra thick, dotted] (-2.5,3.5) -- (-2,4) -- (-1,3) -- (-.5,3.5) -- (2.5,3.5);
\draw[-, ultra thick] (-2.5,2.5) -- (-1.5,2.5) -- (-1,3) -- (0,2) -- (.5,2.5) -- (2.5,2.5);
\draw[-, dashed] (-2.5,1.5) -- (-.5,1.5) -- (0,2) -- (2.5, -0.5);
\draw[-] (-2.5,0.5) -- (.5,0.5) -- (1.5,1.5) -- (2.5, 1.5);
\draw[-] (-2.5,-0.5) -- (.5,-0.5) -- (1.5, -1.5) -- (2.5,-1.5);
\draw[-, ultra thick, dashed] (-2.5,-1.5) -- (-.5,-1.5) -- (0,-2) -- (2.5,.5);
\draw[-, ultra thick] (-2.5,-2.5) -- (-1.5,-2.5) -- (-1,-3) -- (0,-2) -- (.5,-2.5) -- (2.5,-2.5);
\draw[-, ultra thick, dotted] (-2.5,-3.5) -- (-2,-4) -- (-1,-3) -- (-.5,-3.5) -- (2.5,-3.5);
\draw[-] (-2.5,-4.5) -- (-2,-4) -- (-1.5,-4.5) -- (2.5,-4.5);
\end{tikzpicture}\,,
\qquad
\begin{tikzpicture}[scale=.5,baseline=0]
  \node at (-3.1,4.5) {$\tau_5$};
  \node at (-3.1,3.5) {$\tau_4$};  
  \node at (-3.1,2.5) {$\tau_3$};  
  \node at (-3.1,1.5) {$\tau_2$};  
  \node at (-3.1,0.5) {$\tau_1$};  
  \node at (-3.1,-0.5) {$\tau_{\ol1}$};  
  \node at (-3.1,-1.5) {$\tau_{\ol2}$};  
  \node at (-3.1,-2.5) {$\tau_{\ol3}$};  
  \node at (-3.1,-3.5) {$\tau_{\ol4}$};
  \node at (-3.1,-4.5) {$\tau_{\ol5}$};
\draw[-] (-2.5,4.5) -- (-2,4) -- (-1.5,4.5) -- (2.5,4.5);
\draw[-, ultra thick, dotted] (-2.5,3.5) -- (-2,4) -- (-1,3) -- (-.5,3.5) -- (2.5,3.5);
\draw[-, ultra thick] (-2.5,2.5) -- (-1.5,2.5) -- (-1,3) -- (0,2) -- (.5,2.5) -- (2.5,2.5);
\draw[-, dashed] (-2.5,1.5) -- (-.5,1.5) -- (0,2) -- (2,0) -- (2.5,.5);
\draw[-] (-2.5,0.5) -- (.5,0.5) -- (1.5,1.5) -- (2.5, 1.5);
\draw[-] (-2.5,-0.5) -- (.5,-0.5) -- (1.5, -1.5) -- (2.5,-1.5);
\draw[-, ultra thick, dashed] (-2.5,-1.5) -- (-.5,-1.5) -- (0,-2) -- (2,0) -- (2.5,-.5);
\draw[-, ultra thick] (-2.5,-2.5) -- (-1.5,-2.5) -- (-1,-3) -- (0,-2) -- (.5,-2.5) -- (2.5,-2.5);
\draw[-, ultra thick, dotted] (-2.5,-3.5) -- (-2,-4) -- (-1,-3) -- (-.5,-3.5) -- (2.5,-3.5);
\draw[-] (-2.5,-4.5) -- (-2,-4) -- (-1.5,-4.5) -- (2.5,-4.5);
\end{tikzpicture}\,.
\end{equation}
Two $\pi$-bitableaux and two $\tau$-bitableaux of shape $(21,11)$
are
\begin{equation}\label{eq:2111fex}
  \left(\, \tableau[scY]{\pi_5 | \pi_2,\pi_1}\, {,}\,
       \tableau[scY]{\pi_4 | \pi_3}\, \right)\ntnsp {,}
  \qquad
  \left(\, \tableau[scY]{\pi_5 | \pi_2,\pi_3}\, {,}\,
  \tableau[scY]{\pi_1| \pi_4}\, \right)\ntnsp {,}
  \qquad
  \left(\, \tableau[scY]{\tau_5 | \tau_1,\tau_2}\, {,}\,
  \tableau[scY]{\tau_4 | \tau_3}\, \right)\ntnsp {,}
  \qquad
  \left(\, \tableau[scY]{\tau_4 | \tau_1,\tau_2}\, {,}\,
  \tableau[scY]{\tau_5 | \tau_3}\, \right)\ntnsp {.}
\end{equation}
Since paths $\pi_2$ and $\tau_2$
intersect $\pi_{\ol2}$ and $\tau_{\ol2}$, respectively, they are grounded
and must appear in the left tableaux.
The poset $Q(2345\ol1)$ and four marked $Q(2345\ol1)$-bitableaux of shape
$(21,11)$ are
\begin{equation}\label{eq:2111pex}
  \begin{tikzpicture}[scale=.5,baseline=10]
  \draw[fill] (0,3) circle [radius=0.15];
  \draw[fill] (0,1) circle [radius=0.15];
  \draw[fill] (0,-1) circle [radius=0.15];
  \draw[fill] (1,2) circle [radius=0.15];
  \draw[fill] (1,0) circle [radius=0.15];
  \draw[thick] (0,-1) circle [radius=0.35];
\node at (-.5,3) {$\scriptstyle 5$};
\node at (-.5,1) {$\scriptstyle 3$};
\node at (-.6,-1) {$\scriptstyle 1$};
\node at (1.5,2) {$\scriptstyle 4$};
\node at (1.5,0) {$\scriptstyle 2$};
\draw[-, thick] (0,3) -- (0,-1);
\draw[-, thick] (0,3) -- (1,0);
\draw[-, thick] (1,2) -- (0,-1);
\draw[-, thick] (1,2) -- (1,0);
\end{tikzpicture}\ntksp ,
\qquad
  \left(\, \tableau[scY]{3 | 1^\star,2}\, {,}\, \tableau[scY]{5 | 4}\, \right)\ntnsp {,}
  \quad
  \left(\, \tableau[scY]{3 | 1,4}\, {,}\, \tableau[scY]{5 | 2}\, \right)\ntnsp {,}
  \quad
  \left(\, \tableau[scY]{3 |2,1}\, {,}\, \tableau[scY]{4 | 5}\, \right)\ntnsp {,}
  \quad
  \left(\, \tableau[scY]{1^\star | 3,5}\, {,}\, \tableau[scY]{4 | 2}\, \right)\ntnsp {.}\nTksp\nTksp
\end{equation}

If $(U,V)$ is an $F_w$-bitableau or $Q(w)$-bitableau,
then the component tableaux $U$, $V$
could have some of the properties enumerated
between (\ref{eq:posettableaux}) and
Proposition~\ref{p:detpermpsiinterplong}.
Some combinations of these 
may be used to interpret trace evaluations as in
Theorems~\ref{t:epsiloneta} -- \ref{t:psi}.

\begin{thm}\label{t:epsiloneta}
  Let $w \in \bn$ \avoidp{} and let $Q = Q(w)$ be the
  related type-$\msfBC$ unit interval order.
  For each bipartition $(\lambda,\mu) \vdash n$ we have
  \begin{equation*}
    \begin{aligned}
    (\epsilon\epsilon)^{\lambda,\mu}(\btc w1) &=
      \# \{ (U,V) \in \bitab(Q,\lambda^\tr,\mu^\tr) \,|\,
      U, V\, \text{column-strict\,} \},\\
      (\epsilon\eta)^{\lambda,\mu}(\btc w1) &=
      \#\{(U,V) \in \bitab(Q,\lambda^\tr,\mu) \,|\,
      U\, \text{column-strict}, V\, \text{row-semistrict\,} \},\\
    (\eta\epsilon)^{\lambda,\mu}(\btc w1) &=
      \# \{ (U,V) \in \bitab(Q,\lambda,\mu^\tr) \,|\,
      U\, \text{row-semistrict}, V\, \text{column-strict\,} \},\\
      (\eta\eta)^{\lambda,\mu}(\btc w1) &= \# \{ (U,V) \in \bitab(Q,\lambda,\mu) \,|\, U, V\, \text{row-semistrict\,} \}.
      \end{aligned}
    \end{equation*}
  \end{thm}
\begin{proof}
  Let $F = F_w$ have path matrix $A$,
  fix $(\lambda,\mu)\vdash n$ with $|\lambda| = k$,
  and let $\theta = (\zeta \xi)^{\lambda,\mu}$ be one of the characters
  in the theorem.
  By Lemma~\ref{l:tensorimm} and Theorem~\ref{t:charevalimmbc}, we have
  \begin{equation}\label{eq:immchi}
    (\zeta\xi)^{\lambda,\mu}(\btc w1) = 
    \bnimm{(\zeta\xi)^{\lambda,\mu}}(A) =
    \sumsb{I \subseteq [n]\\|I|=k} \simm{k}{\zeta^\lambda}(A^+_{I,I})
    \simm{n-k}{\xi^\mu}(A^-_{[n] \ssm I,[n] \ssm I}).
  \end{equation}
  By (\ref{eq:lmw}), each of the type-$\msfA$ immanants above is a
  sum of products of determinants
  or permanents
  of matrices in
  $\{ A^+_{J,J} \,|\, J \subset I \}$
  or $\{ A^-_{J,J} \,|\, J \subset [n] \ssm I \}$.
  By Proposition~\ref{p:colstrictetc} each
  factor counts
  certain one-row or one-column tableaux and their product counts
  bitableaux of the required shape.
\end{proof}

\begin{cor}\label{c:epsiloneta}
  Let $w \in \bn$ \avoidingp{}
  have zig-zag network $F_w$ and
  type-$\msfBC$ unit interval order $Q = Q(w)$.
  The combinatorial interpretations of trace evaluations in
  Theorem~\ref{t:epsiloneta}
  have several valid alternatives.
  We may replace (marked) column-strict $Q$-tableaux $U$, $V$ of shape $\nu$
  with (marked) colorings of $\inc(Q)$ of type $\nu$.
  We may replace (marked) row-semistrict $Q$-tableaux $U$, $V$ of shape $\nu$
  with either of the following:
  \begin{enumerate}[(i)]
\item (marked) excedance-free $Q$-tableaux of shape $\nu$,
\item (marked) row-closed, left row-strict $F_w$-tableaux of shape $\nu$.
\end{enumerate}
  Furthermore, we have that $(\eta\eta)^{\lambda,\mu}(\btc w1)$ equals
 \begin{enumerate}[(i)]
  \item[(iii)] the number of
  marked acyclic orientations of subgraph sequences
  \end{enumerate}
  \begin{equation}\label{eq:subgraphseq}
    (\inc(Q_{I_1}), \dotsc, \inc(Q_{I_r}), \inc(Q_{J_1}), \dotsc, \inc(Q_{J_t}))
    \end{equation}
  where
  $(I_1,\dotsc,I_{\ell(\lambda)}, J_1,\dotsc,J_{\ell(\mu)})$
  varies over all ordered set partitions of $[n]$ of type
  \begin{equation*}
    (\lambda_1,\dotsc,\lambda_{\ell(\lambda)},\mu_1,\dotsc,\mu_{\ell(\mu)}),
    \end{equation*}
  and all grounded vertices appear in $I_1,\dotsc,I_{\ell(\lambda)}$.
\end{cor}
\begin{proof}
  By Proposition~\ref{p:colstrictetc},
  one-rowed tableaux
  in Theorem~\ref{t:epsiloneta}
  which are row-semistrict
  (descent-free)
  correspond bijectively to one-rowed tableaux
  with properties
  $(i)$ or $(ii)$ 
  and to acyclic orientations $(iii)$
  of $\inc(Q)$.
  Thus there is a correspondence of 
  several-rowed tableaux and of acyclic orientations of subgraph sequences.
  In all cases, we have equality of
  the indices of paths, poset elements, and vertices which are grounded,
  so markings correspond in the obvious way.
\end{proof}
For example, consider evaluating traces at $\btc{2345\ol1}1$ by
applying Theorem~\ref{t:epsiloneta} 
to the poset $Q(2345\ol1)$ in (\ref{eq:2111pex}).
The first bitableau in (\ref{eq:2111pex})
contributes to
$(\epsilon\eta)^{21,11}(\btc{2345\ol1}1)$ and
$(\eta\eta)^{21,11}(\btc{2345\ol1}1)$,
the second to $(\epsilon\epsilon)^{21,2}(\btc{2345\ol1}1)$,
$(\epsilon\eta)^{21,11}(\btc{2345\ol1}1)$,
$(\eta\epsilon)^{21,2}(\btc{2345\ol1}1)$,
$(\eta\eta)^{21,11}(\btc{2345\ol1}1)$,
the third to $(\eta\eta)^{21,11}(\btc{2345\ol1}1)$,
and the fourth to $(\eta\epsilon)^{21,2}(\btc{2345\ol1}1)$
and $(\eta\eta)^{21,11}(\btc{2345\ol1}1)$.
    If we modify the criteria of the theorem
    as in Corollary~\ref{c:epsiloneta} (i),
    the same is true,
    since a tableau with two or
    fewer columns is $Q$-excedance-free if and only
    if it is $Q$-row-semistrict.  

    If instead we modify the criteria of the theorem
    as in Corollary~\ref{c:epsiloneta} (ii) we may
    apply these to path families
    covering $F_{2345\ol1}$ in (\ref{eq:pathex}).
    In this case, we find that the third and fourth
    tableaux in (\ref{eq:2111fex})
    contribute to $(\eta\eta)^{21,11}(\btc{2345\ol1}1)$,
    and that the fourth tableau also contributes to
    $(\eta\epsilon)^{21,2}(\btc{2345\ol1}1)$.
    The tableaux contribute to no additional evaluations listed in
    Theorem~\ref{t:epsiloneta} because
    the left tableau of the first bitableau is not left row-semistrict,
    the left tableau of the second bitableau is not row-closed,
    and the left tableaux of all four bitableaux fail to have the property
    that source and sink indices are equal.

    Finally, we may modify the criteria of the theorem
    by applying Corollary~\ref{c:epsiloneta} (iii) to
    \begin{equation}\label{eq:incQ23451}
      \inc(Q(2345\ol1)) =
      \begin{tikzpicture}[scale=.5,baseline=-3]
  \draw[thick] (0,0) circle [radius=0.35];
  \draw[fill] (0,0) circle [radius=0.15];
  \draw[fill] (1.5,0) circle [radius=0.15];
  \draw[fill] (3,0) circle [radius=0.15];
  \draw[fill] (4.5,0) circle [radius=0.15];
  \draw[fill] (6,0) circle [radius=0.15];
\node at (0,.65) {$\scriptstyle 1$};
\node at (1.5,.65) {$\scriptstyle 2$};
\node at (3,.65) {$\scriptstyle 3$};
\node at (4.5,.65) {$\scriptstyle 4$};
\node at (6,.65) {$\scriptstyle 5$};
\draw[-] (0,0) -- (6,0);
\end{tikzpicture},
\end{equation}      
and by marking and acyclically orienting sequences of subgraphs on
$(2, 1, 1, 1)$ vertices.  Two sequences contributing to
$(\eta\eta)^{21,11}(\wtc{2345\ol1}1)$ are
    \begin{equation}\label{eq:markedor}
\qquad
(
\begin{tikzpicture}[scale=.5,baseline=-3]
\draw[thick] (1,0) circle [radius=0.35];
\draw[fill] (1,0) circle [radius=0.15];
\draw[fill] (2.5,0) circle [radius=0.15];
\node at (1,.65) {$\scriptstyle 1$};
\node at (2.5,.65) {$\scriptstyle 2$};
\draw[->] (1,0) -- (2.3,0);
\end{tikzpicture}\ntnsp,\,\;
\begin{tikzpicture}[scale=.5,baseline=-3]
\draw[fill] (1,0) circle [radius=0.15];
\node at (1,.65) {$\scriptstyle 4$};
\end{tikzpicture}\ntnsp,\,\;
\begin{tikzpicture}[scale=.5,baseline=-3]
\draw[fill] (1,0) circle [radius=0.15];
\node at (1,.65) {$\scriptstyle 5$};
\end{tikzpicture}\ntnsp,\,\;
\begin{tikzpicture}[scale=.5,baseline=-3]
\draw[fill] (1,0) circle [radius=0.15];
\node at (1,.65) {$\scriptstyle 3$};
\end{tikzpicture}
),
\qquad
(
\begin{tikzpicture}[scale=.5,baseline=-3]
\draw[fill] (1,0) circle [radius=0.15];
\draw[fill] (2.5,0) circle [radius=0.15];
\node at (1,.65) {$\scriptstyle 3$};
\node at (2.5,.65) {$\scriptstyle 2$};
\draw[->] (1,0) -- (2.3,0);
\end{tikzpicture}\ntnsp,\,
\begin{tikzpicture}[scale=.5,baseline=-3]
  \draw[thick] (1,0) circle [radius=0.35];
  \draw[fill] (1,0) circle [radius=0.15];
\node at (1,.65) {$\scriptstyle ^{\phantom{\star}}1^\star$};
\end{tikzpicture}\ntksp,\;
\begin{tikzpicture}[scale=.5,baseline=-3]
\draw[fill] (1,0) circle [radius=0.15];
\node at (1,.65) {$\scriptstyle 4$};
\end{tikzpicture}\ntnsp,\,\;
\begin{tikzpicture}[scale=.5,baseline=-3]
\draw[fill] (1,0) circle [radius=0.15];
\node at (1,.65) {$\scriptstyle 5$};
\end{tikzpicture}
).
\end{equation}





\begin{thm}\label{t:chi}
  Let $w \in \bn$ \avoidingp{} have type-$\msfBC$ unit interval order $Q$,
  and fix bipartition $(\lambda,\mu) \vdash n$.
  We have
  \begin{equation*}
    (\chi\chi)^{\lambda,\mu}(\btc w1) =
    \# \{ (U,V) \in \bitab(Q,\lambda,\mu) \,|\,
       U, V \text{ standard } \}.
   \end{equation*}
\end{thm}
\begin{proof}
  Let $A$ be the path matrix of $F_w$ and let $\chi^\lambda$, $\chi^\mu$
  be the characters of $\mfs{k}$, $\mfs{n-k}$ satisfying
  $(\chi\chi)^{\lambda,\mu} =
  (\chi^\lambda \otimes \delta \chi^\mu) \upparrow_{\mfb{k,n-k}}^{\bn}$.
  By
  Lemma~\ref{l:tensorimm} and
  Theorem~\ref{t:charevalimmbc}
  we have
  \begin{equation*}
  (\chi\chi)^{\lambda,\mu}(\btc w1) = \bnimm{(\chi\chi)^{\lambda,\mu}}(A)
  = \sumsb{I \subset [n]\\|I| = k}
  \simm k{\chi^\lambda}(A_{I,I}^+) \simm{n-k}{\chi^\mu}(A_{[n] \ssm I, [n] \ssm I}^-).
  \end{equation*}
  Expanding irreducible character immanants
  $\simm k{\chi^\lambda}$ and $\simm{n-k}{\chi^\mu}$
  in terms of induced trivial character immanants, we obtain
  \begin{equation*}
  \sumsb{I \subset [n]\\|I| = k}
  \sum_{\alpha \vdash k} K_{\alpha,\lambda}^{-1} \imm{\eta^\alpha}(A_{I,I}^+)
  \ntksp \sum_{\beta \vdash n-k} \nTksp
  K_{\beta,\mu}^{-1} \imm{\eta^\beta}(A_{[n]\ssm I,[n] \ssm I}^-).
  \end{equation*}
  Now let $Q$
  be the type-$\msfBC$ unit interval order corresponding to $w$.
  For any subset $J$ of $[n]$ and any partition $\nu \vdash |J|$,
  let $r(Q_J,\nu)$ be the number of row-semistrict
  $Q_J$-tableaux of shape $\nu$, and let $p(J)$
  be the number of elements of $Q_J$ which are grounded.
  By (\ref{eq:lmw}) and Proposition \ref{p:colstrictetc}, we have
  \begin{equation*}
    \simm k{\eta^\alpha}(A_{I,I}^+) = 2^{p(I)} r(Q_I, \alpha),
    \qquad
  \end{equation*}
  and
  \begin{equation*}
    \simm{n-k}{\eta^\beta}(A_{[n]\ssm I,[n]\ssm I}^-) =
    \begin{cases} r(Q_{[n]\ssm I}, \beta) &\text{if $p([n]\ssm I) = 0$}\\
      0 &\text{otherwise}.
    \end{cases}
  \end{equation*} 
  Thus by \cite[Thm.\,4.7 (ii-b), (iii)]{CHSSkanEKL}, the sums
  \begin{equation*}
  \begin{gathered}
    \sum_{\alpha \vdash k} K_{\alpha,\lambda}^{-1} \simm k{\eta^\alpha}(A_{I,I}^+)
    = 2^{p(I)} \sum_{\alpha \vdash k} K_{\alpha,\lambda}^{-1} r(Q_I, \alpha),\\
    \sum_{\beta \vdash n-k} \nTksp
    K_{\beta,\mu}^{-1} \simm{n-k}{\eta^\beta}(A_{[n]\ssm I,[n]\ssm I}^-)
    = \begin{cases}
      \displaystyle{\sum_{\beta \vdash n-k} \nTksp
        K_{\beta,\mu}^{-1} r(Q_{[n]\ssm I}, \beta)} &\text{if $p([n]\ssm I) = 0$},\\
      0 &\text{otherwise}
      \end{cases}
  \end{gathered}
  \end{equation*}
  are equal to the numbers of
  standard $Q_I$-tableaux of shape $\lambda$
  containing any number of grounded elements which may be circled,
  and standard $Q_{[n]\ssm I}$-tableaux of shape $\mu$
  containing no grounded elements, respectively.  
\end{proof}
For example, consider the evaluation $(\chi\chi)^{21,11}(\btc{2345\ol1}1)$
and the poset $Q(2345\ol1)$ in (\ref{eq:2111pex}).
Of the bitableaux shown there, only the second contributes to this evaluation.

\begin{thm}\label{t:psi}
  Let $w \in \bn$ \avoidingp{}
  have type-$\msfBC$ unit interval order $Q$,
  and fix bipartition
  $(\lambda,\mu) \vdash n$.
  We have
  \begin{equation*}
    (\psi\psi)^{\lambda,\mu}(\btc w1) =
    \# \{ (U,V) \in \bitab(Q,\lambda,\mu) \,|\, U,V \text{ cyclically
      row-semistrict } \}.
  \end{equation*}
  \end{thm}
\begin{proof}
  Let $A$ be the path matrix of $F_w$ and let $k = |\lambda|$.
  By Lemma~\ref{l:tensorimm} and
  Theorem~\ref{t:charevalimmbc}
  we have
  \begin{equation}\label{eq:psipsisum}
  (\psi\psi)^{\lambda,\mu}(\btc w1) = \bnimm{(\psi\psi)^{\lambda,\mu}}(A)
  = \sumsb{I \subset [n]\\|I| = m}
  \simm k{\psi^\lambda}(A_{I,I}^+) \simm{n-k}{\psi^\mu}(A_{[n] \ssm I, [n] \ssm I}^-).
  \end{equation}
  By \cite[Prop.\,2.4]{StemConj} we have
  \begin{equation*}
    \simm k{\psi^\lambda}(\bfx)
    = \nTksp \sum_{(J_1,\dotsc,J_r)} \nTksp
    \simm{\lambda_1}{\psi^{\lambda_1}}(\bfx_{J_1,J_1}) \cdots
    \simm{\lambda_r}{\psi^{\lambda_r}}(\bfx_{J_r,J_r}).
  \end{equation*}
  Thus each term in the sum (\ref{eq:psipsisum}) is itself
  a sum of products of type-$\msfA$ single-cycle
  power sum trace immanants, evaluated at $A^+$ or $A^-$.
  By Proposition~\ref{p:colstrictetc} each factor counts
  one-row cyclically row-semistrict tableaux,
  and their product counts
  bitableaux of the required shape.
\end{proof}

\begin{cor}\label{c:psi}
Let $w \in \bn$ \avoidingp{} have zig-zag network $F_w$ and
type-$\msfBC$ unit interval order $Q = Q(w)$.
In Theorem~\ref{t:psi}
we may replace
cyclically row-semistrict $Q$-tableaux
with either of the following:
\begin{enumerate}
\item[(i)]
  record-free $Q$-tableaux,
\item[(ii)]
cylindrical $F_w$-tableaux.
\end{enumerate}
We also have that $(\psi\psi)^{\lambda,\mu}(\btc w1)$ equals
\begin{enumerate}
\item[(iii)]
$       \lambda_1 \cdots \lambda_r \mu_1 \cdots \mu_s \cdot
      \# \{ (U,V) \in \bitab(Q,\lambda,\mu) \,|\,
      U, V \text{ left-anchored, row-semistrict } \}$.
    \item[(iv)]
      the number of marked acyclic orientations of subgraph sequences
      (\ref{eq:subgraphseq})     
      in which each oriented subgraph has one source,
      and all grounded vertices appear in $I_1,\dotsc,I_r$.
      \end{enumerate}
\end{cor}
\begin{proof}
  By Proposition~\ref{p:colstrictetc},
  one-rowed tableaux
  in Theorem~\ref{t:psi}
  which are cyclically
  row-semistrict
  correspond bijectively to one-rowed tableaux
  with properties
  $(i)$ or $(ii)$ 
  and to acyclic orientations $(iii)$ of
  $\inc(Q)$ which have one source.
  Thus there is a correspondence of
  several-rowed tableaux and of acyclic orientations of subgraph sequences.
  In all cases, we have equality of
  the indices of paths, poset elements, and vertices which are grounded,
  so markings correspond in the obvious way.
\end{proof}

For example, consider evaluating $(\psi\psi)^{21,11}(\btc{2345\ol1}1)$ by
applying Theorem~\ref{t:psi} 
to the poset $Q(2345\ol1)$ in (\ref{eq:2111pex}).
Of the tableaux listed there, only the first and third contribute to
this trace evaluation.  If we modify the criteria of the theorem as in
Corollary~\ref{c:psi} (i), the same is true, 
since a tableau with two or
fewer columns is $Q$-record-free if and only
if it is cyclically $Q$-row-semistrict.  
On the other hand, if we modify the criteria of the theorem as in
Corollary~\ref{c:psi} (iii), then the first, second,
and fourth bitableaux contribute.

Now suppose that we modify the criteria of the theorem
    as in Corollary~\ref{c:psi} (ii).  Then we may
    apply these to path families
    covering $F_{2345\ol1}$ in (\ref{eq:pathex}),
    and we find that the first, third and fourth
    tableaux in (\ref{eq:2111fex})
    contribute to $(\psi\psi)^{21,11}(\btc{2345\ol1}1)$.

    Finally, we may modify the criteria of the theorem
    by applying Corollary~\ref{c:psi} (iii) to
    (\ref{eq:incQ23451}).  Both tableaux in (\ref{eq:markedor}) contribute,
    since no component has more than one source.

While
Proposition~\ref{p:detpermpsiinterplong} -- Theorem~\ref{t:wtc1interpssubg}
have known $q$-analogs,
no such $q$-analogs are known for the results in Subsection~\ref{ss:mainBC}.

\bp
State and prove $q$-analogs of the results in
Lemma~\ref{l:twobijections} --
Corollary~\ref{c:psi}
\ep

\section{Some equivalence relations}\label{s:equiv}

Given a Coxeter group $W$ and its Hecke algebra $H = H(W)$,
the trace space $\trsp(H)$ naturally partitions $H$ into equivalence classes
via the relation $\approx$ defined by
\begin{equation}\label{eq:approx}
  D_1 \approx D_2 \text{ if for all } \theta_q \in \trsp(H) \text{ we have }
  \theta_q(D_1) = \theta_q(D_2).
\end{equation}
When $W = \sn$ or $\bn$, we may restrict this relation to the subset
of Kazhdan-Lusztig basis elements indexed by elements of $W$ \avoidingp.
Two more equivalence relations related to this restricted relation
are defined in terms of
isomorphism of the posets and graphs
described in Sections~\ref{s:uio} -- \ref{s:incgraph}:
$P(v) \cong P(w)$ ($Q(v) \cong Q(w)$), or $G(v) \cong G(w)$.
In type $\msfA$ these relations refine the first;
in types $\msfBC$, we conjecture the same to be true
and prove a weaker statement.

\ssec{Type-$\msfA$ equivalence relations}\label{ss:aequiv}





The equivalence relation on
\begin{equation}\label{eq:smoothkl}
  \{ \wtc wq \,|\, w \in \sn \text{ \avoidsp } \}
  \end{equation}
defined by poset isomorphism $P(v) \cong P(w)$
refines that defined by graph isomorphism $G(v) \cong G(w)$,
which in turn refines
the restriction of the relation $\approx$ (\ref{eq:approx})
to this set.
\begin{thm}\label{t:Aequivimplications}
  For $v, w \in \sn$ \avoidingp, we have the implications
  \begin{equation*}
    P(v) \cong P(w)
    \quad \Longrightarrow \quad
    G(v) \cong G(w)
    \quad \Longrightarrow \quad
    \wtc vq \approx \wtc wq.
  \end{equation*}
  \end{thm}
\begin{proof}
  The first implication is clear; the second follows from
  a $q$-analog of Theorem~\ref{t:wtc1interps}~(i-b) in
  \cite[Prop.\,7.3 -- Thm.\,7.4]{CHSSkanEKL} which states that
  all of the evaluations
  $\{\epsilon_q^\lambda(\wtc wq) \,|\, \lambda \vdash n\}$
  are determined by
  $G(w)$.
\end{proof}
The converse of the first implication above is not true (\ref{eq:uioag});
the converse of the second is not known to be true.
\bp
For some $n$ find \pavoiding permutations $v, w \in \sn$
which satisfy
$G(v) \not \cong G(w)$ and
$\wtc vq \approx \wtc wq$
or show that this is impossible.
\ep


By Theorems~\ref{t:312bij} and~\ref{t:Aequivimplications},
the problem of evaluating $\hnq$-traces at (\ref{eq:smoothkl})
reduces to the problem of evaluating these traces at the subset of
(\ref{eq:smoothkl}) indexed by $312$-avoiding
permutations~\cite[Thm.\,5.6]{CHSSkanEKL}.
\begin{cor}\label{c:Acodominant}
  For $w \in \sn$ \avoidingp, there exists $v \in \sn$ avoiding $312$
  such that
  $\wtc wq \approx \wtc vq$.
\end{cor}
Furthermore, some experimentation suggests that
the problem of evaluating traces
at {\em all} Kazhdan--Lusztig basis elements reduces to the
problem of evaluating traces at
the subset of (\ref{eq:smoothkl}) indexed by $312$-avoiding
permutations.
With precise details of such a reduction not yet conjectured,
we have the following problem~\cite[Conj.\,1.9]{ANigroUpdate},
\cite[Conj.\,3.1]{HaimanHecke}.
\bp\label{prob:ANHcodom}
Show that for each $w \in \sn$ there exists a set $S = S(w) \subseteq \sn$
of $312$-avoiding permutations
and a set $\{p_{v,w}(q) \,|\, v \in S \} \subset \mathbb N[q]$
of polynomials such that we have
\begin{equation*}
  \wtc wq \approx \sum_{v \in S} p_{v,w}(q) \wtc vq.
\end{equation*}
\ep
We remark that Haiman~\cite[\S 3]{HaimanHecke}
introduced the name {\em codominant} for $312$-avoiding permutations
because by (\ref{eq:laction}) these have the form
$vw_0$
for $w_0$ the longest element of $\sn$ and
$v$ belonging to the dominant ($132$-avoiding) subset of
the vexillary ($2143$-avoiding) permutations defined in \cite{LS2}.



\ssec{Type-$\msfBC$ equivalence relations}

We distinguish between the type-$\msfBC$ case
of the equivalence relation (\ref{eq:approx})
and its $q=1$ specialization.
For $D_1, D_2 \in \hbnq$, define  
\begin{equation}\label{eq:approxbcq}
  D_1 \approx_q D_2 \text{ if for all } \theta_q \in \trsp(\hbnq)
  \text{ we have } \theta_q(D_1) = \theta_q(D_2).
\end{equation}
For $D_1, D_2 \in \zbn$,
let $D_1 \approx_1 D_2$ be the $q=1$ specialization of the above.
These relations naturally restrict to the Kazhdan--Lusztig bases of
$\hbnq$ and $\zbn$, and to the subsets of these indexed by elements $w \in \bn$
\avoidingp.
It is easy to show that the equivalence relation on
\begin{equation}\label{eq:smoothbc}
  \{ \btc wq \,|\, w \in \bn\ {\text \avoidingp } \}
\end{equation}
defined by type-$\msfBC$ unit interval order
isomorphism $Q(v) \cong Q(w)$
refines that defined by
type-$\msfBC$ indifference graph
isomorphism
$\Gamma(v) \cong \Gamma(w)$,
which in turn refines the relation $\approx_1$.
\begin{thm}\label{t:BCequivimplications}
  For $v, w \in \bn$ \avoidingp, we have the implications
  \begin{equation*}
    Q(v) \cong Q(w)
    \quad \Longrightarrow \quad
    \Gamma(v) \cong \Gamma(w)
    \quad \Longrightarrow \quad
    \btc v1 \approx_1 \btc w1.
  \end{equation*}
  \end{thm}
\begin{proof}
  The first implication is clear; the second follows from
  Corollary~\ref{c:epsiloneta}
\end{proof}
The converse of the first implication above is not true (\ref{eq:uiocg});
the converse of the second is not known to be true.
We conjecture that the second implication can be strengthened.
\begin{conj}\label{conj:BCequivimplications}
  For $v, w \in \bn$ \avoidingp, we have the implication
    $\Gamma(v) \cong \Gamma(w) \ \Longrightarrow \
  \btc vq \approx_q \btc wq$.
\end{conj}
The converse of this conjectured implication is not known to be true.

\bp
For some $n$ find \pavoiding permutations $v, w \in \bn$
which satisfy
$\Gamma(v) \not \cong \Gamma(w)$
and $\btc vq \approx_q \btc wq$,
or show that this is impossible.
\ep


By Theorems~\ref{t:dnetbcpavoid} and \ref{t:BCequivimplications}, 
the problem of evaluating $\bn$-traces at
\begin{equation}\label{eq:smoothklbc}
  \{ \btc w1 \,|\, w \in \bn \ {\text \avoidsp} \}
\end{equation}
reduces to the problem of evaluating these traces at the subset of
(\ref{eq:smoothklbc}) indexed by
elements of $\bn$ \avoidingsignedp.



\begin{cor}\label{c:BCcodominant}
  For $w \in \bn$ \avoidingp, there exists $v \in \bn$ \avoidingsignedp{}
  such that we have $\btc w1 \approx_1 \btc v1$.
\end{cor}
If Conjecture~\ref{conj:BCequivimplications} is true, then the conclusion
of Corollary~\ref{c:BCcodominant} becomes $\btc wq \approx_q \btc vq$.
It would be interesting to discover the extent to which trace evaluations at
\begin{equation*}
  \{ \btc wq \,|\, w \in \bn \ {\text \avoidssignedp} \}
\end{equation*}
describe trace evaluations at the entire Kazhdan--Lusztig basis of $\hbnq$,
as in Problem~\ref{prob:ANHcodom}.
One might call the above $\bn$-elements
{\em 
  codominant}
in analogy to codominant
permutations in $\sn$,
although the author is not aware of a definition of
dominant elements of $\bn$ in the literature.
On the other hand,
it would be interesting to relate
codominant elements of $\bn$
to the subsets of
vexillary,
theta-vexillary,
Grassmanian,
leading, and
amenable
elements of $\bn$, which
appear in
\cite{AFultonDLoci},
\cite{AFultonVexRevisited},
\cite{BilleyLamVex},
\cite{LambertTheta},
\cite{TamvakisAmenable}.


\section{Symmetric functions}\label{s:symm}


For $W = \sn$ or $\bn$, and $H = H(W)$ its Hecke algebra,
bases of $\trsp(W)$ and $\trsp(H(W))$
are often studied in conjunction
with bases of an appropriate
module $\Lambda$
of symmetric functions.
When $W = \sn$, the symmetric function bases consist of
traces of the corresponding Lie group; when $W = \bn$ they do not.
In either case, we have the equalities
\begin{equation*}
  \rnk(\trsp(W)) = \rnk(\trsp(H(W))) = \rnk(\Lambda),
\end{equation*}
which makes $\Lambda$ a convenient setting
in which to define generating functions for trace evaluations.


\ssec{Type-$\msfA$ symmetric functions}


Corresponding to the six commonly used bases 
of $\trsp(\hnq)$ and $\trsp(\sn)$
(\S \ref{ss:snhnq})
are six bases of the $\mathbb Z$-module
\begin{equation*}
  \Lambda_n(x) = \Lambda_n(x_1, x_2, \dotsc ),
\end{equation*}
of homogeneous degree-$n$ symmetric functions:
the Schur basis $\{ s_\lambda \,|\, \lambda \vdash n\}$,
elementary basis $\{ e_\lambda \,|\, \lambda \vdash n \}$,
(complete) homogeneous basis $\{ h_\lambda \,|\, \lambda \vdash n \}$,
power sum basis $\{ p_\lambda \,|\, \lambda \vdash n\}$,
monomial basis $\{ m_\lambda \,|\, \lambda \vdash n \}$,
and forgotten basis $\{ f_\lambda \,|\, \lambda \vdash n \}$.
(See \cite[Ch.\,6]{StanEC2}.)
The correspondence of trace bases and symmetric function bases
is given explicitly by (the $q$-extension of) the Frobenius map
\begin{align}
  \frobch_q: \trsp(\hnq) &\rightarrow \Lambda_n(x)\label{eq:frobch0}\\
  \frobch_q(\theta_q)
  &\defeq \frac1{n!}\sum_{w \in \sn} \theta(w)p_{\ctype(w)}\label{eq:frobch1}\\
  &= \sum_{\mu \vdash n}  \frac1{z_\mu} \theta(\mu)p_\mu,\label{eq:frobch2}
\end{align}
where $\theta = \theta_1$ as in Section~\ref{s:trace},
and $\theta(\mu) \defeq \theta(w)$ for any $w \in \sn$ of cycle type $\mu$.
Specifically, we have
\begin{equation}\label{eq:frobchs}
  \begin{gathered}
  \frobch_q(\chi_q^\lambda) = s_\lambda,
  \qquad \frobch_q(\epsilon_q^\lambda) = e_\lambda,
  \qquad \frobch_q(\eta_q^\lambda) = h_\lambda,\\
  \qquad \frobch_q(\psi_q^\lambda) = p_\lambda,
  \qquad \frobch_q(\phi_q^\lambda) = m_\lambda,
  \qquad \frobch_q(\gamma_q^\lambda) = f_\lambda.
  \end{gathered}
\end{equation}

We construct generating functions for $\hnq$-trace evaluations as follows.
Given element $D \in \mathbb Q(q) \otimes \hnq$, define
the generating function
\begin{equation*}
  Y_q(D) \defeq \sum_{\lambda \vdash n} \epsilon_q^\lambda(D) m_\lambda
  \in \mathbb Q(q) \otimes \Lambda_n(x)
\end{equation*}
for the evaluation of induced sign characters at $D$.
By \cite[Prop.\,2.1]{SkanCCS}, this symmetric function is in fact a generating
function for the evaluation of {\em all} the standard traces at $D$, because it
is equal to
\begin{equation*}
    \sum_{\lambda \vdash n} \eta_q^\lambda(D)f_\lambda
  = \sum_{\lambda \vdash n}\frac{(-1)^{n-\ell(\lambda)}\psi_q^\lambda(D)}{z_\lambda}p_\lambda
  = \sum_{\lambda \vdash n} \chi_q^{\lambda^\tr}(D)s_\lambda
  = \sum_{\lambda \vdash n} \phi_q^\lambda(D)e_\lambda
  = \sum_{\lambda \vdash n} \gamma_q^\lambda(D)h_\lambda.
\end{equation*}
Equivalently, if we let
$\omega: \Lambda_n(x) \rightarrow \Lambda_n(x)$
be the standard involution mapping
\begin{equation}\label{eq:omega}
  s_\lambda \mapsto s_{\lambda^\tr},
  \qquad e_\lambda \mapsto h_\lambda,
  \qquad p_\lambda \mapsto (-1)^{n - \ell(\lambda)}p_\lambda,
  \qquad m_\lambda \mapsto f_\lambda,
  \end{equation}
then we have that
$\omega Y_q(D)$ is equal to  
\begin{equation*}
    \sum_{\lambda \vdash n} \epsilon_q^\lambda(D)f_\lambda
  = \sum_{\lambda \vdash n} \eta_q^\lambda(D)m_\lambda
  = \sum_{\lambda \vdash n}\frac{\psi_q^\lambda(D)}{z_\lambda}p_\lambda
  = \sum_{\lambda \vdash n} \chi_q^\lambda(D)s_\lambda
  = \sum_{\lambda \vdash n} \phi_q^\lambda(D)h_\lambda
  = \sum_{\lambda \vdash n} \ntnsp \gamma_q^\lambda(D)e_\lambda.\nTksp
\end{equation*}
It is not difficult to show that
every symmetric function in $\mathbb Q(q) \otimes \Lambda_n(x)$
is $Y_q(D)$ for some $D \in \mathbb Q(q) \otimes \hnq$.
(See \cite[Prop.\,3]{SkanCCS}.)

Theorem~\ref{t:wtc1interps} ($i$-$b$)
shows that
when
$w \in \sn$ \avoidsp,
the symmetric function $Y_q(\wtc wq)$
is related to
colorings of $\inc(P(w))$.
More generally,
Stanley~\cite{StanSymm} defined
the {\em chromatic symmetric function} of any simple graph $G$ to be 
\begin{equation}\label{eq:chrom}
  X_G \defeq \sum_\kappa x_1^{|\kappa^{-1}(1)|} x_2^{|\kappa^{-1}(2)|} \cdots,
\end{equation}
where the sum is over all proper colorings
$\kappa: V \rightarrow \{1,2,\dotsc, \}$ of $G$ (\S \ref{ss:aincg}).
Expanding in the monomial basis of $\Lambda_n$, we have
\begin{equation*}
  X_G = \sum_\lambda c_{G,\lambda} m_\lambda,
\end{equation*}
where $c_{G,\lambda}$ is the number of
proper colorings of $G$ of type $\lambda$.


Shareshian and Wachs~\cite{SWachsChromQ} defined a quasisymmetric extension
$X_{G,q}$ of the
symmetric function $X_G$.
Given a proper coloring $\kappa$
of $G$,
define $\inv_G(\kappa)$ to be the number of pairs $(i,j) \in E$ with
$i < j$ and $\kappa(i) > \kappa(j)$.
For any composition $\alpha = (\alpha_1, \dotsc, \alpha_\ell) \vDash n$, define
\begin{equation*}
  c_{G,\alpha}(q) = \nTksp \sumsb{\kappa \text{ proper}\\ \type(\kappa) = \alpha} \nTksp
  q^{\inv_G(\kappa)},
\end{equation*}
and let
\begin{equation*}
  M_\alpha = \ntksp \sum_{i_1 < \cdots < i_\ell} \ntksp
  x_{i_1}^{\alpha_1} \cdots x_{i_\ell}^{\alpha_\ell}
\end{equation*}
be the {\em monomial quasisymmetric function} indexed by $\alpha$.
Then we have the definition
\begin{equation}\label{eq:chromq}
  X_{G,q} = \sum_\kappa q^{\inv_G(\kappa)}
  x_1^{|\kappa^{-1}(1)|} x_2^{|\kappa^{-1}(2)|} \cdots
  = \sum_{\alpha \vDash n} c_{G,\alpha}(q) M_\alpha,
\end{equation}
where the first sum is over
proper colorings of $G$.
It is easy to see that the $q=1$ specialization of $X_{G,q}$ satisfies
$X_{G,1} = X_G$.
When $G = \inc(P)$ for a unit interval order $P$
labeled as in Algorithm \ref{a:ptow},
the quasisymmetric function $X_{\inc(P),q}$
is in fact symmetric \cite[Thm.\,4.5]{SWachsChromQF}.
Furthermore, we have the following~\cite[Thm.\,7.4]{CHSSkanEKL}.
\begin{thm}\label{t:YequalsX}
  For $w \in \sn$ \avoidingp,
  and $P = P(w)$ defined as in Subsection~\ref{ss:auio},
  we have
  $Y_q(\wtc wq) = X_{\inc(P),q}$.
\end{thm}


\ssec{Type-$\msfBC$ symmetric functions}

Corresponding to the eleven commonly used bases
of $\trsp(\hbnq)$ and $\trsp(\bn)$
(\S \ref{ss:bnhbnq})
are eleven natural bases of the $\mathbb Z$-module
of {\em type-$\msfBC$ symmetric functions} of degree $n$,
\begin{equation*}
\Lambda_n(x,y) \defeq \bigoplus_{k=0}^n \Lambda_k(x) \otimes \Lambda_{n-k}(y),
\end{equation*}
where $x = (x_1, x_2, \dotsc )$ and $y = (y_1, y_2, \dotsc )$.
These bases of $\Lambda_n(x,y)$ consist
of ten {\em nonplethystic bases}
of the form
$(og)_{\lambda,\mu} \defeq o_\lambda(x)g_\mu(y)$,
\begin{equation}\label{eq:bcnpbases}
  \begin{gathered}
\{(ss)_{\lambda,\mu}\,|\, (\lambda,\mu) \vdash n \}, \quad
\{(ee)_{\lambda,\mu}\,|\, (\lambda,\mu) \vdash n \}, \quad
\{(hh)_{\lambda,\mu}\,|\, (\lambda,\mu) \vdash n \}, \\
\{(eh)_{\lambda,\mu}\,|\, (\lambda,\mu) \vdash n \}, \quad
\{(he)_{\lambda,\mu}\,|\, (\lambda,\mu) \vdash n \}, \\
\{(mm)_{\lambda,\mu}\,|\, (\lambda,\mu) \vdash n \}, \quad
\{(f\ntnsp f)_{\lambda,\mu}\,|\, (\lambda,\mu) \vdash n \}, \quad
\{(mf)_{\lambda,\mu}\,|\, (\lambda,\mu) \vdash n \}, \\
\{(fm)_{\lambda,\mu}\,|\, (\lambda,\mu) \vdash n \}, \quad
\{(pp)_{\lambda,\mu}\,|\, (\lambda,\mu) \vdash n \}, \\
  \end{gathered}
\end{equation}
and the {\em plethystic power sum basis} 
\begin{equation}\label{eq:pbasis}
  \{p^+_\lambda p^-_\mu\,|\, (\lambda,\mu) \vdash n \},
\end{equation}
defined in terms of
ordinary power sum symmetric functions
$p_k(x) \defeq x_1^k + x_2^k + \cdots$ by
\begin{equation}\label{eq:pbasis2}
  \begin{gathered}
    p_\lambda^+ \defeq
    (p_{\lambda_1}(x) + p_{\lambda_1}(y)) \cdots (p_{\lambda_{\ell(\lambda)}}(x) + p_{\lambda_{\ell(\lambda)}}(y)),\\
    p_\mu^- \defeq
    (p_{\mu_1}(x) - p_{\mu_1}(y)) \cdots (p_{\mu_{\ell(\mu)}}(x) - p_{\mu_{\ell(\mu)}}(y)).
  \end{gathered}
\end{equation}
The functions (\ref{eq:pbasis2})
often appear in the literature as
$p_\lambda[X+Y], p_\mu[X-Y]$, respectively.
A correspondence between bases of $\trsp(\hbnq)$ and the bases
(\ref{eq:bcnpbases}) -- (\ref{eq:pbasis}) of
$\Lambda_n(x,y)$ is given
explicitly by the ($q$-extension of the)
plethystic $\msfBC$-Frobenius map~\cite[\S 1, App.~B]{M1}.
(See also~\cite[Eq.\,(2.5)]{AAERCharFormulasHyper}.)
\begin{align}
  \pfrobch_q\ntksp:
  \trsp(\hbnq) &\rightarrow \Lambda_n(x,y) \label{eq:frobchaaer0}\\
  \pfrobch_q(\theta_q) &= \frac1{2^nn!} \ntnsp \sum_{w \in \bn} \ntksp \theta(w)
  p^+_{\alpha(w)} p^-_{\beta(w)} \label{eq:frobchaaer1} \\
  &= \nTksp \sum_{(\lambda,\mu) \vdash n} \ntnsp 
  \frac{\theta(\lambda,\mu)}{z_\lambda z_\mu 2^{\ell(\lambda)+\ell(\mu)}}
  p^+_{\lambda} p^-_{\mu}, \label{eq:frobchaaer2}
  \end{align}
where $\alpha(w)$, $\beta(w)$ are the partitions satisfying
$\sct(w) = (\alpha(w), \beta(w))$, and where we define
$\theta(\lambda,\mu) \defeq \theta(w)$ for any $w \in \bn$
having $\sct(w) = (\lambda,\mu)$.
Specifically, $\pfrobch_q$ maps
\begin{equation}\label{eq:pfrobchtable}
  \begin{gathered}
    \quad (\eta\eta)^{\lambda,\mu}_q \mapsto (hh)_{\lambda,\mu},
  \quad (\eta\epsilon)^{\lambda,\mu}_q \mapsto (he)_{\lambda,\mu},    
  \quad (\epsilon\eta)^{\lambda,\mu}_q \mapsto (eh)_{\lambda,\mu},
  \quad (\epsilon\epsilon)^{\lambda,\mu}_q \mapsto (ee)_{\lambda,\mu},
  \\
  (\phi\phi)^{\lambda,\mu}_q \mapsto (mm)_{\lambda,\mu},
  \quad (\phi\gamma)^{\lambda,\mu}_q \mapsto (mf)_{\lambda,\mu},
  \quad (\gamma\phi)^{\lambda,\mu}_q \mapsto (fm)_{\lambda,\mu},
  \quad (\gamma\gamma)^{\lambda,\mu}_q \mapsto (f\ntnsp f)_{\lambda,\mu},\nTksp\\
  (\chi\chi)^{\lambda,\mu}_q \mapsto (ss)_{\lambda,\mu},
  \quad (\psi\psi)^{\lambda,\mu}_q \mapsto (pp)_{\lambda,\mu},
  \quad \iota^{\lambda,\mu}_q \mapsto p^+_\lambda p^-_\mu.
  \end{gathered}
\end{equation}
We extend the involutive homomorphism $\omega$ on $\Lambda_n(x)$
(\ref{eq:omega})
to an involutive homomorphsim on $\Lambda_n(x,y)$ in the simplest way:
  $\omega(o_\lambda(x)g_\mu(y)) \defeq \omega(o_\lambda(x))\omega(g_\mu(y))$.
This exchanges the symmetric functions on line $1$ of
(\ref{eq:pfrobchtable}) with the corresponding functions on line $2$,
transposes the index shapes of $(ss)_{\lambda,\mu}$ and multiplies
each power sum basis element by $(-1)^{\ell(\lambda)+\ell(\mu)}$.
Transition matrices relating the ten nonplethystic bases
have entries which are simply products of entries of transition
matrices relating type-$\msfA$ symmetric functions, e.g.,
\begin{equation}\label{eq:sstomm}
  (ss)_{\lambda,\mu} = \nTksp \sum_{(\alpha,\beta) \vdash n} \nTksp
  K_{\lambda, \alpha}K_{\mu, \beta}  (mm)_{\alpha,\beta}.
\end{equation}
The plethystic power sum basis can be related to the others via the
nonplethystic Schur basis,
\begin{equation}\label{eq:pptoss}
  p_\lambda^+p_\mu^-  = \nTksp \sum_{(\alpha,\beta)\vdash n} \nTksp
  (\chi\chi)^{\alpha,\beta}(\lambda,\mu) (ss)_{\alpha,\beta}.
\end{equation}

We construct generating functions for $\hbnq$-trace evaluations as follows.
Given any element $D \in \hbnq$, define the generating function
\begin{equation}\label{eq:YBCdefn}
  \YBCq(D) = \sum_{\lambda \vdash n}
  (\epsilon\epsilon)^{\lambda,\mu}_q(D) (mm)_{\lambda,\mu}
  \in \zqq \otimes \Lambda_n(x,y).
\end{equation}
This symmetric function is in fact a generating
function for the evaluation of {\em all} the standard traces at $D$,
in the following sense.
\begin{prop}\label{p:YBCexpansions}
  For $D \in \hbnq$ we have
  \begin{equation*}
  \begin{aligned}
  \YBCq(D)
  &= \sum_{(\lambda,\mu)\vdash n} (\epsilon\epsilon)^{\lambda,\mu}_q(D) (mm)_{\lambda,\mu}
  = \sum_{(\lambda,\mu)\vdash n} (\epsilon\eta)^{\lambda,\mu}_q(D) (mf)_{\lambda,\mu}
  = \sum_{(\lambda,\mu)\vdash n} (\eta\epsilon)^{\lambda,\mu}_q(D) (fm)_{\lambda,\mu}\\
  &= \sum_{(\lambda,\mu)\vdash n} (\eta\eta)^{\lambda,\mu}_q(D) (f\ntnsp f)_{\lambda,\mu}
  = \sum_{(\lambda,\mu)\vdash n} (\chi\chi)^{\lambda^\tr,\mu^\tr}_q(D) (ss)_{\lambda,\mu}
  = \sum_{(\lambda,\mu)\vdash n} (\phi\phi)^{\lambda,\mu}_q(D) (ee)_{\lambda,\mu}\\
  &= \sum_{(\lambda,\mu)\vdash n} (\phi\gamma)^{\lambda,\mu}_q(D) (eh)_{\lambda,\mu}
  = \sum_{(\lambda,\mu)\vdash n} (\gamma\phi)^{\lambda,\mu}_q(D) (he)_{\lambda,\mu}
  = \sum_{(\lambda,\mu)\vdash n} (\gamma\gamma)^{\lambda,\mu}_q(D) (hh)_{\lambda,\mu}\\
  &= \sum_{(\lambda,\mu)\vdash n}
  \frac{(-1)^{\ell(\lambda)+\ell(\mu)}(\psi\psi)^{\lambda,\mu}_q(D)}{z_\lambda z_\mu}
  (pp)_{\lambda,\mu}
  = \sum_{(\lambda,\mu)\vdash n}
  (-1)^{\ell(\lambda)+\ell(\mu)} \iota^{\lambda,\mu}_q(D) p^+_\lambda p^-_\mu.
  \end{aligned}
  \end{equation*}
  Equivalently, $\omega\YBCq(D)$
  is equal to
  \begin{equation}\label{eq:omegabc}
  \begin{aligned}
  &\sum_{(\lambda,\mu)\vdash n} (\epsilon\epsilon)^{\lambda,\mu}_q(D) (f\ntnsp f)_{\lambda,\mu}
  = \sum_{(\lambda,\mu)\vdash n} (\epsilon\eta)^{\lambda,\mu}_q(D) (fm)_{\lambda,\mu}
  = \sum_{(\lambda,\mu)\vdash n} (\eta\epsilon)^{\lambda,\mu}_q(D) (mf)_{\lambda,\mu}\\
  &= \sum_{(\lambda,\mu)\vdash n} (\eta\eta)^{\lambda,\mu}_q(D) (mm)_{\lambda,\mu}
  = \sum_{(\lambda,\mu)\vdash n} (\chi\chi)^{\lambda,\mu}_q(D) (ss)_{\lambda,\mu}
  = \sum_{(\lambda,\mu)\vdash n} (\phi\phi)^{\lambda,\mu}_q(D) (hh)_{\lambda,\mu}\\
  &= \sum_{(\lambda,\mu)\vdash n} (\phi\gamma)^{\lambda,\mu}_q(D) (he)_{\lambda,\mu}
  = \sum_{(\lambda,\mu)\vdash n} (\gamma\phi)^{\lambda,\mu}_q(D) (eh)_{\lambda,\mu}
  = \sum_{(\lambda,\mu)\vdash n} (\gamma\gamma)^{\lambda,\mu}_q(D) (ee)_{\lambda,\mu}\\
  &= \sum_{(\lambda,\mu)\vdash n} \frac{(\psi\psi)^{\lambda,\mu}_q(D)}{z_\lambda z_\mu} (pp)_{\lambda,\mu}
  = \sum_{(\lambda,\mu)\vdash n} \iota^{\lambda,\mu}_q(D) p^+_\lambda p^-_\mu.
  \end{aligned}
  \end{equation}
\end{prop}
\begin{proof}
  Consider the fourth and fifth sums in (\ref{eq:omegabc}),
  in which the symmetric functions and traces satisfy
  \begin{equation}\label{eq:kostkasum}
    (ss)_{\lambda,\mu} = \nTksp \sum_{(\alpha,\beta) \vdash n} \nTksp
    K_{\lambda, \alpha}K_{\mu, \beta}  (mm)_{\alpha,\beta},
    \qquad
    (\eta\eta)^{\alpha,\beta}_q = \nTksp \sum_{(\lambda,\mu) \vdash n} \nTksp
    K_{\lambda, \alpha}K_{\mu, \beta} (\chi\chi)^{\lambda,\mu}_q.
  \end{equation}
  Using (\ref{eq:kostkasum}) to expand the fifth sum in the
  monomial symmetric function basis, we have
  \begin{equation*}
    \begin{aligned}
      \sum_{(\lambda,\mu)\vdash n} \nTksp (\chi\chi)^{\lambda,\mu}_q(D)
      \nTksp \sum_{(\alpha,\beta) \vdash n} \nTksp
      K_{\lambda,\alpha}K_{\mu,\beta} (mm)_{\alpha,\beta}
      &= \sum_{(\alpha,\beta) \vdash n} \sum_{(\lambda,\mu)\vdash n}
      K_{\lambda,\alpha}K_{\mu,\beta} (\chi\chi)^{\lambda,\mu}_q(D) (mm)_{\alpha,\beta} \\
      &= \sum_{(\alpha,\beta) \vdash n} \ntksp
      (\eta\eta)^{\alpha,\beta}_q(D) (mm)_{\alpha,\beta},
    \end{aligned}
  \end{equation*}
     i.e., it is equal to the fourth sum.
    Similarly, for each of the remaining sums of the form
    $\sum_{(\lambda,\mu) \vdash n} (\zeta\xi)^{\lambda,\mu}_q(D) o_\alpha(x) g_\beta(y)$
    in (\ref{eq:omegabc}),
    there is a matrix
    $(M_{(\lambda,\mu),(\alpha,\beta)})_{(\lambda,\mu) \vdash n, (\alpha,\beta) \vdash n}$
    and equations
    \begin{equation*}
      (ss)_{\lambda,\mu} = \sum_{(\alpha,\beta)}
      M_{(\lambda,\mu),(\alpha,\beta)} o_\alpha(x)g_\beta(y),
    \qquad
    (\zeta\xi)^{\alpha,\beta}_q = \sum_{(\lambda,\mu)}
    M_{(\lambda, \mu),(\alpha,\beta)} (\chi\chi)^{\lambda,\mu}_q,
    \end{equation*}
    relating it to the fifth sum.
    In particular,
    we have
    $M_{(\lambda,\mu),(\alpha,\beta)} =
    K_{\lambda^\tr,\alpha}K_{\mu^\tr,\beta}$,
    $K_{\lambda^\tr,\alpha}K_{\mu,\beta}$,
    $K_{\lambda,\alpha}K_{\mu^\tr,\beta}$,
    $K_{\ntnsp\alpha,\lambda}^{-1}K_{\ntnsp\beta,\mu}^{-1}$,
    $K_{\ntnsp\alpha,\lambda}^{-1}K_{\ntnsp\beta,\mu^\tr}^{-1}$,
    $K_{\ntnsp\alpha,\lambda^\tr}^{-1}K_{\ntnsp\beta,\mu}^{-1}$,
    $K_{\ntnsp\alpha,\lambda^\tr}^{-1}K_{\ntnsp\beta,\mu^\tr}^{-1}$,
    $\chi^\lambda(\alpha)\chi^\mu(\beta)$,
    respectively.    
    (See \cite[\S 3]{RemTrans}.)
    Relating the last sum to the fifth sum, we have equations
    \begin{equation*}
      (ss)_{\lambda,\mu} = \sum_{(\alpha,\beta)}
      \frac{(\chi\chi)^{\lambda,\mu}(\alpha,\beta)}
           {z_\alpha z_\beta 2^{\ell(\alpha) + \ell(\beta)}}
      p^+_\alpha p^-_\beta,
      \qquad
      \iota^{\alpha,\beta} = \sum_{(\lambda,\mu)}
      \frac{(\chi\chi)^{\lambda,\mu}(\alpha,\beta)}
           {z_\alpha z_\beta 2^{\ell(\alpha) + \ell(\beta)}}
      (\chi\chi)^{\lambda,\mu}.
    \end{equation*}
\end{proof}

To say that the functions $\{ Y_q(D) \,|\, D \in \zbn \}$ arise often
in the study of type-$\msfBC$ symmetric functions would be an
understatement;
in fact, {\em every} element of $\mathbb Z[q] \otimes \Lambda_n(x,y)$
has this form.


\begin{prop}\label{p:everysymmfnbcq}
  Every symmetric function in $\mathbb Z[q] \otimes \Lambda_n(x,y)$
  has the form $Y_q(D)$ for some element
  $D \in \mathbb Q(q)[\bn]$.
\end{prop}
\begin{proof}
  Fix a symmetric function in $\Lambda_n(x,y)$
  and expand it in the plethystic power sum basis as
 $\sum_{(\lambda,\mu) \vdash n} a_{\lambda,\mu}(q) p_\lambda^+p_\mu^-$.
  Then choose one representative $w_{\kappa,\nu}$
  of each conjugacy class of $\bn$ and consider the trace evaluations
  (\ref{eq:iotaqdef})
  \begin{equation*}
    \iota^{\lambda,\mu}_q(T_{w_{\kappa,\nu}}) =
    \sum_{(\alpha,\beta)\vdash n} (\chi\chi)^{\alpha,\beta}(\lambda,\mu)
    (\chi\chi)_q^{\alpha,\beta}(T_{w_{\kappa,\nu}}).
  \end{equation*}
  Since the matrices
  $((\chi\chi^{\alpha,\beta}(\lambda,\mu))_{(\lambda,\mu),(\alpha,\beta)}$
  and
  $((\chi\chi_q^{\alpha,\beta}(T_{w_{\kappa,\nu}}))_{(\alpha,\beta),(\kappa,\nu)}$
  are both invertible, so is their product
  $((\iota_q^{\lambda,\mu}(T_{w_{\kappa,\nu}}))_{(\lambda,\mu),(\kappa,\nu)}$.
  Call the inverse of this product $B = (b_{(\alpha,\beta),(\lambda,\mu)}(q))$
  and for each $(\alpha,\beta) \vdash n$ define
  \begin{equation*}
    U_{\alpha,\beta} =
    \sum_{(\kappa,\nu)\vdash n} b_{(\alpha,\beta),(\kappa,\nu)}(q) T_{w_{\kappa,\nu}}
    \in \mathbb Q(q) \otimes \hbnq.
  \end{equation*}
  Then
  we have
  \begin{equation*}
    \iota_q^{\lambda,\mu}(U_{\alpha,\beta}) = \begin{cases}
      1 &\text{if $(\lambda,\mu) = (\alpha,\beta)$},\\
      0 &\text{otherwise}.
    \end{cases}
  \end{equation*}
  Now define the $\mathbb Q(q)[\bn]$
  element
  \begin{equation*}
    D = \sum_{(\alpha,\beta) \vdash n}
    a_{\alpha,\beta}(q) U_{w_{\alpha,\beta}}.
  \end{equation*}
  By (\ref{eq:iotaqdef}) and (\ref{eq:omegabc}) we have
  \begin{equation*}
    Y_q(D) = \sum_{(\lambda,\mu) \vdash n}
    \iota_q^{\lambda,\mu}
    \Big( \sum_{(\alpha,\beta) \vdash n}
    a_{\alpha,\beta}(q)
    U_{\alpha,\beta} \Big) p_\lambda^+p_\mu^-
    = \sum_{(\lambda,\mu) \vdash n} a_{\lambda,\mu}(q) p_\lambda^+p_\mu^-,
  \end{equation*}
  as desired.
\end{proof}



For any type-$\msfBC$
incomparability graph $\Gamma = (V,E)$,
we define the
{\em type-$\msfBC$ chromatic symmetric function of $\Gamma$} to be 
\begin{equation}\label{eq:XBCGdef}
  X^{\msfBC}_G \defeq \sum_\kappa
  (x_1^{|\kappa^{-1}(1)|} x_2^{|\kappa^{-1}(2)|} \cdots )
  (y_1^{|\kappa^{-1}(-1)|} y_2^{|\kappa^{-1}(-2)|} \cdots ),
\end{equation}
where the sum is over all proper $\msfBC$-colorings
$\kappa: V \rightarrow \mathbb Z \ssm \{0 \}$ of $\Gamma$
(\S \ref{ss:bcincg}).

For the incomparability graph $\inc(Q)$ of a $\msfBC$-poset $Q$,
we may express the symmetric function
$X^{\msfBC}_{\inc(Q)}$ in terms of decompositions of $P$ into chains.
Letting $c_{Q,\lambda,\mu}$ be the number of column-strict marked $Q$-bitableaux
of shape $(\lambda^\tr, \mu^\tr)$, we have
\begin{equation*}
  X^{\msfBC}_{\inc(Q)} = \sum_\lambda c_{Q,\lambda,\mu} (mm)_{\lambda,\mu}.
\end{equation*}


\begin{thm}\label{t:bcxy}
  For $w \in \bn$ \avoidingp{} and $Q = Q(w)$ defined as in
  Subsection~\ref{ss:bcuio}, we have
  $Y^{\msfBC}(\btc w1) = X^{\msfBC}_{\inc(Q)}$.
\end{thm}
\begin{proof}
  The coefficient of $(mm)_{\lambda,\mu}$ in $X^{\msfBC}_{\inc(Q)}$
  equals the number
  of proper $\msfBC$-colorings of $\inc(Q)$ of type $(\lambda,\mu)$.
  This is exactly the number of marked column-strict
  $\inc(Q)$-bitableaux of shape $(\lambda^\tr,\mu^\tr)$.
  By Theorem~\ref{t:epsiloneta},
  this is $(\epsilon\epsilon)^{\lambda,\mu}(\btc w1)$,
  and by (\ref{eq:YBCdefn}),
  we have the desired equality.
\end{proof}
It would be interesting to state and prove
a $q$-analog of Theorem~\ref{t:bcxy}.
\begin{prob}\label{p:qanalogofbcchrom}
  Define a statistic $\stat$ on proper $\msfBC$-colorings $\kappa$ of
  type-$\msfBC$ interval graphs so that
  for $w \in \bn$ \avoidingp,
  $Q = Q(w)$ and
  $\Gamma = \inc(Q)$,
  we have a $q$-chromatic symmetric function of the form
  \begin{equation*}
    X^{\msfBC}_{\Gamma,q} = \sum_\kappa q^{\stat(\kappa)} 
  (x_1^{|\kappa^{-1}(1)|} x_2^{|\kappa^{-1}(2)|} \cdots )
  (y_1^{|\kappa^{-1}(-1)|} y_2^{|\kappa^{-1}(-2)|} \cdots )
  \end{equation*}
  which satisfies the identity
    $Y_q^{\msfBC}(\btc wq) = X^{\msfBC}_{\Gamma,q}$.
\end{prob}


\ssec{Another approach to type-$\msfBC$ symmetric functions}

From the plethystic power sum symmetric functions
$  \{ p_\lambda^+ \,|\, \lambda \vdash k\}
  \cup
  \{ p_\lambda^- \,|\, \lambda \vdash k\}
  \subseteq \Lambda_k(x,y)$,
one can define other plethystic
symmetric functions~\cite[\S 3]{RemTrans}, \cite{StemOrdRep}
\begin{equation*}
  \begin{gathered}
  \{s_\lambda^+ \,|\, \lambda \vdash k\}, \quad
  \{h_\lambda^+ \,|\, \lambda \vdash k\}, \quad
  \{e_\lambda^+ \,|\, \lambda \vdash k\}, \quad
  \{m_\lambda^+ \,|\, \lambda \vdash k\}, \quad
  \{f_\lambda^+ \,|\, \lambda \vdash k\},\\ \quad
  \{s_\lambda^- \,|\, \lambda \vdash k\}, \quad
  \{h_\lambda^- \,|\, \lambda \vdash k\}, \quad
  \{e_\lambda^- \,|\, \lambda \vdash k\}, \quad
  \{m_\lambda^- \,|\, \lambda \vdash k\}, \quad
  \{f_\lambda^- \,|\, \lambda \vdash k\}
  \end{gathered}
\end{equation*}
to be those symmetric functions in $\Lambda_k(x,y)$ related to
$\{p_\lambda^+ \,|\, \lambda \vdash k \}$
or
$\{p_\lambda^- \,|\, \lambda \vdash k \}$
just as
$\{s_\lambda \,|\, \lambda \vdash k\}$,
$\{h_\lambda \,|\, \lambda \vdash k\}$,
$\{e_\lambda \,|\, \lambda \vdash k\}$,
$\{m_\lambda \,|\, \lambda \vdash k\}$,
$\{f_\lambda \,|\, \lambda \vdash k\}$
in $\Lambda_k(x)$ are related to
$\{p_\lambda \,|\, \lambda \vdash k \}$.
Such functions often appear in the literature as
$s_\lambda[X+Y], \dotsc, f_\lambda[X+Y]$
and $s_\lambda[X-Y],\dotsc,f_\lambda[X-Y]$.
Certain products of pairs of these
form nine more {\em plethystic} bases of the space $\Lambda_n(x,y)$:
\begin{equation}\label{eq:bcpbases}
  \begin{gathered}
  \{s_\lambda^+s_\mu^- \,|\, (\lambda,\mu) \vdash n \}\\
  \{h_\lambda^+h_\mu^- \,|\, (\lambda,\mu) \vdash n \}, \quad
  \{h_\lambda^+e_\mu^- \,|\, (\lambda,\mu) \vdash n \},\quad
  \{e_\lambda^+h_\mu^- \,|\, (\lambda,\mu) \vdash n \},\quad
  \{e_\lambda^+e_\mu^- \,|\, (\lambda,\mu) \vdash n \},\\
  \{m_\lambda^+m_\mu^- \,|\, (\lambda,\mu) \vdash n \},\quad
  \{m_\lambda^+f_\mu^- \,|\, (\lambda,\mu) \vdash n \},\quad
  \{f_\lambda^+m_\mu^- \,|\, (\lambda,\mu) \vdash n \},\quad
  \{f_\lambda^+f_\mu^- \,|\, (\lambda,\mu) \vdash n \}.
  \end{gathered}
\end{equation}
Naturally, one may study $\Lambda_n(x,y)$ in terms of these bases
and
\begin{equation*}
  \{ p_\lambda^+ p_\mu^- \,|\, (\lambda,\mu) \vdash n\}, \quad
  \{ (pp)_{\lambda,\mu} \,|\, (\lambda,\mu) \vdash n\},
  \end{equation*}
instead of using the eleven bases in (\ref{eq:bcnpbases}) -- (\ref{eq:pbasis}).
It is straightforward to show
that matrices relating bases of the forms
$\{ o_\lambda^+g_\mu^- \,|\, (\lambda,\mu) \vdash n \}$
to
$\{ (pp)_{\lambda,\mu} \,|\, (\lambda,\mu) \vdash n \}$
are the same as those relating bases of the forms
$\{ (og)_{\lambda,\mu} \,|\, (\lambda,\mu) \vdash n \}$
to
$\{ p_\lambda^+p_\mu^-/2^{\ell(\lambda)+\ell(\mu)} \,|\, (\lambda,\mu) \vdash n \}$.
Formulas for the matrix entries are given in \cite[App.\,A]{RemTrans}.

A correspondence between the plethystic bases and those of the trace space
$\trsp(\hbnq)$ is given explicitly by the ($q$-extension of the)
{\em nonplethystic} $\msfBC$-Frobenius map
\cite[\S 3]{RemTrans}, \cite{StemOrdRep}
\begin{align}
  \npfrobch_q\ntksp:
  \trsp(\hbnq) &\rightarrow \Lambda_n(x,y) \label{eq:afrobchaaer0}\\
  \npfrobch_q(\theta_q) &= \frac1{2^nn!}
  \ntnsp \sum_{w \in \bn} \ntksp 2^{\ell(\alpha(w)) + \ell(\beta(w))} \theta(w)
  (pp)_{\alpha(w),\beta(w)} \label{eq:afrobchaaer1} \\
  &= \nTksp \sum_{(\lambda,\mu) \vdash n} \ntnsp 
  \frac{\theta(\lambda,\mu)}{z_\lambda z_\mu}
  (pp)_{\lambda,\mu}, \label{eq:afrobchaaer2}
\end{align}
analogous to (\ref{eq:frobchaaer0}) -- (\ref{eq:frobchaaer2}).
Specifically, $\npfrobch_q$ maps
\begin{equation*}
  \begin{gathered}
  \quad (\epsilon\epsilon)^{\lambda,\mu}_q \mapsto e_\lambda^+e_\mu^-,
  \quad (\epsilon\eta)^{\lambda,\mu}_q \mapsto e_\lambda^+h_\mu^-,
  \quad (\eta\epsilon)^{\lambda,\mu}_q \mapsto h_\lambda^+e_\mu^-,
  \quad (\eta\eta)^{\lambda,\mu}_q \mapsto h_\lambda^+h_\mu^-,\\
  (\phi\phi)^{\lambda,\mu}_q \mapsto m_\lambda^+m_\mu^-,
  \quad (\phi\gamma)^{\lambda,\mu}_q \mapsto m_\lambda^+f_\mu^-,
  \quad (\gamma\phi)^{\lambda,\mu}_q \mapsto f_\lambda^+m_\mu^-,
  \quad (\gamma\gamma)^{\lambda,\mu}_q \mapsto f_\lambda^+f_\mu^-,\\
  (\chi\chi)^{\lambda,\mu}_q \mapsto s_\lambda^+s_\mu^-,
  \quad (\psi\psi)^{\lambda,\mu}_q \mapsto p_\lambda^+p_\mu^-,
  \quad \iota^{\lambda,\mu}_q \mapsto 2^{\ell(\lambda)+\ell(\mu)}(pp)_{\lambda,\mu}.
  \end{gathered}
\end{equation*}
Defining a trace generating function in $\zqq \otimes \Lambda_n(x,y)$
in terms of characters $(\epsilon\epsilon)_q^{\lambda,\mu}$, plethystic monomial
symmetric functions $m_\lambda^+m_\mu^-$, and elements $D \in \hbnq$, we have
expansions analogous to those in Proposition~\ref{p:YBCexpansions},
\begin{equation*}
  \begin{aligned}
&\sum_{(\lambda,\mu)\vdash n} (\epsilon\epsilon)_q^{\lambda,\mu}(D) m_\lambda^+m_\mu^-
  = \sum_{(\lambda,\mu)\vdash n} (\epsilon\eta)_q^{\lambda,\mu}(D) m_\lambda^+f_\mu^-
  = \cdots
  = \sum_{(\lambda,\mu)\vdash n} (\gamma\gamma)_q^{\lambda,\mu}(D) h_\lambda^+h_\mu^-\\
  &= \sum_{(\lambda,\mu)\vdash n}
  \frac{(-1)^{\ell(\lambda)+\ell(\mu)}(\psi\psi)_q^{\lambda,\mu}(D)}{z_\lambda z_\mu}
  p^+_\lambda p^-_\mu  
  = \sum_{(\lambda,\mu)\vdash n}
  (-1)^{\ell(\lambda)+\ell(\mu)} \iota_q^{\lambda,\mu}(D)
  (pp)_{\lambda,\mu}.
  \end{aligned}
  \end{equation*}
While one could define $Y^{\msfBC}_q(D)$ to be the above symmetric function
instead of
that in (\ref{eq:YBCdefn}), this would lead to a less natural
connection to the
$\msfBC$-chromatic symmetric
function (\ref{eq:XBCGdef}).

\section{Hessenberg varieties}\label{s:hess}


Hessenberg varieties are subvarieties of flag varieties which were first
studied~\cite{DPSHessenberg}, \cite{DSGenEuler}
in conjunction with questions
concerning eigenvalues of linear operators.
More recent work reveals connections to other varieties,
representation theory, and combinatorics.

The
Hessenberg varieties of Coxeter types $\msfA$, $\msfBB$, $\sfC$
are parametrized by
certain vector spaces called Hessenberg spaces. Specifically,
given a reducive algebraic group
$\mathrm G$, Borel subgroup $\mathrm B$,
and Lie algebras $\mathfrak g$, $\mathfrak b$ of these,
call a subspace
$\HH \subseteq \mathfrak g$ a {\em Hessenberg space}
if it satisfies
the Lie algebra containment conditions
\begin{equation}\label{eq:contcond}
  \mathfrak b \subseteq \HH, \qquad
  \big[ \mathfrak b, \HH \big] \subseteq \HH.
\end{equation}
Hessenberg spaces may be parametrized by appropriate sets of
roots in root systems.
(See \cite{Hum2}
for definitions.)
Let $\Phi$ be the root system of type $\msfA$, $\msfBB$ or $\msfC$,
let $\Delta \subset \Phi$ be a simple system, let $\Phi^- \subseteq \Phi$
be the set of negative roots,
and for $\gamma \in \Phi$ let $\mathfrak g_\gamma$
be the root space of $\mathfrak g$ corresponding to $\gamma$.
Define the {\em root poset} on $\Phi$ by
$\alpha \leq_\Phi \beta$ if $\beta - \alpha \in \spn_{\mathbb N} \Delta$,
and define the {\em negative root poset} to be the subposet of this
induced by $\Phi^-$.
Call subset $I \subseteq \Phi^-$ a {\em dual order ideal} of $\Phi^-$ if
for all $\alpha, \beta \in \Phi^-$ we have
\begin{equation*}
  \begin{matrix}
    \alpha \in I, \\ \alpha \leq_{\Phi^-} \beta
  \end{matrix}
  \bigg\} \Rightarrow \beta \in I.
\end{equation*}
In particular, we have the following~\cite[Lem.\,1]{DPSHessenberg}.
\begin{prop}
  A bijection between Hessenberg spaces $\HH \subseteq \mathfrak g$
  and dual order ideals $I$ of the
  negative root poset on $\Phi^-$
  is given by
\begin{equation}\label{eq:hessideal}
  I \mapsto \HH(I) \defeq \mathfrak b \oplus
  \bigoplus_{\gamma \in I} \mathfrak g_\gamma.
\end{equation}
\end{prop}

For Hessenberg space $\HH$ and matrix $S \in \mathfrak g$,
define a subvariety of the flag variety $\mathrm G/\mathrm B$
by
\begin{equation}\label{eq:hessvardef}
  \hess(\HH) = \hess(\HH,S) \defeq
  \{ g\mathrm B \in \mathrm G/\mathrm B \,|\, g^{-1}\ntnsp Sg \in \HH \},
\end{equation}
and
call this a
{\em Hessenberg variety associated to $\HH$}.
If $S$
is a regular semisimple element of $\mathfrak g$,
then
call $\hess(\HH)$ a {\em regular semisimple Hessenberg variety}.
In this case, its
cohomology
vanishes in odd degree~\cite[\S 3]{DPSHessenberg},
\begin{equation}\label{eq:cohvanish}
  \coh^*(\mathcal \hess(\HH)) = \bigoplus_{j\geq0} \coh^{2j}(\hess(\HH)).
\end{equation}
(See also \cite{PrecupAffine}, \cite{Ty1}.)
Tymozcko~\cite{Ty3}, \cite{Ty2} defined a graded $W$-module sructure
on (\ref{eq:cohvanish}),
where $W$ is the Weyl group of $G$.
Let
\begin{equation}\label{eq:frobcharhess}
  \mathrm{ch}(\coh^{2j}(\hess(\HH)))
\end{equation}
be the Frobenius characteristic of the character
of the submodule $\coh^{2j}(\hess(\HH))$.
For regular semisimple Hessenberg varieties of type $\msfA$,
the
Frobenius characteristics
(\ref{eq:frobcharhess})
are closely related to trace evaluations and graph coloring.
In types $\msfBB$ and $\msfC$,
no such relation
is known.

\ssec{Regular semisimple Hessenberg varieties of type $\msfA$}

Define type-$\msfA$ Hessenberg spaces
as in (\ref{eq:contcond})
with
$\mathrm G = \mathrm{GL}_n(\mathbb C)$,
$\mathfrak g = \mathfrak{gl}_n(\mathbb C)$.
In addition to the bijection (\ref{eq:hessideal}) with dual order ideals,
we have the following bijection with codominant
elements of $\sn$.
\begin{prop}\label{p:hessparambyw}
  A bijective correspondence between
  $312$-avoiding permutations in $\sn$ and
  Hessenberg spaces in
  $\mathfrak{gl}_n(\mathbb C)$
  is given by
\begin{equation}\label{eq:hspace}
  w \mapsto \HH(w) \defeq \{ A = (a_{i,j}) \in \mathfrak{gl}_n(\mathbb C) \,|\, a_{i,j} = 0
  \text{ for all } j > \max(w_1,\dotsc,w_i) \}.
\end{equation}
\end{prop}
\begin{proof}
  Given $w \in \sn$ avoiding the pattern $312$,
  define
  $m_i = \max(w_1,\dotsc,w_i)$ for $i=1,\dotsc,n$.
  By Theorem~\ref{t:312bij}
  the map $w \mapsto m_1 \cdots m_n$ is bijective,
  and
  it is easy to see that $m_1 \cdots m_n$
  satisfies the defining conditions of
  a {\em Hessenberg function}:
  \begin{enumerate}
\item $i \leq m_i \leq n$ for $i = 1,\dotsc,n$,
\item $m_1 \leq \cdots \leq m_n$.
\end{enumerate}
  Thus the $\frac1{n+1}\tbinom{2n}{n}$
  spaces $\HH(w)$ in (\ref{eq:hspace})
  are precisely the Hessenberg spaces
  usually denoted by $\HH(m_1 \cdots m_n)$ in the literature.
  (See, e.g., \cite[Eq.\.(2.2)]{ADGHGeoHess}.)
\end{proof}

Let $\ahess(\HH)$ denote
a
type-$\msfA$
regular semisimple Hessenberg variety,
defined as in (\ref{eq:hessvardef}) with $\HH$ as above
and $S = \diag(\lambda_1,\dotsc,\lambda_n)$,
where $\lambda_1,\dotsc,\lambda_n$ are distinct.
For $w \in \sn$ avoiding the pattern $312$, define the generating function
\begin{equation}\label{eq:gradedchA}
  \mathrm{Fr}_q^{\msfA}(\HH(w)) \defeq
  \sum_{j=0}^{\ell(w)} \mathrm{ch}(\coh^{2j}(\ahess(\HH(w))))q^j
\end{equation}
for the type-$\msfA$ Frobenius characteristics
(\ref{eq:frobcharhess}).
Shareshian and Wachs conjectured
that $\omega \mathrm{Fr}_q^{\msfA}(\HH(w))$
is equal to a chromatic symmetric function,
viewed as a polynomial in $q$ with coefficients
in $\Lambda_n$~\cite[Conj.\,10.1]{SWachsChromQF}.
This was first proved by Brosnan--Chow~\cite{BChowUIODot}.
(See also \cite{ANigroGeoAppro}, \cite{GPSecondPf},
\cite{KiemLeeBirational}, \cite{KiemLeeTwin}.)
Combining the equality with Theorem~\ref{t:YequalsX}
we obtain
the following.
\begin{thm}\label{t:tymXY}
  For $w \in \sn$ avoiding the pattern $312$
  and unit interval order $P = P(w)$,
  we have
  \begin{equation}\label{eq:tymXY}
    \mathrm{Fr}_q^{\msfA}(\HH(w))
    = \omega X_{\inc(P),q}
    = \omega Y_q(\wtc wq).
  \end{equation}
\end{thm}
We can see in three ways that the function in (\ref{eq:tymXY})
belongs to $\spn_{\mathbb N[q]} \{ s_\lambda \,|\, \lambda \vdash n \}$:
\begin{enumerate}
\item The character of
  $\coh^{2j}(\ahess(\HH(w)))$
  belongs to $\spn_{\mathbb N}\{ \chi^\lambda \,|\, \lambda \vdash n \}$.
  Thus its Frobenius characteristic (\ref{eq:frobchs})
  belongs to
  $\spn_{\mathbb N}\{ s_\lambda \,|\, \lambda \vdash n \}$.
\item By \cite[Thm.\,6.3]{SWachsChromQF} the coefficient of
  $s_\lambda$ in $X_{\inc(P),q}$ belongs to $\mathbb N[q]$.
\item By \cite[Lem.\,1.1]{HaimanHecke} we have
  $\chi_q^\lambda(\wtc wq) \in \mathbb N[q]$.
\end{enumerate}
Furthermore, the function (\ref{eq:tymXY}) is conjectured
to belong to $\spn_{\mathbb N[q]} \{ h_\lambda \,|\, \lambda \vdash n \}$.
This open problem also can be viewed in three ways.
\begin{prob}
  Prove one of the following.
  \begin{enumerate}
  \item \cite[Conj.\,10.4]{SWachsChromQF}
    For each
    permutation $w \in \sn$ avoiding the pattern $312$
    and index $j = 0,\dotsc,\ell(w)$,
    the $\sn$-module $\coh^{2j}(\hess^{\msfA}(\HH(w)))$
    is a permutation module in which each point stabilizer is a Young subgroup.
  \item 
    \cite[Conj.\,5.1]{SWachsChromQF}
    For each unit interval order $P$ labeled as in
    Algorithm~\ref{a:ptow},
    the function $X_{\inc(P),q}$ belongs to
    $\spn_{\mathbb N[q]} \{ e_\lambda \,|\, \lambda \vdash n \}$.
    (See also \cite[Conj.~5.5]{StanStemIJT}.)
  \item 
    \cite[Conj.\,2.1]{HaimanHecke}
    For each permutation $w \in \sn$ and each partition $\lambda \vdash n$,
    we have $\phi_q^\lambda(\wtc wq) \in \mathbb N[q]$.
    (See also \cite[Conj.~2.1]{StemConj}.)  
  \end{enumerate}
\end{prob}
The statements above satisfy the implications
$(1) \Leftrightarrow (2) \Leftarrow (3)$.
For progress on these problems, see, e.g., \cite{ANigroSplit},
\cite{HPTPermutationBases}, \cite{HikitaProof},
\cite{RSEschers}, and references
cited in \cite[\S 3.5]{SkanCCS}.

\ssec{Regular semisimple Hessenberg varieties of types $\msfBB$ and $\msfC$}

Define type-$\msfBB$ Hessenberg spaces
as in (\ref{eq:contcond}) with
$\mathrm G = \mathrm{SO}_{2n+1}(\mathbb C)$,
$\mathfrak g = \mathfrak{so}_{2n+1}(\mathbb C)$;
define type-$\msfC$ Hessenberg spaces
similarly
with
$\mathrm G = \mathrm{SP}_{2n}(\mathbb C)$,
$\mathfrak g = \mathfrak{sp}_{2n}(\mathbb C)$.
Let $\mathcal M^{\msfBB}_n$,
$\mathcal M^{\msfC}_n$
denote the sets of these spaces, respectively.
The cardinalities
of these sets are equal to
the number of order ideals appearing in (\ref{eq:hessideal}).
By \cite[Thm.\,3.1]{CelliniPapi}, 
this is $\tbinom{2n}{n}$.
Thus neither collection of spaces
corresponds bijectively to the $\smash{\tfrac{1}{n+2}\tbinom{2n+2}{n+1}}$
``codominant'' elements of $\bn$ \avoidingsignedp.
This suggests the following problem.
(See also the type-$\msfBB$ and $\msfC$
Hessenberg functions defined in
\cite[\S 10]{AHMMSHess}.)
\bp
State a type-$\msfBB$ or $\msfC$ analog of
Proposition~\ref{p:hessparambyw}
in terms of a subset of $\tbinom{2n}n$ elements of $\bn$.
\ep
\noindent


Let $\bhess(\HH)$ denote the
type-$\msfBB$
regular semisimple Hessenberg variety
defined as in (\ref{eq:hessvardef})
with $\HH \in \mathcal M^{\msfBB}_n$
and $S \in \mathfrak{so}_{2n+1}(\mathbb C)$ having distinct eigenvalues
$(0, \lambda_1, \dotsc, \lambda_n, -\lambda_1, \dotsc, -\lambda_n)$.
Similarly, let $\chess(\HH)$ denote the
type-$\msfC$
regular semisimple Hessenberg variety
defined as in (\ref{eq:hessvardef})
with $\HH \in \mathcal M^{\msfC}_n$
and $S \in \mathfrak{sp}_{2n}(\mathbb C)$ having distinct eigenvalues
$(\lambda_1, \dotsc, \lambda_n, -\lambda_1, \dotsc, -\lambda_n)$.

In analogy to (\ref{eq:gradedchA}) we 
define the generating functions
\begin{equation}\label{eq:gradedchBC}
  \mathrm{Fr}_q^{\msfBB}(\HH)
  \defeq
  \sum_j \ntnsp \mathrm{ch}(\coh^{2j}(\bhess(\HH)))q^j,
  \quad \ 
  \mathrm{Fr}_q^{\sfC}(\HH)
  \defeq
  \sum_j \ntnsp \mathrm{ch}(\coh^{2j}(\chess(\HH)))q^j.
\end{equation}

It would be interesting
to connect the symmetric functions
$\mathrm{Fr}_q^{\msfBB}(\HH)$, $\mathrm{Fr}_q^{\msfC}(\HH)$
to type-$\msfBC$ chromatic symmetric functions and
type-$\msfBC$ trace generating functions,
i.e., to formulate type-$\msfBB$ and $\msfC$ analogs
of Theorem~\ref{t:tymXY}.
By Theorem~\ref{t:bcxy}
the $q=1$ specializations of the
chromatic symmetric functions and trace generating functions
are equal;
general equality is stated as
Problem~\ref{p:qanalogofbcchrom}.
One could also consider
interpreting
$\mathrm{Fr}_q^{\msfBB}(\HH)$,
$\mathrm{Fr}_q^{\msfC}(\HH)$ as
trace generating functions.
By Proposition~\ref{p:everysymmfnbcq},
such interpretations must exist.
\bp\label{p:hesstrace}
Find families
$\{ D^{\msfBB}_\HH \,|\, \HH \in \mathcal M^{\msfBB}_n \}$,
$\{ D^{\msfC}_\HH \,|\, \HH \in \mathcal M^{\msfC}_n \}$ of
elements of
$\hbnq$
whose trace generating functions satisfy
\begin{enumerate}
  \item $\omega \YBCq(D^{\msfBB}_\HH) = \mathrm{Fr}_q^{\msfBB}(\HH)$ for all
    $\HH \in \mathcal M^{\msfBB}_n$,
  \item $\omega \YBCq(D^{\msfC}_\HH) = \mathrm{Fr}_q^{\msfC}(\HH)$ for all
    $\HH \in \mathcal M^{\msfC}_n$.
\end{enumerate}    
\ep
\noindent
Similarly, one could try to interpret
$\mathrm{Fr}_q^{\msfBB}(\HH)$, $\mathrm{Fr}_q^{\sfC}(\HH)$
in terms of graph coloring~\cite{ShareshianComm18}.

\bp\label{p:hesschrom}
Define families
$\mathcal G^{\msfBB} =\{ G^{\msfBB}(\HH) \,|\, \HH \in \mathcal M^{\msfBB}_n \}$,
$\mathcal G^{\msfC} =\{ G^{\msfC}(\HH) \,|\, \HH \in \mathcal M^{\msfC}_n \}$
of graphs and families
$\{ X_{G,q}^{\msfBB} \,|\, G \in \mathcal G^{\msfBB} \}$,
$\{ X_{G,q}^{\msfC} \,|\, G \in \mathcal G^{\msfC} \}$
of chromatic symmetric functions
so that we have
\begin{enumerate}
  \item $\omega X^{\msfBB}_{G^{\msfBB}(\HH),q} = \mathrm{Fr}_q^{\msfBB}(\HH)$ for all
    $\HH \in \mathcal M^{\msfBB}_n$,
  \item $\omega X^{\msfC}_{G^{\msfC}(\HH),q} = \mathrm{Fr}_q^{\msfC}(\HH)$ for all
    $\HH \in \mathcal M^{\msfC}_n$.
\end{enumerate}    
\ep
\noindent
In addition,
it would
be interesting to connect these problems to
the Kazhdan--Lusztig basis of $\hbnq$ and to the type-$\msfBC$
indifference graphs defined in Subsection~\ref{ss:bcincg}.
\begin{prob}
  Decide to what extent the symmetric functions in
  Problems~\ref{p:hesstrace} -- \ref{p:hesschrom}
  can be chosen to simultaneously solve Problem~\ref{p:qanalogofbcchrom}.
\end{prob}
Finally, it would be interesting to study
appropriate $q$-extensions of symmetric functions $X^{\msfBC}_{\inc(Q)}$
(as in Problem~\ref{p:qanalogofbcchrom})
in conjunction with
type-$\msfBB$ and $\msfC$ Lusztig varieties
in addition to Hessenberg varieties.
(See \cite{ANigroGeoAppro}.)

\section{Acknowledgements}

The author is grateful to
Jonathan Boretsky,
Patrick Brosnan,
Mahir Can,
Tim Chow,
Theodossios Douvropoulos,
Emma Gray,
Sean Griffin,
Angela Hicks,
Jongwon Kim,
Nate Lesnevich,
Diep Luong,
Jeremy Martin,
Nicholas Mayers,
Antonio Nigro,
Rosa Orellana,
Tommy Parisi,
Martha Precup,
Arun Ram,
Bruce Sagan,
George Seelinger,
John Shareshian,
Eric Sommers,
John Stembridge,
Michelle Wachs,
Jiayuan Wang,
and
Andy Wilson
for helpful conversations,
and to anonymous reviewers for valuable suggestions.

\bibliography{../skan}

\begin{thebibliography}{10}

\bibitem{ADGHGeoHess}
{\sc H.~Abe, L.~DeDieu, F.~Galetto, and M.~Harada}.
\newblock Geometry of {H}essenberg varieties with applications to
  {N}ewton--{O}kounkov bodies.
\newblock {\em Selecta Math. (N.S.)\/}, {\bf 24}, 3 (2018) pp. 2129--2163.

\bibitem{AHMMSHess}
{\sc T.~Abe, T.~Horiguchi, M.~Masuda, S.~Murai, and T.~Sato}.
\newblock Hessenberg varieties and hyperplane arrangements.
\newblock {\em J. Reine Angew. Math.\/}, {\bf 764} (2020) pp. 241--286.

\bibitem{ANigroSplit}
{\sc A.~Abreu and A.~Nigro}.
\newblock Splitting the cohomology of {Hessenberg} varieties and $e$-positivity
  of chromatic symmetric functions (2023).
\newblock Preprint {\tt math.AG/2304.10644} on {ArXiv}.

\bibitem{ANigroUpdate}
{\sc A.~Abreu and A.~Nigro}.
\newblock An update on {H}aiman's conjectures.
\newblock {\em Forum Math. Sigma\/}, {\bf 12} (2024).
\newblock Paper No. e86, 15.

\bibitem{ANigroGeoAppro}
{\sc A.~Abreu and A.~Nigro}.
\newblock A geometric approach to characters of {H}ecke algebras.
\newblock {\em J. Reine Angew. Math.\/}, {\bf 821} (2025) pp. 53--114.

\bibitem{AAERCharFormulasHyper}
{\sc R.~Adin, C.~Athanasiadis, S.~Elizalde, and Y.~Roichman}.
\newblock Character formulas and descents for the hyperoctahedral group.
\newblock {\em Adv. in Appl. Math.\/}, {\bf 87} (2017) pp. 128--169.

\bibitem{AFultonDLoci}
{\sc D.~Anderson and W.~Fulton}.
\newblock Degeneracy loci, pfaffians, and vexillary signed permutations (2012).
\newblock Preprint {\tt math.AG:1210.2066} on {ArXiv}.

\bibitem{AFultonVexRevisited}
{\sc D.~Anderson and W.~Fulton}.
\newblock Vexillary signed permutations revisited.
\newblock {\em Algebr. Comb.\/}, {\bf 3}, 5 (2020) pp. 1041--1057.

\bibitem{AthanPSE}
{\sc C.~Athanasiadis}.
\newblock Power sum expansion of chromatic quasisymmetric functions.
\newblock {\em Electron. J. Combin.\/}, {\bf 22}, 2 (2015).
\newblock Paper 2.7, 9 pages.

\bibitem{RemTrans}
{\sc D.~Beck, J.~Remmel, and T.~Whitehead}.
\newblock The combinatorics of transition matrices between the bases of the
  symmetric functions and the {$B_n$} analogues.
\newblock {\em Discrete Math.\/}, {\bf 153} (1996) pp. 3--27.

\bibitem{BilleyPattern}
{\sc S.~Billey}.
\newblock Pattern avoidance and rational smoothness of {Schubert} varieties.
\newblock {\em Adv. Math.\/}, {\bf 139} (1998) pp. 141--156.

\bibitem{BilleyLak}
{\sc S.~Billey and V.~Lakshmibai}.
\newblock {\em Singular loci of {S}chubert varieties\/}, vol. 182 of {\em
  Progress in Mathematics\/}.
\newblock Birkh\"auser Boston Inc., Boston, MA (2000).

\bibitem{BilleyLamVex}
{\sc S.~Billey and T.~K. Lam}.
\newblock Vexillary elements in the hyperoctahedral group.
\newblock {\em J. Algebraic Combin.\/}, {\bf 8}, 2 (1998) pp. 139--152.

\bibitem{BWHex}
{\sc S.~Billey and G.~Warrington}.
\newblock {Kazhdan--Lusztig} polynomials for $321$-hexagon-avoiding
  permutations.
\newblock {\em J. Algebraic Combin.\/}, {\bf 13}, 2 (2001) pp. 111--136.

\bibitem{BBCoxeter}
{\sc A.~Bj{\"o}rner and F.~Brenti}.
\newblock {\em Combinatorics of {C}oxeter groups\/}, vol. 231 of {\em Graduate
  Texts in Mathmatics\/}.
\newblock Springer, New York (2005).

\bibitem{BChowUIODot}
{\sc P.~Brosnan and T.~Chow}.
\newblock Unit interval orders and the dot action on the cohomology of regular
  semisimple {H}essenberg varieties.
\newblock {\em Adv. Math.\/}, {\bf 329} (2018) pp. 955--1001.

\bibitem{CelliniPapi}
{\sc P.~Cellini and P.~Papi}.
\newblock Ad-nilpotent ideals of a {Borel} subalgebra.
\newblock {\em J.\ Algebra\/}, {\bf 225} (2000) pp. 130--141.

\bibitem{CHSSkanEKL}
{\sc S.~Clearman, M.~Hyatt, B.~Shelton, and M.~Skandera}.
\newblock Evaluations of {H}ecke algebra traces at {K}azhdan-{L}usztig basis
  elements.
\newblock {\em Electron. J. Combin.\/}, {\bf 23}, 2 (2016).
\newblock Paper 2.7, 56 pages.

\bibitem{CSkanTNNChar}
{\sc A.~Clearwater and M.~Skandera}.
\newblock Total nonnegativity and {H}ecke algebra trace evaluations.
\newblock {\em Ann.\ Combin.\/}, {\bf 25} (2021) pp. 757--787.

\bibitem{CMRIndexSpec}
{\sc V.~Coll, Jr., N.~Mayers, and N.~Russoniello}.
\newblock The index and spectrum of {L}ie poset algebras of types {B}, {C}, and
  {D}.
\newblock {\em Electron. J. Combin.\/}, {\bf 28}, 3 (2021).
\newblock Paper 3.47, 23 pages.

\bibitem{CsarThesis}
{\sc S.~Csar}.
\newblock {\em Root and weight semigroup rings for signed posets\/}.
\newblock Ph.D. thesis, University of Minnesota, Minneapolis, MN (2014).

\bibitem{DPSHessenberg}
{\sc F.~De~Mari, C.~Procesi, and M.~A. Shayman}.
\newblock Hessenberg varieties.
\newblock {\em Trans. Amer. Math. Soc.\/}, {\bf 332}, 2 (1992) pp. 529--534.

\bibitem{DSGenEuler}
{\sc F.~De~Mari and M.~A. Shayman}.
\newblock Generalized {E}ulerian numbers and the topology of the {H}essenberg
  variety of a matrix.
\newblock {\em Acta Appl. Math.\/}, {\bf 12}, 3 (1988) pp. 213--235.

\bibitem{Deodhar90}
{\sc V.~Deodhar}.
\newblock A combinatorial setting for questions in {Kazhdan--Lusztig} theory.
\newblock {\em Geom. Dedicata\/}, {\bf 36}, 1 (1990) pp. 95--119.

\bibitem{DJRepHeckeB}
{\sc R.~Dipper and G.~James}.
\newblock Representations of {H}ecke algebras of type {$B_n$}.
\newblock {\em J.\ Algebra\/}, {\bf 146} (1992) pp. 454--481.

\bibitem{Ehresmann}
{\sc C.~Ehresmann}.
\newblock Sur la topologie de certains espaces {homog\`enes}.
\newblock {\em Ann. Math.\/}, {\bf 35} (1934) pp. 187--198.

\bibitem{SFischerThesis}
{\sc S.~Fischer}.
\newblock {\em Signed poset homology and $q$-analog {M\"obius} functions\/}.
\newblock Ph.D. thesis, University of Michigan, Ann Arbor, MI (1993).

\bibitem{Fish}
{\sc P.~C. Fishburn}.
\newblock {\em {Interval Graphs and Interval Orders}\/}.
\newblock Wiley, New York (1985).

\bibitem{GPCharHecke}
{\sc M.~Geck and G.~Pfeiffer}.
\newblock {\em Characters of finite {C}oxeter groups and {I}wahori-{H}ecke
  algebras\/}, vol.~21 of {\em London Mathematical Society Monographs. New
  Series\/}.
\newblock The Clarendon Press Oxford University Press, New York (2000).

\bibitem{GV}
{\sc I.~Gessel and G.~Viennot}.
\newblock Determinants and plane partitions (1989).
\newblock Preprint.

\bibitem{GJMaster}
{\sc I.~Goulden and D.~Jackson}.
\newblock Immanants, {S}chur functions, and the {M}ac{M}ahon master theorem.
\newblock {\em Proc. Amer. Math. Soc.\/}, {\bf 115}, 3 (1992) pp. 605--612.

\bibitem{GPSecondPf}
{\sc M.~Guay-Paquet}.
\newblock A second proof of the {Shareshian--Wachs} conjecture, by way of a new
  {Hopf} algebra (2016).
\newblock Preprint {\tt math.CO/1601.05498} on {ArXiv}.

\bibitem{HaimanHecke}
{\sc M.~Haiman}.
\newblock Hecke algebra characters and immanant conjectures.
\newblock {\em J. Amer. Math. Soc.\/}, {\bf 6}, 3 (1993) pp. 569--595.

\bibitem{HPTPermutationBases}
{\sc M.~Harada, M.~Precup, and J.~Tymoczko}.
\newblock Toward permutation bases in the equivariant cohomology rings of
  regular semisimple {H}essenberg varieties.
\newblock {\em Matematica\/}, {\bf 1}, 1 (2022) pp. 263--316.

\bibitem{HararyNBSG}
{\sc F.~Harary}.
\newblock On the notion of balance of a signed graph.
\newblock {\em Michigan Math J.\/}, {\bf 2}, 2 (1953/1954) pp. 143--146.

\bibitem{HikitaProof}
{\sc T.~Hikita}.
\newblock A proof of the {Stanley}--{Stembridge} conjecture (2024).
\newblock Preprint {\tt math.CO/2410.12758} on {ArXiv}.

\bibitem{HoefsmitThesis}
{\sc P.~Hoefsmit}.
\newblock {\em Representations of Hecke algebras of finite groups with BN-pairs
  of classical type\/}.
\newblock ProQuest LLC, Ann Arbor, MI (1974).
\newblock Thesis (Ph.D.)--The University of British Columbia (Canada).

\bibitem{Hum2}
{\sc J.~Humphreys}.
\newblock {\em {Introduction to Lie Algebras and Representation Theory}\/}.
\newblock No.~9 in GTM. Springer-Verlag, New York (1972).

\bibitem{KLSBasesQMBIndSgn}
{\sc R.~Kaliszewski, J.~Lambright, and M.~Skandera}.
\newblock Bases of the quantum matrix bialgebra and induced sign characters of
  the {H}ecke algebra.
\newblock {\em J. Algebraic Combin.\/}, {\bf 49}, 4 (2019) pp. 475--505.

\bibitem{KMG}
{\sc S.~Karlin and J.~{McGregor}}.
\newblock Coincidence probabilities.
\newblock {\em Pacific J. Math.\/}, {\bf 9} (1959) pp. 1141--1164.

\bibitem{KLRepCH}
{\sc D.~Kazhdan and G.~Lusztig}.
\newblock {Representations of Coxeter groups and Hecke algebras}.
\newblock {\em Invent. Math.\/}, {\bf 53} (1979) pp. 165--184.

\bibitem{KLSchub}
{\sc D.~Kazhdan and G.~Lusztig}.
\newblock {Schubert varieties and Poincar\'e duality}.
\newblock {\em Proc. Symp. Pure. Math., A.M.S.\/}, {\bf 36} (1980) pp.
  185--203.

\bibitem{KiemLeeBirational}
{\sc Y.-H. Kiem and D.~Lee}.
\newblock Birational geometry of generalized {H}essenberg varieties and the
  generalized {S}hareshian-{W}achs conjecture.
\newblock {\em J. Combin. Theory Ser. A\/}, {\bf 206} (2024).
\newblock Paper No. 105884, 45.

\bibitem{KiemLeeTwin}
{\sc Y.-H. Kiem and D.~Lee}.
\newblock Geometry of the twin manifolds of regular semisimple {H}essenberg
  varieties and unicellular {LLT} polynomials.
\newblock {\em Algebr. Comb.\/}, {\bf 7}, 3 (2024) pp. 861--885.

\bibitem{KSkanQGJ}
{\sc M.~Konvalinka and M.~Skandera}.
\newblock Generating functions for {H}ecke algebra characters.
\newblock {\em Canad.\ J.\ Math.\/}, {\bf 63}, 2 (2011) pp. 413--435.

\bibitem{KurTsuChrom}
{\sc M.~Kuroda and S.~Tsujie}.
\newblock Chromatic signed-symmetric functions of signed graphs (2021).
\newblock Preprint {\tt math.CO/2101.03018} on {ArXiv}.

\bibitem{LakSan}
{\sc V.~Lakshmibai and B.~Sandhya}.
\newblock Criterion for smoothness of {Schubert} varieties in {$SL(n)/B$}.
\newblock {\em Proc. Indian Acad. Sci. (Math Sci.)\/}, {\bf 100}, 1 (1990) pp.
  45--52.

\bibitem{LambertTheta}
{\sc J.~Lambert}.
\newblock Theta-vexillary signed permutations.
\newblock {\em Electron. J. Combin.\/}, {\bf 25}, 4 (2018).
\newblock Paper No. 4.53, 30.

\bibitem{LS2}
{\sc A.~Lascoux and M.-P. Sch\"utzenberger}.
\newblock {Schubert Polynomials and the Littlewood--Richardson Rule}.
\newblock {\em Letters in Math. Physics\/}, {\bf 10} (1985) pp. 111--124.

\bibitem{LinVrep}
{\sc B.~{Lindstr\"om}}.
\newblock On the vector representations of induced matroids.
\newblock {\em Bull. London Math. Soc.\/}, {\bf 5} (1973) pp. 85--90.

\bibitem{LittlewoodTGC}
{\sc D.~E. Littlewood}.
\newblock {\em The {T}heory of {G}roup {C}haracters and {M}atrix
  {R}epresentations of {G}roups\/}.
\newblock Oxford University Press, New York (1940).

\bibitem{M1}
{\sc I.~Macdonald}.
\newblock {\em {Symmetric Fuctions and Hall Polynomials}\/}.
\newblock Oxford University Press, Oxford (1979).

\bibitem{MerWatIneq}
{\sc R.~Merris and W.~Watkins}.
\newblock Inequalities and identities for generalized matrix functions.
\newblock {\em Linear Algebra Appl.\/}, {\bf 64} (1985) pp. 223--242.

\bibitem{PrecupAffine}
{\sc M.~Precup}.
\newblock Affine pavings of {H}essenberg varieties for semisimple groups.
\newblock {\em Selecta Math. (N.S.)\/}, {\bf 19}, 4 (2013) pp. 903--922.

\bibitem{ProctorBruhat}
{\sc R.~Proctor}.
\newblock Classical {Bruhat} orders and lexicographic shellability.
\newblock {\em J. Algebra\/}, {\bf 77} (1982) pp. 104--126.

\bibitem{ReinerParset}
{\sc V.~Reiner}.
\newblock Signed posets.
\newblock {\em J.\ Combin.\ Theory Ser.\ A\/}, {\bf 62}, 2 (1993) pp. 324--360.

\bibitem{RSEschers}
{\sc A.~Rok and A.~Szenes}.
\newblock {Eschers} and {Stanley's} chromatic $e$-positivity conjecture (2023).
\newblock Preprint {\tt math.CO/2305.00963} on {ArXiv}.

\bibitem{ShareshianComm18}
{\sc J.~Shareshian} (2018).
\newblock Personal communication.

\bibitem{SWachsChromQ}
{\sc J.~Shareshian and M.~Wachs}.
\newblock Chromatic quasisymmetric functions and {Hessenberg} varieties.
\newblock In {\em Configuration Spaces\/} ({\sc A.~Bjorner, F.~Cohen,
  C.~{De~Concini}, C.~Procesi, and M.~Salvetti}, eds.). Edizione Della Normale,
  Pisa (2012), pp. 433--460.

\bibitem{SWachsChromQF}
{\sc J.~Shareshian and M.~Wachs}.
\newblock Chromatic quasisymmetric functions.
\newblock {\em Adv. Math.\/}, {\bf 295} (2016) pp. 497--551.

\bibitem{SkanNNDCB}
{\sc M.~Skandera}.
\newblock On the dual canonical and {Kazhdan}--{Lusztig} bases and 3412,
  4231-avoiding permutations.
\newblock {\em J.\ Pure Appl.\ Algebra\/}, {\bf 212} (2008).

\bibitem{SkanCCS}
{\sc M.~Skandera}.
\newblock Characters and chromatic symmetric functions.
\newblock {\em Electron. J. Combin.\/}, {\bf 28}, 2 (2021).
\newblock Research Paper P2.19, 39 pages.

\bibitem{SkanGenFnWreath}
{\sc M.~Skandera}.
\newblock Generating functions for monomial characters of wreath products
  $\mathbb{Z}/d\mathbb{Z} \wr \mathfrak{S}_n$.
\newblock {\em Enum. Combin. Appl.\/}, {\bf 1}, 2 (2021).
\newblock Research Paper S2R10, 10 pages.

\bibitem{StanSymm}
{\sc R.~Stanley}.
\newblock {A symmetric function generalization of the chromatic polynomial of a
  graph}.
\newblock {\em Adv. Math.\/}, {\bf 111} (1995) pp. 166--194.

\bibitem{StanEC1}
{\sc R.~Stanley}.
\newblock {\em {Enumerative Combinatorics}\/}, vol.~1.
\newblock Cambridge University Press, Cambridge (1997).

\bibitem{StanEC2}
{\sc R.~Stanley}.
\newblock {\em {Enumerative Combinatorics}\/}, vol.~2.
\newblock Cambridge University Press, Cambridge (1999).

\bibitem{StanPos}
{\sc R.~Stanley}.
\newblock Positivity problems and conjectures.
\newblock In {\em Mathematics: Frontiers and Perspectives\/} ({\sc V.~Arnold,
  M.~Atiyah, P.~Lax, and B.~Mazur}, eds.). American Mathematical Society,
  Providence, RI (2000), pp. 295--319.

\bibitem{StanStemIJT}
{\sc R.~Stanley and J.~Stembridge}.
\newblock {On immanants of Jacobi--Trudi matrices and permutations with
  restricted positions}.
\newblock {\em J. Combin. Theory Ser. A\/}, {\bf 62} (1993) pp. 261--279.

\bibitem{StemOrdRep}
{\sc J.~Stembridge}.
\newblock Ordinary representations of {$B_n$}.
\newblock Unpublished manuscript.

\bibitem{StemConj}
{\sc J.~Stembridge}.
\newblock Some conjectures for immanants.
\newblock {\em Canad.\ J.\ Math.\/}, {\bf 44}, 5 (1992) pp. 1079--1099.

\bibitem{TamvakisAmenable}
{\sc H.~Tamvakis}.
\newblock Degeneracy locus formulas for amenable {Weyl} group elements (2019).
\newblock Preprint {\tt math.AG:1909.06398} on {ArXiv}.

\bibitem{Trott}
{\sc W.~T. Trotter}.
\newblock {\em Combinatorics and Partially Ordered Sets: Dimension Theory\/}.
\newblock Johns Hopkins University Press, Baltimore (1992).

\bibitem{Ty1}
{\sc J.~Tymoczko}.
\newblock Linear conditions imposed on flag varieties.
\newblock {\em Amer. J. Math.\/}, {\bf 128}, 6 (2006) pp. 1587--1604.

\bibitem{Ty3}
{\sc J.~Tymoczko}.
\newblock Permutation actions on equivariant cohomology of flag varieties.
\newblock In {\em Toric topology\/}, vol. 460 of {\em Contemp. Math\/}. Amer.
  Math. Soc., Providence, RI (2008), pp. 365--384.

\bibitem{Ty2}
{\sc J.~Tymoczko}.
\newblock Permutation representations on {S}chubert varieties.
\newblock {\em Amer. J. Math.\/}, {\bf 130}, 5 (2008) pp. 1171--1194.

\bibitem{ZaslavskySGC}
{\sc T.~Zaslavsky}.
\newblock Signed graph coloring.
\newblock {\em Discrete Math.\/}, {\bf 39} (1982) pp. 215--228.

\end{thebibliography}
\end{document}